\documentclass[draftcls, onecolumn]{IEEEtran}

\usepackage{pifont}
\usepackage{multirow}
\usepackage{mathtools}
\usepackage{amsmath}
\usepackage{amsthm}
\usepackage{dsfont}
\usepackage{blkarray}
\usepackage{bigstrut}
\usepackage{amsfonts}
\usepackage{mathptmx}
\usepackage{bm}
\usepackage{graphicx}
\usepackage{epstopdf}
\usepackage{subfigure}
\usepackage{soul}
\usepackage{xr}
\usepackage{xcite}
\usepackage{url}
\usepackage{amssymb}
\usepackage{algorithm}
\usepackage{cite}
\usepackage{hyperref}
\usepackage[dvipsnames]{xcolor}
\usepackage[misc,geometry]{ifsym}
\usepackage[noend]{algpseudocode}
\usepackage[normalem]{ulem}

\renewcommand{\qedsymbol}{\ensuremath{\blacksquare}}
\theoremstyle{plain}
\newtheorem{theorem}{\bf Theorem}
\newtheorem{remark}{\bf Remark}
\newtheorem{proposition}{\bf Proposition}
\newtheorem{lem}{\bf Lemma}

\definecolor{cardinal}{rgb}{0.77, 0.12, 0.23}
\definecolor{darkspringgreen}{rgb}{0.09, 0.45, 0.27}
\definecolor{princetonorange}{rgb}{1.0, 0.56, 0.0}
\definecolor{plum}{rgb}{0.3,0,0.7}

\def \P {{\mathbf P}}

\def \v {{\mathbf v}}
\def \u {{\mathbf u}}
\def \a {{\mathbf a}}
\def \b {{\mathbf b}}
\def \e {{\mathbf e}}
\def \x {{\mathbf x}}
\def \y {{\mathbf y}}

\def \A {{\mathbf A}}
\def \B {{\mathbf B}}

\def \X {{\mathbf X}}

\def \Z {{\mathbf Z}}

\def \B {{\mathbf B}}

\def \I {{\mathbf I}}

\newcommand{\probP}{\text{I\kern-0.15em P}}
\newcommand{\cmark}{\ding{51}}%
\newcommand{\xmark}{\ding{55}}
\newcommand{\norm}[1]{\ensuremath{\left\|#1\right\|}}

\DeclarePairedDelimiter\floor{\lfloor}{\rfloor}

\DeclareMathOperator*{\argmin}{arg\,min}
\DeclarePairedDelimiter{\abs}{\lvert}{\rvert}

\begin{document}

\title{Boundary Conditions for Linear Exit Time Gradient Trajectories Around Saddle Points: Analysis and Algorithm}

\author{Rishabh Dixit, Mert G\"urb\"uzbalaban, and Waheed U.\ Bajwa
%
\thanks{R.\ Dixit (Department of Electrical and Computer Engineering), M.\ G\"urb\"uzbalaban (Departments of Management Science and Information Systems and Electrical \& Computer Engineering), and W.\ U.\ Bajwa (Departments of Electrical \& Computer Engineering and Statistics) are at Rutgers University--New Brunswick, NJ 08854 (Emails: {\tt \{rishabh.dixit,~mg1366,~waheed.bajwa\}@rutgers.edu}).}%
%
\thanks{This work was supported in part by the National Science Foundation under grants CCF-1453073, CCF-1907658, CCF-1814888, DMS-2053485, and CCF-1910110, by the Army Research Office under grants W911NF-17-1-0546 and W911NF-21-1-0301, by the Office of Naval Research under grant N00014-21-1-2244 and by the DARPA Lagrange Program under ONR/SPAWAR contract N660011824020.}}

\maketitle

\begin{abstract}
Gradient-related first-order methods have become the workhorse of large-scale numerical optimization problems. Many of these problems involve nonconvex objective functions with multiple saddle points, which necessitates an understanding of the behavior of discrete trajectories of first-order methods within the geometrical landscape of these functions. This paper concerns convergence of first-order discrete methods to a local minimum of nonconvex optimization problems that comprise strict-saddle points within the geometrical landscape. To this end, it focuses on analysis of discrete gradient trajectories around saddle neighborhoods, derives sufficient conditions under which these trajectories can escape strict-saddle neighborhoods in linear time, explores the contractive and expansive dynamics of these trajectories in neighborhoods of strict-saddle points that are characterized by gradients of moderate magnitude, characterizes the non-curving nature of these trajectories, and highlights the inability of these trajectories to re-enter the neighborhoods around strict-saddle points after exiting them. Based on these insights and analyses, the paper then proposes a simple variant of the vanilla gradient descent algorithm, termed Curvature Conditioned Regularized Gradient Descent (CCRGD) algorithm, which utilizes a check for an initial boundary condition to ensure its trajectories can escape strict-saddle neighborhoods in linear time. Convergence analysis of the CCRGD algorithm, which includes its rate of convergence to a local minimum, is also presented in the paper. Numerical experiments are then provided on a test function as well as a low-rank matrix factorization problem to evaluate the efficacy of the proposed algorithm.
\end{abstract}

	
\section{Introduction}
The gradient descent method and its (stochastic) variants have been at the forefront of nonconvex optimization for nearly a decade. Many of these variants stem from the earliest works like \cite{curry1944method, hestenes1952methods, rosenbrock1960automatic}, the interior-point method \cite{karmarkar1984new, mehrotra1992implementation, nesterov1994interior}, and their stochastic counterparts. But the highly complicated geometrical landscape of many nonconvex functions often puts the efficacy of these algorithms to question, which otherwise have robust performance in convex settings. Indeed, problems involving matrix factorization \cite{lee2001algorithms}, neural networks \cite{sanger1989optimal}, rank minimization \cite{broyden1970convergence}, etc., can be highly nonconvex, wherein the function geometry can possess many saddle points that create regions of very small magnitude gradients, something which the gradient-related methods rely upon heavily. As a consequence, travel times for trajectories generated by these methods in such regions could be exponentially large, thereby defeating the purpose of optimization. However, the large travel times around saddle points for gradient-based methods is not always the case; see, e.g., \cite{dixit2022exit} that gives a linear exit-time bound for first-order approximations of gradient trajectories provided some necessary boundary conditions are satisfied by the trajectories. Such analysis suggests existence of gradient-based methods capable of `fast' traversal of geometrical landscapes of nonconvex functions under appropriate conditions. Development of such methods, however, necessitates a deeper geometric analysis of the saddle neighborhoods so as to leverage any initial boundary conditions required by the faster gradient trajectories around saddle points in order to reduce the total travel time on the entire function landscape.

To this end, we \emph{first} study in this paper the problem of developing sufficient boundary conditions for gradient trajectories around any saddle point $\x^*$ of some nonconvex function $f(\x)$ that can guarantee linear exit time, i.e., $K_{exit} = \mathcal{O}(\log(\epsilon^{-1}))$, from the open saddle neighborhood $\mathcal{B}_{\epsilon}(\x^*)$. This problem focuses on a closed neighborhood $\mathcal{\bar{B}}_{\epsilon}(\x^*)$ around the saddle point $\x^*$, with the current iterate $\x_0$ sitting on the boundary of this neighborhood, i.e., $\x_0  \in \mathcal{\bar{B}}_{\epsilon}(\x^*) \backslash \mathcal{B}_{\epsilon}(\x^*)$. Suppose also that the gradient trajectory starting at $\x_0$ has approximately linear exit time from this region $\mathcal{{B}}_{\epsilon}(\x^*)$. (Existence of such trajectories is guaranteed because of the analysis in \cite{dixit2022exit}.) Then, the question posed here is what are the sufficient conditions on $\x_0$ such that the trajectory can escape $\mathcal{B}_{\epsilon}(\x^*)$ in almost linear time of order $\mathcal{O}(\log({\epsilon^{-1}}))$. Once the sufficient conditions have been derived, we \emph{next} study the question of whether it is possible to get linear rates of travel by the same gradient trajectory in some bigger neighborhood $\mathcal{{B}}_{\xi}(\x^*) \supset \mathcal{{B}}_{\epsilon}(\x^*)$. Note that unlike the matrix perturbation-based analysis in \cite{dixit2022exit}, the radius $\xi$ of the bigger neighborhood needs to be characterized by a fundamentally different proof technique. This is since the eigenspace of the Hessian $\nabla^2 f(\x)$ for any $\x \in \mathcal{{B}}_{\xi}(\x^*) \backslash \mathcal{{B}}_{\epsilon}(\x^*)$ cannot be obtained by perturbing the eigenspace of $\nabla^2 f(\x^*)$ since the series expansion of $\nabla^2 f(\x)$ about $\nabla^2 f(\x^*)$ may not necessarily converge from matrix perturbation theory. \emph{Third}, after such linear rates have been obtained, we then study whether it is possible to develop a robust algorithm that leverages the boundary conditions so as to steer the gradient trajectory away from $ \mathcal{{B}}_{\epsilon}(\x^*)$ in almost linear time. \emph{Finally}, we seek an answer to the question of whether the developed algorithm converges to a neighborhood of a local minimum and, if so, what would be its rate of convergence within the global landscape of the nonconvex function.

To address all these problems effectively, we engage in a rigorous analysis of trajectories of the vanilla gradient descent method, starting off directly where we left in \cite{dixit2022exit}.\footnote{Since this work is a continuation of \cite{dixit2022exit}, we refrain from elaborating certain terminologies and definitions that were covered in detail in \cite{dixit2022exit}, though a summary of all the required concepts is provided in Sec.~\ref{ssec:preface} to make this a self-contained paper.} First, we utilize tools from the matrix perturbation theory to develop sufficient conditions on $\x_0 \in \mathcal{\bar{B}}_{\epsilon}(\x^*) \backslash \mathcal{B}_{\epsilon}(\x^*)$ for which the subsequent gradient trajectory has linear exit time from $\mathcal{B}_{\epsilon}(\x^*)$. Next, we prove a rather intuitive yet extremely powerful result, termed the \emph{sequential monotonicity of gradient trajectories}, which establishes that the gradient trajectories in a neighborhood of the saddle point first exhibit contractive dynamics up to some point and there onward strictly expansive dynamics. Next, we provide an analysis of the travel time for the gradient trajectory in the region $\mathcal{{B}}_{\xi}(\x^*) \backslash \mathcal{{B}}_{\epsilon}(\x^*)$ using the sequential monotonicity result. Finally, we develop a novel gradient-based algorithm, termed Curvature Conditioned Regularized Gradient Descent (CCRGD), around the idea of sufficient boundary conditions with a robust check condition guaranteeing almost linear exit time from $\mathcal{{B}}_{\epsilon}(\x^*)$. In doing so, we also prove certain qualitative lemmas about the local behavior of gradient trajectories around saddle points. Thereafter, the asymptotic convergence and the rate of convergence for CCRGD to a local minimum is proved using these lemmas. Finally, the performance of CCRGD is evaluated on two problems: a test function for nonconvex optimization and a low-rank matrix factorization problem.

\subsection{Relation to Prior Work}
Since this work directly extends the results in \cite{dixit2022exit}, we steer away from repeating the discussion in \cite[Sec.~1.1]{dixit2022exit} in relation to existing convergence guarantees for gradient-related methods in nonconvex settings. Instead, we primarily focus in this section on presenting comparisons and highlighting key differences between our contributions and the existing literature. In addition, given the vast interest of the optimization community in nonconvex optimization using gradient-related methods, we also discuss some additional relevant works in here.

{Similar to \cite{lee2017first}, which focuses on the gradient descent method, we prove in Theorem \ref{thmglobmain} that the trajectories generated by the proposed CCRGD algorithm (see Algorithm~\ref{algo1}) converge to a local minimum. But unlike \cite{lee2017first}, which fundamentally uses the Stable Manifold Theorem \cite{kelley1966stable}, we also develop in this paper a proof of convergence of CCRGD to a local minimum and obtain algorithmic convergence rates using the geometry of function landscape near saddle points and in regions that have {sufficiently large} gradient magnitudes.} Though this idea of rate analysis has been well summarized in \cite{polyak1964some} for gradient-related sequences and more recently in \cite{nesterov2006cubic} for Newton-type methods, yet these works do not utilize the nonconvex geometry to its fullest extent. Specifically, we categorize the function geometry in our work into `\emph{regions near}' and `\emph{regions away}' from the stationary points so as to better analyze `\emph{escape conditions}' from saddle neighborhoods and at the same time generate convergence guarantees to a local minimum.  Within the regions of `moderate gradients' around saddle points, i.e., the \emph{shell} $\mathcal{B}_{\xi}(\x^*)\backslash \mathcal{B}_{\epsilon}(\x^*)$, we show using the sequential monotonicity property (detailed in Theorem \ref{thm2}) that the sequence $\{\norm{\x_k-\x^*}\}$ is {strictly} monotonic whenever the iterate $\{\x_k\}$ has expansive dynamics with respect to $\x^*$, while the function value sequence $\{f(\x_k)\}$ satisfies the Polyak--{\L}ojasiewicz (PL) condition \cite{karimi2016linear} whenever the iterate sequence $\{\x_k\}$ has contractive dynamics with respect to $\x^*$ (see Lemma~\ref{polyaklem}). Consequently, linear rates of contraction to a point on the boundary $\mathcal{\bar{B}}_{\epsilon}(\x^*)\backslash \mathcal{B}_{\epsilon}(\x^*)$ are derived using the PL condition and linear rates of expansion to a point on the boundary  $\mathcal{\bar{B}}_{\xi}(\x^*)\backslash \mathcal{B}_{\xi}(\x^*)$ are obtained using the sequential monotonicity property from Theorem \ref{thm2}, both of which aid in our convergence analysis. Note that the PL condition cannot be applied directly around a saddle point since that would yield a trivial lower bound of $0$ on the gradient norm (see Lemma \ref{polyaklem}). This particular analytical approach of separately analyzing the contractive and expansive dynamics locally around a saddle point and exploiting the PL condition restricted to contractive dynamics is in contrast to the existing works {that focus on the problem} of escaping saddle points for nonconvex optimization. In {addition, while} the PL condition or the more general Kurdyka--{\L}ojasiewicz property \cite{Lojasiewicz1959probleme} are often used for local or even global analysis such as in \cite{attouch2013convergence} and \cite{bolte2007Lojasiewicz}{, they have not} been used in the context of analyzing local contractive dynamics of iterates w.r.t. a strict saddle point. 
In terms of the analytical tools used, regions near the saddle points in this work are analysed using the matrix perturbation theory, yielding {sharp} bounds ('sharp' in terms of the condition number, problem dimension, and spectral gap) on the initial conditions, whereas regions away from the saddle points utilize properties like the sequential monotonicity (cf.~Theorem~\ref{thm2}). {Such local analysis distinguishing sufficiently small saddle neighborhoods from moderately small saddle neighborhoods} seems to be quite novel and has not been carried out in any previous work to our knowledge.

Next, to the best of our knowledge, no other work has provided sufficient boundary conditions for escape from saddle neighborhoods for  the case of \emph{discrete-time} gradient descent-related algorithms. Though the idea is not necessarily new and has been explored while dealing with continuous-time dynamical systems, specifically the boundary value problems, yet it is still nascent when it comes to analyzing saddle points. The continuous-time works such as \cite{kifer1981exit, hu2017fast, bolte2007Lojasiewicz} have been discussed in detail in \cite{dixit2022exit}. However even these works do not analyze the boundary conditions for continuous trajectories. The work \cite{hu2017fast} does take into account cascaded saddles encountered by continuous trajectories, which gets a detailed treatment in our work in Theorems \ref{thm4} and \ref{thm5} for discrete trajectories.

The Stochastic Differential Equation (SDE) setup has also been utilized in a recent work~\cite{ShiSuEtAl.arxiv20} to study gradient-based (stochastic) methods for nonconvex optimization in the continuous-time setting. Interestingly, this work considers the set of \emph{index-$1$ saddle points} in the function's geometry and thereby obtains a stochastic rate of convergence to a global minimum, where the rate is of the order `a constant term plus a geometric term'. While the rate is linear/geometric, \cite{ShiSuEtAl.arxiv20} assumes the \emph{coercivity condition} (sufficient growth condition on the function away from the origin) and the \emph{Villani condition} (growth of gradient's norm), whereas only the former condition of coercivity is assumed in our work. Also, the constant in the non-geometric term of the rate is dependent on the horizon $T$ obtained from discretization of the SDE, which could be large. Moreover, it is not clear how the SDE approach in \cite{ShiSuEtAl.arxiv20} would apply to the discrete-time setting of this paper.

Recently, within the class of discrete-time non-acceleration-based methods, \cite{du2017gradient, jin2017escape} provide the rates for escaping saddles using perturbed gradient descent, \cite{zhou2017stochastic} utilizes the notion of variational coherence between stochastic mirror gradient and descent direction in quasi convex and nonconvex problems for obtaining ergodic rates of convergence to a local/global minimum (under certain conditions), and \cite{daneshmand2018escaping} provides rates and escape guarantees under certain strong assumptions of high correlation between the negative curvature direction and a random perturbation vector. However, none of these stochastic variants explore the idea of initial boundary conditions near saddle points so as to obtain linear rates. It should be noted that the work in \cite{du2017gradient} shows the time to escape cascaded saddles scales {exponentially} with dimension, whereas we show in Theorem \ref{thm5} that the time to escape cascaded saddles is not exponential in dimension. Rather, the number of cascaded saddles encountered by the trajectory is upper bounded and this bound scales only linearly with the inverse of the gradient norms in regions away from the stationary points of the objective. Further, this upper bound on the number of saddles encountered is independent of the problem dimension.

The next set of related discrete-time gradient-based methods includes first-order methods leveraging acceleration and momentum techniques. For instance, the work in \cite{reddi2017generic} provides an extension of SGD to methods like the Stochastic Variance Reduced Gradient (SVRG) algorithm for escaping saddles. Recently, methods approximating the second-order information of the function that preserve the first-order nature of the algorithm have also been employed to escape the saddles. Examples include \cite{jin2017accelerated}, where the authors prove that an acceleration step in gradient descent guarantees escape from saddle points, and the method in \cite{xu2018first}, which utilizes the second-order nature of the acceleration step combined with a stochastic perturbation to guarantee escape rates. Moreover, both \cite{allen2018natasha, allen2018neon2} build on the idea of utilizing acceleration as a source of finding the negative curvature direction. Due to the low computational cost of evaluating gradients, we also make use of such connections between the curvature magnitude and the gradient difference in our proposed algorithm (Algorithm \ref{algo1}). In the class of first-order algorithms, there also exist trust region-based methods. The work in \cite{fang2019sharp} is one such method that presents a novel stopping criterion with a heavy ball controlled mechanism for escaping saddles using the SGD method. If the SGD iterate escapes some neighborhood in a certain number of iterations, the algorithm is restarted with the next round of SGD, else the ergodic average of the iterate sequence is designated to be a second-order stationary solution. In a similar vein, we formally derive in Lemma \ref{lemma5} the escape guarantees from a neighborhood around a saddle point and utilize that result within the proposed Algorithm \ref{algo1}.

Lastly, higher-order methods are discussed in \cite{paternain2019newton, mokhtari2018escaping}, which utilize either Hessian-based {approaches} 
or a second-order step combined with first-order algorithms so as to reach local minimum with fast speed while trading off with computational costs. Going a step even further, the work in \cite{anandkumar2016efficient} poses the escape problem with second-order saddles, thereby {motivating the use of higher-order methods}. 
Though these techniques optimize well over certain pathological functions like those having `degenerate' saddles or very ill-conditioned geometries, yet they suffer heavily in terms of complexity; e.g., the work \cite{anandkumar2016efficient} requires third-order methods to solve for a feasible descent direction. This further motivates us to develop a hybrid algorithm for the saddle escape problem that captures the advantages of a Hessian-based method and at the same time is low on computational complexity.

Table \ref{table:2} draws comparisons between our work and other existing works within the realm of {saddle escape in deterministic nonconvex optimization problems}. Though there {is a} plethora of works that {study the} saddle escape {problem}, only those works are listed here that address the simple unconstrained optimization problem of minimizing a smooth nonconvex function $f(\cdot)$ and propose {perturbation of} deterministic gradient-based methods {for saddle escape}. Many of the other related works discussed in this section tackle stochastic optimization problems and are therefore not included in the table.

\begin{table}[ht]
\centering
\renewcommand\thempfootnote{\arabic{mpfootnote}}
\begin{minipage}{\textwidth}
\caption{\small Summary of the similarities and differences between this work and some related prior works.}
\label{table:2}
\resizebox{1\columnwidth}{!}{%
\begin{tabular}{||c c c c c c||}
 \hline
 \multirow{2}*{\textbf{References}} & \multirow{2}*{\textbf{Method of saddle escape} } &
 \multirow{2}*{\textbf{Base algorithm}} & \textbf{Explicit dependence on}  &
 \multirow{2}*{\textbf{Convergence rate}} & \multirow{2}*{\textbf{Type of convergence rate}}
 \\
 & & & \textbf{number of saddles}  & &
 \\ [0.5ex]
 \hline\hline
 \cite{jin2017escape} & One-step noise & Gradient descent method & \xmark  &  $\mathcal{O}\bigg(\frac{1}{\epsilon^2} \log^4\bigg(\frac{1}{\epsilon^2}\bigg) \bigg)$ & probabilistic
   \\
 \hline
 \cite{jin2017accelerated} & One-step noise with & Accelerated gradient method & \xmark & $\mathcal{O}\bigg(\frac{1}{\epsilon^{7/4}} \log^6\bigg(\frac{1}{\epsilon}\bigg) \bigg)$  & probabilistic
   \\
   & negative curvature search &  & &  &\\
 \hline
 \cite{paternain2019newton} &  One-step noise with &  Second-order Newton method & \cmark & $\mathcal{O}\bigg( T\log\bigg(\frac{1}{\epsilon}\bigg) + T\log \log\bigg(\frac{1}{\epsilon}\bigg) \bigg)$; &  probabilistic  \\ & negative curvature search &  & & $T$ is the number of saddles encountered &\\
 \hline
  \cite{carmon2018accelerated} &  Multi-step noise with &  Accelerated gradient method & \xmark & $\mathcal{O}\bigg(\frac{1}{\epsilon^{7/4}} \log\bigg(\frac{1}{\epsilon}\bigg) \bigg)$ & probabilistic \\ & negative curvature search & & &  & \\
 \hline
  \cite{liu2018adaptive} &  Multi-step noise with &  Adaptive negative curvature descent & \xmark & $\mathcal{O}\bigg(\frac{1}{\epsilon^{2}}  \bigg)$ & probabilistic \\ & negative curvature search & & &  & \\
   \hline
  \cite{zhang2021escape} &  One-step noise followed by  &  Accelerated gradient method & \xmark & $\mathcal{O}\bigg(\frac{1}{\epsilon^{7/4}} \log\bigg( \frac{1}{\epsilon}\bigg)  \bigg)$ & probabilistic \\ & multi-step negative curvature search & & &  & \\
 \hline
\textbf{This work} &  One second-order step \emph{only} &  Gradient descent method & \cmark & $\mathcal{O}\bigg(T\log\bigg(\frac{1}{\epsilon}\bigg) \bigg)  + \mathcal{O} \bigg(T\log \bigg(\frac{\xi}{\epsilon}\bigg)\bigg)  + \mathcal{O}\bigg(\frac{1}{\epsilon^{2{\upsilon}}}\bigg)$;  & deterministic \\
 & when curvature condition fails & &  &  for locally analytic, coercive Morse functions;  &    \\
   & & &  &   $T = \mathcal{O}\bigg(\frac{1}{\epsilon^{\upsilon}}\bigg)$ is the number of saddles and \footnote{The parameter $\upsilon$ is defined in Proposition \ref{proposition3} and it controls the function geometry in regions away from its critical points.}$\upsilon \in [0,1)$  &    \\
 \hline
\end{tabular}%
}
\end{minipage}
\end{table}
	
\subsection{Our Contributions}

This work starts off directly from the point where we left off in \cite{dixit2022exit}, where we obtained exit time bounds for $\epsilon$-precision gradient descent trajectories around saddle points and derived a necessary condition on the initial unstable subspace projection value for linear exit time. The first novel result in this work is the development of a bound on the initial unstable subspace projection value in Theorem~\ref{thm1} that approximately guarantees the linear exit time bound from \cite[Theorem 3.2]{dixit2022exit}. Our second contribution is Theorem~\ref{thm2}, in which we analyze the behavior of gradient descent trajectories in some region $\mathcal{B}_{\xi}(\x^*) \supset \mathcal{B}_{\epsilon}(\x^*)$ where the approximate analysis from matrix perturbation theory may not necessarily hold. In such augmented neighborhood of the strict saddle point $\x^*$, we prove that the gradient descent trajectories have a sequential monotonic behavior, i.e., there exists some $\xi$ such that the trajectory inside $\mathcal{B}_{\xi}(\x^*) $  first exhibits contractive dynamics moving towards $\x^*$ and then has expansive dynamics for the remainder of the time as long as it stays inside $\mathcal{B}_{\xi}(\x^*) $. Though this property may appear to be trivial for trajectories around saddle points, yet it is extremely important in developing improved rates/travel times of the gradient descent trajectories inside $\mathcal{B}_{\xi}(\x^*)$, which follows from our next contribution. Our third contribution is Theorem~\ref{thm3}, in which we obtain upper bounds on the travel time of gradient trajectory inside the shell $\mathcal{B}_{\xi}(\x^*) \backslash \mathcal{B}_{\epsilon}(\x^*)$ that we denote by $K_{shell}$. This particular region is specifically of great importance since we can categorize it as a region of ``moderate'' gradients (gradient magnitude not too small) that still inherits certain geometric properties such as the minimum curvature from the smaller saddle neighborhood $\mathcal{B}_{\epsilon}(\x^*)$. Without taking such properties into consideration, the journey time in this shell could only be naively upper bounded as $K_{shell} = \mathcal{O}(\epsilon^{-2})$ using the gradient Lipschitz condition. Hence, it is imperative to separately analyze the journey time inside the shell $\mathcal{B}_{\xi}(\x^*) \backslash \mathcal{B}_{\epsilon}(\x^*)$ so as to improve upon the standard nonconvex rate of $\mathcal{O}(\epsilon^{-2})$.

Our next set of contributions corresponds to Lemmas~\ref{lemma1}--\ref{lemma5}, in which we provide insights into certain qualitative properties of the gradient descent trajectories around saddle points. Lemma~\ref{lemma1} talks about the approximate hyperbolic nature of the gradient trajectories near saddle points, while Lemma~\ref{lemma2} proves that trajectories with linear exit time approximately never curve around saddle points. Lemma~\ref{lemma3} shows that the gradient trajectory can only exit $\mathcal{B}_{\epsilon}(\x^*) $ at those points where the function value is strictly less than $f(\x^*)$. Lemma~\ref{lemma4} establishes that the gradient trajectory, once it exits the neighborhood $\mathcal{B}_{\epsilon}(\x^*)$, can never re-enter it, while Lemma \ref{lemma5} extends the same result to the bigger neighborhood $\mathcal{B}_{\xi}(\x^*)$ under certain stricter conditions. Our next contribution 
is the development of the Curvature Conditioned Regularized Gradient Descent (CCRGD) algorithm (cf.~Algorithm~\ref{algo1}) that provably escapes saddle neighborhoods and gives second-order stationary solutions. The asymptotic convergence of the proposed algorithm is established from Theorem \ref{thmglobmain}, which is proved using Lemmas \ref{lem6}, \ref{lem7}, \ref{lem8} and the Global Convergence Theorem (Theorem \ref{thmglob}) from \cite{luenberger1984linear}. The algorithm checks for a curvature condition near the saddle neighborhood and makes the decision of whether to perform a second-order iteration for one step or continue using the vanilla gradient descent method. The curvature condition (Step~\ref{algocurvaturecondition} in Algorithm~\ref{algo1}) is derived from our proof of convergence of the algorithm; in addition, Algorithm~\ref{algo1} is tested for its efficacy on a modified Rastrigin function (a test function for nonconvex optimization) and the matrix factorization problem as part of numerical experiments. Last, but not the least, the final contribution of this work is derivation of the rate of convergence of an iterate sequence generated from Algorithm~\ref{algo1} to a local minimum. The rates are obtained for a more general setting of cascaded saddles where the number of saddles encountered and the total time of convergence are bounded from Theorems~\ref{thm4} and~\ref{thm5}, {respectively}.

\subsection{Notations}
All vectors in the paper are in bold lower-case letters, all matrices are in bold upper-case letters, $\mathbf{0}$ is the $n$-dimensional null vector, $\mathbf{I}$ represents the $n \times n$ identity matrix, and $\langle\cdot, \cdot\rangle $ represents the inner product of two vectors. In addition, unless otherwise stated, all vector norms $\norm{\cdot}$ are $\ell_2$ norms, while the matrix norm $\|\cdot\|_2$ denotes the operator norm. Further, the symbol $ (\cdot)^T$ is the transpose operator, the symbol $\mathcal{O}$ represents the Big-O notation and sometimes we use $a \ll b \iff a = \mathcal{O}(b)$, the symbol $\Omega$ is the Big-Omega notation and $\Theta$ represents the Big-Theta notation, $\otimes$ represents the kronecker product, i.o. means infinitely often,  $\mathrm{id}$ represents the identity map, and $W(\cdot)$ is the Lambert $W$ function~\cite{corless1996lambertw}. Throughout the paper, $k$ and $K$ are used for the discrete time. {Next, $\gtrapprox$ and $\lessapprox$ represent the `approximately greater than' and `approximately less than' symbols, respectively, where $a \lessapprox b $ implies $a \leq  b + g(\epsilon) $
 and $a \gtrapprox b $ implies $a + g(\epsilon)\geq  b  $ for some absolutely continuous function $g(\cdot)$ of $\epsilon$ where $g(\cdot) \geq 0$ and $g(\epsilon) \to 0$ as $\epsilon \to 0$.} Also, for any matrix expressed as $\Z+\mathcal{O}(c)$ with $c$ being a scalar, the matrix-valued perturbation term $\mathcal{O}(c)$ is with respect to the Frobenius norm. Finally, the operator $\mathbf{dist}(\cdot,\cdot)$ gives the distance between two sets whereas $\mathbf{diam}(\cdot)$ gives the diameter of a set.
	
\section{Problem Formulation}
Consider a nonconvex smooth function $f(\cdot)$ that has strict first-order saddle points in its geometry. By strict first-order saddle points, we mean that the Hessian of function $f(\cdot)$ at these points has at least one negative eigenvalue, i.e., the function has negative curvature. Next, consider some (open) neighborhood $\mathcal{B}_{\epsilon}(\x^*)$ around a given saddle point $\x^*$, where the neighborhood radius $\epsilon$ is bounded above by $\Theta({L}{M^{-1}})$ (see \cite[Theorem~3.2]{dixit2022exit} for the exact form) with $L$ and $M$ being the gradient and Hessian Lipschitz constants of $f(\cdot)$. Also, it is given that the initial iterate $\x_0$ of the gradient trajectory sits on the boundary of the neighborhood, i.e.,  $\x_{0} \in \mathcal{\bar{B}}_{\epsilon}(\x^*) \backslash \mathcal{B}_{\epsilon}(\x^*)$, and the gradient trajectory exits $\mathcal{B}_{\epsilon}(\x^*)$ in linear time bounded by \cite[Theorem~3.2]{dixit2022exit}. With this information, we are first interested in finding the sufficient conditions on $\x_0$ that guarantee the linear exit time. In addition, we need to analyze the gradient trajectories in some larger neighborhood $\mathcal{B}_{\xi}(\x^*) \supset \mathcal{B}_{\epsilon}(\x^*)$ such that the trajectories first contract towards the saddle point and then expand away from it. More importantly, we are interested in finding such $\xi > \epsilon$ for which the gradient trajectory has linear travel time in the shell  $\mathcal{B}_{\xi}(\x^*) \backslash \mathcal{B}_{\epsilon}(\x^*)$. Next, we are required to find certain local properties of $f(\cdot)$ for which the gradient trajectories, having escaped it once, can never re-enter the neighborhood $\mathcal{B}_{\xi}(\x^*) $.  Finally, we have to develop a robust low-complexity algorithm that utilizes the sufficient conditions to traverse the {landscape} of saddle neighborhoods in linear time and also provide its rate of convergence to some local minimum.

Having briefly stated the problem, we now formally state the set of assumptions that are required for this problem to be tackled in this work.
	\subsection{Assumptions}
		\begin{itemize}
	\item[] \textbf{A1.} \textit{The function $f:\mathbb{R}^n \to \mathbb{R}$ is coercive, i.e., $ \lim_{\norm{\x}\to \infty} f(\x) = \infty$, is globally $\mathcal{C}^2$, i.e., twice continuously differentiable, and locally $\mathcal{C}^{\omega}$ in sufficiently large neighborhoods of its saddle points, i.e., all the derivatives of this function are continuous around saddle points and the function $f(\cdot)$ also admits Taylor series expansion in these neighborhoods.\footnote{{By sufficiently large neighborhoods, we mean that the diameter of such neighborhoods is $\Omega(1)$.}}
	}
		\item[] \textbf{A2.} \textit{The gradient of function $f(\cdot)$ is $L-$Lipschitz continuous:} $ \norm{\nabla f(\x) - \nabla f(\y)} \leq L \norm{\x - \y}$.
		\item[] \textbf{A3.} \textit{The Hessian of function $f(\cdot)$ is $M-$Lipschitz continuous:} $\norm{\nabla^{2} f(\x) - \nabla^{2} f(\y)}_2 \leq M \norm{\x - \y}$.		
		\item[] \textbf{A4.} \textit{The function $f(\cdot)$ has only well-conditioned first-order stationary points, i.e., no eigenvalue of the function's Hessian is close to zero {around} these points. Formally, if $\x^{*}$ is the first-order stationary point for $f(\cdot)$, then}
		\begin{align}
		\nabla f(\x^{*}) &= \mathbf{0},  \ \text{and}\nonumber \\
		\min_i \abs{\lambda_i(\nabla^{2}f(\x^*))} &> \beta, \nonumber
		\end{align}
		\textit{where $\lambda_i(\nabla^{2}f(\x^*))$ denotes the $i^{th}$ eigenvalue of the matrix $ \nabla^{2}f(\x^*)$ and $\beta > 0$. Note that such a function is termed a Morse function. Also, there exists {an open} neighborhood $ \mathcal{W}$ of $\x^*$ such that}
		\begin{align}
		\forall \x \in \mathcal{W}, \ \min_i \abs{\lambda_i(\nabla^{2}f(\x))} &> \beta. \nonumber
		\end{align}
	\end{itemize}
	
\begin{remark}
The coercivity of $f(\cdot)$ is only required from Section \ref{algosection} onward, where we prove the convergence of Algorithm \ref{algo_1}. Also, Section \ref{mainsec2} requires $f(\cdot)$ to be only $\mathcal{C}^2$ {Hessian-Lipschitz} Morse function, unlike Section \ref{mainsec1} in which the additional assumption of local analyticity is required around saddle points.
\end{remark}

Note that Assumption \textbf{A1} may seem too restrictive since it requires $f(\cdot)$ to be locally real analytic, while the theory of nonconvex optimization is often developed around only the assumption that $f \in \mathcal{C}^2$ with Lipschitz-continuous Hessian. It is worth reminding the reader, however, that many practical nonconvex problems such as quadratic programs, low-rank matrix completion, phase retrieval, etc., with appropriate smooth regularizers satisfy this assumption of real analyticity around the saddle neighborhoods; see, e.g., the formulations discussed in \cite{ma2020implicit}.  Similarly, 
{many} of the loss functions in nonconvex optimization are coercive, i.e., they grow arbitrarily large asymptotically due to the presence of some form of regularization. As for the other assumptions, gradient Lipschitz continuity (Assumption \textbf{A2}) and Hessian Lipschitz continuity (Assumption \textbf{A3}) are {invoked routinely in} the nonconvex optimization literature, while Assumption \textbf{A4} implies $f(\cdot)$ is a Morse function. In particular, since Morse functions are dense in the class of $\mathcal{C}^2$ functions \cite{matsumoto2002introduction}, we are not giving up much by {making this assumption}. We now state two propositions that {follow from our assumptions and that} will be routinely used in our analysis. 

\begin{proposition}\label{morseprop}
{Under Assumption~\textbf{A4}, the function $f(\cdot)$ has only first-order saddle points in its geometry. Moreover, these first-order saddle points are strict saddle, i.e., for any first-order saddle point $\x^*$, there exists at least one eigenvalue $\lambda_i$ of $\nabla^2 f (\x^*)$ that satisfies $\lambda_i(\nabla^{2}f(\x^*)) < - \beta$.}
\end{proposition}
\begin{proof}
For any $ \mathcal{C}^m$-smooth function $f(\cdot)$ with $m \geq 2$, if $\x^*$ is its second- or higher-order saddle point then it must necessarily satisfy $ \nabla f(\x^*) = \mathbf{0}$ and $\nabla^2 f(\x^*) \succeq \mathbf{0}$, where at least one of the eigenvalues of $ \nabla^2 f(\x^*)$ is $0$. {But this is not possible in our case because of Assumption~\textbf{A4}.}
\end{proof}

\begin{proposition}\label{eigenprop}
Under Assumption \textbf{A4}, for any sufficiently small $\epsilon$ where $\epsilon \ll \beta$, we can group the eigenvalues of the Hessian $ \nabla^{2}f(\x^*)$ at any strict saddle point $\x^*$ into $m$ disjoint sets $\{\mathcal{G}_1, \mathcal{G}_2, \dots, \mathcal{G}_m \}$ with $2\leq m\leq n$ based on the level of degeneracy of eigenvalues (closeness to one another) such that for some $\delta = \Omega(\epsilon^{1-a})$ where $a \in (0,1]$, we have the following conditions:
\begin{align}
			\mathbf{dist}(\mathcal{G}_p,\mathcal{G}_q) &\geq \delta \hspace{0.2cm} \forall \hspace{0.2cm} \mathcal{G}_p,\mathcal{G}_q \hspace{0.2cm} \text{s.t.} \hspace{0.2cm} p \neq q, \ \text{and}
		\\
		\max_{p}\{\mathbf{diam} (\mathcal{G}_p)\}  &= \mathcal{O}(\epsilon^{1-a}).			\end{align}
\end{proposition}

\begin{proof}
From Assumption \textbf{A4}, the eigenvalues of the Hessian $ \nabla^{2}f(\x^*)$ at any strict saddle point $\x^*$ can always be separated into two distinct groups, one consisting of positive eigenvalues and the other comprising negative eigenvalues. By this construction, the distance between these groups will be at least $2 \beta$. Since $\epsilon \ll \beta$, we get a $\delta = 2 \beta$ for this construction which satisfies the constraint $\delta = \Omega(1)$. Next, we check whether the diameter of these two groups is larger than $\Theta(\epsilon^{1-a})$; if yes then we split that particular group into two more groups at the first eigenvalue where the consecutive eigenvalue gap within that group exceeds $\Theta(\epsilon^{1-a})$. This eigenvalue gap becomes our new $\delta$ and by construction it will satisfy the constraint $\delta = \Omega(\epsilon^{1-a})$ for some $a > 0$ since $\delta > \Theta(\epsilon^{1-a})$. Repeating this process recursively, we would have constructed the disjoint sets $\{\mathcal{G}_1, \mathcal{G}_2, ..., \mathcal{G}_m \}$ with $2\leq m\leq n$. Since $n$ is finite, this process will terminate in finite steps (maximum $n-1$ steps) and therefore after the final splitting, we will obtain $\delta = \Omega(\epsilon^{1-a})$ for some $a \in (0,1]$ such that $\max_{p}\{\mathbf{diam} (\mathcal{G}_p)\}  = \mathcal{O}(\epsilon^{1-a}) $.
\end{proof}

Proposition \ref{eigenprop} describes a fundamental property of any $\mathcal{C}^2$ function that arises due to the algebraic multiplicity / (approximate) degeneracy of the eigenvalues of its Hessian at the saddle points. Note that, as a consequence of the strict-saddle property (Assumption~\textbf{A4} / Proposition~\ref{morseprop}) and Proposition~\ref{eigenprop}, we get the following necessary condition:
\begin{align}
	\beta \geq \frac{\delta}{2}.
\end{align}

\section{Boundary Conditions for Linear Exit Time From a Saddle Neighborhood}\label{mainsec1}
\subsection{Preface}\label{ssec:preface}
Given a saddle neighborhood $\mathcal{B}_{\epsilon}(\x^*)$ for some strict saddle point $\x^*$ {and $\varepsilon>0$}, the goal is selecting those gradient trajectories in $\mathcal{B}_{\epsilon}(\x^*)$ for which the exit time is of the order $K_{exit} = \mathcal{O}(\log (\epsilon^{-1}))$, i.e., of linear rate. Formally, the exit time for an iterate sequence $\{\x_k\}$ of some trajectory in the ball $\mathcal{B}_{\epsilon}(\x^*)$ is defined as the smallest positive index $K$ such that  $\norm{\x_K-\x^*} \geq \epsilon $ and we are required to obtain such sequence $\{\x_k\}$ generated by the gradient descent method for which the exit time from the saddle neighborhood $\mathcal{B}_{\epsilon}(\x^*)$ is linear. To conduct such analysis, certain essential concepts and definitions need to be elaborated, most of which were developed in a previous work (for reference see \cite{dixit2022exit}).

First, due to the strict-saddle property, for any $\x$ in an $\epsilon$-neighborhood of $\x^*$, i.e., $\x \in \mathcal{B}_{\epsilon}(\x^*)$, the vector $\x - \x^*$ belongs to a vector space $\mathcal{E} = \mathcal{E}_{S} \bigoplus \mathcal{E}_{US}$, where
{
\begin{align*}
  	\mathcal{E}_{S} &= span\{\v_i | \lambda_{i} > 0\}, \hspace{0.05cm} \mathcal{N}_{S} = \{i | \lambda_{i} > 0\}, \\
  		\mathcal{E}_{US} &= span\{\v_i  | \lambda_{i} < 0\}, \hspace{0.05cm} \mathcal{N}_{US} = \{j | \lambda_{j} < 0\},
\end{align*}
and} $(\lambda_{i}, \v_i)$ are the $i^{th}$ eigenvalue--eigenvector pair of the Hessian $\nabla^2 f(\x^*)$.

Second, using the `degenerate' matrix perturbation theory \cite{MBT, MBT1}, the Hessian $\nabla^2 f(\x)$ at any point $\x = \x^* + p \u$, where $p \in [0,1]$ and $\norm{\u} \leq \epsilon$, can be given as
\begin{align}
      	\nabla^2 f(\x) &= \nabla^2 f(\x^*) + p \norm{\u} \mathbf{H}(\hat{\u}) + \mathcal{O}(\epsilon^2),
\label{perturbation1paper1}      	\end{align}
      	where  $\u := \x - \x^*$ is termed the \textbf{radial vector}, $\hat{\u}= \frac{\u}{\norm{\u}}$ is the unit radial vector and we have that
      	\begin{align}		\mathbf{H}(\hat{\u}) &=  \sum_{i=1}^{n} \bigg(\langle \v_i, \mathbf{H}(\hat{\u})\v_i\rangle \v_i\v_i^{T} +  \lambda_i\sum_{l \not\in \mathcal{G}_i} \frac{\langle \v_l, \mathbf{H}(\hat{\u})\v_i\rangle}{\lambda_{i}-\lambda_{l}}\bigg(\v_l\v_i^{T} + \v_i\v_l^{T}\bigg) \bigg)
      \label{perturbation2paper1}	\end{align}      	
      	with $\mathcal{G}_i = \{\hspace{0.2cm} j \hspace{0.2cm}\vert \hspace{0.2cm} \lambda_{j}=\lambda_{i} \pm   \mathcal{O}(\epsilon)\}$. For details, see Lemma 3.3 from \cite{dixit2022exit}.
      	
      	The third concept can be regarded as the most important tool for developing the proof machinery of linear exit time; see Lemmas 3.4 and 3.5 from \cite{dixit2022exit} for details. Specifically, it can be summarized as the ``Approximation Lemma" for a linear dynamical system. Given some initialization of the radial vector $\u_{0}$ and sufficiently small $\epsilon$, we have for any iteration $K$ that
{ $\u_{K} = \prod_{k=0}^{K-1} \bigg[\A_k + \epsilon \P_k   \bigg]\u_{0}$}, where $\epsilon \P_k = \B_k+\mathcal{O}(\epsilon^2)$, $\B_k = \mathcal{O}(\epsilon)$ for $\x_{k} \in \mathcal{B}_{\epsilon}(\x^*)$, $\{\A_k\}$ and $ \{\B_k\}$ are sequences of real symmetric matrices, and $\A_k$'s are invertible.

When $K\epsilon \ll 1$ and $\epsilon <  \norm{\A^{-1}}_2^{-1}\norm{\P}_2^{-1}$, we have the condition
{
\begin{align*}
\small \norm{\A^{-1}}_2^{-K} \bigg( 1- K\epsilon \frac{\norm{\P}_2}{\norm{\A^{-1}}^{-1}_2} - \mathcal{O}\bigg((K\epsilon)^2\bigg) \bigg) \leq \nu_n \leq \dots \leq \nu_1  \leq  \norm{\A}_2^K \bigg( 1+ K\epsilon \frac{\norm{\P}_2}{\norm{\A}_2} + \mathcal{O}\bigg((K\epsilon)^2\bigg) \bigg),
\end{align*}
where} $\nu_n \leq \dots  \leq \nu_1$ are absolute values of the eigenvalues of matrix $\prod_{k=0}^{K-1} \bigg[\A_k + \epsilon \P_k \bigg]$ and $ \sup_{0\leq k \leq K-1}\norm{\A_k}_2 = \norm{\A}_2$, $ \sup_{0\leq k \leq K-1}\norm{\A_k^{-1}}_2 = \norm{\A^{-1}}_2$, $\sup_{0\leq k \leq K-1}\norm{\P_k}_2 = \norm{\P}_2$ for some matrices $\A$ and $\P$. {Hence, $\u_K = \prod_{k=0}^{K-1} \bigg[\A_k + \epsilon \P_k   \bigg]\u_{0}$ can be expanded to first order in $\epsilon$ with the first-order approximation called $\tilde{\u}_K$ and the trajectory generated by the sequence $\{\tilde{\u}_K\}$ is termed $\epsilon$--precision trajectory}. Thus the gradient update $ \x_{K+1} = \x_K - \alpha \nabla f(\x_K)$ near $\x^*$ can be written as $ \u_{K} = \prod_{k=0}^{K-1} \bigg[\A_k + \epsilon \P_k   \bigg]\u_{0}$ for $\u_K = \x_K- \x^*$, $\A_K = \mathbf{I}-\alpha \nabla^2 f(\x^*)$ and $\epsilon \P_K = -\frac{\alpha\norm{\u_K}}{2}\mathbf{H}(\hat{\u}_K) + \mathcal{O}(\epsilon^2)$.

Fourth, from Lemma 3.6 of \cite{dixit2022exit}, the `minimal' $\epsilon$--precision trajectory has the maximum exit time. More rigorously, let $S_{\epsilon}=\Big\{  \{\tilde{\u}_K^{\tau}\}_{K=1}^{K_{exit}^{\tau}}  \Big| \u_{0} \Big\}$ be the set of $\tau$-parametrized $\epsilon$--\textbf{precision trajectories} generated by expanding $\u_K$ to first order in $\epsilon$, where $\tau$ varies with variations in the perturbation sequence $\{\P_k\}_{k=0}^{K}$. Let $K_{exit}^{\tau}$ be the exit time of the $\tau$-parametrized trajectory $\{\tilde{\u}_K^{\tau}\}_{K=1}^{K_{exit}^{\tau}}$ from the ball $\mathcal{B}_{\epsilon}(\x^*)$, where we have $ K_{exit}^{\tau} = \inf_{K\geq 1} \bigg\{K \hspace{0.2cm}\bigg| \hspace{0.2cm} \norm{\tilde{\u}_K^{\tau}}^2 > \epsilon^2 \bigg\}.$
	Let $K^{\iota}$ be defined as
	\begin{align}
	K^{\iota} & = \inf_{K\geq 1} \bigg\{K \hspace{0.2cm}\bigg| \hspace{0.2cm} \inf_{\tau}\bigg\{\norm{\tilde{\u}_K^{\tau}}^2 \bigg \}> \epsilon^2 \bigg\}.
\label{kiota}	\end{align}
Then the following inequality holds:
	\begin{align*}
	K^{\iota} &\geq \sup_{\tau} \bigg\{K_{exit}^{\tau} \bigg\}  = \sup_{\tau} \inf_{K\geq 1} \bigg\{K \hspace{0.2cm}\bigg| \hspace{0.2cm} \norm{\tilde{\u}_K^{\tau}}^2 > \epsilon^2 \bigg\}.
	\end{align*}
	
Finally, the linear exit time theorem for the $\epsilon$--precision trajectories (Theorem 3.2 in \cite{dixit2022exit}) states that for gradient descent with $\alpha = \frac{1}{L}$ where $\epsilon < \frac{2\beta}{M}$, and some \textbf{minimum projection value $ \sum_{j \in \mathcal{N}_{US}} ({\theta}^{us}_j)^2 \geq  \Delta$ of the initial radial vector $\u_{0}$ on $\mathcal{E}_{US}$} with $ \u_{0} = \epsilon\sum_{i \in \mathcal{N}_{S} } {\theta}^{s}_{i} \v_i + \epsilon\sum_{j \in \mathcal{N}_{US} } {\theta}^{us}_{j} \v_j
$, there exist $\epsilon$--precision trajectories $\{\tilde{\u}_{K}\}_{K=1}^{K_{exit}}$ with linear exit time. Moreover their exit time $K_{exit}$ from $\mathcal{B}_{\epsilon}(\x^*)$ is approximately upper bounded as
\begin{align}
K_{exit} < K^{\iota} & \lessapprox \frac{\log \bigg(\bigg(2 + \frac{ \epsilon M}{2L}\bigg) \log\bigg(\frac{2 + \frac{ \epsilon M}{2L}}{1 + \frac{\beta}{L} - \frac{\epsilon M}{2L}}\bigg)\frac{2 \delta}{\epsilon M n}\bigg)}{2\log\bigg(\frac{2 + \frac{ \epsilon M}{2L}}{1 + \frac{\beta}{L} - \frac{\epsilon M}{2L}}\bigg)}. \label{linearexittimebound}
\end{align}	
In \cite[Theorem~3.2]{dixit2022exit}, we provide a necessary initial condition for the linear exit time bound, which is \begin{align*}
	\Delta >\epsilon \frac{ M Ln}{ \delta (L+ \beta )} = \mathcal{O}(\epsilon),
	\end{align*}
	where it is required that $ \sum_{j \in \mathcal{N}_{US}} ({\theta}^{us}_j)^2 \geq  \Delta$. In this work we provide the sufficient boundary conditions for linear exit time $\epsilon$--precision trajectories.
	
	Before moving to the next section that details the sufficient conditions, we show that the $\epsilon$--precision trajectory $\{\tilde{\u}_K\}_{K=0}^{K_{exit}}$ generated by expanding the matrix product in the expression $ \u_{K} = \prod_{k=0}^{K-1} \bigg[\A_k + \epsilon \P_k   \bigg]\u_{0}  $ to first order in $\epsilon$ has a very small relative error compared to the exact trajectory.
	
	\subsubsection{Relative Error Margin in the $\epsilon$--Precision Trajectory}\label{sectionrelativeerror}
	By the definition of the $\epsilon$--precision trajectory, we have that
	\begin{align}
	    \tilde{\u}_K = \prod_{k=0}^{K-1} \A_k \u_{0} + \epsilon \sum\limits_{r=0}^{K-1}\prod_{k=0}^{r} \A_k \P_r  \prod_{k=r+1}^{K-1} \A_k   \u_{0},
	\end{align}
	which is obtained by expanding the matrix product $ \prod_{k=0}^{K-1} \bigg[\A_k + \epsilon \P_k   \bigg]$ to first order in $\epsilon$. Now using the ``Approximation Lemma" discussed above for $K \epsilon \ll 1$ and $ \epsilon < \norm{\A^{-1}}_2^{-1} \norm{\P}_2^{-1}$ where $ \sup_{0\leq k \leq K-1}\norm{\A_k}_2 = \norm{\A}_2$, $ \sup_{0\leq k \leq K-1}\norm{\A_k^{-1}}_2 = \norm{\A^{-1}}_2$, $\sup_{0\leq k \leq K-1}\norm{\P_k}_2 = \norm{\P}_2$ for some matrices $\A$ and $\P$, we get that:
	\begin{align}
	    \u_K &= \prod_{k=0}^{K-1} \bigg[\A_k + \epsilon \P_k   \bigg] \u_0\\
	    &= \prod_{k=0}^{K-1} \A_k \u_{0} + \epsilon \sum\limits_{r=0}^{K-1}\prod_{k=0}^{r} \A_k \P_r  \prod_{k=r+1}^{K-1} \A_k   \u_{0} + \mathcal{O}\bigg(\norm{\A}_2^K(K\epsilon)^2 \frac{\norm{\P}_2^2}{\norm{\A}_2^2} \norm{\u_0} \bigg)\\
	   &=  \tilde{\u}_K + \mathcal{O}\bigg(\norm{\A}_2^K(K\epsilon)^2 \epsilon \bigg).  \label{errormargin1}
	\end{align}
	Next, from the proof of \cite[Lemma~3.4]{dixit2022exit} we recall that $\A_k = \sum\limits_{i \in \mathcal{N}_S} c_i^s(k)\v_i \v_i^T + \sum\limits_{j \in \mathcal{N}_{US}} c_j^{us}(k)\v_j \v_j^T $ where  $c_i^s(k) = 1 - \alpha \lambda_i^{s} + \mathcal{O}(\epsilon)  $, $c_j^{us}(k) = 1 - \alpha \lambda_j^{us} + \mathcal{O}(\epsilon)  $ and $\lambda_i^{s}, \v_i$ and $\lambda_j^{us},\v_j$ are the eigenvalue-eigenvector pairs corresponding to the stable and unstable subspaces of $\nabla^2 f(\x^*)$, respectively. Also, $ \u_{0} = \epsilon\sum_{i \in \mathcal{N}_{S} } {\theta}^{s}_{i} \v_i + \epsilon\sum_{j \in \mathcal{N}_{US} } {\theta}^{us}_{j} \v_j
$ and for $\alpha = \frac{1}{L}$ we have the bounds $ 1 + \frac{\beta}{L} - \frac{\epsilon M}{2L}  \leq c_j^{us}(k) \leq 2 + \frac{\epsilon M}{2L}$ and $  -\frac{\epsilon M}{2L} \leq c_i^{s}(k) \leq 1 - \frac{\beta}{L} + \frac{\epsilon M}{2L}$ (see \cite[Lemma~3.4]{dixit2022exit}). Hence we have that:
\begin{align}
    \norm{\u_K} &= \norm{\prod_{k=0}^{K-1} \bigg[\A_k + \epsilon \P_k   \bigg] \u_0}\\
	    & \geq \norm{\prod_{k=0}^{K-1} \A_k \u_{0}} - \norm{ \epsilon \sum\limits_{r=0}^{K-1}\prod_{k=0}^{r} \A_k \P_r  \prod_{k=r+1}^{K-1} \A_k   \u_{0}} - \mathcal{O}\bigg(\norm{\A}_2^K(K\epsilon)^2 \frac{\norm{\P}_2^2}{\norm{\A}_2^2} \norm{\u_0} \bigg)\\
	     & \geq \norm{\prod_{k=0}^{K-1} \A_k \u_{0}} -  \mathcal{O}\bigg(\norm{\A}_2^K(K\epsilon) \frac{\norm{\P}_2}{\norm{\A}_2} \norm{\u_0} \bigg)\\
	    & = \norm{\bigg(\prod_{k=0}^{K-1}c_i^{s}(k)\bigg) \epsilon\sum_{i \in \mathcal{N}_{S} } {\theta}^{s}_{i} \v_i  +  \bigg(\prod_{k=0}^{K-1}c_j^{us}(k)\bigg) \epsilon\sum_{j \in \mathcal{N}_{US} } {\theta}^{us}_{j} \v_j } - \mathcal{O}\bigg(\norm{\A}_2^{K}(K\epsilon) \epsilon \bigg) \\
	    & \geq  \epsilon\bigg(\inf\{c_j^{us}(k)\}\bigg)^K \sqrt{\bigg(\frac{\inf\{c_i^{s}(k)\}}{\inf\{c_j^{us}(k)\}}\bigg)^{2K} \sum_{i \in \mathcal{N}_{S} } ({\theta}^{s}_{i})^2   +  \sum_{j \in \mathcal{N}_{US} } ({\theta}^{us}_{j})^2} - \mathcal{O}\bigg(\norm{\A}_2^{K}(K\epsilon) \epsilon \bigg) \\
	    & \approx \epsilon\bigg( 1 + \frac{\beta}{L} - \frac{\epsilon M}{2L}\bigg)^K \sqrt{\sum_{j \in \mathcal{N}_{US} } ({\theta}^{us}_{j})^2}  - \mathcal{O}\bigg(\norm{\A}_2^{K}(K\epsilon) \epsilon \bigg), \label{errormargin2}
\end{align}
where we used $ \inf\{c_j^{us}(k)\} = \bigg( 1 + \frac{\beta}{L} - \frac{\epsilon M}{2L}\bigg) $, $ \inf\{c_i^{s}(k)\} =  - \frac{\epsilon M}{2L}$ and $\epsilon^{2K} \approx 0$ (here $\epsilon \ll 1$ since $K \epsilon \ll 1$). Simplifying \eqref{errormargin1} by using the substitution $ \norm{\A}_2 = \sup\{\norm{\A_k}_2\}= \sup\{c_j^{us}(k)\} = 2 + \frac{\epsilon M}{2 L}$ and taking norm yields
\begin{align}
 \norm{ \u_K - \tilde{\u}_K} = \mathcal{O}\bigg(\norm{\A}_2^K(K\epsilon)^2 \epsilon \bigg) = \mathcal{O}\bigg(\bigg(2 + \frac{\epsilon M}{2 L}\bigg)^K(K\epsilon)^2 \epsilon \bigg). \label{errormargin3}
\end{align}
Finally, dividing \eqref{errormargin3} by \eqref{errormargin2} we get the following bound on the relative error:
\begin{align}
    \frac{\norm{ \u_K - \tilde{\u}_K}}{\norm{\u_K}} &\leq \frac{1}{\epsilon\bigg( 1 + \frac{\beta}{L} - \frac{\epsilon M}{2L}\bigg)^K \sqrt{\sum_{j \in \mathcal{N}_{US} } ({\theta}^{us}_{j})^2}  - \mathcal{O}\bigg(\norm{\A}_2^{K}(K\epsilon) \epsilon \bigg)}\mathcal{O}\bigg(\bigg(2 + \frac{\epsilon M}{2 L}\bigg)^K(K\epsilon)^2 \epsilon \bigg) \\
    &\leq \frac{1}{ \sqrt{\sum_{j \in \mathcal{N}_{US} } ({\theta}^{us}_{j})^2}  - \mathcal{O}\bigg(\frac{\bigg(2 + \frac{\epsilon M}{2 L}\bigg)^K}{\bigg( 1 + \frac{\beta}{L} - \frac{\epsilon M}{2L}\bigg)^K}(K\epsilon) \bigg)}\mathcal{O}\bigg(\frac{\bigg(2 + \frac{\epsilon M}{2 L}\bigg)^K}{\bigg( 1 + \frac{\beta}{L} - \frac{\epsilon M}{2L}\bigg)^K}(K\epsilon)^2  \bigg) \\
    &\leq \frac{1}{ \sqrt{\sum_{j \in \mathcal{N}_{US} } ({\theta}^{us}_{j})^2}  - \mathcal{O}\bigg(\frac{1}{\sqrt{\epsilon}}\bigg(\log\bigg(\frac{1}{\epsilon} \bigg)\epsilon\bigg)  \bigg)}\mathcal{O}\bigg(\frac{1}{\sqrt{\epsilon}}\bigg(\log\bigg(\frac{1}{\epsilon} \bigg)\epsilon\bigg)^2  \bigg), \label{relativeerrorprojectionbound}
\end{align}
where we have substituted the upper bound on $K_{exit}$ from \eqref{linearexittimebound} into $K$. Now, if $ \sqrt{\sum_{j \in \mathcal{N}_{US} } ({\theta}^{us}_{j})^2} > \mathcal{O}\bigg(\frac{1}{\sqrt{\epsilon}}\bigg(\log\bigg(\frac{1}{\epsilon} \bigg)\epsilon\bigg)  \bigg)$ then the relative error is of the order $\mathcal{O}\bigg(\frac{1}{\sqrt{\epsilon}}\bigg(\log\bigg(\frac{1}{\epsilon} \bigg)\epsilon\bigg)^2  \bigg)$, which goes to $0$ as $\epsilon \to 0$.

\subsection{Sufficient Conditions for Linear Exit Time}

Our first theorem states that the first order approximation of any gradient descent trajectory starting from an $\epsilon$ neighborhood of any strict saddle point $\x^*$ will escape this neighborhood in linear time, i.e., $\mathcal{O}(\log(\epsilon^{-1}) )$,  provided the projection value of its initialization on the unstable subspace of $\nabla^2 f(\x^*)$ is lower bounded.

\begin{theorem} \label{thm1}
{The $\epsilon$--precision trajectory $\{\tilde{\u}_K\}_{K=0}^{K_{exit}}$ generated by the gradient descent method for step-size $\alpha = \frac{1}{L}$ {on any function satisfying Assumptions \textbf{A1-A4}} has linear exit time \eqref{linearexittimebound} from the strict saddle neighborhood $\mathcal{B}_{\epsilon}(\x^*)$ {provided the projection value} of the initialization $\u_0$ onto the unstable subspace $\mathcal{E}_{US}$ of the Hessian $\nabla^2 f(\x^*)$, given by $ \sum_{j \in \mathcal{N}_{US}} ({\theta}^{us}_j)^2$, is lower bounded as:}
\begin{align}
\hspace{0.2cm}\sum_{j \in \mathcal{N}_{US}} ({\theta}^{us}_j)^2     \gtrapprox \frac{\bigg(2 + \frac{ \epsilon M}{2L}\bigg)\bigg(\frac{2 \delta \mu\log\bigg(1 + \frac{\beta}{L} - \frac{\epsilon M}{2L}\bigg)}{ Mn}
	\bigg)}{\frac{1}{a}\log \bigg(\frac{2 \delta\bigg(2 + \frac{ \epsilon M}{2L}\bigg)\log\bigg(1 + \frac{\beta}{L} - \frac{\epsilon M}{2L}\bigg)\log \bigg(\frac{2 + \frac{ \epsilon M}{2L}}{1 + \frac{\beta}{L} - \frac{\epsilon M}{2L}}\bigg)}{   \epsilon M n \log\bigg(2 + \frac{ \epsilon M}{2L}\bigg)}\bigg)+ 1},
\end{align}
{where $\sqrt[a]{\mu} =  \frac{    M n \log\bigg(2 + \frac{ \epsilon M}{2L}\bigg)}{2 \delta\bigg(2 + \frac{ \epsilon M}{2L}\bigg)\log\bigg(1 + \frac{\beta}{L} - \frac{\epsilon M}{2L}\bigg)\log \bigg(\frac{2 + \frac{ \epsilon M}{2L}}{1 + \frac{\beta}{L} - \frac{\epsilon M}{2L}}\bigg)}$, $a=\frac{\log \bigg(2+ \frac{\epsilon M}{2L}\bigg)}{\log\bigg(2 + \frac{ \epsilon M}{2L}\bigg) - \log \bigg({1 + \frac{\beta}{L} - \frac{\epsilon M}{2L}}\bigg)}$ and we require that}:  \begin{align}
   \epsilon <\min \bigg\{\inf_{{\norm{\u}=1}}\bigg(\limsup_{j\to \infty} \sqrt[j]{\frac{r_j(\u)}{j!}}\bigg)^{-1},\frac{2L\delta}{M(2Ln^2 -\delta )} + \mathcal{O}(\epsilon^2), \frac{2 \beta}{M}\bigg\}, \label{epbound}
\end{align}
{where $r_j(\u) = \norm{\bigg(\frac{d^j }{dw^j}\nabla^2 f(\x^*+w\u)\bigg\vert_{w=0}\bigg)}_2$, $ \u_{0} = \epsilon\sum_{i \in \mathcal{N}_{S} } {\theta}^{s}_{i} \v_i + \epsilon\sum_{j \in \mathcal{N}_{US} } {\theta}^{us}_{j} \v_j
$ and $\v_i, \v_j$ are the eigenvectors of the Hessian $\nabla^2 f(\x^*)$ and {$\delta$ is as in Proposition \ref{eigenprop}}.}

{In terms of order notation, we require the following lower bound on the projection} $ \sum_{j \in \mathcal{N}_{US}} ({\theta}^{us}_j)^2 $:
\begin{align}
\sum_{j \in \mathcal{N}_{US}} ({\theta}^{us}_j)^2  \gtrapprox  \mathcal{O}{}\bigg(\frac{1}{\log(\epsilon^{-1})}\bigg).
\end{align}

\end{theorem}
The proof of this theorem is given in Appendix \ref{Appendix A}.

Recall from \eqref{relativeerrorprojectionbound} that for relative error in the $\epsilon$--precision trajectory to be bounded, we require that $ \sqrt{\sum_{j \in \mathcal{N}_{US} } ({\theta}^{us}_{j})^2} > \mathcal{O}\bigg(\frac{1}{\sqrt{\epsilon}}\bigg(\log\bigg(\frac{1}{\epsilon} \bigg)\epsilon\bigg) \bigg)$. However, this condition is already satisfied by the sufficient condition $ \sum_{j \in \mathcal{N}_{US}} ({\theta}^{us}_j)^2 \gtrapprox \mathcal{O}\bigg(\frac{1}{\log (\epsilon^{-1})}\bigg)$ in terms of order since $ \mathcal{O}\bigg(\sqrt{\frac{1}{\log (\epsilon^{-1})}}\bigg) > \mathcal{O}\bigg(\frac{1}{\sqrt{\epsilon}}\bigg(\log\bigg(\frac{1}{\epsilon} \bigg)\epsilon\bigg)  \bigg)$ as $\epsilon \to 0$.

{The above result can be interpreted as follows: for any sufficiently small $\epsilon$ bounded from \eqref{epbound} if a gradient descent trajectory at the surface of any saddle neighborhood $\mathcal{B}_{\epsilon}(\x^*)$ has a projection value of order $\Theta\bigg(\frac{1}{\log(\epsilon^{-1})}\bigg)$ on the unstable subspace of $\nabla^2 f(\x^*)$, then this trajectory is guaranteed to exit the saddle neighborhood in linear time. This result is crucial since it furthers the findings of the state of the art \cite{lee2016gradient} where a non-zero projection value guarantees almost sure escape from the saddle point but does not provide any insights into whether a non-zero projection value could lead to fast escaping trajectories, something which Theorem \ref{thm1} establishes rigorously. Moreover the projection value bound in Theorem \ref{thm1} is {insightful} in the sense that it {illustrates the dependency to the} quantities like condition number, problem dimension, spectral gap, etc.  Since this result ensures that fast escaping gradient trajectories are indeed dense with respect to random initialization on the surface of the ball $ \mathcal{B}_{\epsilon}(\x^*)$, we can safely say that fast escaping trajectories for gradient descent method from small saddle neighborhoods of Morse functions will be a {generic} 
phenomenon. In case if the sufficient condition is not satisfied, one can perform a single step perturbation to land on a point which satisfies this condition. Then reverting back to gradient descent update, linear exit time from the saddle neighborhood will be guaranteed. This particular idea will serve as a basis for the development of a single step perturbation based gradient descent method for escaping saddle points {faster}.
}

{We now move to the next section which provides a rate analysis in regions outside the small saddle neighborhood $\mathcal{B}_{\epsilon}(\x^*)$ where the local analyticity property no longer exists and we are only left with the class of $\mathcal{C}^2$ gradient and Hessian Lipschitz, Morse functions, i.e., functions satisfying assumptions \textbf{A2-A4}.}

\section{Sequential Monotonicity}\label{mainsec2}

{The first theorem in this section {establishes a monotonicity property of the}  
gradient descent trajectories in a strict saddle neighborhood. This property is termed as ``sequential monotonicity" which implies that within some neighborhood of the strict saddle point $\x^*$ any gradient trajectory, which does not converge to $\x^*$, first continuously contracts towards $\x^*$ up to some point and from there onward expands continuously away from $\x^*$ until it escapes this neighborhood. }
\begin{theorem} \label{thm2}
{{On the class of $\mathcal{C}^2$ gradient and Hessian Lipschitz, Morse functions,} if a gradient trajectory with respect to some stationary point $\x^*$ has non-contractive dynamics at any iteration $k=K$, then it has expansive dynamics for all iterations $k>K$ provided $\norm{\x_k - \x^*}$ is bounded above by some $\xi>0$ where $\{\x_k\}$ is the sequence that generates the gradient trajectory. This property of the sequence of radial distances $\{\norm{\x_k-\x^*}\}$ can be termed as the sequential monotonicity.}

{Moreover, in the case of $\x^*$ being a strict saddle point, we have for gradient trajectories with step-size $\alpha = \frac{1}{L}$ that $\xi<\frac{1}{ \varsigma M } \frac{\bigg((1+\frac{\beta}{L})^2+ \frac{1}{4 (1+\frac{\beta}{L})^2}-\frac{5}{4}\bigg)}{6}$ for some $\varsigma >2$. Specifically, consider the tuple $ (\x, \x^+, \x^{++} )$ that is equivalent to the tuple $ (\x_k, \x_{k+1}, \x_{k+2} )$ for any $k$. Let $\norm{\x^+ - \x^*} > \norm{\x-\x^*}$ and $\norm{\x - \x^*}<\xi$. Then the following holds: }
\begin{align}
\textbf{a.} \hspace{0.5cm} \norm{\x^{++}-\x^*} &\geq \bar{\rho}(\x)\norm{\x^+-\x^*} - \sigma(\x), \quad \text{and}\\
\textbf{b.} \hspace{0.5cm}  \norm{\x^{++}-\x^*} &> \norm{\x^+-\x^*},
\end{align}
{where $\sigma(\x) = \mathcal{O}(\norm{\x-\x^*}^2)$ and $\bar{\rho}(\x)> 1 + \frac{\bigg((1+\frac{\beta}{L})^2+ \frac{1}{4 (1+\frac{\beta}{L})^2}-\frac{5}{4}\bigg)}{12}$.}
\end{theorem}

The proof of this theorem is given in Appendix \ref{Appendix B}.

\begin{remark}\label{proprem}
The upper bound on $\xi$ given by the quantity $\frac{1}{ \varsigma M } \frac{\bigg((1+\frac{\beta}{L})^2+ \frac{1}{4 (1+\frac{\beta}{L})^2}-\frac{5}{4}\bigg)}{6}$ for $\varsigma>2$ is always positive and  is equal to $0$ only when $ \beta = 0$. Moreover, for Morse functions that are {well conditioned at their stationary points, i.e., $0 \ll \frac{\beta}{L}< 1$,} this quantity can be treated as a constant. Moreover this bound on $\xi$ also makes sure that there cannot be any other critical point within a radius of $\frac{1}{ \varsigma M } \frac{\bigg((1+\frac{\beta}{L})^2+ \frac{1}{4 (1+\frac{\beta}{L})^2}-\frac{5}{4}\bigg)}{6}$ for $\varsigma>2$ from $\x^*$. If another stationary point did exist within this radius of $\x^*$  say $\x^*_1$ then $ \norm{\nabla f(\x^*_1)} \geq \beta \norm{\x^*- \x^*_1} > 0$ from \eqref{interimbound4} which contradicts the fact that $\x^*_1$ is a critical point of $f$. This seemingly trivial result will be of utility in Proposition \ref{proposition1} where we define separation between critical points.
\end{remark}

{In words, Theorem \ref{thm2} states that within any $\xi$ neighborhood of the saddle point where $\xi<\frac{1}{ \varsigma M } \frac{\bigg((1+\frac{\beta}{L})^2+ \frac{1}{4 (1+\frac{\beta}{L})^2}-\frac{5}{4}\bigg)}{6}$ for some $\varsigma >2$, every gradient descent trajectory first contracts continuously towards $\x^*$. The first iteration after the end of contraction phase is either marked by expansion or preservation of radial distance, i.e., no expansion or contraction. In both cases the trajectory from here onward expands continuously till it exits $ \mathcal{B}_{\xi}(\x^*)$ where in the latter case it is assumed that the trajectory didn't already contract to $\x^*$. Furthermore expansion happens at an almost geometric rate as evident from part \textbf{(a.)} of the theorem which can be leveraged to obtain linear rate for the expansion phase of trajectories inside  $ \mathcal{B}_{\xi}(\x^*)$.    }

{So far we have been able to develop a machinery that will help us in providing linear rate of expansion inside  $ \mathcal{B}_{\xi}(\x^*)$. It remains to develop a proof technique which can generate linear rates of contraction inside $ \mathcal{B}_{\xi}(\x^*)$. In order to do so we introduce certain terms that are required for better understanding the contraction and expansion dynamics of the trajectory.} In this regard, let $\hat{K}_{exit} $ be the first exit time of the gradient descent trajectory from the ball $\mathcal{B}_{\xi}(\x^*)$, where we assume that the trajectory starts at the boundary of the ball $\mathcal{B}_{\xi}(\x^*)$, i.e., $\x_0 \in \mathcal{\bar{B}}_{\xi}(\x^*)\backslash \mathcal{B}_{\xi}(\x^*) $ and $\xi$ is bounded from Theorem \ref{thm2}. Next, for any $\epsilon < \xi$, let $\bar{\mathcal{B}}_{\xi}(\x^*) \backslash \mathcal{B}_{\epsilon}(\x^*)$ be a compact shell centered at $\x^*$. Let $k=K_c$ be the last iteration for which the gradient trajectory has contractive dynamics inside the shell and $k=K_e$ be the first iteration for which the gradient trajectory has expansive dynamics inside the shell. {Note that $K_c$ and $K_e$ are equal iff either the trajectory starts expanding before reaching the ball $\mathcal{B}_{\epsilon}(\x^*)$ or the trajectory just touches the surface of the ball $\mathcal{B}_{\epsilon}(\x^*)$ and then expands from there onward.}

{The next lemma provides 
{further insights} into the behavior of function sequence $ \{f(\x_k)\}_{k=0}^{K_c} $ associated with iterate sequence $ \{\x_k\}_{k=0}^{K_c}$ where $0 \leq k \leq K_c$ are the iterations with contraction dynamics.}

\begin{lem}\label{polyaklem}
{{On the class of $\mathcal{C}^2$ gradient and Hessian Lipschitz, Morse functions}, the function sequence $ \{f(\x_k)\}_{k=0}^{K_c} $ associated with iterate sequence $ \{\x_k\}_{k=0}^{K_c}$ for $\norm{\x_{K_c} - \x^*} < \frac{3\beta^2}{4ML}$ and $K_c < K_e$ satisfies the Polyak--{\L}ojasiewicz condition \cite{karimi2016linear} where for any $0 \leq k \leq K_c$ we have that:
\begin{align}
    0 < f(\x_k) - f(\x^*) \leq  \frac{L}{2 \beta^2} \norm{\nabla f(\x_k)}^2. \nonumber
\end{align}}
\end{lem}
The proof of this lemma is given in Appendix \ref{Appendix C}. {Using this lemma, it can be readily checked that the function sequence $\{f(\x_k)\}_{k=0}^{K_c}$ is strongly monotonic in the contraction phase of the trajectory. Formally, for $0 \leq k \leq K_c$ using Lemma \ref{polyaklem} and the gradient Lipschitz condition we will have the inequality $ f(\x_{k+1}) - f(\x^*)  \leq \bigg( 1 - \frac{\beta^2}{L^2}\bigg) \bigg( f(\x_{k}) - f(\x^*) \bigg)$. Therefore linear rates for the contraction phase of trajectory can be recovered using this result. It should however be noted that the function sequence $\{f(\x_k)\}_{k=K_e}^{\hat{K}_{exit}}$ associated with the expansion phase of the trajectory does not satisfy the Polyak--{\L}ojasiewicz condition from Lemma \ref{polyaklem} and therefore we require Theorem \ref{thm2} to generate linear rates of expansion for the trajectory in its expansion phase (see discussion within the proof of {Lemma} \ref{polyaklem} for details).}

{Before stating the final theorem of this section we introduce the term '\textit{sojourn time}'. It is defined as the time the trajectory spends inside the shell $\bar{\mathcal{B}}_{\xi}(\x^*) \backslash \mathcal{B}_{\epsilon}(\x^*)$ before leaving this region. The sojourn time will be the sum of contraction time (derived using Lemma \ref{polyaklem}) and the expansion time (derived using Theorem \ref{thm1}) for any trajectory inside the shell $\bar{\mathcal{B}}_{\xi}(\x^*) \backslash \mathcal{B}_{\epsilon}(\x^*)$. We are now ready to state the theorem.}

\begin{theorem}\label{thm3}
{The sojourn time $K_{shell}$ for a gradient trajectory inside the compact shell $\bar{\mathcal{B}}_{\xi}(\x^*) \backslash \mathcal{B}_{\epsilon}(\x^*)$ for a strict saddle point $\x^*$ {of any $\mathcal{C}^2$ gradient and Hessian Lipschitz, Morse function} is bounded by
\begin{align}
  \hspace{-0.3cm} K_{shell} \leq    \frac{\log\bigg(\frac{L}{2}\xi^2\bigg) - \log\bigg(\frac{\beta^2}{2L}\epsilon^2 - \frac{2M}{3}\epsilon^3\bigg) }{\log\bigg( 1 - \frac{\beta^2}{L^2}\bigg)^{-1}} + \frac{\log (\xi) - \log (\epsilon)}{\log \bigg(\frac{\inf\{\bar{\rho}(\x_{k-2})\}}{1 + M \xi}\bigg)} + 3, \nonumber
\end{align} where  $K_{shell} = \hat{K}_{exit} + K_c - K_e $ with $K_c \leq \frac{\log\bigg(\frac{L}{2}\xi^2\bigg) - \log\bigg(\frac{\beta^2}{2L}\epsilon^2 - \frac{2M}{3}\epsilon^3\bigg) }{\log\bigg( 1 - \frac{\beta^2}{L^2}\bigg)^{-1}} +1$, $\hat{K}_{exit}  - K_e \leq \frac{\log (\xi) - \log (\epsilon)}{\log \bigg(\frac{\inf\{\bar{\rho}(\x_{k-2})\}}{1 + M \xi}\bigg)} + 2$, and infimum in the term $ \inf\{\bar{\rho}(\x_{k-2})\}$ is {taken} over {the indices} $K_e+2 \leq k \leq \hat{K}_{exit}$. Further, $\hat{K}_{exit} - K_e$ is the time for which the gradient trajectory has expansive dynamics inside the shell $\bar{\mathcal{B}}_{\xi}(\x^*) \backslash \mathcal{B}_{\epsilon}(\x^*)$, and $K_c$ is the time for which the gradient trajectory has contractive dynamics inside the shell $\bar{\mathcal{B}}_{\xi}(\x^*) \backslash \mathcal{B}_{\epsilon}(\x^*)$. Also, $\xi  \leq   \frac{1}{ \varsigma M } \frac{\bigg((1+\frac{\beta}{L})^2+ \frac{1}{4 (1+\frac{\beta}{L})^2}-\frac{5}{4}\bigg)}{6}  $ with $ \varsigma > 2$, $\epsilon < \frac{3\beta^2}{4ML}$ and $\frac{\inf\{\bar{\rho}(\x_{k-2})\}}{1 + M \xi}> 1$.}

{In terms of order notation, $ K_{shell}$ has the following rate:}
\begin{align}
    K_{shell} =  \mathcal{O} \bigg(\log \bigg(\frac{\xi}{\epsilon}\bigg)\bigg) + \mathcal{O} (1),  \label{shellratelaw}
\end{align}
{where $K_c =  \mathcal{O} \bigg(\log \bigg(\frac{\xi}{\epsilon}\bigg)\bigg)$, $ \hat{K}_{exit} - K_e = \mathcal{O} \bigg(\log \bigg(\frac{\xi}{\epsilon}\bigg)\bigg) + \mathcal{O} (1)$.}
\end{theorem}
The proof of this theorem is given in Appendix \ref{Appendix CC}.

{Theorem \ref{thm3} provides an upper bound on the travel time of the trajectory inside the shell $\bar{\mathcal{B}}_{\xi}(\x^*) \backslash \mathcal{B}_{\epsilon}(\x^*)$. The upper bound is linear since it is the sum of rates in the contraction and expansion phase of the trajectory and both these rates are linear by virtue of Lemma \ref{polyaklem} and Theorem \ref{thm2} respectively. In contrast to the linear exit time bound \eqref{linearexittimebound} which only holds for very small values of $\epsilon$ from Theorem \ref{thm1}, this rate holds for much bigger $\xi$ neighborhoods and at the same time does not require the function to be analytic. The power of Theorem \ref{thm3} will become more apparent once we develop a {fast} algorithm
for escaping strict saddle points of Morse functions. 
This theorem will facilitate in keeping the algorithm very close to the gradient descent method since it proves that any escaping gradient descent trajectory from some small ball $ \mathcal{B}_{\epsilon}(\x^*)$ will leave a larger ball $ \mathcal{B}_{\xi}(\x^*)$ at a linear rate irrespective of its exit point  on $ \mathcal{B}_{\epsilon}(\x^*)$. Hence any algorithm, which exits some small ball $ \mathcal{B}_{\epsilon}(\x^*)$ using the gradient descent update, can keep on performing gradient descent updates so as to have linear rate of escape from a larger ball $ \mathcal{B}_{\xi}(\x^*)$. }

\section{Additional Lemmas}
We now discuss some additional yet important lemmas instrumental in analysing the gradient trajectory/approximate trajectory behavior in saddle neighborhoods of any strict saddle point $\x^*$. Also, in the remainder of this section, we do not consider the effects of {first-order perturbations}, i.e., $\mathcal{O}(\epsilon)$ terms, in the Hessian (see \cite[Lemma~3.3]{dixit2022exit}) since we no longer quantify the exit times / boundary conditions and {are only interested in approximate trajectory behavior.} Hence most of the {results in this section are qualitative.} {Assumptions \textbf{A2-A4} hold for all the lemmas in this section where Lemmas \ref{lemma1}, \ref{lemma2} use the extra assumption of local analyticity around the strict saddle point.} The proofs of the lemmas in this section are given in Appendix \ref{Appendix D}.

\begin{lem}\label{lemma1}
The gradient trajectories $\{{\u}_K\}_{K=0}^{K_{exit}}$  inside the ball $\mathcal{B}_{\epsilon}(\x^*)$ with linear exit time and satisfying the initial condition $\sqrt{\sum_{j \in \mathcal{N}_{US} } ({\theta}^{us}_{j})^2}  > \mathcal{O}\bigg(\frac{1}{\sqrt{\epsilon}}\bigg(\log\bigg(\frac{1}{\epsilon} \bigg)\epsilon\bigg)  \bigg) $ approximately exhibit hyperbolic behavior {in the sense} 
that they first move exponentially fast towards the saddle point $\x^*$, reach some point of minimum distance from $\x^*$, denoted by $\x_{critical}$, and then move exponentially fast away from $\x^*$ for some iterations so as to escape the saddle region. For the case when $\x_{critical} \to \x^*$, their first-order approximation or the $\epsilon$--precision trajectories can take very large time to exit the ball $\mathcal{B}_{\epsilon}(\x^*)$, i.e., $K^{\iota} \to \infty$ where $K^{\iota}$ is defined in \eqref{kiota}. When $\x_{critical} = \x^*$, we have  $K_{exit} = K^{\iota} = \infty$, which implies that the $\epsilon$--precision trajectories and hence the gradient trajectory can never escape the saddle region.
\end{lem}

\begin{lem}\label{lemma2}
In the ball $\mathcal{B}_{\epsilon}(\x^*)$, gradient descent trajectories with linear exit time and satisfying the initial condition $\sqrt{\sum_{j \in \mathcal{N}_{US} } ({\theta}^{us}_{j})^2}  > \mathcal{O}\bigg(\frac{1}{\sqrt{\epsilon}}\bigg(\log\bigg(\frac{1}{\epsilon} \bigg)\epsilon\bigg)  \bigg) $ approximately{\footnote{{When we say this condition holds approximately, we mean that it holds for a first-order approximation of the gradient descent trajectory (see the proof of Lemma \ref{lemma2} for further details).}}} never curve around the stationary point $\x^*$. Moreover, all the linear exit time gradient descent trajectories lie approximately inside some orthant of the ball $\mathcal{B}_{\epsilon}(\x^*)$, i.e., the entry and exit point approximately subtend an angle less than or equal to $\frac{\pi}{2}$ at the point $\x^*$.
\end{lem}

\begin{lem}\label{lemma3}
The function value at the exit point on the ball $\mathcal{B}_{\epsilon}(\x^*)$ for any gradient descent trajectory is strictly less than $f(\x^*)$ provided $\epsilon$ is sufficiently small.
\end{lem}

\begin{lem}\label{lemma4}
For any $\epsilon \ll 2^{-\frac{2}{\kappa^2}} $ where $\kappa = \frac{\beta}{L}$, a gradient trajectory having exited the ball $\mathcal{B}_{{\epsilon}}(\x^*)$ can never re-enter this ball.
\end{lem}

\begin{lem}\label{lemma5}
The gradient descent trajectories exiting the ball $\mathcal{B}_{\xi}(\x^*)$, where $\xi$ is defined in Theorem \ref{thm3}, can never re-enter this ball provided the gradient magnitudes outside the ball $\mathcal{B}_{\xi}(\x^*)$ are sufficiently large with $\norm{\nabla f(\x)} \geq \gamma >  \frac{1}{\sqrt{2}}L \xi$.
\end{lem}

{Note that Lemma \ref{lemma3} is
used in {our analysis for} establishing that the function sequence $\{f(\x_k)\}_{k=K_e}^{\hat{K}_{exit}}$ associated with the expansion phase of the trajectory inside $\mathcal{B}_{\xi}(\x^*)$ does not satisfy the Polyak--{\L}ojasiewicz condition from Lemma \ref{polyaklem}. Lemmas \ref{lemma4} and \ref{lemma5} are termed as the ``no-return conditions" to $\epsilon$ and $\xi$ radius saddle neighborhoods respectively. Choosing $\epsilon$ from Lemma \ref{lemma4} will guarantee that any gradient trajectory can visit the saddle neighborhood $\mathcal{B}_{\epsilon}(\x^*)$ at most once. In particular, if the function satisfies the condition of large gradient magnitudes for certain $\xi$ from Lemma \ref{lemma5} then any gradient trajectory can visit the saddle neighborhood $\mathcal{B}_{\xi}(\x^*)$ at most once, and such a function is called a \textit{well-structured} function (see discussion after Proposition \ref{proposition3} for details).}

\section{Proposed Algorithm}\label{algosection}

{Since we have established the preliminaries on our unstable projection value and the sequential monotonicity property}, we propose a method called the Curvature Conditioned {Regularized} Gradient Descent (CCRGD) (Algorithm \ref{algo_1}) that can guarantee escaping saddle points in approximately linear time for Morse functions, by virtue of  {Theorems~\ref{thm1} and \ref{thm3}}, and that is also guaranteed to converge to a local minimum.

\begin{algorithm}[h]
	\caption{Curvature Conditioned {Regularized} Gradient Descent (CCRGD)}\label{algo_1}
	\begin{algorithmic}[1]
		\State {\textbf{Initialize}} $\{\x_{0}, \y_0, \y_1\}$ to $\mathbf{0}$, a  radius $\epsilon$ bounded by Theorem \ref{thm1}, constants $L,M, \beta, \delta$, minimum unstable  projection value $P_{min}(\epsilon)$ from the lower bound in $\eqref{projectionwellconditioned}$, condition flag $\Xi=0$, $\kappa = \frac{\beta}{L}$ and step-size $\alpha = \frac{1}{L}$
		\State {\textbf{for} $k=0,1,\cdots,K$} \textbf{do}
		\State \ \ \ \ \textbf{Obtain} $\nabla f(\x_k)$ \textbf{from first-order oracle}
		\State \ \ \ \ \textbf{If} $\norm{\nabla {f}(\x_k)} > L\epsilon$  	\textbf{then}	
		\State \ \ \ \ \ \ \ \ \textbf{Update} $ \x_{k+1} \gets \x_{k} - \alpha \nabla {f}(\x_k) $
		\State \ \ \ \ \ \ \ \ \textbf{If} $\Xi = 1$ \textbf{then update condition flag} $\Xi \gets 0$
		\State \ \ \ \ \textbf{Else}
		\State \ \ \ \ \ \ \ \ \textbf{If}
		$\norm{\nabla {f}(\x_k)} \leq L\epsilon$ \textbf{and} $\Xi = 1$	\textbf{then}
			\State \ \ \ \ \ \ \ \ \ \ \textbf{Update} $ \x_{k+1} \gets \x_{k} - \alpha \nabla {f}(\x_k) $
				\State \ \ \ \ \ \ \ \ \textbf{Else If} $\norm{\nabla {f}(\x_k)} \leq L\epsilon$ \textbf{and} $\Xi = 0$	\textbf{then}
		\State \ \ \ \ \ \ \ \ \ \ \textbf{Set} $\y_0 \gets \x_k$
		\State \ \ \ \ \ \ \ \ \ \ \textbf{Update} $\y_1 \gets \y_0 -\alpha \nabla {f}(\y_0) $	
		\State \ \ \ \ \ \ \ \ \ \ \textbf{Compute} $V_1 \gets \langle \y_1 - \y_0, \y_1 - \y_0\rangle$
		\State \ \ \ \ \ \ \ \ \ \ \textbf{Compute} $V_2 \gets \alpha \langle \y_1 - \y_0, \nabla {f}(\y_1) - \nabla {f}(\y_0)\rangle$
		\State \ \ \ \ \ \ \ \ \ \ \ \ \ \ \textbf{If}  $\frac{4\epsilon^2}{27\kappa^2}< V_1 - V_2 < \bigg(\frac{50 P_{min}(\epsilon) +4}{27}\bigg)\frac{\epsilon^2}{\kappa^2}$ \textbf{then} \hfill \Comment{Curvature Check Condition} \label{algocurvaturecondition}
						\State \ \ \ \ \ \ \ \ \ \ \ \ \ \ \ \ \textbf{Obtain} $\textbf{H} \gets \alpha\nabla^2 f(\x_k)$ \textbf{from second-order oracle}
\State \ \ \ \ \ \ \ \ \ \ \ \ \ \ \ \ \textbf{Solve the constrained eigenvalue problem:}
    $\x_{k+1} \in \argmin_{\norm{\x -\x_k}= \frac{\norm{\nabla f(\x_k)}}{\beta}  } \bigg( \frac{1}{2}(\x - \x_k)^{T}\textbf{H}(\x - \x_k)  \bigg)$ \label{ballconstraint}
	\State \ \ \ \ \ \ \ \ \ \ \ \ \ \ \textbf{Else If} $ 0< V_1 - V_2 \leq  \frac{4\epsilon^2}{27\kappa^2}$ \textbf{then} \Comment{Curvature Check Condition}
						\State \ \ \ \ \ \ \ \ \ \ \ \ \ \ \ \ \textbf{Obtain} $\textbf{H} \gets \alpha\nabla^2 f(\x_k)$ \textbf{from second-order oracle}
												\State \ \ \ \ \ \ \ \ \ \ \ \ \ \ \ \ \textbf{Compute} $\lambda_{min}(\textbf{H})$
																								\State \ \ \ \ \ \ \ \ \ \ \ \ \ \ \ \ \textbf{If} $\lambda_{min}(\textbf{H}) < 0$ \textbf{then}
\State \ \ \ \ \ \ \ \ \ \ \ \ \ \ \ \ \ \ \textbf{Solve the constrained eigenvalue problem:}
$\x_{k+1} \in \arg \min_{\norm{\x -\x_k}= \frac{\norm{\nabla f(\x_k)}}{\beta}  } \bigg( \frac{1}{2}(\x - \x_k)^{T}\textbf{H}(\x - \x_k)  \bigg)$
\State \ \ \ \ \ \ \ \ \ \ \ \ \ \ \ \ \textbf{Else}  \textbf{break}
		\State \ \ \ \ \ \ \ \ \ \ \textbf{Update condition flag} $\Xi \gets 1$
		\State \textbf{end for }
				\State \textbf{Second-Order Stationary Solution } $ = \x_k$
	\end{algorithmic}
	\label{algo1}
\end{algorithm}

{We first establish that the proposed algorithm escapes any saddle point of a function satisfying assumptions \textbf{A1-A4} at a linear rate and the function values generated by the algorithm decrease monotonically.}

\begin{lem}\label{saddleescapelemma}
{The trajectory generated by the CCRGD algorithm \ref{algo_1} in some $\epsilon$ neighborhood $\mathcal{B}_{\epsilon}(\x^*)$ of any strict saddle point $\x^*$ of a function satisfying assumptions \textbf{A1-A4} where $\epsilon$ is bounded by Theorem \ref{thm1}, exits  $\mathcal{B}_{\epsilon}(\x^*)$ in approximately linear time\footnote{The term ``\textit{approximately linear time}" implies that $K_{exit} \leq \mathcal{O}(\log(\epsilon^{-1})) + g(\epsilon)$ where $g(\cdot)$ is some absolutely continuous positive function such that $g(\epsilon) \to 0$ as $\epsilon \to 0$. See the exact expression for $g(\cdot)$ in \eqref{gfunction} from Appendix \ref{Appendix F}.} where the exit time is bounded by \eqref{linearexittimebound}.}
\end{lem}

The proof of this lemma is given in Appendix \ref{Appendix E}.

\begin{lem}\label{functionseqmonotonic}
{The function value sequence $\{f(\x_k)\}$ generated by the CCRGD algorithm \ref{algo_1} in some $\epsilon$ neighborhood $\mathcal{B}_{\epsilon}(\x^*)$ of any strict saddle point $\x^*$ of a function satisfying assumptions \textbf{A1-A4} where $\epsilon$ is bounded by Theorem \ref{thm1} decreases monotonically.} 
\end{lem}
The proof of this lemma is given in Appendix \ref{Appendix E}.

\begin{remark}
Note that the second-order step after the curvature check condition \ref{algocurvaturecondition} of Algorithm \ref{algo_1} can be replaced by Perturbed Gradient Descent (GD) type of update from \cite{jin2017escape} since one-step noise injection is known to escape saddle points. However there is no guarantee that such replacement will provably generate trajectories that exit the saddle neighborhood in linear time. The best one can achieve with a Perturbed GD type of update is fast escape with high probability. Since the focus of this work is to develop a deterministic algorithm that generates trajectories with linear exit time, we refrain from analyzing the class of Perturbed GD type methods, which are designed for saddle escape but not necessarily with a linear rate.
\end{remark}

\section{Convergence Rates to a Minimum}
Now that we have developed an algorithm that escapes saddle neighborhoods in approximately linear time, our goal is to show that it (Algorithm \ref{algo_1}) converges to some local minimum and obtain its rate of convergence.

\subsection{Asymptotic convergence}
First, we show that the iterate sequence $\{\x_k\}$ generated by Algorithm \ref{algo_1} avoids strict saddle points.
\begin{lem}\label{lem6}
The iterate sequence $\{\x_k\}$ generated by Algorithm \ref{algo_1} or any of its subsequence on the class of $\mathcal{C}^2$ gradient and Hessian Lipschitz, Morse functions does not converge to a strict saddle point.
\end{lem}
The proof of this lemma is given in Appendix \ref{Appendix J}.

The next 2 lemmas establish that the function sequence $\{f(\x_k)\}$ converges to a limit within a compact set in $\mathbb{R}^n$ and the trajectory of $\{\x_k\}$ generated by Algorithm \ref{algo_1} encounters at most finitely many saddle points. These lemmas will also be instrumental in providing global rates of convergence.

\begin{lem}\label{lem7}
The sequence $\{f(\x_k)\}$, where $\{\x_k\}$ is the iterate sequence generated by Algorithm \ref{algo_1} on the class of $\mathcal{C}^2$ gradient and Hessian Lipschitz, coercive functions, converges {to a limit value while the iterates $\x_k$ stay in} a compact set in $\mathbb{R}^n$.
\end{lem}
The proof of this lemma is given in Appendix \ref{Appendix J}.

\begin{lem}\label{lem8}
The iterate sequence $\{\x_k\}$ generated by Algorithm \ref{algo_1} on the class of $\mathcal{C}^2$ gradient and Hessian Lipschitz, coercive Morse functions stays within a compact subset of $\mathbb{R}^n$ and encounters at most finitely many saddle points.
\end{lem}
The proof of this lemma is given in Appendix \ref{Appendix J}.

It is needless to state that finite critical points imply isolated critical points\footnote{The condition of isolated critical points means that there is some separation between the critical points.}. {The condition of isolated critical points however holds in general for the class of Morse functions.} We now state the Global Convergence Theorem from \cite{luenberger1984linear} which is instrumental in establishing the asymptotic convergence of Algorithm \ref{algo_1} to a local minimum. Its proof is detailed in section 7.7 of \cite{luenberger1984linear} so we do not present its proof here and directly use this theorem.

\begin{theorem}[\textbf{Global Convergence Theorem~\cite{luenberger1984linear}}]\label{thmglob}
Let $\A$ be an algorithm on a vector space ${X}$, and suppose that, given $\x_0$ the
sequence
$\{\x_k\}_{k=0}^{\infty}$
is generated satisfying
$\x_{k+1} \in \A(\x_k)$.
Let a solution set ${S} \subset {X}$ be given, and suppose
\begin{enumerate}
    \item all points $\x_k$ are contained in a compact set $ {D} \subset {X}$,
    \item there is a continuous function $Z$ on ${X}$ such that:
    \begin{itemize}
        \item if $\x \not\in {S}$, then $Z(\y) < Z(\x)$ for all $\y \in \A(\x)$,
        \item if $\x \in {S}$, then $Z(\y) \leq Z(\x)$ for all $\y \in \A(\x)$,
    \end{itemize}
    \item the mapping $\A$ is closed at points outside ${S}$.
\end{enumerate}
Then the limit of any convergent subsequence of $\{\x_k\}$ is a solution. If under the conditions of the Global Convergence Theorem, $S$ consists of a
single point $\bar{\x}$, then the sequence $\{\x_k\}$ converges to $\bar{\x}$.
\end{theorem}

Using Theorem \ref{thmglob} and {Lemmas} \ref{lem6}-\ref{lem8} we now establish the asymptotic convergence of the sequence $ \{\x_k\}$ to a local minimum.

\begin{theorem}\label{thmglobmain}
The iterate sequence $\{\x_k\}$ generated by Algorithm \ref{algo_1} on the class of $\mathcal{C}^2$ gradient and Hessian Lipschitz, coercive Morse functions has a convergent subsequence that converges to a local minimum. Since the local minimum is a fixed point of Algorithm \ref{algo_1}, the sequence  $\{\x_k\}$ also converges to this local minimum.
\end{theorem}

The proof of this theorem is given in Appendix \ref{Appendix J}.

\subsection{Global rate of convergence}
To develop rate of convergence of the sequence $\{\x_k\}$ to some local minimum $\x^{*}_{optimal}$ of $f(\cdot)$ we {first introduce certain propositions}.

\begin{proposition}\label{proposition1}
{
In some compact domain $\mathcal{U}$, let $\mathcal{S_{*}}$ be the set of all critical points of a function $f(\cdot)$ satisfying assumptions \textbf{A1-A4}, where $\x^*_j \in \mathcal{S_{*}}$ denotes the $j^{\text{th}}$ critical point with $\abs{\mathcal{S_{*}}}= l$ and $l$ is finite. Then the distance between any two critical points of the function $f(\cdot)$ is lower bounded by some $R>0$ where $R>\frac{1}{ \varsigma M } \frac{\bigg((1+\frac{\beta}{L})^2+ \frac{1}{4 (1+\frac{\beta}{L})^2}-\frac{5}{4}\bigg)}{6}$ for $\varsigma>2$, i.e., $\norm{\x^*_i - \x^*_j} \geq R$ for any $\x^*_i$ and $\x^*_j$ in $\mathcal{S_{*}}$ and $\xi$ is chosen such that $\xi \ll R $ where $\xi$ is bounded from Theorem \ref{thm3}.}
\end{proposition}

\begin{proof}
{Since a Morse function on a compact manifold has finitely many critical points \cite{matsumoto2002introduction}, the compact domain $\mathcal{U}$ will have finitely many critical points. The lower bound $R>\frac{1}{ \varsigma M } \frac{\bigg((1+\frac{\beta}{L})^2+ \frac{1}{4 (1+\frac{\beta}{L})^2}-\frac{5}{4}\bigg)}{6}$ for $\varsigma>2$ follows from remark \ref{proprem}.}
\end{proof}

\begin{proposition}\label{proposition2}
{ Let the sequence $\{\x_k\}$ generated by Algorithm \ref{algo_1} on a function $f(\cdot)$ satisfying assumptions \textbf{A1-A4} converges to the local minimum $\x^{*}_{optimal} \in \mathcal{S_{*}}$ from Theorem \ref{thmglobmain} and we have $\norm{\x_0 - \x^{*}_{optimal}} \leq \zeta$ for some $\zeta>0$, where $\x_0$ is the initialization point for Algorithm \ref{algo_1}. Also, without loss of generality we can assume the following condition on the initialization:
		$$ \norm {\x_0 - \x^*_j} \leq \xi$$  for some strict saddle point $\x^*_j$.}
\end{proposition}

\begin{proof}
{From Theorem \ref{thmglobmain} the sequence $\{\x_k\}$ generated by Algorithm  \ref{algo_1} converges to some
 local minimum $\x^{*}_{optimal} $ and this local minimum lies in a compact set in $\mathbb{R}^n$ from Lemma \ref{lem8}. Hence the compact set can be taken to be the compact domain $\mathcal{U}$ from Proposition \ref{proposition1} where we have $\x_0 \in \mathcal{U}$ and $\x^{*}_{optimal} \in \mathcal{S_{*}} \subset \mathcal{U}$. Finally $\norm{\x_0 - \x^{*}_{optimal}} \leq \zeta$ follows from the compactness of $\mathcal{U}$.}
\end{proof}

\begin{proposition}\label{proposition3}
{ For any Morse function, the gradient magnitude at any $\x \in \mathcal{U} \backslash \bigcup_{j=1}^{l} \bar{\mathcal{B}}_{\xi}(\x^*_j)$ for any sufficiently small $\xi$ is lower bounded by some $\gamma$ where we have that:
		$$ \norm{\nabla f(\x)} \geq \gamma = \Omega (\xi)$$ and $\xi$ is bounded from Theorem \ref{thm3}. Further, for any sufficiently small $\epsilon$ where $\epsilon \ll 1$, we can write $\gamma = \Theta( \epsilon^{\upsilon})$ where $\upsilon \in [0,1)$ is a $\xi$ dependent parameter that controls the function geometry in regions away from its critical points\footnote{The value of $\upsilon$ cannot be greater than or equal to $1$ since by definition $\gamma = \Omega(\xi)$ and $\xi > \epsilon$ which implies $\gamma = \Omega(\epsilon)$.}. Hence, very small values of $\upsilon$ imply \emph{well-structured functions}, i.e., functions whose gradients are almost of constant order in regions away from its critical points whereas $\upsilon \uparrow 1$ implies \emph{ill-structured functions}, i.e., functions whose gradients are almost of $\epsilon$ order in regions away from their critical points.   }
\end{proposition}
	
\begin{proof}
{For any Morse function on a compact domain $ \mathcal{U}$, the region away from its critical points defined by $ \mathcal{U} \backslash \bigcup_{j=1}^{l} \bar{\mathcal{B}}_{\xi}(\x^*_j)$ can be categorized into three sub-regions on the basis of gradient magnitudes in these regions. Expressing the gradient magnitudes as function of $\epsilon$ and some $\xi$ where $\epsilon < \xi$ and $\epsilon \ll 1$, we can write $ \norm{\nabla f(\x)} \geq \gamma = \Theta( \epsilon^{\upsilon})$ for any $\x \in \mathcal{U} \backslash \bigcup_{j=1}^{l} \bar{\mathcal{B}}_{\xi}(\x^*_j)$. The parameter $\upsilon \geq 0$ is a function of $\xi$ which controls gradient magnitudes in regions away from the function's critical points. Since $\xi$ is a free variable that is bounded above from Theorem \ref{thm3}, we can choose $\xi$ such that $  \gamma = \Omega(\xi)$ so as to restrict $\upsilon$ in the interval $[0,1)$. Then based on the values of $\upsilon$ we have:
\begin{itemize}
    \item regions with ``large" gradient magnitudes when $\gamma = \Theta( \epsilon^{\upsilon})$ is a constant for $\upsilon \downarrow 0$,
    \item regions with ``moderate to small" gradient magnitudes when $\gamma = \Theta( \epsilon^{\upsilon})$ is moderate or small for $0 < \upsilon < 1$, and
    \item regions with sufficiently ``small" gradient magnitudes when $\gamma = \Theta( \epsilon^{\upsilon})$ is almost of $\epsilon$ order for $ \upsilon \uparrow 1$.
\end{itemize}
Since only the above three cases or their combinations are  possible in regions away from critical points, Proposition \ref{proposition3} captures every possible Morse function. When a function in regions away from its critical points satisfies a combination of two or more of these cases, then $\gamma$ is automatically the minimum of the occurring cases as $\norm{\nabla f(\x)}$ is lower bounded by $\gamma$.}
\end{proof}

Note that from Proposition \ref{proposition3} for $\upsilon $ close to $0$ the quantity {$\gamma$ is of constant order, i.e., $\gamma \approx \Theta(1)$  .} Since $ \gamma = \Omega(\xi)$ and $\gamma$ is of constant order hence we will have that $\gamma \gg \xi$ which implies $ \gamma > \frac{1}{\sqrt{2}}L\xi$ for moderate values of $\xi$ and therefore the \textbf{no-return condition} to such $\xi-$saddle neighborhood holds from Lemma \ref{lemma5}. For all other choices of $\upsilon $ we have $\gamma = \Theta (\epsilon^{\upsilon}) $ and therefore $\xi = \mathcal{O}(\epsilon^{\upsilon})$ where $\epsilon \ll 1$ due to which \textbf{no-return condition} to a small $\xi-$saddle neighborhood holds from Lemma \ref{lemma4}.

Our next lemma establishes the  Lipschitz continuity of $f(\cdot)$ in the compact domain $\mathcal{U}$.
\begin{lem}\label{lipschitzlem1}
As a consequence of Proposition \ref{proposition2}, the function $f(\cdot)$ is Lipschitz continuous in the compact domain $\mathcal{U}$, where the Lipschitz constant is given by $ L \textbf{diam}(\mathcal{U})$.
\end{lem}
\begin{proof}
By the gradient Lipschitz continuity of $f$ for any $\x \in \mathcal{U}$ where $\mathcal{U}$ has atleast one critical point $\x^*$ of $f$, we have the following bound:
\begin{align}
   \norm{\nabla f(\x)} &\leq L \norm{\x - \x^*} \leq L\textbf{diam}(\mathcal{U}) \\
   \implies \sup_{\x \in \mathcal{U}} \norm{\nabla f(\x)} &\leq L\textbf{diam}(\mathcal{U}).
\end{align}
	 From the Mean value theorem, for any $\x, \y$ in $\mathcal{U}$ we have that:
\begin{align}
    f(\x) - f(\y) &\leq \sup_{\x \in \mathcal{U}} \norm{\nabla f(\x)} \norm{\x -\y} \leq  {L} \textbf{diam}(\mathcal{U})  \norm{\x - \y}.  \label{functionlipschitz}
\end{align}
\end{proof}

{The above lemma will help us in developing global rates of convergence in terms of the iterate sequence $\{\x_k\}$. In the absence of this lemma global rates of convergence can still be obtained however such rates would be in terms of the function value sequence $\{f(\x_k)\}$. Since the condition $\x_k \to \x^*_{optimal}$ implies strong convergence whereas the condition $f(\x_k) \to f(\x^*_{optimal})$ implies weak convergence, lemma \ref{lipschitzlem1} becomes absolutely necessary for establishing a stronger convergence result.}

{Now that we are interested in developing convergence rates for the iterate sequence, we need a handle on the largest distance our iterate $\x_k$ can possibly travel from the initialization $\x_0$  within some compact domain $\mathcal{U}$ before converging to a neighborhood of $\x^*_{optimal}$. Quantifying this distance is essential since the total number of iterations or the travel time of any trajectory depends on how much distance it travelled before converging to some local minimum neighborhood. In the best case the trajectory could take a bee line path between $\x_0$ and  $\x^*_{optimal}$ whereas in the worst case a trajectory could possibly travel much farther than $\x^*_{optimal}$ before turning back and eventually converging. The next theorem provides a precise bound on the farthest distance any worst case trajectory could travel to before returning back for good. In doing so it also provides a handle on the number of saddle point neighborhoods encountered in the path of such trajectory.}

\begin{theorem} \label{thm4}
{On a function satisfying assumptions \textbf{A1-A4}, the trajectory generated from the iterate sequence $\{\x_k\}$ by Algorithm \ref{algo_1} that has escaped some ball $\mathcal{B}_{R_0}(\x^*_0)$ cannot escape the ball $\mathcal{B}_{R_{\omega}}(\x^*_0) \supset \mathcal{B}_{R_0}(\x^*_0)$ if it has to re-enter the ball $\mathcal{B}_{R_0}(\x^*_0)$ in finite number of iterations, where we have that $\x_0 \in \mathcal{B}_{\xi}(\x^*_0)$ and $\x^*_0 \in \mathcal{S}_*$ is a strict saddle point {provided that} the radius $R_{\omega}$ satisfies the condition:}
\begin{align}
    R_{\omega} & \leq R_0 +  2L\textbf{diam}(\mathcal{U}) \frac{R_0}{\gamma}+ N_0 \bigg( \frac{1}{\beta} + \frac{L}{2 \beta^2}\bigg) \frac{L^2\epsilon^2}{\gamma} +  N_0 ({K}_{exit}+K_{shell}) \xi
\end{align}
{where $N_0 = \frac{2L\textbf{diam}(\mathcal{U}) \frac{R_0}{R}}{\bigg( \frac{\gamma}{2}-\bigg( \frac{1}{\beta} + \frac{L}{2 \beta^2}\bigg) \frac{L^2\epsilon^2}{R} - \gamma ({K}_{exit}+K_{shell}) \frac{\xi}{R}\bigg)}$ is an upper bound on the number of {strict saddle neighborhoods of radius $\xi$} encountered by the trajectory of $\{\x_k\}$. Note that here $K_{exit}$ is upper bounded by \eqref{linearexittimebound}, $K_{shell} $ is upper bounded by Theorem \ref{thm3} and the compact domain $\mathcal{U}$ contains the ball $\mathcal{B}_{R_{\omega}}(\x^*_0) $, i.e.,  $\mathcal{U} \supset \mathcal{B}_{R_{\omega}}(\x^*_0)  $.}
\end{theorem}

The proof of this theorem is given in Appendix \ref{Appendix F}.

\begin{remark}
In order to characterize the convergence rate for Algorithm~\ref{algo1} we need to focus on the worst-case trajectories that can be generated by it. Theorem~\ref{thm4} helps capture the behavior of such worst-case trajectories by finding {the radius of the largest possible ball whose boundary can be reached} by such trajectories.
\end{remark}

We are now ready to state the final theorem of this work which quantifies the convergence rate of Algorithm \ref{algo_1} to some $\epsilon$-neighborhood of a local minimum.
\begin{theorem} \label{thm5}
{On a function satisfying assumptions \textbf{A1-A4}, the total time $K_{max}$ for the trajectory of $\{\x_k\}$ generated from Algorithm \ref{algo_1} to converge to a sufficiently small $\epsilon$-neighborhood of a local minimum $\x^*_{optimal}$ is bounded by:}
\begin{align}
    K_{max}
           & <  T\bigg( {K}_{exit}+K_{shell}  \bigg)+ 4L\textbf{diam}(\mathcal{U}) \frac{\zeta L}{\gamma^2} +  2T\bigg( \frac{1}{\beta} + \frac{L}{2 \beta^2}\bigg) \frac{\epsilon^2  }{\gamma^2} + \frac{\log\bigg(\frac{\xi}{\epsilon} \bigg)}{\log\bigg(1- \frac{\beta}{L} \bigg)^{-1}},
\end{align}
where $ T < \frac{2L\textbf{diam}(\mathcal{U}) \frac{\zeta}{R}}{\bigg( \frac{\gamma}{2}-\bigg( \frac{1}{\beta} + \frac{L}{2 \beta^2}\bigg) \frac{L^2\epsilon^2}{R} - \gamma ({K}_{exit}+K_{shell}) \frac{\xi}{R}\bigg)}$ is the total number of $\xi$ radius saddle neighborhoods encountered, $\epsilon$ and $\xi$ are bounded from Theorems \ref{thm1}, \ref{thm3} {and $\x_0$ is initialized in a $\xi$-neighborhood of any strict saddle point.}

\end{theorem}
The proof of this theorem is given in Appendix \ref{Appendix F}.

{In terms of the order notation, using \eqref{linearexittimebound} and \eqref{shellratelaw} followed by choosing some sufficiently small $\epsilon$ where $\epsilon$ is bounded by theorem \ref{thm1}, some moderately small $\xi$ from Propositions \ref{proposition1}, \ref{proposition3} and substituting $\gamma = \Theta( \epsilon^{\upsilon})$,  $K_{max}$ has the following {dependency on $\varepsilon$}: 
\begin{align}
    K_{max} = \mathcal{O}\bigg(T\log\bigg(\frac{1}{\epsilon}\bigg) \bigg)  + \mathcal{O} \bigg(T\log \bigg(\frac{\xi}{\epsilon}\bigg)\bigg)  + \mathcal{O}\bigg(\frac{1}{\epsilon^{2{\upsilon}}}\bigg)  \label{complexityrate}
\end{align}
 where $T= \mathcal{O}\bigg(\frac{1}{\epsilon^{\upsilon}}\bigg)$ is the number of saddles encountered and ${\upsilon}\in [0,1)$ is a parameter of the function $f(\cdot)$ defined in Proposition \ref{proposition3} which controls the function geometry in regions away from its critical points. The third term on the right hand side of \eqref{complexityrate} is $\mathcal{O}\bigg(\frac{1}{\epsilon^{2{\upsilon}}}\bigg)$ which quantifies the travel time of the trajectory in the region $ \mathcal{U} \backslash \bigcup_{j=1}^{l} \bar{\mathcal{B}}_{\xi}(\x^*_j)$ (for details, see proof of Theorem \ref{thm5} in Appendix \ref{Appendix F}).
}

{Observe that the dominant term in the expression of convergence rate from \eqref{complexityrate} is $\mathcal{O}\bigg(\frac{1}{\epsilon^{2{\upsilon}}}\bigg)$ where $\upsilon \in [0,1)$. Compared to the state of the art\footnote{{While Table \ref{table:2} lists various state-of-the-art algorithms, all those listed works except \cite{jin2017escape} use either accelerated gradient methods or Newton method as their base algorithm. Hence for sake of fairness, the rate comparison is done only with the Perturbed GD method of \cite{jin2017escape}.}} Perturbed GD method \cite{jin2017escape} which has a convergence rate of order $\mathcal{O}\bigg(\frac{1}{\epsilon^2} \log^4\bigg(\frac{1}{\epsilon^2}\bigg) \bigg)$, there is no poly-logarithmic dependence in our term $\mathcal{O}\bigg(\frac{1}{\epsilon^{2{\upsilon}}}\bigg)$ and in the worst case this term is still better than $\mathcal{O}\bigg(\frac{1}{\epsilon^{2}}\bigg)$ provided $\epsilon$ and $\xi$ are chosen to be sufficiently small from {Proposition} 
\ref{proposition3}. In particular, for well-structured functions which have large gradient magnitudes in regions away from critical points, we will have $\frac{1}{\epsilon^{2{\upsilon}}} \ll \frac{1}{\epsilon^{2}}$ thereby yielding a superior convergence rate to sufficiently small neighborhood of a local minimum. This improvement over the rate $\mathcal{O}\bigg(\frac{1}{\epsilon^{2}}\bigg)$ is only possible because of Theorem \ref{thm3} which gives a linear travel time within $\xi$ radius saddle neighborhoods. In the absence of Theorem \ref{thm3}, we would not have $\xi$ radius saddle neighborhoods within which fast travel is possible. Then we only have a much smaller $\epsilon$ radius saddle neighborhood from Theorem \ref{thm1} and outside such neighborhood, the travel time of the trajectory will be $\mathcal{O}\bigg(\frac{1}{\epsilon^{2}}\bigg)$. Existence of larger saddle neighborhoods from Theorem \ref{thm3} enables us to invoke Proposition \ref{proposition3} using which we can choose our $\epsilon$ sufficiently small and a certain $\xi$ so that the gradient magnitude in the region $ \mathcal{U} \backslash \bigcup_{j=1}^{l} \bar{\mathcal{B}}_{\xi}(\x^*_j)$ is lower bounded by $\gamma = \Omega(\xi) = \Theta(\epsilon^{\upsilon})$ for some $\upsilon \in [0,1)$. Then we get the improved rate of $\mathcal{O}\bigg(\frac{1}{\epsilon^{2{\upsilon}}}\bigg)$ in the region $ \mathcal{U} \backslash \bigcup_{j=1}^{l} \bar{\mathcal{B}}_{\xi}(\x^*_j)$ for our trajectory. It should however be noted that the value of parameter $\upsilon$ is not known explicitly since it depends on the function landscape in the region $ \mathcal{U} \backslash \bigcup_{j=1}^{l} \bar{\mathcal{B}}_{\xi}(\x^*_j)$. Specifying certain value for $\upsilon$ would require more assumptions on the function landscape which is beyond the scope of this work.}

\section{Numerical Results}
To test the efficacy of the proposed method, we simulate Algorithm \ref{algo_1} on two different problems, a modified Rastrigin function and a low-rank matrix factorization problem.

\subsection{Modified Rastrigin Function} The Rastrigin function is a nonconvex function that was first proposed in \cite{rastrigin1974systems} and the generalized versions appeared in \cite{muhlenbein1991parallel, hoffmeister1990genetic}. The function is given by
\begin{align}
    f(\x) & = An + \sum\limits_{i=1}^{n}(x_i^2 -  \cos{(2\pi x_i)}), \label{rastrigin}
\end{align}
where $A=10$ and $x_i \in [-5.12,5.12]$, and $f(\cdot)$ has a global minimum at $\x=\mathbf{0}$. In this section, we use a modified version of \eqref{rastrigin} given by:
\begin{align}
    f(\x) = \sum\limits_{i=1}^{n} a_i\cos{(b_i x_i)}, \label{rastrigin1}
\end{align}
where \eqref{rastrigin1} differs from \eqref{rastrigin} in the sense that \eqref{rastrigin1} does not have a quadratic term added to it (hence possibly some local minima are global minima). The modified formulation of the Rastrigin function is analytic and locally Morse at its critical point $\x^* = \mathbf{0}$ for the choice of parameters given below. It satisfies all the listed assumptions \textbf{A1-A4} in this work except coercivity due to the fact that we removed the quadratic growth term from it. {In particular, for the formulation \eqref{rastrigin1} we will have $L \leq \sum_i \abs{a_i b_i}$, $M \leq \sum_i\abs{a_i b_i^2}$ and $\beta$, $\delta$ are evaluated from the simulations.} This particular example highlights the fact that convergence to a local minimum is possible even without the coercivity assumption.

For simulations, we set $a_i = 1$ for $i=1$ and $a_i = -1$ elsewhere, $b_i = 1$ for $1 \leq i \leq \floor*{\frac{n}{2}}$ and $b_i = 0.4$ for $\floor*{\frac{n}{2}} +1 \leq i \leq n$. The point $\x^* = \mathbf{0}$ is a strict saddle point in our case and the initialization of the proposed CCRGD algorithm (Algorithm \ref{algo_1}) and the gradient descent (GD) method is done in an $\epsilon$ neighborhood of $\x^*$. Specifically, the iterate {$\x_0$} is initialized in an $\epsilon$ neighborhood of the strict saddle point $\x^*$ with a {very small unstable subspace projection value}, {i.e., $\frac{\norm{\pi_{\mathcal{E}_{US}}(\x_0-\x^*)}}{\norm{\x_0-\x^*}}<10^{-4}$ where $ \mathcal{E}_{US}$ is the unstable subspace of $\nabla^2f(\x^*)$} and the initialization point is same for both methods. In addition, the step-size for both methods is set to $\alpha =\frac{1}{L}$, where $L$ is the maximum absolute eigenvalue of the Hessian {we estimated} in the saddle neighborhood.

The results of our simulations are reported in Figures \ref{fig11}(a)--(d), where each subfigure has a total of three plots for a different combination of $(n,\epsilon)$. In each of the subfigures, the top-left plot shows that the gradient norm of the proposed CCRGD method first increases and then decreases while the GD method struggles to increase its gradient norm for many iterations. The top-right plot in each subfigure shows the initial and final eigenvalues of the Hessian at an iterate generated by the two methods, while the blue stem subplot in there shows the eigenvalue spectrum at the initialization (which is the same for both methods). It can be seen from the two plots in each subfigure that the GD method fails to converge to a second-order stationary point in the given number of iterations, while the CCRGD method easily converges to a local minimum.

Finally, the bottom plot in each subfigure shows the evolution of distance of the iterate from the initialization for the two methods. This plot also marks the iteration where the CCRGD method first exited the initial saddle neighborhood (this iteration index is the ``First Exit Time'') and also marks those iteration indices where the CCRGD method invoked the second-order Step \ref{algocurvaturecondition} in Algorithm \ref{algo1}.

\begin{figure}
\centering
\begin{tabular}{cc}
    \includegraphics[width=2.7in]{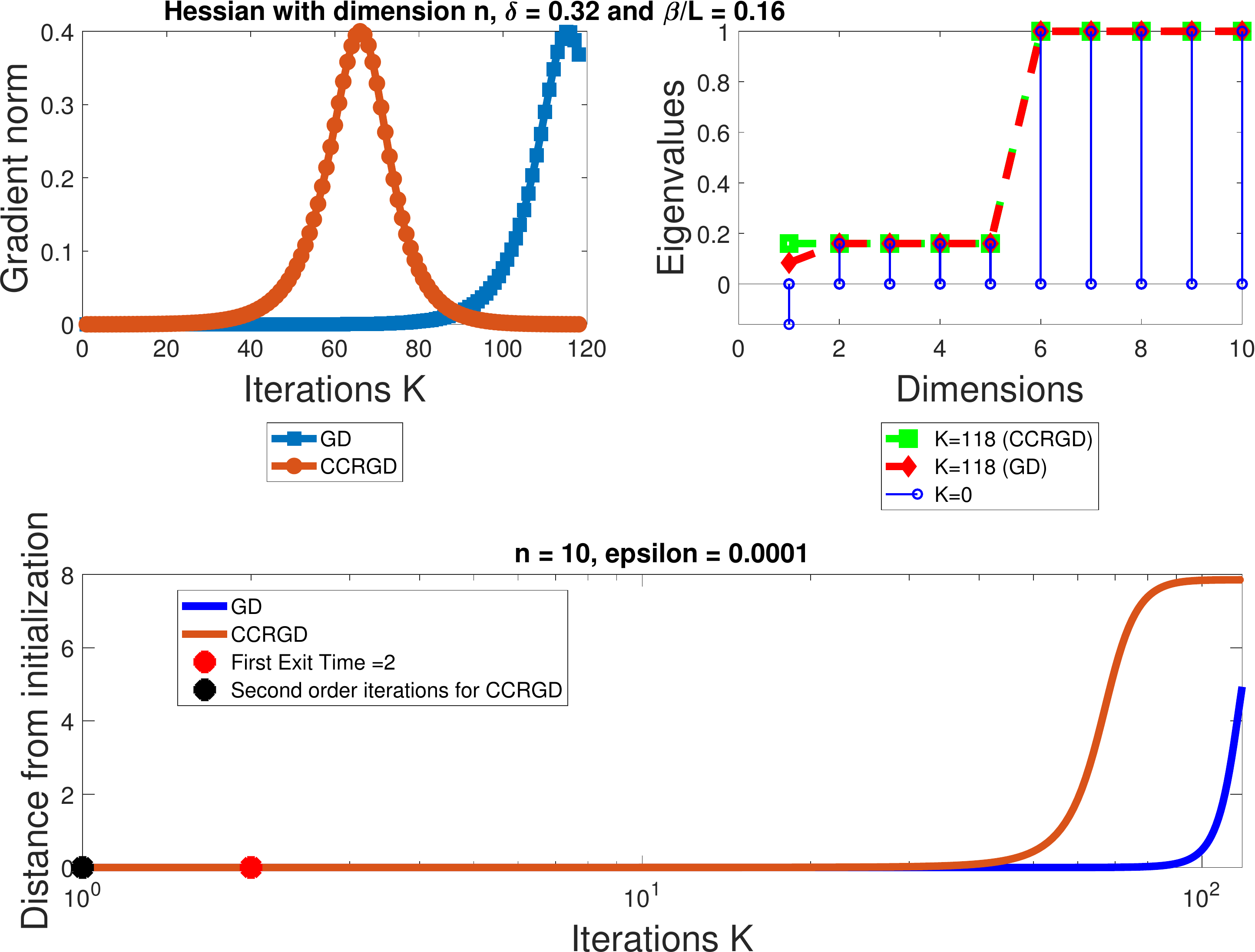} & \includegraphics[width=2.7in]{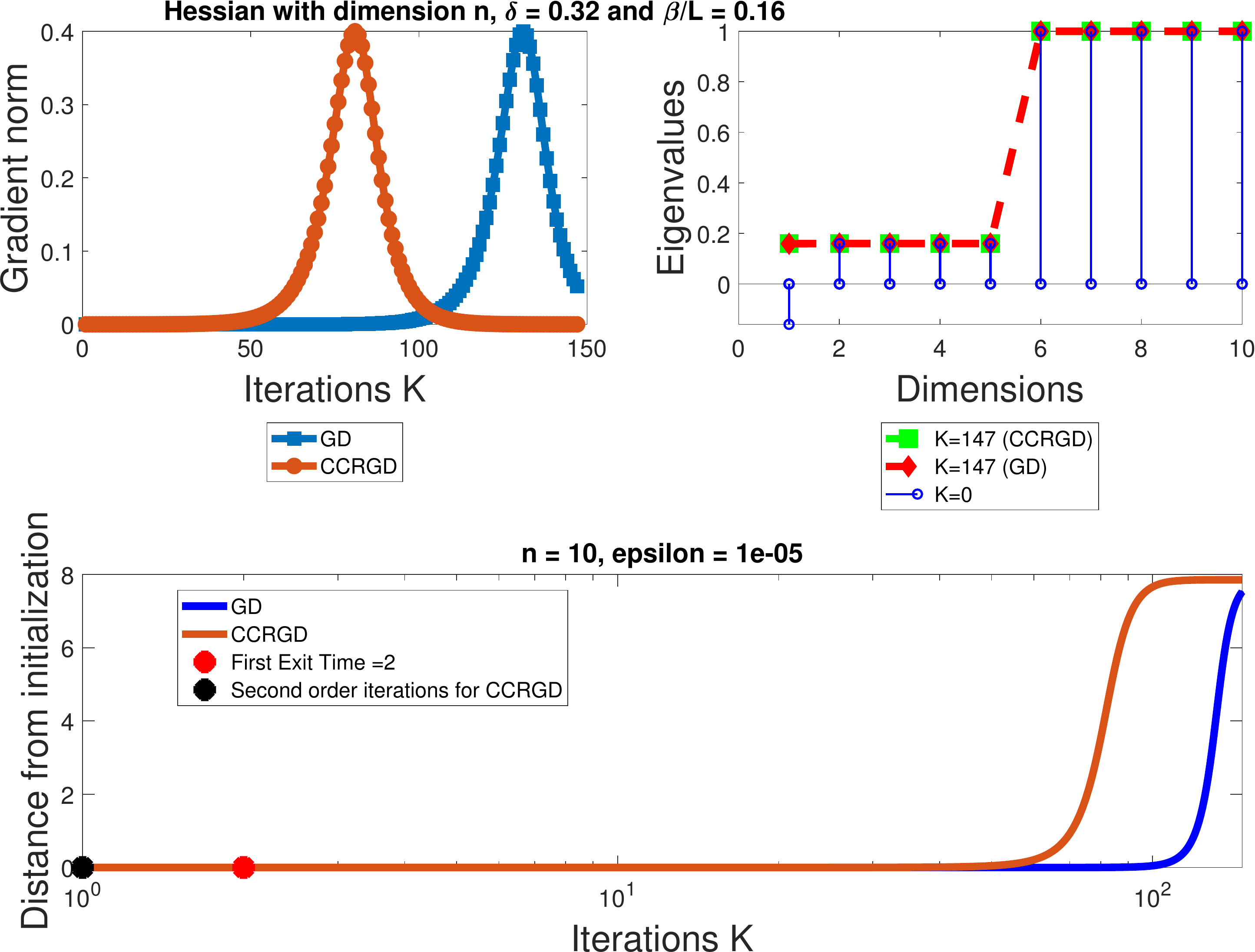} \\
    (a) & (b) \\
    & \\
    \includegraphics[width=2.7in]{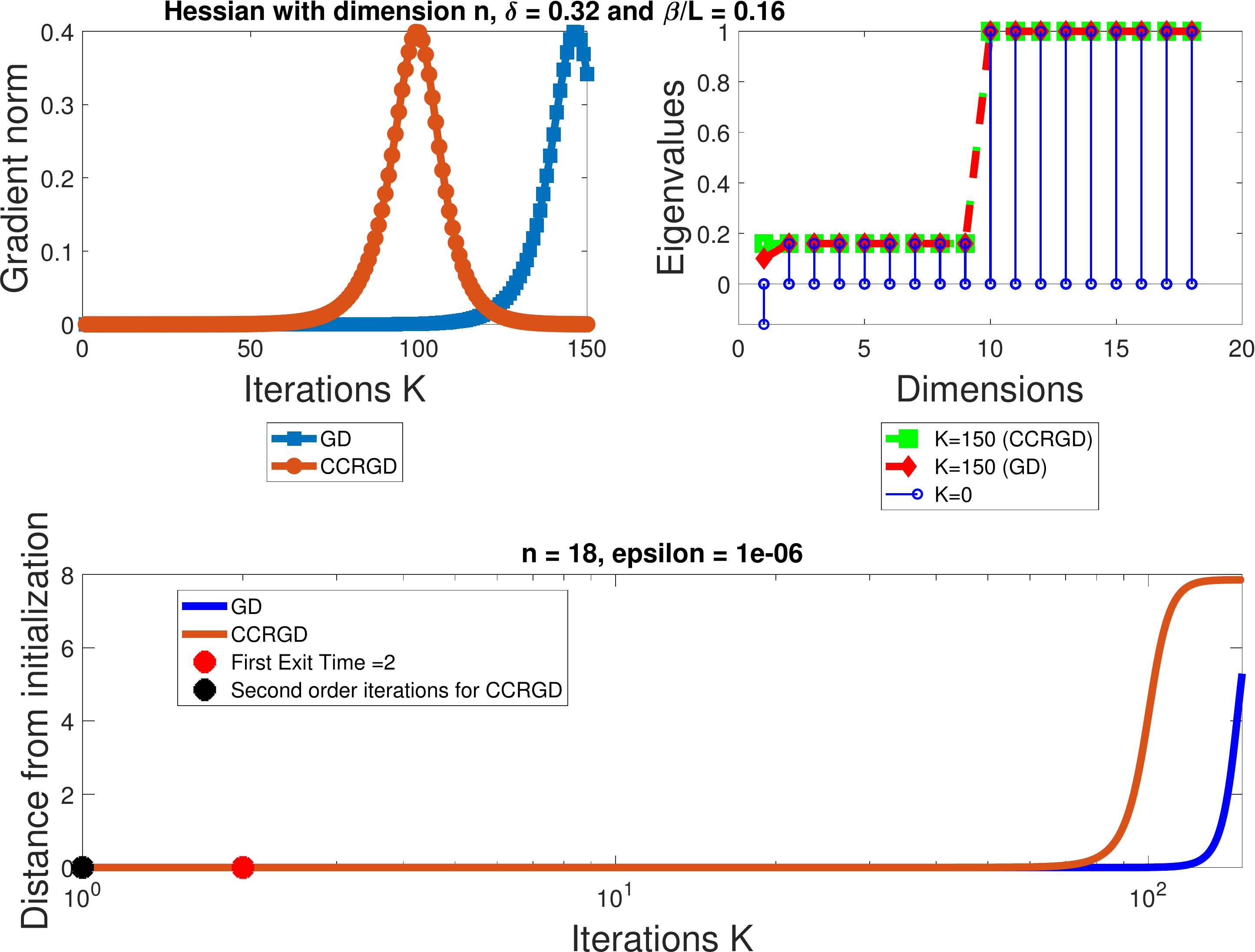} & \includegraphics[width=2.7in]{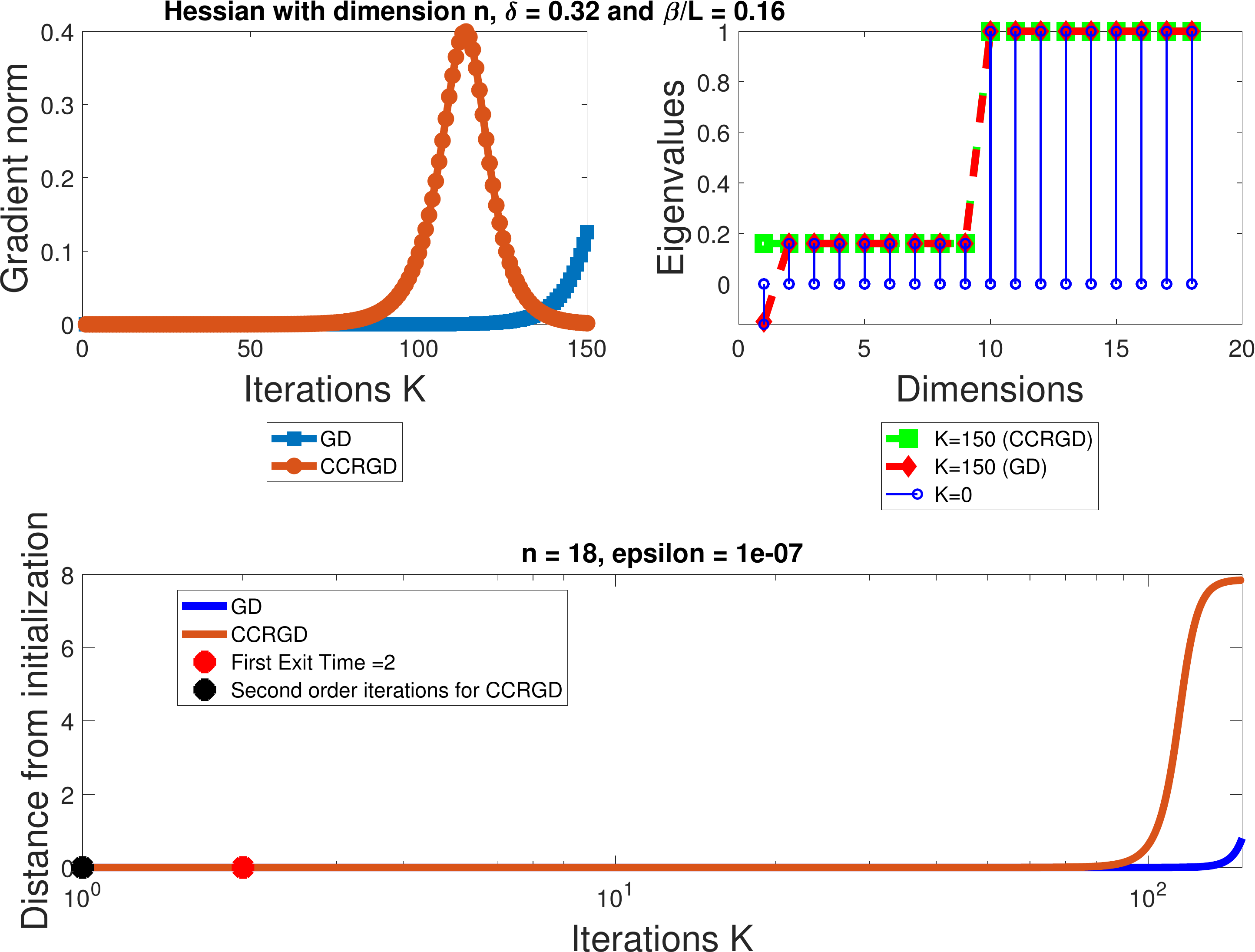} \\
    (c) & (d)
\end{tabular}
\caption{Simulation results on the modified Rastrigin function for various values of $n$ and $\epsilon$.}
\label{fig11}
\end{figure}

\subsection{Low-Rank Matrix Factorization}
The objective function for the problem in consideration is as follows:
\begin{align}
    f(\X_1,\X_2) = \frac{1}{4}\norm{\mathbf{M} - \X_1\X_2^T }^2_F + \varpi_1 \norm{\X_1}^2_F + \varpi_2 \norm{\X_2}^2_F, \label{simulate1}
\end{align}
where $\mathbf{M}\in \mathbb{R}^{n_1 \times n_2}$, $\X_1\in \mathbb{R}^{n_1 \times r}$ and $\X_2\in \mathbb{R}^{n_2 \times r}$ such that $r \leq \min \{n_1,n_2\}$ is the rank of matrix $\mathbf{M}$.

\begin{figure}
\centering
\begin{tabular}{cc}
     \includegraphics[width=2.7in]{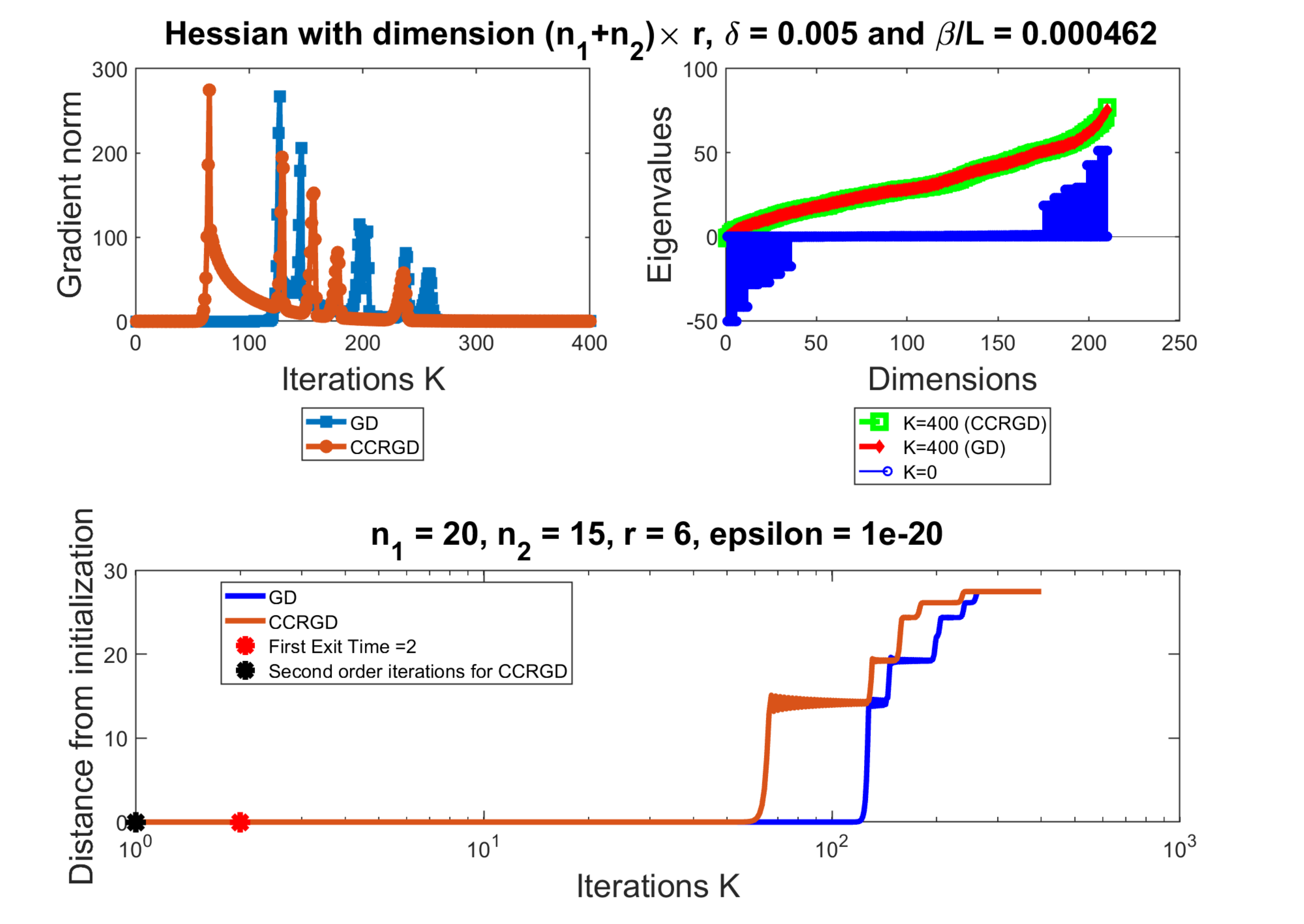} &  \includegraphics[width=2.7in]{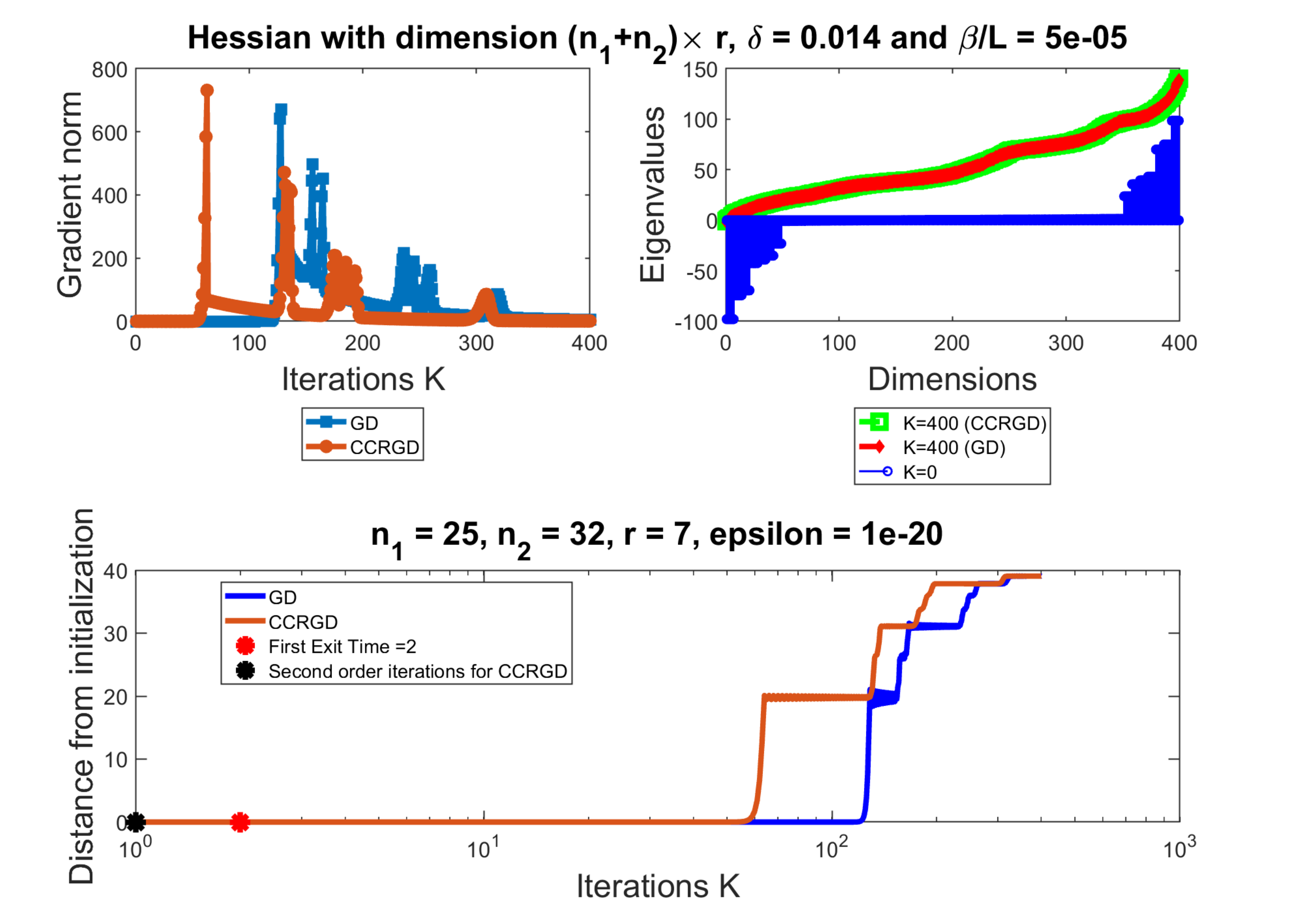}\\
     (a) & (b) \\
     & \\
     \includegraphics[width=2.7in]{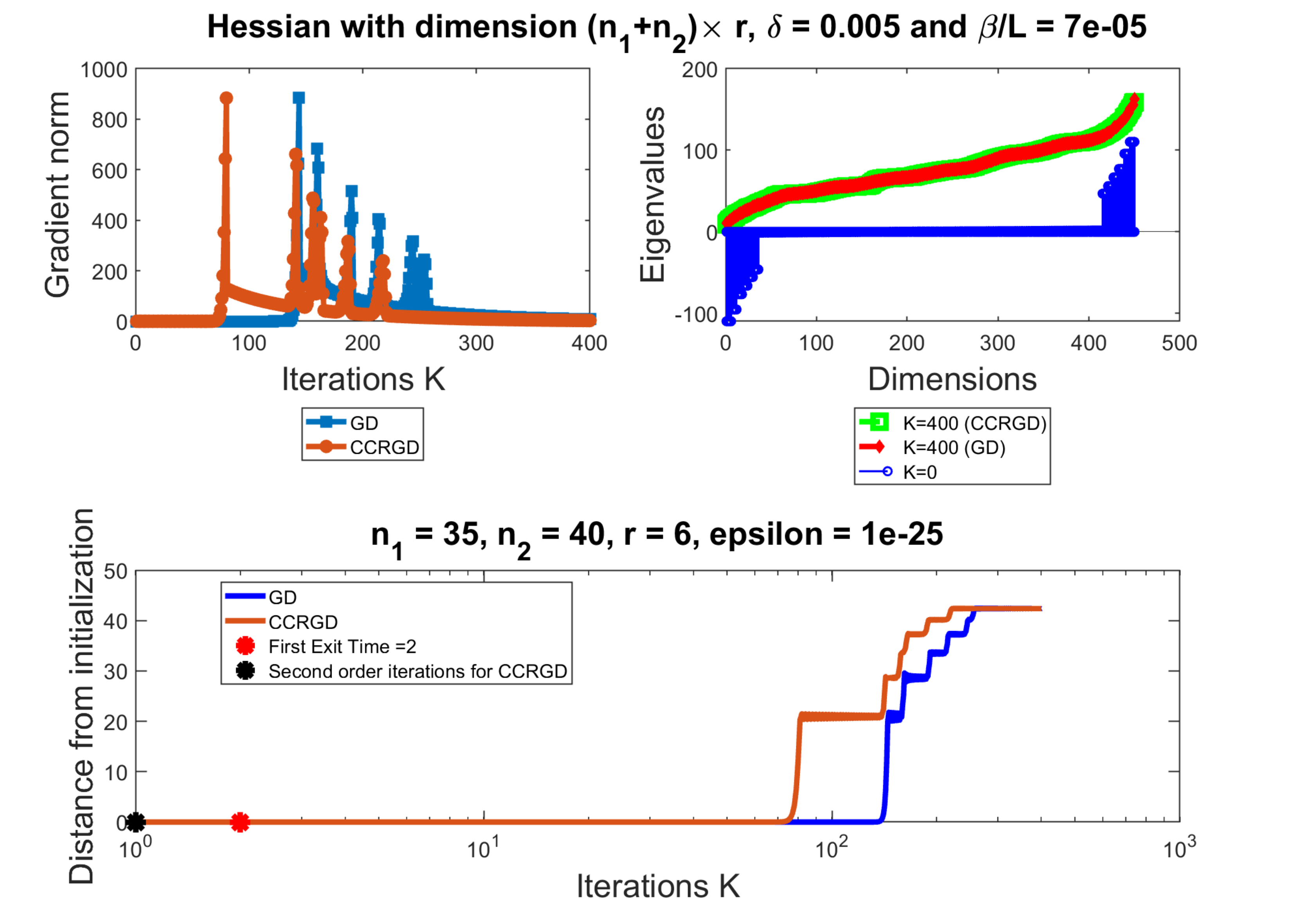} &  \includegraphics[width=2.7in]{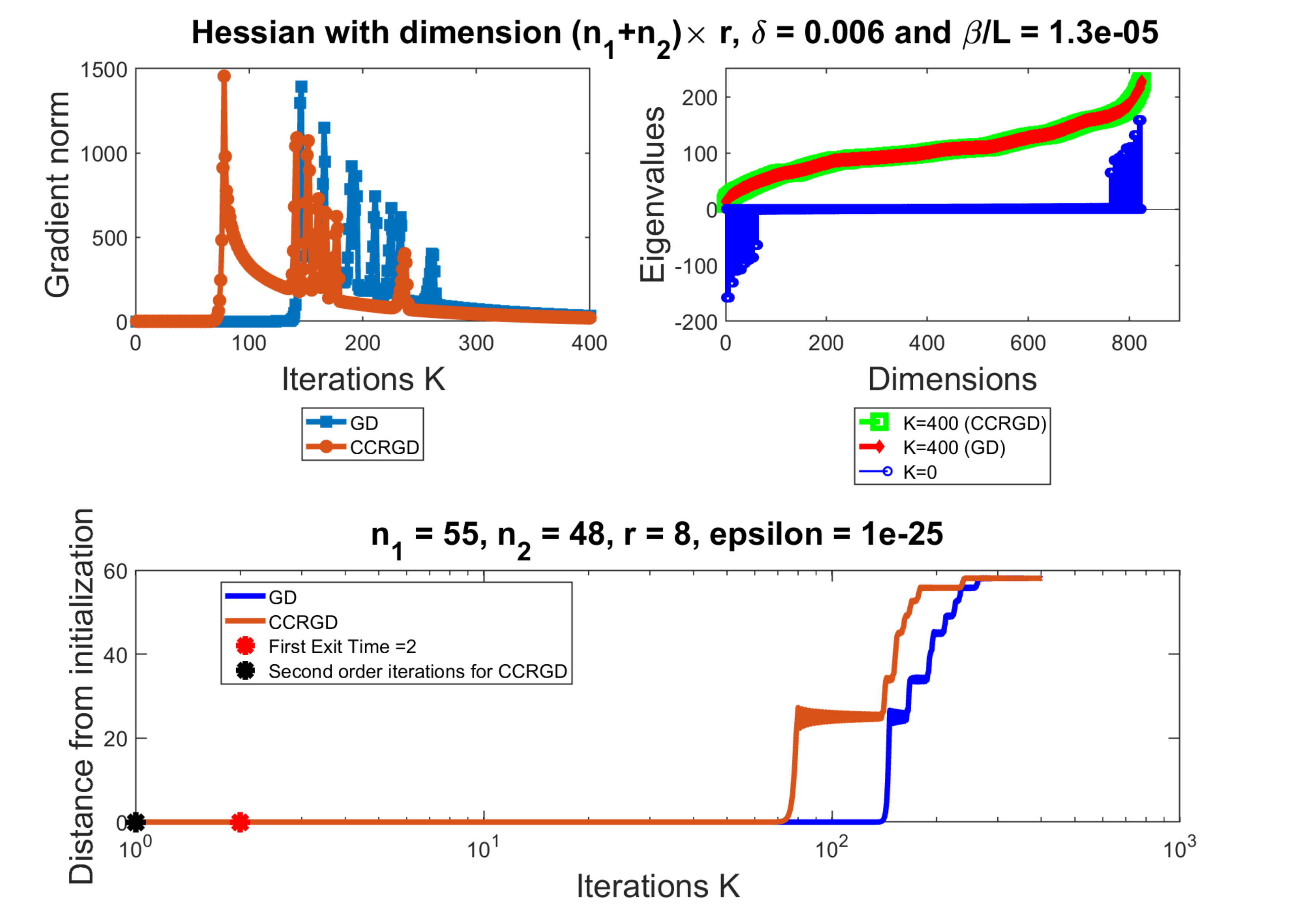}\\
     (c) & (d) \\
     & \\
     \includegraphics[width=2.7in]{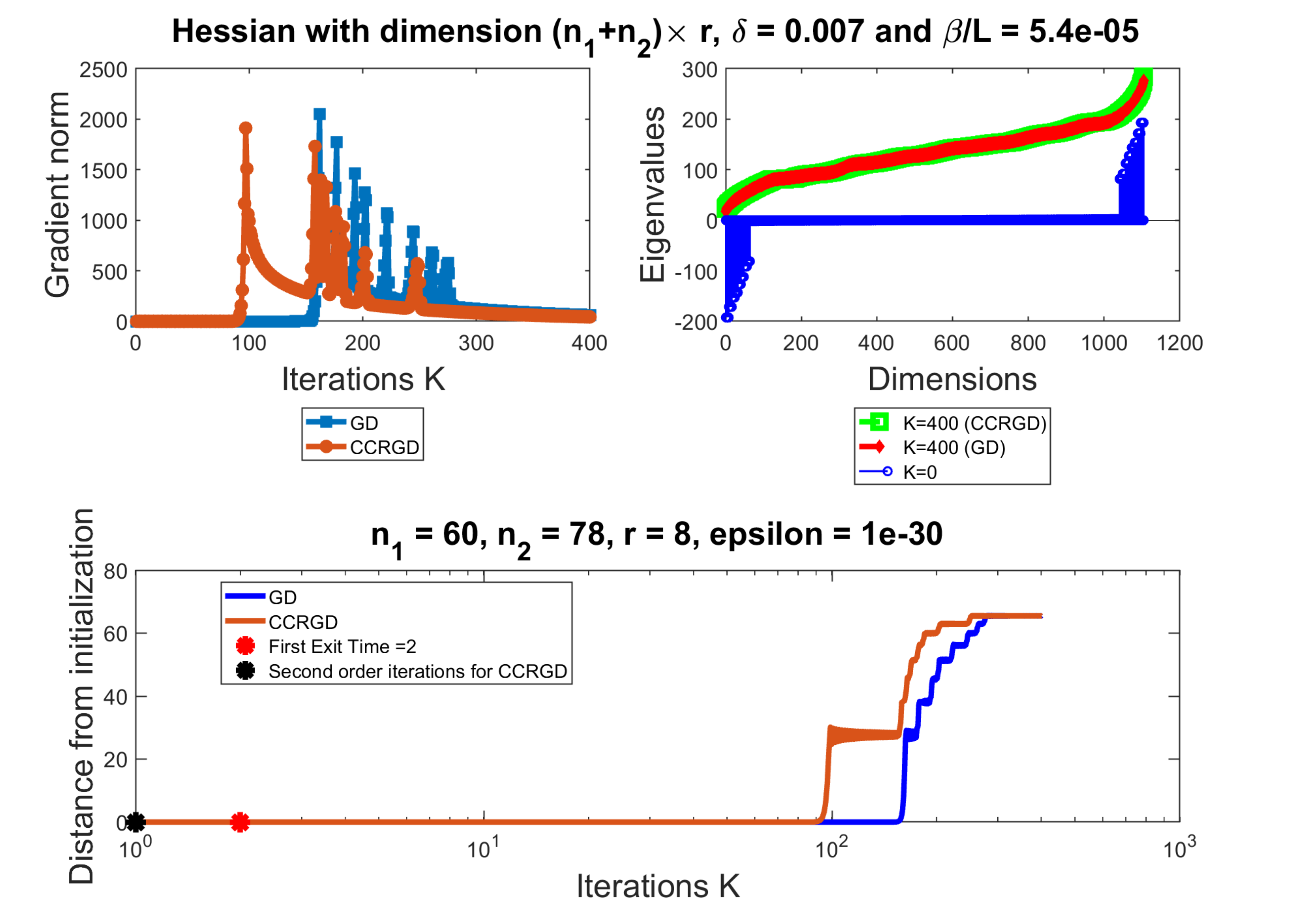} &  \includegraphics[width=2.7in]{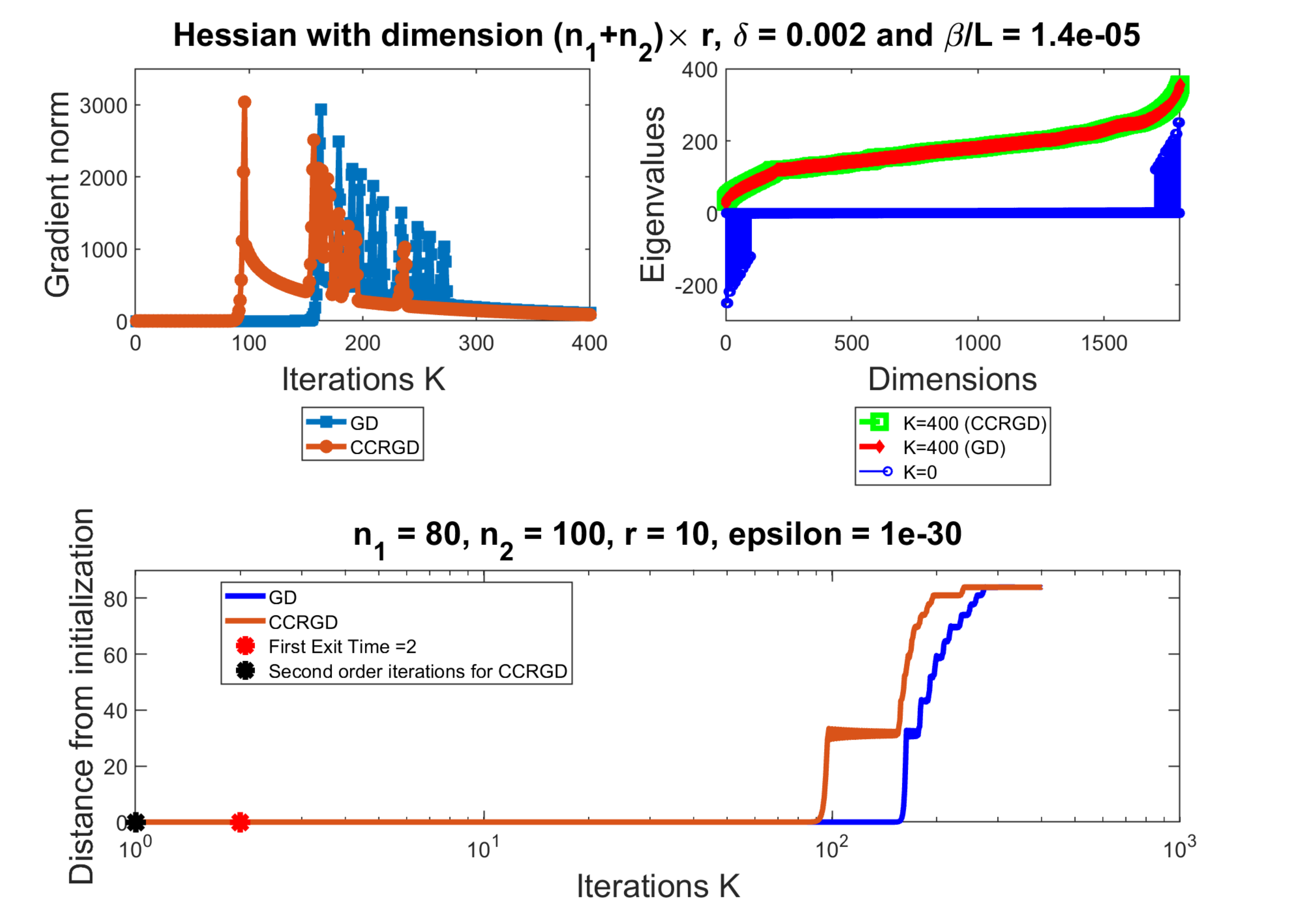}\\
     (e) & (f)
\end{tabular}
\caption{Simulation results of a low-rank matrix factorization problem for various values of $n_1$, $n_2$, $r$, and $\epsilon$.}
\label{fig12}
\end{figure}

To simplify the problem structure so as to make \eqref{simulate1} some function of a single variable $\X$, let $\X_1$ and $\X_2$ be blocks of the variable $\X$ such that
$$     \X = \left[\begin{array}{c}
\X_1\\
\X_2\\
\end{array}\right],
$$ where we have $\X_1 = \B_1 \X$ and $\X_2 = \B_2 \X$ with $ \B_1 = \left[\begin{array}{cc}
\I_{n_1 \times n_1} \hspace{0.2cm}\vert\hspace{0.2cm} \mathbf{0}_{n_1 \times n_2}\\
\end{array}\right]$ and $ \B_2 = \left[\begin{array}{cc}
\mathbf{0}_{n_2 \times n_1} \hspace{0.2cm}\vert\hspace{0.2cm} \I_{n_2 \times n_2}\\
\end{array}\right]$. Here $\I_{n_1 \times n_1}$, $\I_{n_2 \times n_2}$ represent the identity matrices and $\mathbf{0}_{n_1 \times n_2}$, $\mathbf{0}_{n_2 \times n_1}$ represent the null rectangular matrices of appropriate dimensions. Using this change of variable, \eqref{simulate1} can be written as a function of $\X$:
\begin{align}
     f(\X) =  \frac{1}{4}\norm{\mathbf{M} - \B_1\X\X^T\B_2^T }^2_F + \varpi_1 \norm{\B_1\X}^2_F + \varpi_2 \norm{\B_2\X}^2_F. \label{simulate2}
\end{align}
Next, $\nabla f(\X)$ can be given as follows:
\begin{align}
   \nabla f(\X) = &\frac{1}{2}(\B_1^T\B_1 \X \X^T \B_2^T\B_2 + \B_2^T\B_2 \X \X^T \B_1^T\B_1 )\X - \frac{1}{2}(\B_2^T \mathbf{M}^T \B_1 + \B_1^T \mathbf{M} \B_2)\X + \nonumber\\ &\qquad 2 \varpi_1 \B_1^T\B_1 \X +  2 \varpi_2 \B_2^T\B_2 \X. \label{gradientsimulate}
\end{align}
Since the gradient in \eqref{gradientsimulate} is a matrix, hence the corresponding Hessian will be a tensor, whereas our analysis assumes the Hessian to be a matrix. To circumvent this problem, we make use of \cite[Theorem 9]{magnus1985matrix} by vectorizing matrix $\X$ so that $\nabla^2 f(vec(\X))$ is a Jacobian matrix.

The closed form expression for the Jacobian is as follows:
\begin{align}
  \hspace{-1cm}   \nabla^2 f(vec(\X)) &= \frac{1}{2}\bigg( ((\X^T\B_2^T\B_2) \otimes \I_{n \times n})( (\X \otimes \I_{n \times n}) (\I_{r \times r} \otimes (\B_1^T\B_1)) + (\I_{n \times n} \otimes (\B_1^T \B_1 \X))) \nonumber \\ &\qquad + (\I_{r \times r} \otimes (\B_1^T\B_1 \X \X^T) )(\I_{r \times r} \otimes (\B_2^T\B_2))  \bigg)  + \frac{1}{2}\bigg( ((\X^T\B_1^T\B_1) \otimes \I_{n \times n})( (\X \otimes \I_{n \times n}) (\I_{r \times r} \otimes (\B_2^T\B_2)) \nonumber \\ &\qquad  + (\I_{n \times n} \otimes (\B_2^T \B_2 \X))) + (\I_{r \times r} \otimes (\B_2^T\B_2 \X \X^T) )(\I_{r \times r} \otimes (\B_1^T\B_1))  \bigg)  \nonumber \\ &\qquad - \frac{1}{2}\bigg( \I_{r \times r} \otimes ( \B_2^T \mathbf{M}^T \B_1 + \B_1^T \mathbf{M} \B_2 ) \bigg) + 2\bigg( \I_{r \times r} \otimes (\varpi_1 \B_1^T\B_1 + \varpi_1 \B_2^T\B_2) \bigg), \label{simulatehessian}
\end{align}
where $n=n_1 + n_2$. For simulations, matrix $\mathbf{M}$ was generated randomly using the relation $$ \mathbf{M} = \mathbf{U}_1 \mathbf{U}_2^T + \varrho^2 \mathbf{N},$$ where $ \mathbf{U}_1 \in \mathbb{R}^{n_1 \times r}$, $ \mathbf{U}_2 \in \mathbb{R}^{n_2 \times r}$ and the entries of these matrices were independently sampled from a standard normal distribution. Matrix $ \mathbf{N} \in \mathbb{R}^{n_1 \times n_2}$ is the additive noise generated from a normal  distribution whose variance is scaled by $\varrho$. The formulation \eqref{simulate1} is analytic and {the Hessian at the critical point $\X = \mathbf{0}$ is invertible} but the function at $\X = \mathbf{0}$ has a poor condition number {which will be evident from the simulations}. It is coercive, Hessian Lipschitz and satisfies all the assumptions in this work. The highly ill conditioned nature of the problem however could possibly make the function non-Morse at other critical points. Since the closed form expression of the Hessian in \eqref{simulatehessian} is very complex, we steer away from the computation of its eigenvalues at critical points other than $\X= \mathbf{0}$.

For the experiments, we use $\varpi_1 = \varpi_2 = 0.5$, $\varrho = 0.5$, and step-size $\alpha = \frac{1}{L}$ where $L = \lambda_{max}(\nabla^2 f(vec(\X)))$. Also, for the particular selection of parameters, $\X=\mathbf{0}$ is a strict saddle point. Hence, $\X$ is initialized on the boundary of ball $\mathcal{B}_{\epsilon}(\mathbf{0})$ and $\epsilon$ is varied in the simulations along with $n_1, n_2, r$. Finally, the proposed method is plotted against the standard gradient descent method where the metric is $\norm{\X_k - \X_{init}}_F$ with $\X_{init}$ being the common initialization for the two methods.

The simulation results for Algorithm \ref{algo_1} are presented in Figures \ref{fig12}(a)--(f) and comparisons are made with the GD method. For the sake of uniformity, the plots within each subfigure of Figure \ref{fig12} follow the same convention as the plots within each subfigure of Figure \ref{fig11}. From the plots, it is evident that the functions are not well-conditioned for different cases and both GD and CCRGD encounter cascaded saddles. Still CCRGD performs remarkably better than GD in terms of convergence to a local minimum, which is evident from the eigenvalues of the Hessian at final iterate. Moreover in every case CCRGD is able to escape the first saddle neighborhood much more faster than GD due to a single second order step which is invoked only once over all iterations.

\section{Conclusion}
This work focuses on the global analysis of gradient trajectories for a class of nonconvex functions that have strict saddle points in their geometry. Building on top of the results from our earlier work \cite{dixit2022exit}, sufficient boundary conditions are developed here that guarantee approximate linear exit time of gradient trajectories from saddle neighborhoods. Further, the gradient trajectories are analyzed in an augmented saddle neighborhood and it is proved that the trajectories exhibit sequential monotonicity. Using this result, bounds on the total travel time are given for trajectories in this region. A robust algorithm is also developed in this work that uses the sufficient boundary conditions to check whether a given trajectory will exit saddle neighborhood in linear time and invokes a second-order step otherwise. Several intuitive yet important lemmas are proved, characterizing the behaviour of gradient trajectories in saddle neighborhoods and two theorems are proved that provide rate of convergence of the algorithm to a local minimum.

\begin{appendices}
\section{} \label{Appendix A}

In order to prove Theorem \ref{thm1} we first establish 3 supporting lemmas.

\begin{lem}\label{suplem_1}
The smooth extension of the lower bound on the trajectory function ${\Psi}(K)$ (Theorem 3.1, \cite{dixit2022exit}) given by the function $\underline{\Psi}(K)$ for $\alpha= \frac{1}{L}$ slopes upward for some small positive values of $K$ and then it slopes downward for very large values of $K$, i.e., $\underline{\Psi}(K)$ becomes a decreasing function for large values of $K$ ($\underline{\Psi}(K) \to -\infty$ as $K \to \infty$) provided the initial unstable projection value satisfies the necessary condition $\sum_{j \in \mathcal{N}_{US}} ({\theta}^{us}_j)^2 > \Delta$ where $\Delta > \frac{\epsilon M L n}{\delta(L+\beta)}$.
\end{lem}

\begin{proof}
From Theorem 3.1 in \cite{dixit2022exit}, for every value of parameter $\tau$, there exists a lower bound on the squared radial distance $\norm{\tilde{\u}_K^{\tau}}^2$ for all $K$ in the range $1 \leq K \leq \sup_{\tau} \bigg\{K_{exit}^{\tau} \bigg\}  $ provided $K\epsilon \ll 1$. Moreover, this lower bound can be expressed using a function of $K$ called the trajectory function $\Psi(K)$. Formally, we have that:
	  \begin{align}
	   \epsilon^2 \geq\inf_{\tau}\norm{\tilde{\u}_K^{\tau}}^2 > &  \epsilon^2 \Psi(K) \label{inversionineq},
	  \end{align}
	  where the the trajectory function $\Psi(K)$ is given by:
\begin{align}
\hspace{-0.2cm}\Psi(K) =& \bigg(c_1^{2K} -2Kc_2^{2K-1}b_1 - b_2c_3^Kc_2^K - b_2c_3^{2K}\bigg)\sum_{i \in \mathcal{N}_{S}}({\theta}^{s}_i)^2 + \bigg( c_4^{2K} - 2Kc_3^{2K-1}b_1- b_2c_3^Kc_2^K -b_2c_3^{2K}\bigg)\sum_{j \in \mathcal{N}_{US}}({\theta}^{us}_j)^2
     \end{align}
with $c_1 = \bigg(1 - \alpha L - \frac{\alpha \epsilon M}{2}- \mathcal{O}(\epsilon^2) \bigg)$, $c_2 =\bigg(1 - \alpha \beta + \frac{\alpha \epsilon M}{2}+\mathcal{O}(\epsilon^2)\bigg) $, $c_3 = \bigg(1 + \alpha L + \frac{\alpha \epsilon M}{2}+\mathcal{O}(\epsilon^2)\bigg)$, $c_4 =\bigg(1 + \alpha \beta - \frac{\alpha \epsilon M}{2}-\mathcal{O}(\epsilon^2)\bigg)$, $b_1 = \bigg(\frac{\alpha\epsilon M L n}{2 \delta}+\mathcal{O}(\epsilon^2)\bigg)$, and $b_2 = \frac{\bigg(\frac{\alpha\epsilon M L n}{2 \delta}+\mathcal{O}(\epsilon^2)\bigg)\bigg(1+\mathcal{O}(K\epsilon)\bigg)}{ \bigg( \alpha L + \alpha\beta + \mathcal{O}(\epsilon^2)\bigg)}$.

Substituting these coefficients in the expression for $\Psi(K)$ followed by dropping order $\mathcal{O}(\epsilon^2)$ and $\mathcal{O}(K\epsilon)$ terms (for $K\epsilon \ll 1$) appearing on its right hand side, we get the following approximate expression for $\Psi(K)$:
\begin{align}
\hspace{-2cm} \Psi(K) \approx  &  \bigg(\bigg[ \bigg(1 - \alpha L - \frac{\alpha \epsilon M}{2}\bigg)^{2K}  - 2K\bigg(1 - \alpha \beta + \frac{\alpha \epsilon M}{2}\bigg)^{2K-1}  \frac{\alpha \epsilon M Ln}{2 \delta} \bigg] \sum_{i \in \mathcal{N}_{S}} ({\theta}^{s}_i)^2 + \nonumber \\ & \bigg[\bigg(1 + \alpha \beta - \frac{\alpha \epsilon M}{2}\bigg)^{2K} -  2K\bigg(1 + \alpha L + \frac{\alpha \epsilon M}{2}\bigg)^{2K-1}  \frac{\alpha \epsilon M Ln}{2 \delta}\bigg]  \sum_{j \in \mathcal{N}_{US}} ({\theta}^{us}_j)^2  \nonumber \\ &-\frac{\alpha\epsilon M L n}{2 \delta ( \alpha L + \alpha\beta )}\bigg(1 + \alpha L + \frac{\alpha \epsilon M}{2}\bigg)^K\bigg(1 - \alpha \beta + \frac{\alpha \epsilon M}{2}\bigg)^{K}\bigg(\sum_{i \in \mathcal{N}_{S}}({\theta}^{s}_i)^2 + \sum_{j \in \mathcal{N}_{US}}({\theta}^{us}_j)^2 \bigg)  \nonumber \\ & -\frac{\alpha\epsilon M Ln }{2 \delta ( \alpha L + \alpha\beta)}\bigg(1 + \alpha L + \frac{\alpha \epsilon M}{2}\bigg)^{2K}\bigg)\bigg(\sum_{i \in \mathcal{N}_{S}}({\theta}^{s}_i)^2 + \sum_{j \in \mathcal{N}_{US}}({\theta}^{us}_j)^2 \bigg)
 \\
\hspace{-2cm} \Psi(K) \gtrapprox  &  \bigg(\bigg[ \bigg(1 - \alpha L - \frac{\alpha \epsilon M}{2}\bigg)^{2K}  - 2K\bigg(1 - \alpha \beta + \frac{\alpha \epsilon M}{2}\bigg)^{2K-1}  \frac{\alpha \epsilon M Ln}{2 \delta} \bigg] \sum_{i \in \mathcal{N}_{S}} ({\theta}^{s}_i)^2 + \nonumber \\ & \bigg[\bigg(1 + \alpha \beta - \frac{\alpha \epsilon M}{2}\bigg)^{2K} -  2K\bigg(1 + \alpha L + \frac{\alpha \epsilon M}{2}\bigg)^{2K-1}  \frac{\alpha \epsilon M Ln}{2 \delta}\bigg]  \sum_{j \in \mathcal{N}_{US}} ({\theta}^{us}_j)^2   - \epsilon M Ln \frac{\bigg(1 + \alpha L + \frac{\alpha \epsilon M}{2}\bigg)^{2K} }{ \delta (L+\beta )}
\bigg) \label{approxtrajec1},
\end{align}
where in the last step we used the relation $\bigg(\sum_{i \in \mathcal{N}_{S}}({\theta}^{s}_i)^2 + \sum_{j \in \mathcal{N}_{US}}({\theta}^{us}_j)^2 \bigg) =1 $ and the inequality $\bigg(1 - \alpha \beta + \frac{\alpha \epsilon M}{2}\bigg)< \bigg(1 + \alpha L + \frac{\alpha \epsilon M}{2}\bigg)$. Now for $\alpha = \frac{1}{L}$, \eqref{approxtrajec1} becomes the following approximate inequality:
\begin{align}
\hspace{-2cm}\Psi(K) \gtrapprox  &  \bigg(\bigg[ \bigg( - \frac{ \epsilon M}{2L}\bigg)^{2K}  - 2K\bigg(1 - \frac{\beta}{L} + \frac{\epsilon M}{2L}\bigg)^{2K-1}  \frac{ \epsilon M n}{2 \delta} \bigg] \sum_{i \in \mathcal{N}_{S}} ({\theta}^{s}_i)^2 + \nonumber \\ & \bigg[\bigg(1 + \frac{\beta}{L} - \frac{\epsilon M}{2L}\bigg)^{2K} -  2K\bigg(2 + \frac{ \epsilon M}{2L}\bigg)^{2K-1}  \frac{\epsilon M n}{2 \delta}\bigg]  \sum_{j \in \mathcal{N}_{US}} ({\theta}^{us}_j)^2   - \epsilon M Ln \frac{\bigg(2 + \frac{ \epsilon M}{2L}\bigg)^{2K} }{ \delta (L+ \beta )}
\bigg)  \label{originallowerbound}\\
\hspace{-2cm}\Psi(K) \gtrapprox  &  \bigg(\bigg[  - 2K\bigg(1 - \frac{\beta}{L} + \frac{\epsilon M}{2L}\bigg)^{2K-1}  \frac{ \epsilon M n}{2 \delta} \bigg] \sum_{i \in \mathcal{N}_{S}} ({\theta}^{s}_i)^2 + \nonumber \\ & \bigg[\bigg(1 + \frac{\beta}{L} - \frac{\epsilon M}{2L}\bigg)^{2K} -  2K\bigg(2 + \frac{ \epsilon M}{2L}\bigg)^{2K-1}  \frac{\epsilon M n}{2 \delta}\bigg]  \sum_{j \in \mathcal{N}_{US}} ({\theta}^{us}_j)^2 - \epsilon M Ln \frac{\bigg(2 + \frac{ \epsilon M}{2L}\bigg)^{2K} }{ \delta ( L+\beta )}
\bigg).  \label{lbtranscend}
\end{align}

We first assume that the approximate lower bound on $\Psi(K)$ from \eqref{lbtranscend} is a continuous function of $K$ so as to allow differentiation of this lower bound with respect to variable $K$. This continuous extension is possible since the approximate lower bound on $\Psi(K)$ from \eqref{lbtranscend} is a well-defined function of $K$. Note that we do not use the lower bound from \eqref{originallowerbound} since we are looking for values of $K$ greater than $1$ and the derivative of $\bigg(-\frac{ \epsilon M}{2L}\bigg)^{2K}$ is of at most order $\mathcal{O}(\epsilon^{2K-1})$ for $K > 1$ with small $\epsilon$. Representing this approximate lower bound in \eqref{lbtranscend} as $\underline{\Psi}(K)$ where we have that $\Psi(K) \gtrapprox \underline{\Psi}(K)$, followed by differentiating it with respect to $K$ yields:
\begin{align}
\underline{\Psi}(K) = & \bigg(\bigg[  - 2K\bigg(1 - \frac{\beta}{L} + \frac{\epsilon M}{2L}\bigg)^{2K-1}  \frac{ \epsilon M n}{2 \delta} \bigg] \sum_{i \in \mathcal{N}_{S}} ({\theta}^{s}_i)^2 + \nonumber \\ & \bigg[\bigg(1 + \frac{\beta}{L} - \frac{\epsilon M}{2L}\bigg)^{2K} -  2K\bigg(2 + \frac{ \epsilon M}{2L}\bigg)^{2K-1}  \frac{\epsilon M n}{2 \delta}\bigg]  \sum_{j \in \mathcal{N}_{US}} ({\theta}^{us}_j)^2 - \epsilon M Ln \frac{\bigg(2 + \frac{ \epsilon M}{2L}\bigg)^{2K} }{ \delta (2 \beta - \epsilon M)}
\bigg) \label{trajectoryfunctionmodified}\\
\frac{d\underline{\Psi}(K)}{dK} & = \bigg(\bigg[ -4K \log\bigg(1 - \frac{\beta}{L} + \frac{\epsilon M}{2L}\bigg)  - 2 \bigg] \bigg(1 - \frac{\beta}{L} + \frac{\epsilon M}{2L}\bigg)^{2K-1}\frac{ \epsilon M n}{2 \delta}\sum_{i \in \mathcal{N}_{S}} ({\theta}^{s}_i)^2 + \nonumber \\ & \bigg[2\bigg(1 + \frac{\beta}{L} - \frac{\epsilon M}{2L}\bigg)^{2K}\log\bigg(1 + \frac{\beta}{L} - \frac{\epsilon M}{2L}\bigg) -  2\bigg(2 + \frac{ \epsilon M}{2L}\bigg)^{2K-1}   \frac{\epsilon M n}{2 \delta} -\nonumber \\ & 4K\bigg(2 + \frac{ \epsilon M}{2L}\bigg)^{2K-1}   \frac{\epsilon M n}{2 \delta} \log\bigg(2 + \frac{ \epsilon M}{2L}\bigg) \bigg]  \sum_{j \in \mathcal{N}_{US}} ({\theta}^{us}_j)^2 -  2\epsilon M Ln \frac{\bigg(2 + \frac{ \epsilon M}{2L}\bigg)^{2K}}{ \delta (2 \beta - \epsilon M)}
\log\bigg(2 + \frac{ \epsilon M}{2L}\bigg) \bigg) \label{derivative}
\end{align}
It can be inferred from the above equation \eqref{derivative} that for $\epsilon < \frac{2\beta}{M} $ and $\sum_{j \in \mathcal{N}_{US}} ({\theta}^{us}_j)^2 > \Delta$ where $\Delta > \frac{\epsilon M L n}{\delta(L+\beta)}$, the function $\underline{\Psi}(K)$ slopes upward for some small positive values of $K$ and then it slopes downward for very large values of $K$, i.e., $\underline{\Psi}(K)$ becomes a decreasing function for large values of $K$ ($\underline{\Psi}(K) \to -\infty$ as $K \to \infty$).
\end{proof}

\begin{lem}\label{suplem_2}
The sufficient condition (though not necessary) which guarantees the escape of the approximate lower bound $\underline{\Psi}(K)$ on the trajectory function $\Psi(K)$ from the ball $\mathcal{B}_{\epsilon}(\x^*)$ is as follows:
\begin{align}
1 & \leq \sup_{K \in  G_{\underline{\Psi}}}\bigg\{\underline{\Psi}(K)\bigg\}
\end{align}
where $ G_{\underline{\Psi}} = \bigg\{K   \bigg\vert K \in (0, K^{\iota}], \frac{d^2\underline{\Psi}(K)}{dK^2}  < 0 , \frac{d\underline{\Psi}(K)}{dK}  = 0 \bigg\} $ and $K^{\iota} = \mathcal{O}(\log(\epsilon^{-1})) $. Moreover, there exists some $K_0 = \mathcal{O}(\log(\epsilon^{-1}))$ in the set $G_{\underline{\Psi}} $ implying that the set $G_{\underline{\Psi}} $ is non empty.
\end{lem}

\begin{proof}
Recall that from the condition  \eqref{inversionineq}, the exit time is obtained by evaluating the first $K$ where $\Psi(K)>1$.
From the inequality \eqref{lbtranscend}, by setting the right hand side greater than equal to $1$ for some given $K$ of order $\mathcal{O}(\log(\epsilon^{-1}))$, we will have $\Psi(K)\gtrapprox 1 $. Hence the sufficient condition on the unstable  projection value $ \sum_{j \in \mathcal{N}_{US}} ({\theta}^{us}_j)^2$ for escaping saddle with linear rate can be obtained from \eqref{lbtranscend} by setting its right hand side greater than equal to $1$. Notice that for very large $K$, the right hand side of \eqref{lbtranscend} is always less than $1$.
Moreover, there exists some $K_{min}\geq1$ and $K_{max}>1$ such that the approximate lower bound of \eqref{lbtranscend} can become greater than $1$ only in the interval $(K_{min}, K_{max})$.
 Therefore we only need to find some $K_0 \in (K_{min}, K_{max})$ where the function $\underline{\Psi}(K)$ has zero slope and the value $\underline{\Psi}(K_0)$ is greater than or equal to $1$ for guaranteeing escape. The condition $\underline{\Psi}(K_0) \geq 1 $ would imply ${\Psi}(K_0) \gtrapprox 1$ thereby approximately guaranteeing escape from the condition \eqref{inversionineq} which gets reversed for $K=K_0$.

The above condition can be achieved in many different ways. However, to ensure that the so-called sufficient conditions have minimal restrictions, we must have $K_0$ to be the local maximum of the function $\underline{\Psi}(K)$ on the interval $ K \in (0,C]$ where $C$ is some arbitrary positive finite value with $C \leq K_{max}$. Note that $K_0$ is a root of the equation $ \frac{d\underline{\Psi}(K)}{dK}  = 0$. The condition that $K_0$ is the local maximum of $\underline{\Psi}(K)$ on the interval $ K \in (0,C]$ ensures existence of at least one value of $K_0$ such that $\underline{\Psi}(K_0)\geq 1$ and hence ${\Psi}(K_0)\gtrapprox \underline{\Psi}(K_0)\geq 1$.

Next, recall that from Theorem 3.2 in \cite{dixit2022exit} we have the condition of $K_{exit} < K^{\iota} \lessapprox \mathcal{O}(\log(\epsilon^{-1}))$ for $\epsilon$--precision trajectories with linear exit time. Note that the linear exit time was obtained explicitly by solving for the roots of equation $\underline{\Psi}(K)=1$. Now $K_0$ is the local maximum of the function $\underline{\Psi}(K)$ on the interval $ K \in (0,C]$ and we have $\underline{\Psi}(K_0)\geq 1$ hence we can set $C= K^{\iota}$ which is valid since $C$ was arbitrary with $K_{exit} < C \leq K_{max}$. Similarly, $K_{max}$ was arbitrary hence we can set $K_{max} = 2K^{\iota}$. Therefore we will have $\norm{\tilde{\u}_{K_0}^{\tau}}^2 > \epsilon^2$ for all values of $\tau$ where $\{\tilde{\u}_K^{\tau}\}_{K=0}^{K_{exit}}$ was the $\epsilon$--precision trajectory defined in \cite{dixit2022exit}.

Then the sufficient condition (though not necessary) which guarantees the escape of the approximate lower bound $\underline{\Psi}(K)$ on the trajectory function $\Psi(K)$ from the ball $\mathcal{B}_{\epsilon}(\x^*)$ is as follows:
\begin{align}
1 & \leq \sup_{K \in  G_{\underline{\Psi}}}\bigg\{\underline{\Psi}(K)\bigg\} \label{sufficiency1}
\end{align}
where $ G_{\underline{\Psi}} = \bigg\{K   \bigg\vert K \in (0, K^{\iota}], \frac{d^2\underline{\Psi}(K)}{dK^2}  < 0 , \frac{d\underline{\Psi}(K)}{dK}  = 0 \bigg\} $.

The condition \eqref{sufficiency1} can be relaxed to obtain $ \underline{\Psi}(K_0) \geq 1$ for some $K_0 \in G_{\underline{\Psi}}$.
Note that the set $ G_{\underline{\Psi}}$ is non-empty since the function $\underline{\Psi}(K)$ slopes upwards for small positive $K$ whereas $\underline{\Psi}(K) \to -\infty$ as $K \to \infty$.
Simplifying the derivative condition \eqref{derivative} by setting it to $0$ we get the following:
\begin{align}
0=&\frac{d\underline{\Psi}}{dK}\bigg\vert_{K=K_0}  = \bigg(\bigg[ -4K_0 \log\bigg(1 - \frac{\beta}{L} + \frac{\epsilon M}{2L}\bigg)  - 2 \bigg] \bigg(1 - \frac{\beta}{L} + \frac{\epsilon M}{2L}\bigg)^{2K_0-1}\frac{ \epsilon M n}{2 \delta}\sum_{i \in \mathcal{N}_{S}} ({\theta}^{s}_i)^2 + \nonumber \\ & \bigg[2\bigg(1 + \frac{\beta}{L} - \frac{\epsilon M}{2L}\bigg)^{2K_0}\log\bigg(1 + \frac{\beta}{L} - \frac{\epsilon M}{2L}\bigg) -  2\bigg(2 + \frac{ \epsilon M}{2L}\bigg)^{2K_0-1}   \frac{\epsilon M n}{2 \delta} -\nonumber \\ &  4K_0\bigg(2 + \frac{ \epsilon M}{2L}\bigg)^{2K_0-1}   \frac{\epsilon M n}{2 \delta} \log\bigg(2 + \frac{ \epsilon M}{2L}\bigg) \bigg]  \sum_{j \in \mathcal{N}_{US}} ({\theta}^{us}_j)^2 -  2\epsilon M Ln \frac{\bigg(2 + \frac{ \epsilon M}{2L}\bigg)^{2K_0}}{ \delta (2 \beta - \epsilon M)}
\log\bigg(2 + \frac{ \epsilon M}{2L}\bigg) \bigg) \\
0 = & \bigg(\bigg[ -4K_0 \log\bigg(1 - \frac{\beta}{L} + \frac{\epsilon M}{2L}\bigg)  - 2 \bigg] \bigg(\frac{1 - \frac{\beta}{L} + \frac{\epsilon M}{2L}}{2 + \frac{ \epsilon M}{2L}}\bigg)^{2K_0}\bigg(1 - \frac{\beta}{L} + \frac{\epsilon M}{2L}\bigg)^{-1}\frac{ \epsilon M n}{2 \delta}\sum_{i \in \mathcal{N}_{S}} ({\theta}^{s}_i)^2 + \nonumber \\ & \bigg[2\bigg(\frac{1 + \frac{\beta}{L} - \frac{\epsilon M}{2L}}{2 + \frac{ \epsilon M}{2L}}\bigg)^{2K_0}\log\bigg(1 + \frac{\beta}{L} - \frac{\epsilon M}{2L}\bigg) -  2\bigg(2 + \frac{ \epsilon M}{2L}\bigg)^{-1}   \frac{\epsilon M n}{2 \delta} -\nonumber \\ &  4K_0\bigg(2 + \frac{ \epsilon M}{2L}\bigg)^{-1}   \frac{\epsilon M n}{2 \delta} \log\bigg(2 + \frac{ \epsilon M}{2L}\bigg) \bigg]  \sum_{j \in \mathcal{N}_{US}} ({\theta}^{us}_j)^2 -  2\epsilon M Ln \frac{1}{ \delta (2 \beta - \epsilon M)}
\log\bigg(2 + \frac{ \epsilon M}{2L}\bigg) \bigg)  \label{tk_0}
\end{align}
Observe that the roots of this equation cannot be explicitly computed due to the transcendental nature of this equation. However, the roots can be obtained if the order of $K_0$ is known with respect to $\epsilon$. Since $K_0 \in G_{\underline{\Psi}}$, we will have $K_0 < K^{\iota} \lessapprox \mathcal{O}(\log(\epsilon^{-1}))$. Therefore, we compute only those values of $K_0$ which are linear, i.e., $K_0 = \mathcal{O}(\log(\epsilon^{-1}))$. For such a $K_0$, setting $\frac{1}{\bigg(2 + \frac{ \epsilon M}{2L}\bigg)^{2K_0}} = \mu \epsilon^a $ where $\mu>0$, $a>0$ and $ \bigg(1 - \frac{\beta}{L} + \frac{\epsilon M}{2L}\bigg)^{2K_0} = \eta \epsilon^b$ where $\eta>0$, $b>0$ provided $\epsilon< \frac{2\beta}{M} $, the above equality \eqref{tk_0} becomes:
\begin{align}
0 = & \bigg(\underbrace{\bigg[ -4K_0 \log\bigg(1 - \frac{\beta}{L} + \frac{\epsilon M}{2L}\bigg)  - 2 \bigg] \bigg(1 - \frac{\beta}{L} + \frac{\epsilon M}{2L}\bigg)^{-1}\frac{\mu \eta \epsilon^{(1+a+b)} M n}{2 \delta}\sum_{i \in \mathcal{N}_{S}} ({\theta}^{s}_i)^2}_{F_1} + \nonumber \\ & \bigg[2\bigg(\frac{1 + \frac{\beta}{L} - \frac{\epsilon M}{2L}}{2 + \frac{ \epsilon M}{2L}}\bigg)^{2K_0}\log\bigg(1 + \frac{\beta}{L} - \frac{\epsilon M}{2L}\bigg) -  2\bigg(2 + \frac{ \epsilon M}{2L}\bigg)^{-1}   \frac{\epsilon M n}{2 \delta} -\nonumber \\ &  4K_0\bigg(2 + \frac{ \epsilon M}{2L}\bigg)^{-1}   \frac{\epsilon M n}{2 \delta} \log\bigg(2 + \frac{ \epsilon M}{2L}\bigg) \bigg]  \sum_{j \in \mathcal{N}_{US}} ({\theta}^{us}_j)^2 -  2\epsilon M Ln \frac{1}{ \delta (2 \beta - \epsilon M)}
\log\bigg(2 + \frac{ \epsilon M}{2L}\bigg) \bigg) \label{transcendode_beta} \\
  & \hspace{-1cm}\bigg(\frac{1 + \frac{\beta}{L} - \frac{\epsilon M}{2L}}{2 + \frac{ \epsilon M}{2L}}\bigg)^{2K_0} \approx   \bigg(2 + \frac{ \epsilon M}{2L}\bigg)^{-1}   \frac{\epsilon M n}{2 \delta} \frac{\log\bigg(2 + \frac{ \epsilon M}{2L}\bigg)}{\log\bigg(1 + \frac{\beta}{L} - \frac{\epsilon M}{2L}\bigg)} 2K_0   + \bigg(2 + \frac{ \epsilon M}{2L}\bigg)^{-1}   \frac{\epsilon M n}{2 \delta\log\bigg(1 + \frac{\beta}{L} - \frac{\epsilon M}{2L}\bigg)}+ \nonumber \\ & \hspace{2cm} \frac{\epsilon M Ln}{ \delta (2 \beta - \epsilon M)\log\bigg(1 + \frac{\beta}{L} - \frac{\epsilon M}{2L}\bigg)\sum_{j \in \mathcal{N}_{US}} ({\theta}^{us}_j)^2}
\log\bigg(2 + \frac{ \epsilon M}{2L}\bigg) \label{transcendode}
\end{align}
where in the last step, we dropped the term $F_1$ (since this term $F_1 = \mathcal{O}(K_0 \epsilon^{(1+a+b)}) = \mathcal{O}( \epsilon^{(1+a+b)}\log(\epsilon^{-1}))$) to obtain the approximate equality \eqref{transcendode}. The approximate solution for \eqref{transcendode} can be obtained using a transcendental equation of the form $q^x = cx + d$ where $x=2K_0$ and the coefficients are as follows:
\begin{align}
& q =  \bigg(\frac{1 + \frac{\beta}{L} - \frac{\epsilon M}{2L}}{2 + \frac{ \epsilon M}{2L}}\bigg) , c = \bigg(2 + \frac{ \epsilon M}{2L}\bigg)^{-1}   \frac{\epsilon M n}{2 \delta} \frac{\log\bigg(2 + \frac{ \epsilon M}{2L}\bigg)}{\log\bigg(1 + \frac{\beta}{L} - \frac{\epsilon M}{2L}\bigg)} \\
& d = \bigg(2 + \frac{ \epsilon M}{2L}\bigg)^{-1}   \frac{\epsilon M n}{2 \delta\log\bigg(1 + \frac{\beta}{L} - \frac{\epsilon M}{2L}\bigg)}+  \frac{\epsilon M Ln\log\bigg(2 + \frac{ \epsilon M}{2L}\bigg) }{ \delta (2 \beta - \epsilon M)\log\bigg(1 + \frac{\beta}{L} - \frac{\epsilon M}{2L}\bigg)\sum_{j \in \mathcal{N}_{US}} ({\theta}^{us}_j)^2}.
\end{align}
The solution for this equation is given by the following relation:
\begin{align}
x & = - \frac{W(-\frac{\log q}{c}q^{-\frac{d}{c}})}{\log q} - \frac{d}{c} \leq  \frac{\log(-\frac{\log q}{c}q^{-\frac{d}{c}})}{\log q^{-1}} - \frac{d}{c} = \frac{\log(-\frac{\log q}{c})}{\log q^{-1}}
\end{align}
where $W(.)$ is the Lambert W function and we have that $W(y) \leq \log(y)$ for large $y$. Substituting these coefficients in \eqref{transcendode}, we obtain the following approximate condition:

\begin{align}
K_0 \lessapprox \underbrace{\frac{1}{2}\log_{\bigg(\frac{2 + \frac{ \epsilon M}{2L}}{1 + \frac{\beta}{L} - \frac{\epsilon M}{2L}}\bigg)}\bigg(\frac{2 \delta\bigg(2 + \frac{ \epsilon M}{2L}\bigg)\log\bigg(1 + \frac{\beta}{L} - \frac{\epsilon M}{2L}\bigg)\log \bigg(\frac{2 + \frac{ \epsilon M}{2L}}{1 + \frac{\beta}{L} - \frac{\epsilon M}{2L}}\bigg)}{   \epsilon M n \log\bigg(2 + \frac{ \epsilon M}{2L}\bigg)}\bigg)}_{\hat{K}_0} \label{transcendodebound}
\end{align}
where ${\hat{K}_0}$ is the approximate upper bound on $K_0$. However, for the condition $K_0 \in G_{\underline{\Psi}}$ to hold, we also require $ \frac{d^2\underline{\Psi}}{dK^2}\bigg\vert_{K={K}_0} < 0$ condition to hold. It can be readily checked that $ \frac{d\underline{\Psi}}{dK}\bigg\vert_{K=\hat{K}_0} < 0$ whereas $\frac{d\underline{\Psi}}{dK}$ is positive for very small values of $K$. Hence, there must exist a local maximum at some $K_0 < {\hat{K}_0}$ which would imply $ \frac{d^2\underline{\Psi}}{dK^2}\bigg\vert_{K={K}_0} < 0$. Hence, it is not required to explicitly solve the condition $ \frac{d^2\underline{\Psi}}{dK^2}\bigg\vert_{K={K}_0} < 0$.

It is worth mentioning that dropping the term $F_1$ to obtain the approximate equality \eqref{transcendode} is justified. Observe that in the two approximate transcendental equations \eqref{transcendode_beta} and \eqref{transcendode} with $K_0$ as the variable, the right-hand sides will be greater than their left-hand sides respectively at the value $K_0={\hat{K}_0}$. Also, for small values of $K_0$ the respective left-hand sides of \eqref{transcendode_beta} and \eqref{transcendode} dominate, hence the approximate equality occurs for some $K_0<{\hat{K}_0}$. Now, we are only left to prove that the approximations \eqref{transcendode_beta} and \eqref{transcendode} are almost equal at $K_0={\hat{K}_0}$. This can be established by proving that the term $F_1=\mathcal{O}({\hat{K}_0} \epsilon^{(1+a+b)}) = \mathcal{O}( \epsilon^{(1+a+b)}\log(\epsilon^{-1}))$ is negligible w.r.t. other terms in \eqref{transcendode_beta} at $K_0={\hat{K}_0}$. From the particular approximate upper bound in \eqref{transcendodebound}, it can be verified that $a>1$. Using the substitution $\frac{1}{\bigg(2 + \frac{ \epsilon M}{2L}\bigg)^{2{\hat{K}_0}}} = \mu \epsilon^a $ where $\mu>0$, $a>0$, taking log both sides followed by substituting the approximate upper bound ${\hat{K}_0}$ from \eqref{transcendodebound} yields:
\begin{align}
a \log\bigg(\frac{1}{\sqrt[a]{\mu}\epsilon}\bigg) & = 2{\hat{K}_0} \log\bigg(2+ \frac{\epsilon M}{2L}\bigg) \\
a \log\bigg(\frac{1}{\sqrt[a]{\mu}\epsilon}\bigg) & = \frac{\log\bigg(\frac{2 \delta\bigg(2 + \frac{ \epsilon M}{2L}\bigg)\log\bigg(1 + \frac{\beta}{L} - \frac{\epsilon M}{2L}\bigg)\log \bigg(\frac{2 + \frac{ \epsilon M}{2L}}{1 + \frac{\beta}{L} - \frac{\epsilon M}{2L}}\bigg)}{   \epsilon M n \log\bigg(2 + \frac{ \epsilon M}{2L}\bigg)}\bigg)}{\log\bigg(\frac{2 + \frac{ \epsilon M}{2L}}{1 + \frac{\beta}{L} - \frac{\epsilon M}{2L}}\bigg)}\log \bigg(2+ \frac{\epsilon M}{2L}\bigg) \\
a & = \frac{\log \bigg(2+ \frac{\epsilon M}{2L}\bigg)}{\log\bigg(2 + \frac{ \epsilon M}{2L}\bigg) - \log \bigg({1 + \frac{\beta}{L} - \frac{\epsilon M}{2L}}\bigg)}>1, \label{agreaterthan1}
\end{align}
where in the last step we have that $\frac{1}{\sqrt[a]{\mu}\epsilon} = \frac{2 \delta\bigg(2 + \frac{ \epsilon M}{2L}\bigg)\log\bigg(1 + \frac{\beta}{L} - \frac{\epsilon M}{2L}\bigg)\log \bigg(\frac{2 + \frac{ \epsilon M}{2L}}{1 + \frac{\beta}{L} - \frac{\epsilon M}{2L}}\bigg)}{   \epsilon M n \log\bigg(2 + \frac{ \epsilon M}{2L}\bigg)}$. Now with $a>1$ we have the following condition for any $b>0$:
\begin{align}
\lim_{\epsilon \to 0^+} \frac{\epsilon^{(1+a+b)}\log(\epsilon^{-1})}{\epsilon^2} = 0.
\end{align}
Hence, for sufficiently small $\epsilon$, term $F_1$ can be of at most order $\mathcal{O}(\epsilon^2)$.
\end{proof}

\begin{lem}\label{suplem_3}
There exists some $K_0 = \mathcal{O}(\log(\epsilon^{-1}))$ in the set $G_{\underline{\Psi}} $ such that $\underline{\Psi}(K_0)\geq1 $ provided the lower bound on the unstable projection value $ \sum_{j \in \mathcal{N}_{US}} ({\theta}^{us}_j)^2$ has the following order:
\begin{align}
\sum_{j \in \mathcal{N}_{US}} ({\theta}^{us}_j)^2  \gtrapprox  \mathcal{O}\bigg(\frac{1}{\log(\epsilon^{-1})}\bigg).
\end{align}
\end{lem}

\begin{proof}
Recall that from the relaxation of condition \eqref{sufficiency1}, we require $\underline{\Psi}(K_0)\geq1$. Since $K_0$ is not explicitly available and we only have the approximate upper bound $\hat{K}_0$ from \eqref{transcendodebound}, hence we use the substitution $ K_0=\chi \hat{K}_0$ for some $ 0<\chi \leq 1$ and set the value of function $\underline{\Psi}$ at this point greater than equal to $1$.

Substituting $K_0=\chi\hat{K}_0$ from \eqref{transcendodebound} into the condition  $\underline{\Psi}({K}_0)\geq 1$, dropping the first term on the right hand side of \eqref{trajectoryfunctionmodified} (it is of order $\mathcal{O}(\chi\hat{K}_0 \epsilon^{(1+a+b)}) = \mathcal{O}( \epsilon^{(1+a+b)}\log(\epsilon^{-1}))$ as before, substituting $\frac{1}{\bigg(2 + \frac{ \epsilon M}{2L}\bigg)^{2{K}_0}} = \mu \epsilon^{\chi a} $ for $\mu>0, \epsilon>0$, using \eqref{transcendode} for $ K_0 = \chi \hat{K}_0$ followed by rearranging, we get:
\begin{align}
& \hspace{-0.5cm}\bigg( \bigg[\bigg(\frac{1 + \frac{\beta}{L} - \frac{\epsilon M}{2L}}{2 + \frac{ \epsilon M}{2L}}\bigg)^{2{K}_0} -  2{K}_0\bigg(2 + \frac{ \epsilon M}{2L}\bigg)^{-1}  \frac{\epsilon M n}{2 \delta}\bigg]  \sum_{j \in \mathcal{N}_{US}} ({\theta}^{us}_j)^2 -  \frac{\epsilon M Ln}{ \delta (2 \beta - \epsilon M)}
\bigg)  \gtrapprox \frac{1}{\bigg(2 + \frac{ \epsilon M}{2L}\bigg)^{2{K}_0}}  \\
& \hspace{-1.5cm}\bigg( \bigg(2 + \frac{ \epsilon M}{2L}\bigg)^{-1}   \frac{\epsilon M n}{2 \delta} \frac{\log\bigg(2 + \frac{ \epsilon M}{2L}\bigg)}{\log\bigg(1 + \frac{\beta}{L} - \frac{\epsilon M}{2L}\bigg)} 2{K}_0  - \bigg(2 + \frac{ \epsilon M}{2L}\bigg)^{-1}  \frac{\epsilon M n}{2 \delta}2{K}_0   + \bigg(2 + \frac{ \epsilon M}{2L}\bigg)^{-1}   \frac{\epsilon M n }{2 \delta\log\bigg(1 + \frac{\beta}{L} - \frac{\epsilon M}{2L}\bigg)} \bigg)\sum_{j \in \mathcal{N}_{US}} ({\theta}^{us}_j)^2     \gtrapprox \nonumber \\ & \hspace{1cm}  \mu \epsilon^{\chi a} + \frac{\epsilon M Ln}{ \delta (2 \beta - \epsilon M)} - \frac{\epsilon M Ln}{ \delta (2 \beta - \epsilon M)\log\bigg(1 + \frac{\beta}{L} - \frac{\epsilon M}{2L}\bigg)}
\log\bigg(2 + \frac{ \epsilon M}{2L}\bigg)
\end{align}
Since $ \bigg(  \frac{\epsilon M n}{2 \delta} \frac{\log\bigg(2 + \frac{ \epsilon M}{2L}\bigg)}{\log\bigg(1 + \frac{\beta}{L} - \frac{\epsilon M}{2L}\bigg)} 2{K}_0  - \frac{\epsilon M n}{2 \delta}2{K}_0 +    \frac{\epsilon M n }{2 \delta\log\bigg(1 + \frac{\beta}{L} - \frac{\epsilon M}{2L}\bigg)} \bigg) > 0$, dividing both sides by this quantity yields the following sufficient condition on unstable  projection value $\sum_{j \in \mathcal{N}_{US}} ({\theta}^{us}_j)^2 $:
\begin{align}
& \hspace{0.5cm}\sum_{j \in \mathcal{N}_{US}} ({\theta}^{us}_j)^2     \gtrapprox  \frac{\bigg(2 + \frac{ \epsilon M}{2L}\bigg)\bigg(\mu \epsilon^{\chi a} + \frac{\epsilon M Ln}{ \delta (2 \beta - \epsilon M)} - \frac{\epsilon M Ln}{ \delta (2 \beta - \epsilon M)\log\bigg(1 + \frac{\beta}{L} - \frac{\epsilon M}{2L}\bigg)}
\log\bigg(2 + \frac{ \epsilon M}{2L}\bigg) \bigg)}{\bigg(  \frac{\epsilon M n}{2 \delta} \frac{\log\bigg(2 + \frac{ \epsilon M}{2L}\bigg)}{\log\bigg(1 + \frac{\beta}{L} - \frac{\epsilon M}{2L}\bigg)} 2{K}_0  - \frac{\epsilon M n}{2 \delta}2{K}_0    +    \frac{\epsilon M n }{2 \delta\log\bigg(1 + \frac{\beta}{L} - \frac{\epsilon M}{2L}\bigg)} \bigg)}\\
& \hspace{0.5cm}\sum_{j \in \mathcal{N}_{US}} ({\theta}^{us}_j)^2    \gtrapprox  \frac{\bigg(2 + \frac{ \epsilon M}{2L}\bigg)\bigg(\frac{2 \delta \mu \epsilon^{(\chi a-1)}}{Mn} + \frac{ 1}{  (\frac{\beta}{L} - \frac{\epsilon M}{2L})} - \frac{	\log\bigg(2 + \frac{ \epsilon M}{2L}\bigg)}{  (\frac{\beta}{L} - \frac{\epsilon M}{2L})\log\bigg(1 + \frac{\beta}{L} - \frac{\epsilon M}{2L}\bigg)}
	\bigg)}{2\chi\hat{K}_0\bigg(  \frac{\log\bigg(2 + \frac{ \epsilon M}{2L}\bigg) }{\log\bigg(1 + \frac{\beta}{L} - \frac{\epsilon M}{2L}\bigg)}  - 1   \bigg)+  \frac{1 }{\log\bigg(1 + \frac{\beta}{L} - \frac{\epsilon M}{2L}\bigg)}}  \label{projectionloose}.
\end{align}

Now, recall that from \eqref{agreaterthan1} we have $a>1$ and we also know that $\hat{K}_0 \gtrapprox {K}_0 = \chi \hat{K}_0$. Since $K_0$ is not explicitly known we can choose a surrogate for $\chi$ to obtain the sufficient condition. Notice that $\chi$ is a quantity between $0$ and $1$. Choosing a large value for $\chi$ say close to $1$ will yield the following order bound $\sum\limits_{j \in \mathcal{N}_{US}} (\theta_j^{us})^2 \gtrapprox \mathcal{O}\bigg( \frac{\epsilon^{a -1}}{\log(\epsilon^{-1})}\bigg)$. Recall that from \eqref{relativeerrorprojectionbound} we require $ \sqrt{\sum_{j \in \mathcal{N}_{US} } ({\theta}_{j}^{us})^2} > \mathcal{O}\bigg(\frac{1}{\sqrt{\epsilon}}\bigg(\log\bigg(\frac{1}{\epsilon} \bigg)\epsilon\bigg)  \bigg)$. However this bound may then contradict the sufficient condition $\sum\limits_{j \in \mathcal{N}_{US}} (\theta_j^{us})^2 \gtrapprox \mathcal{O}\bigg( \frac{\epsilon^{a -1}}{\log(\epsilon^{-1})}\bigg)$ if $a>2$, i.e., we have $\mathcal{O}\bigg(\frac{1}{{\epsilon}}\bigg(\log\bigg(\frac{1}{\epsilon} \bigg)\epsilon\bigg)^2  \bigg) > \mathcal{O}\bigg( \frac{\epsilon^{a -1}}{\log(\epsilon^{-1})}\bigg)$ as $\epsilon \to 0$ (for well conditioned problems, i.e., $\frac{\beta}{L}$ close to $1$, it can be checked using \eqref{agreaterthan1} that $a$ becomes arbitrarily large). Next, choosing very small values of $\chi$ say close to $0$ will cause the approximation \eqref{transcendode} to fail since the term $F_1$ in \eqref{transcendode_beta} can no longer be dropped (this term is of order $\mathcal{O}(\epsilon)$ for $\chi = 0$).

However, the choice $\chi = \frac{1}{a}$ is able to strike a balance between both the requirements (dropping of the term $F_1$ in \eqref{transcendode_beta} and satisfying the bound on $ \sum\limits_{j \in \mathcal{N}_{US}} (\theta_j^{us})^2$ from \eqref{relativeerrorprojectionbound}). Observe that by setting $\chi = \frac{1}{a}$, we can get rid of the $\epsilon$ dependency in the numerator of \eqref{projectionloose} which generates the order bound $\sum\limits_{j \in \mathcal{N}_{US}} (\theta_j^{us})^2 \gtrapprox \mathcal{O}\bigg( \frac{1}{\log(\epsilon^{-1})}\bigg)$ that agrees with the condition $ \sqrt{\sum_{j \in \mathcal{N}_{US} } ({\theta}_{j}^{us})^2} > \mathcal{O}\bigg(\frac{1}{\sqrt{\epsilon}}\bigg(\log\bigg(\frac{1}{\epsilon} \bigg)\epsilon\bigg)  \bigg)$ from  \eqref{relativeerrorprojectionbound} for any $a>0$. Also, it can be easily checked that the term $F_1$ from \eqref{transcendode_beta} for $K_0 = \chi \hat{K}_0 = \frac{1}{a}\hat{K}_0$ has the order $\mathcal{O}(\epsilon^{(2+b)}\log(\epsilon^{-1}))$ for some $b>0$ hence the term $F_1$ can be dropped to get the approximation \eqref{transcendode}. Substituting $\hat{K}_0$ from \eqref{transcendodebound} and $ \chi = \frac{1}{a}$ in \eqref{projectionloose} followed by further simplification gives the following result:
\begin{align}
\hspace{0.2cm}\sum_{j \in \mathcal{N}_{US}} ({\theta}^{us}_j)^2     \gtrapprox \frac{\bigg(2 + \frac{ \epsilon M}{2L}\bigg)\bigg(\frac{2 \delta \mu\log\bigg(1 + \frac{\beta}{L} - \frac{\epsilon M}{2L}\bigg)}{ Mn} + \frac{ \log\bigg(1 + \frac{\beta}{L} - \frac{\epsilon M}{2L}\bigg)}{  (\frac{\beta}{L} - \frac{\epsilon M}{2L})} - \frac{	\log\bigg(2 + \frac{ \epsilon M}{2L}\bigg)}{  (\frac{\beta}{L} - \frac{\epsilon M}{2L})}
	\bigg)}{\frac{1}{a}\log \bigg(\frac{2 \delta\bigg(2 + \frac{ \epsilon M}{2L}\bigg)\log\bigg(1 + \frac{\beta}{L} - \frac{\epsilon M}{2L}\bigg)\log \bigg(\frac{2 + \frac{ \epsilon M}{2L}}{1 + \frac{\beta}{L} - \frac{\epsilon M}{2L}}\bigg)}{   \epsilon M n \log\bigg(2 + \frac{ \epsilon M}{2L}\bigg)}\bigg)+ 1}   \label{unstabletightbound01}
\end{align}

Finally, for $\epsilon < \frac{2\beta}{M} $, dropping the negative term $ \frac{ \log\bigg(1 + \frac{\beta}{L} - \frac{\epsilon M}{2L}\bigg)}{  (\frac{\beta}{L} - \frac{\epsilon M}{2L})} - \frac{	\log\bigg(2 + \frac{ \epsilon M}{2L}\bigg)}{  (\frac{\beta}{L} - \frac{\epsilon M}{2L})}$ from the numerator of \eqref{unstabletightbound01} and setting the condition:
\begin{align}
\hspace{0.2cm}\sum_{j \in \mathcal{N}_{US}} ({\theta}^{us}_j)^2     \gtrapprox \frac{\bigg(2 + \frac{ \epsilon M}{2L}\bigg)\bigg(\frac{2 \delta \mu\log\bigg(1 + \frac{\beta}{L} - \frac{\epsilon M}{2L}\bigg)}{ Mn}
	\bigg)}{\frac{1}{a}\log \bigg(\frac{2 \delta\bigg(2 + \frac{ \epsilon M}{2L}\bigg)\log\bigg(1 + \frac{\beta}{L} - \frac{\epsilon M}{2L}\bigg)\log \bigg(\frac{2 + \frac{ \epsilon M}{2L}}{1 + \frac{\beta}{L} - \frac{\epsilon M}{2L}}\bigg)}{   \epsilon M n \log\bigg(2 + \frac{ \epsilon M}{2L}\bigg)}\bigg)+ 1},  \label{unstabletightbound}
\end{align}
the approximate lower bound in  \eqref{unstabletightbound01} is guaranteed.
Now using the upper bound on $K_0$ from \eqref{transcendodebound} in the expression $\mu \epsilon^a = \frac{1}{\bigg(2+\frac{\epsilon M}{2 L}\bigg)^{2\hat{K}_0}}$, we have that:
\begin{align}
\sqrt[a]{\mu} &=  \frac{    M n \log\bigg(2 + \frac{ \epsilon M}{2L}\bigg)}{2 \delta\bigg(2 + \frac{ \epsilon M}{2L}\bigg)\log\bigg(1 + \frac{\beta}{L} - \frac{\epsilon M}{2L}\bigg)\log \bigg(\frac{2 + \frac{ \epsilon M}{2L}}{1 + \frac{\beta}{L} - \frac{\epsilon M}{2L}}\bigg)}
\end{align}
where $a=\frac{\log \bigg(2+ \frac{\epsilon M}{2L}\bigg)}{\log\bigg(2 + \frac{ \epsilon M}{2L}\bigg) - \log \bigg({1 + \frac{\beta}{L} - \frac{\epsilon M}{2L}}\bigg)}>1$.
Hence, the approximate lower bound on the unstable projection value $ \sum_{j \in \mathcal{N}_{US}} ({\theta}^{us}_j)^2$ has the following order:
\begin{align}
\sum_{j \in \mathcal{N}_{US}} ({\theta}^{us}_j)^2  \gtrapprox  \mathcal{O}\bigg(\frac{1}{\log(\epsilon^{-1})}\bigg). \label{projectionwellconditioned}
\end{align}
It is also worth mentioning that the lower bound on the unstable  projection value $ \sum_{j \in \mathcal{N}_{US}} ({\theta}^{us}_j)^2$ from \eqref{projectionwellconditioned} is an increasing function of $\epsilon$.
\end{proof}

\subsection*{Proof of Theorem \ref{thm1}}

Using Lemmas \ref{suplem_1}, \ref{suplem_2} and \ref{suplem_3} we have established that there exists some $K_0 = \mathcal{O}(\log(\epsilon^{-1}))$ in the set $G_{\underline{\Psi}} $ such that $\underline{\Psi}(K_0)\geq1 $ provided the initial condition of \eqref{projectionwellconditioned} holds. Since $K_0 \in G_{\underline{\Psi}} $ we will have $K_0 \leq K^{\iota}$ where $K^{\iota} $ is upper bounded by the linear exit time bound from \eqref{linearexittimebound}. Then using the fact that $\underline{\Psi}(K_0)\geq1 $ we get that $\Psi(K_0)> \underline{\Psi}(K_0)\geq1 $ implying $ \inf_{\tau}\norm{\tilde{\u}_{K_0}^{\tau}}^2 >   \epsilon^2 \Psi(K_0) > \epsilon^2$ from  \eqref{inversionineq}. Hence the approximate trajectories $\{\tilde{\u}_{K}^{\tau}\}$ exit $\mathcal{B}_{\epsilon}(\x^*)$ at $K< K_0 < K^{\iota}$ under the sufficient initial condition of \eqref{projectionwellconditioned}. This completes the proof of Theorem \ref{thm1}.

Finally, using the fact that $\epsilon < \frac{2\beta}{M} $ and Theorem 3.2 of \cite{dixit2022exit}, we can upper bound $\epsilon$ as follows:
\begin{align}
   \epsilon <\min \bigg\{\inf_{{\norm{\u}=1}}\bigg(\limsup_{j\to \infty} \sqrt[j]{\frac{r_j(\u)}{j!}}\bigg)^{-1},\frac{2L\delta}{M(2Ln^2 -\delta )} + \mathcal{O}(\epsilon^2), \frac{2\beta}{M}\bigg\}
\end{align}
where $r_j(\u) = \norm{\bigg(\frac{d^j }{dw^j}\nabla^2 f(\x^*+w\u)\bigg\vert_{w=0}\bigg)}_2$.

\hfill
 \qedhere
 \qedsymbol

\section{} \label{Appendix B}

We prove Theorem \ref{thm2} by first proving a sequence of lemmas.

\begin{lem}\label{suplem1}
For an iterative gradient mapping given by $\x^+ = \x - \alpha \nabla f(\x)$ in some neighborhood of $\x^*$, if $\norm{\x^+-\x^*} > \norm{\x-\x^*} $ then the following holds:
\begin{align}
\textbf{a.} \hspace{0.5cm} \norm{\x^{++}-\x^*} &\geq \bar{\rho}(\x)\norm{\x^+-\x^*} - \sigma(\x) \label{partb.}\\
\textbf{b.} \hspace{0.5cm}  \norm{\x^{++}-\x^*} &> \norm{\x^+-\x^*} \label{parta.}
\end{align}
where $\sigma(\x) = \mathcal{O}(\norm{\x-\x^*}^2)$, $\bar{\rho}(\x)>1 $ and \eqref{parta.} is termed as the sequential monotonicity property.
\end{lem}

\begin{proof}
For an iterative gradient mapping given by $\x^+ = \x - \alpha \nabla f(\x)$ in any neighborhood of $\x^*$, we have:
\begin{align}
   \nabla f(\x) = \bigg(\nabla f(\x^*) + \int_{p=0}^{p=1}\nabla^2 f(\x^* + p(\x-\x^*)) (\x-\x^*) dp\bigg). \label{gradtaylor}
\end{align}
provided function $f(\cdot)$ is twice continuously differentiable.
Using this substitution in the iterative gradient mapping, we have the following result:
\begin{align}
\norm{\x^+-\x^*} &= \norm{\x - \alpha \nabla f(\x) - \x^*}\\
& = \norm{(\x-\x^*) - \alpha\bigg(\nabla f(\x^*) + \int_{p=0}^{p=1}\nabla^2 f(\x^* + p(\x-\x^*)) (\x-\x^*) dp\bigg)} \\
& = \norm{(\x-\x^*) - \alpha \int_{p=0}^{p=1}\nabla^2 f(\x^* + p(\x-\x^*)) dp(\x-\x^*)} \\
& = \norm{\bigg(\mathbf{I} - \alpha \int_{p=0}^{p=1}\nabla^2 f(\x^* + p(\x-\x^*)) dp\bigg)(\x-\x^*)} \\
& = \sqrt{\bigg(\sum_{j \in \mathcal{I}_{US}}(\nu^{us}_j\langle \hat{\u}, \e_j^{us}\rangle)^2  + \sum_{i \in \mathcal{I}_{S}} (\nu^{s}_i  \langle \hat{\u}, \e_i^{s}\rangle)^2\bigg)}\norm{\x-\x^*}  \label{weyl}
\end{align}
where $\u = \x - \x^*$, $\hat{\u}= \frac{\u}{\norm{\u}}$, $\x - \x^* = \norm{\u}\bigg(\sum_{j \in \mathcal{I}_{US}}\langle \hat{\u}, \e_j^{us}\rangle \e_j^{us} + \sum_{i \in \mathcal{I}_{S}}  \langle \hat{\u}, \e_i^{s}\rangle\e_i^{s}\bigg)$ and $(\e_j^{us}, \nu_j^{us})$, $(\e_j^{s}, \nu_j^{s})$ are the eigenvector-eigenvalue pair of the matrix $\mathbf{D}(\x)$ where $\mathbf{D}(\x)=\bigg(\mathbf{I} - \alpha \int_{p=0}^{p=1}\nabla^2 f(\x^* + p(\x-\x^*)) dp\bigg)$ with $\nu^s_i<1$ for all $i \in \mathcal{I}_{S}$, $\nu^{us}_j\geq 1$ for all $j \in \mathcal{I}_{US}$ and $ \mathcal{I}_{US}, \mathcal{I}_{S}$ are the index sets associated respectively with these subspaces respectively.

We consider the case of strict expansive dynamics in the current iteration.
Given: $\norm{\x^+-\x^*} > \norm{\x-\x^*} $ or equivalently
\begin{align}
    \norm{\x^+-\x^*} =\sqrt{\bigg(\sum_{j \in \mathcal{I}_{US}}(\nu^{us}_j\langle \hat{\u}, \e_j^{us}\rangle)^2  + \sum_{i \in \mathcal{I}_{S}} (\nu^{s}_i  \langle \hat{\u}, \e_i^{s}\rangle)^2\bigg)}\norm{\u} > \norm{\u}. \label{Dx_eigen}
\end{align}

This implies:
\begin{align}
   \sqrt{\bigg(\sum_{j \in \mathcal{I}_{US}}(\nu^{us}_j\langle \hat{\u}, \e_j^{us}\rangle)^2  + \sum_{i \in \mathcal{I}_{S}} (\nu^{s}_i  \langle \hat{\u}, \e_i^{s}\rangle)^2\bigg)} > 1. \label{assertion1}
\end{align}
We now show that the claim in \eqref{partb.} holds.

Since $\x^{++} = \x^+ - \alpha \nabla f(\x^+)$, we have the following:
\begin{align}
\norm{\x^{++}-\x^*} &= \norm{\x^+ - \alpha \nabla f(\x^+) - \x^*} \\
& \hspace{-1cm}= \norm{(\x^+-\x^*) - \alpha\bigg(\nabla f(\x^*) + \int_{p=0}^{p=1}\nabla^2 f(\x^* + p(\x^+-\x^*)) (\x^+-\x^*) dp\bigg)} \\
& \hspace{-1cm}= \norm{(\x^+-\x^*) - \alpha \int_{p=0}^{p=1}\nabla^2 f(\x^* + p(\x^+-\x^*)) dp(\x^+-\x^*)} \\
& \hspace{-1cm}= \norm{\bigg(\mathbf{I} - \alpha \int_{p=0}^{p=1}\nabla^2 f(\x^* + p(\x^+-\x^*)) dp\bigg)(\x^+-\x^*)}\\
&\hspace{-1cm} = \norm{\bigg(\mathbf{I} - \alpha \int_{p=0}^{p=1}\nabla^2 f(\x^* + p(\x^+-\x^*)) dp\bigg)\bigg(\mathbf{I} - \alpha \int_{p=0}^{p=1}\nabla^2 f(\x^* + p(\x-\x^*)) dp\bigg)(\x-\x^*)}\\
& \hspace{-1cm}= \norm{\bigg(\mathbf{I} - \alpha \int_{p=0}^{p=1}\nabla^2 f(\x^* + p(\x-\x^*)) dp - \alpha \mathbf{P}(\x)\bigg)\bigg(\mathbf{I} - \alpha \int_{p=0}^{p=1}\nabla^2 f(\x^* + p(\x-\x^*)) dp\bigg)(\x-\x^*)} \label{checkpt1}
\end{align}
where in the last step we used the following substitution:
\begin{align}
 \int_{p=0}^{p=1}\nabla^2 f(\x^* + p(\x^+-\x^*))dp =  \int_{p=0}^{p=1}\nabla^2 f(\x^* + p(\x-\x^*))dp + \mathbf{P}(\x). \label{perturbation_1}
\end{align}
and we have that $ \norm{\mathbf{P}(\x)} = \mathcal{O}(\norm{\nabla f(\x)})$ which can be verified from Assumption \textbf{A3}. Rearranging \eqref{perturbation_1} and taking norm both sides we get:
\begin{align}
\norm{\mathbf{P}(\x)}_2 & = \norm{\int_{p=0}^{p=1}\bigg(\nabla^2 f(\x^* + p(\x^+-\x^*))-  \nabla^2 f(\x^* + p(\x-\x^*))\bigg)dp }_2 \\
& \leq \int_{p=0}^{p=1}\norm{\bigg(\nabla^2 f(\x^* + p(\x^+-\x^*))-  \nabla^2 f(\x^* + p(\x-\x^*))\bigg)}_2 dp \\
& \leq \int_{p=0}^{p=1}M\norm{p(\x^+ - \x)}dp \\
&= M  \norm{\x^+ - \x}\int_{p=0}^{p=1}p dp = \frac{M \alpha \norm{\nabla f(\x)}}{2}.
\end{align}
Now recall that  $\mathbf{D}(\x) = \bigg(\mathbf{I} - \alpha \int_{p=0}^{p=1}\nabla^2 f(\x^* + p(\x-\x^*)) dp\bigg)$ hence further simplifying \eqref{checkpt1} yields the following:
\begin{align}
\norm{\x^{++}-\x^*} & = \norm{\bigg(\mathbf{D}(\x)\bigg)^2(\x-\x^*) - \alpha\bigg(\mathbf{D}(\x)\mathbf{P}(\x)(\x - \x^*)\bigg)} \\
& \geq \sqrt{\bigg(\sum_{j \in \mathcal{I}_{US}}(\nu^{us}_j)^4(\langle \hat{\u}, \e_j^{us}\rangle)^2  + \sum_{i \in \mathcal{I}_{S}} (\nu^{s}_i)^4  (\langle \hat{\u}, \e_i^{s}\rangle)^2\bigg)}\norm{\x-\x^*} - \alpha\norm{\mathbf{D}(\x)}_2\norm{\mathbf{P}(\x)}_2 \norm{\x-\x^*} \\
& \geq \sqrt{\bigg(\sum_{j \in \mathcal{I}_{US}}(\nu^{us}_j)^4(\langle \hat{\u}, \e_j^{us}\rangle)^2  + \sum_{i \in \mathcal{I}_{S}} (\nu^{s}_i)^4  (\langle \hat{\u}, \e_i^{s}\rangle)^2\bigg)}\norm{\x-\x^*} -   \frac{\sup_{j}\{\nu^{us}_j\} M \alpha \norm{\nabla f(\x)}\norm{\x-\x^*}}{2} \\
& \geq \sqrt{\bigg(\sum_{j \in \mathcal{I}_{US}}(\nu^{us}_j)^4(\langle \hat{\u}, \e_j^{us}\rangle)^2  + \sum_{i \in \mathcal{I}_{S}} (\nu^{s}_i)^4  (\langle \hat{\u}, \e_i^{s}\rangle)^2\bigg)}\norm{\x-\x^*} -  \frac{\sup_{j}\{\nu^{us}_j\} M L\alpha \norm{\x-\x^*}^2}{2}\label{parta_partb}
\end{align}
where we used the fact that $\norm{\nabla f(\x)} \leq L \norm{\x - \x^*}$ by Lipschitz continuity of $\nabla f(\x)$. Now with $\sigma(\x) =\frac{\sup_{j}\{\nu^{us}_j\} M L\alpha \norm{\x-\x^*}^2}{2}  = \mathcal{O}(\norm{\x-\x^*}^2)$ we are left to prove:
 $$ \sqrt{\bigg(\sum_{j \in \mathcal{I}_{US}}(\nu^{us}_j)^4(\langle \hat{\u}, \e_j^{us}\rangle)^2  + \sum_{i \in \mathcal{I}_{S}} (\nu^{s}_i)^4  (\langle \hat{\u}, \e_i^{s}\rangle)^2\bigg)}\norm{\x-\x^*} > \norm{\x^+-\x^*} $$ or equivalently the following result:
 \begin{align}
 \sqrt{\bigg(\sum_{j \in \mathcal{I}_{US}}(\nu^{us}_j)^4(\langle \hat{\u}, \e_j^{us}\rangle)^2  + \sum_{i \in \mathcal{I}_{S}} (\nu^{s}_i)^4  (\langle \hat{\u}, \e_i^{s}\rangle)^2\bigg)}\norm{\u} &>\sqrt{\bigg(\sum_{j \in \mathcal{I}_{US}}(\nu^{us}_j\langle \hat{\u}, \e_j^{us}\rangle)^2  + \sum_{i \in \mathcal{I}_{S}} (\nu^{s}_i  \langle \hat{\u}, \e_i^{s}\rangle)^2\bigg)}\norm{\u} \\
  \sqrt{\bigg(\sum_{j \in \mathcal{I}_{US}}(\nu^{us}_j)^4(\langle \hat{\u}, \e_j^{us}\rangle)^2  + \sum_{i \in \mathcal{I}_{S}} (\nu^{s}_i)^4  (\langle \hat{\u}, \e_i^{s}\rangle)^2\bigg)} &>\sqrt{\bigg(\sum_{j \in \mathcal{I}_{US}}(\nu^{us}_j\langle \hat{\u}, \e_j^{us}\rangle)^2  + \sum_{i \in \mathcal{I}_{S}} (\nu^{s}_i  \langle \hat{\u}, \e_i^{s}\rangle)^2\bigg)}. \label{checkresult1}
 \end{align}

This will hold true if:
\begin{align}
\sqrt{\bigg(\sum_{j \in \mathcal{I}_{US}}(\nu^{us}_j\langle \hat{\u}, \e_j^{us}\rangle)^2  + \sum_{i \in \mathcal{I}_{S}} (\nu^{s}_i  \langle \hat{\u}, \e_i^{s}\rangle)^2\bigg)} > 1. \label{checkresult2}
\end{align}

Recall that $(\e_j^{us}, \nu_j^{us})$, $(\e_j^{s}, \nu_j^{s})$ are respectively the eigenvector-eigenvalue pair of the matrix $\mathbf{D}(\x) =\bigg(\mathbf{I} - \alpha \int_{p=0}^{p=1}\nabla^2 f(\x^* + p(\x-\x^*)) dp\bigg)$ with $\nu^s_i<1$ for all $i \in \mathcal{I}_{S}$, $\nu^{us}_j\geq 1$ for all $j \in \mathcal{I}_{US}$.
Then the condition \eqref{checkresult1} can be written as:
\begin{align}
\sqrt{\langle \hat{\u}, (\mathbf{D}(\x))^4\hat{\u} \rangle} & > \sqrt{\langle  \hat{\u}, (\mathbf{D}(\x))^2\hat{\u}\rangle} \\
\implies \langle \hat{\u}, (\mathbf{D}(\x))^4\hat{\u} \rangle & > \langle  \hat{\u}, (\mathbf{D}(\x))^2\hat{\u}\rangle
\end{align}
where $\hat{\u}$ is a unit vector. Also we are given \eqref{assertion1} that can be written as:
\begin{align}
\sqrt{\langle  \hat{\u}, (\mathbf{D}(\x))^2\hat{\u}\rangle} &> 1 = \sqrt{\langle \hat{\u}, \hat{\u} \rangle} \\
\implies \langle \hat{\u}, ((\mathbf{D}(\x))^2-\mathbf{I})\hat{\u} \rangle &> 0.
\end{align}
Now consider the following difference:
\begin{align}
\langle \hat{\u}, (\mathbf{D}(\x))^4\hat{\u} \rangle  - \langle  \hat{\u}, (\mathbf{D}(\x))^2\hat{\u}\rangle &= \underbrace{\langle \hat{\u}, ((\mathbf{D}(\x))^2-\mathbf{I})^2\hat{\u} \rangle}_{\geq 0} + \underbrace{\langle \hat{\u}, ((\mathbf{D}(\x))^2-\mathbf{I})\hat{\u} \rangle}_{>0} > 0 \label{checkproof1} \\
\implies \langle \hat{\u}, (\mathbf{D}(\x))^4\hat{\u} \rangle & > \langle  \hat{\u}, (\mathbf{D}(\x))^2\hat{\u}\rangle
\end{align}
which completes the proof for \eqref{checkresult1}.
We are now ready to prove the result $ \norm{\x^{++}-\x^*} \geq \bar{\rho}(\x)\norm{\x^+-\x^*} - \sigma(\x)$. Recall that from \eqref{weyl} we have that:
\begin{align}
\norm{\x^{+}-\x^*} &= \sqrt{\bigg(\sum_{j \in \mathcal{I}_{US}}(\nu^{us}_j\langle \hat{\u}, \e_j^{us}\rangle)^2  + \sum_{i \in \mathcal{I}_{S}} (\nu^{s}_i  \langle \hat{\u}, \e_i^{s}\rangle)^2\bigg)}\norm{\x-\x^*}\\
&= \sqrt{\langle  \hat{\u}, (\mathbf{D}(\x))^2\hat{\u}\rangle}  \norm{\x-\x^*} \label{subs_1}
\end{align}
Now from \eqref{parta_partb} we have the following:
\begin{align}
\norm{\x^{++}-\x^*} & \geq \sqrt{\bigg(\sum_{j \in \mathcal{I}_{US}}(\nu^{us}_j)^4(\langle \hat{\u}, \e_j^{us}\rangle)^2  + \sum_{i \in \mathcal{I}_{S}} (\nu^{s}_i)^4  (\langle \hat{\u}, \e_i^{s}\rangle)^2\bigg)}\norm{\x-\x^*} - \mathcal{O}(\norm{\x-\x^*}^2) \\
& = \sqrt{\langle  \hat{\u}, (\mathbf{D}(\x))^4\hat{\u}\rangle} \norm{\x-\x^*} - \sigma(\x) \\
& = \frac{\sqrt{\langle  \hat{\u}, (\mathbf{D}(\x))^4\hat{\u}\rangle} }{\sqrt{\langle  \hat{\u}, (\mathbf{D}(\x))^2\hat{\u}\rangle} } \sqrt{\langle  \hat{\u}, (\mathbf{D}(\x))^2\hat{\u}\rangle}  \norm{\x-\x^*} - \sigma(\x)  \\
& =\frac{\sqrt{\langle  \hat{\u}, (\mathbf{D}(\x))^4\hat{\u}\rangle} }{\sqrt{\langle  \hat{\u}, (\mathbf{D}(\x))^2\hat{\u}\rangle} }   \norm{\x^+-\x^*} - \sigma(\x)
\end{align}
where in the last step we used the substitution from \eqref{subs_1}. Next, note that $\langle \hat{\u}, (\mathbf{D}(\x))^4\hat{\u} \rangle  > \langle \hat{\u}, (\mathbf{D}(\x))^2\hat{\u} \rangle >1 $ and hence we can set $\bar{\rho}(\x)=\frac{\sqrt{\langle  \hat{\u}, (\mathbf{D}(\x))^4\hat{\u}\rangle} }{\sqrt{\langle  \hat{\u}, (\mathbf{D}(\x))^2\hat{\u}\rangle} }  > 1$ to complete the proof.

Next, we show that the claim in \eqref{parta.} holds, i.e., $\norm{\x^{++}-\x^*} > \norm{\x^+-\x^*} $ provided $ \norm{\x-\x^*}$ is bounded above. It can be done using \eqref{partb.} of the result where we lower bound the right hand side of \eqref{partb.} with $ \norm{\x^+-\x^*}$ to get:
\begin{align}
\norm{\x^{++}-\x^*} \geq \bar{\rho}(\x)\norm{\x^+-\x^*} - \sigma(\x) &> \norm{\x^+-\x^*} \\
\implies (\bar{\rho}(\x)-1)\norm{\x^+-\x^*} &> \sigma(\x) . \label{partb_variation}
\end{align}
Since $\sigma(\x) = \mathcal{O}(\norm{\x-\x^*}^2)$, hence $\norm{\x-\x^*}$ should be sufficiently small for \eqref{partb_variation} to hold. Now, if \eqref{partb_variation} condition holds true, then we will have the condition $$\norm{\x^{++}-\x^*} \geq \bar{\rho}(\x)\norm{\x^+-\x^*} - \sigma(\x) > \norm{\x^+-\x^*} $$ or equivalently $ \norm{\x^{++}-\x^*}  > \norm{\x^+-\x^*}$. Next, for some $k=K$ let $\x = \x_K$, $\x^+ = \x_{K+1}$, $ \x^{++} = \x_{K+2}$ and we have $\norm{\x_{K+1} - \x^*} > \norm{\x_K -\x^*}$ with the condition \eqref{partb_variation} satisfied, then we also have $\norm{\x_{K+2} - \x^*} > \norm{\x_{K+1} -\x^*}$. Using induction, we then get $\norm{\x_{k+1} - \x^*} > \norm{\x_k -\x^*}$  for all $ k \geq K+1 $ provided \eqref{partb_variation} holds true with $\x = \x_k$.

\end{proof}

Hence, the claim of sequential monotonicity has been proved partially, i.e., if a gradient trajectory has expansive dynamics w.r.t. stationary point $\x^*$ at some $k=K$, then it has expansive dynamics for all iterations $k>K$ provided $\norm{\x_k - \x^*}$ remains bounded above.\footnote{Notice that $\x^*$ can be any stationary point and not just the strict saddle point. Since the stationary points of the function are non-degenerate from our assumptions, the extension of this proof to other types of stationary points is left as an easy exercise to the reader.}
Now, we are only left with proving the complete claim, i.e., sequential monotonicity holds even if the gradient trajectory has non-contraction dynamics w.r.t. stationary point $\x^*$ at some $k=K$. Before completing the proof of this claim, we need to do provide a bound on the expansion factor $\bar{\rho}(\x)$.

\begin{lem}\label{suplem2}
The expansion factor $\bar{\rho}(\x)$ in \eqref{partb.} is bounded as $\bar{\rho}(\x)  >1 + \frac{\bigg((1+\frac{\beta}{L})^2+ \frac{1}{4 (1+\frac{\beta}{L})^2}-\frac{5}{4}\bigg)}{12} $.
\end{lem}

\begin{proof}
From the condition \eqref{partb_variation}, we require $\sigma(\x)$ to be upper bounded. Notice that the upper bound on $\sigma(\x)$ goes to $0$ as $\bar{\rho}(\x)$ approaches $1$. Then, the particular theorem cannot be applied recursively since $\sigma(\x)$ is a positive quantity that comes from \eqref{parta_partb} and \eqref{partb_variation} would then fail to hold. Hence, in order to exploit the property \eqref{partb_variation}, we require $\bar{\rho}(\x)$ to be bounded away from $1$. Using \eqref{subs_1} in \eqref{partb_variation} and simplifying $\bar{\rho}(\x)$, we get that:
\begin{align}
(\bar{\rho}(\x)-1)\norm{\x^+-\x^*} &> \sigma(\x) \\
\implies \bigg( \frac{\sqrt{\langle  \hat{\u}, (\mathbf{D}(\x))^4\hat{\u}\rangle} }{\sqrt{\langle  \hat{\u}, (\mathbf{D}(\x))^2\hat{\u}\rangle} }-1\bigg)\sqrt{\langle  \hat{\u}, (\mathbf{D}(\x))^2\hat{\u}\rangle}\norm{\x-\x^*} &> \sigma(\x) \\
\implies \bigg(\sqrt{\langle  \hat{\u}, (\mathbf{D}(\x))^4\hat{\u}\rangle} - \sqrt{\langle  \hat{\u}, (\mathbf{D}(\x))^2\hat{\u}\rangle}\bigg)\norm{\x-\x^*} &> \sigma(\x) \label{sigmaerrorradii}
\end{align}
where we require the term $ \bigg(\sqrt{\langle  \hat{\u}, (\mathbf{D}(\x))^4\hat{\u}\rangle} - \sqrt{\langle  \hat{\u}, (\mathbf{D}(\x))^2\hat{\u}\rangle}\bigg)$ to be bounded away from $0$. This will hold true due to the following fact:
\begin{align}
    \sqrt{\langle  \hat{\u}, (\mathbf{D}(\x))^4\hat{\u}\rangle} &= \sqrt{ \bigg(\sum_{j \in \mathcal{I}_{US}}(\nu^{us}_j)^4(\langle \hat{\u}, \e_j^{us}\rangle)^2  + \sum_{i \in \mathcal{I}_{S}} (\nu^{s}_i)^4  (\langle \hat{\u}, \e_i^{s}\rangle)^2\bigg)} \\
    & \geq \bigg(\sum_{j \in \mathcal{I}_{US}}\sqrt{(\nu^{us}_j)^4}(\langle \hat{\u}, \e_j^{us}\rangle)^2  + \sum_{i \in \mathcal{I}_{S}} \sqrt{(\nu^{s}_i)^4}  (\langle \hat{\u}, \e_i^{s}\rangle)^2\bigg) \\
    & = \langle  \hat{\u}, (\mathbf{D}(\x))^2\hat{\u}\rangle > \sqrt{\langle  \hat{\u}, (\mathbf{D}(\x))^2\hat{\u}\rangle}
\end{align}
where we used the Jensen's inequality for square root function followed by the fact that $\langle  \hat{\u}, (\mathbf{D}(\x))^2\hat{\u}\rangle > 1$. But in order to develop a bound on the radius of ball inside which sequential monotonicity holds, we require something more. Notice that if we plug in the naive lower bound just obtained into \eqref{sigmaerrorradii}, all we can get is a projection dependent term which does not generalize to the class of functions being studied. The goal here is to obtain some bound that is independent of $\hat{\u}$ and solely depends on the function parameters like condition number, etc. The next steps develop a generalized lower bound for $\bigg(\sqrt{\langle  \hat{\u}, (\mathbf{D}(\x))^4\hat{\u}\rangle} - \sqrt{\langle  \hat{\u}, (\mathbf{D}(\x))^2\hat{\u}\rangle}\bigg) $ independent of $\hat{\u}$.

Since we have $ \sqrt{\langle  \hat{\u}, (\mathbf{D}(\x))^2\hat{\u}\rangle}>1$, we can write:
\begin{align}
 \bigg(\sqrt{\langle  \hat{\u}, (\mathbf{D}(\x))^4\hat{\u}\rangle} - \sqrt{\langle  \hat{\u}, (\mathbf{D}(\x))^2\hat{\u}\rangle}\bigg)& = \frac{{\langle  \hat{\u}, (\mathbf{D}(\x))^4\hat{\u}\rangle} - \langle  \hat{\u}, (\mathbf{D}(\x))^2\hat{\u}\rangle}{\sqrt{\langle  \hat{\u}, (\mathbf{D}(\x))^4\hat{\u}\rangle} + \sqrt{\langle  \hat{\u}, (\mathbf{D}(\x))^2\hat{\u}\rangle}} \label{rho_bound}
\end{align}
where we require $ \langle  \hat{\u}, (\mathbf{D}(\x))^4\hat{\u}\rangle >  \langle  \hat{\u}, (\mathbf{D}(\x))^2\hat{\u}\rangle$.
Next, substituting $ \langle  \hat{\u}, (\mathbf{D}(\x))^4\hat{\u}\rangle =  \bigg(\sum_{j \in \mathcal{I}_{US}}(\nu^{us}_j)^4(\langle \hat{\u}, \e_j^{us}\rangle)^2  + \sum_{i \in \mathcal{I}_{S}} (\nu^{s}_i)^4  (\langle \hat{\u}, \e_i^{s}\rangle)^2\bigg)$ and
 $\langle  \hat{\u}, (\mathbf{D}(\x))^2\hat{\u}\rangle =  \bigg(\sum_{j \in \mathcal{I}_{US}}(\nu^{us}_j\langle \hat{\u}, \e_j^{us}\rangle)^2  + \sum_{i \in \mathcal{I}_{S}} (\nu^{s}_i  \langle \hat{\u}, \e_i^{s}\rangle)^2\bigg)$ in the left-hand side of \eqref{rho_bound} followed by simplification yields:
\begin{align}
\frac{{\langle  \hat{\u}, (\mathbf{D}(\x))^4\hat{\u}\rangle} - \langle  \hat{\u}, (\mathbf{D}(\x))^2\hat{\u}\rangle}{\sqrt{\langle  \hat{\u}, (\mathbf{D}(\x))^4\hat{\u}\rangle} + \sqrt{\langle  \hat{\u}, (\mathbf{D}(\x))^2\hat{\u}\rangle}} &= \frac{\bigg(\sum_{j \in \mathcal{I}_{US}}\bigg({(\nu^{us}_j)^4}-(\nu^{us}_j)^2\bigg)(\langle \hat{\u}, \e_j^{us}\rangle)^2  + \sum_{i \in \mathcal{I}_{S}} \bigg({(\nu^{s}_i)^4}-(\nu^{s}_i)^2\bigg)( \langle \hat{\u}, \e_i^{s}\rangle)^2\bigg)}{\sqrt{\langle  \hat{\u}, (\mathbf{D}(\x))^4\hat{\u}\rangle} + \sqrt{\langle  \hat{\u}, (\mathbf{D}(\x))^2\hat{\u}\rangle}}. \label{rho_bound0}
\end{align}

Now recall that we had $\mathbf{D}(\x) = \bigg(\mathbf{I} - \alpha \int_{p=0}^{p=1}\nabla^2 f(\x^* + p(\x-\x^*)) dp\bigg)$, hence for any eigenvalue $\nu_l$ of the matrix $\mathbf{D}(\x)$ where $\nu_l = 1 - \alpha\lambda_l(\int_{p=0}^{p=1}\nabla^2 f(\x^* + p(\x-\x^*)) dp)$ and $1 \leq l\leq n$ with $\nu_l \geq 0$ and $\lambda_l$ is the corresponding eigenvalue of $\int_{p=0}^{p=1}\nabla^2 f(\x^* + p(\x-\x^*)) dp$, we have that:
\begin{align}
        {\norm{ \bigg( \int_{p=0}^{p=1}\bigg(\mathbf{I} -\alpha\nabla^2 f(\x^* + p(\x-\x^*)) \bigg)dp\bigg)^{-1}}_2^{-1}}
     & \leq    \nu_l  \leq \norm{  \int_{p=0}^{p=1}\bigg(\mathbf{I} -\alpha\nabla^2 f(\x^* + p(\x-\x^*)) \bigg) dp}_2 \\
      \int_{p=0}^{p=1}\norm{\bigg(\mathbf{I} -\alpha\nabla^2 f(\x^* + p(\x-\x^*)) \bigg)^{-1}}_2^{-1}dp
     & \leq    \nu_l  \leq  \int_{p=0}^{p=1}\norm{ \bigg(\mathbf{I} -\alpha\nabla^2 f(\x^* + p(\x-\x^*)) \bigg)}_2 dp \\
1 - \alpha \int_{p=0}^{p=1}\sup_{l}\lambda_l(\nabla^2 f(\x^* + p(\x-\x^*))) dp &\leq \nu_l \leq 1 - \alpha \int_{p=0}^{p=1}\inf_{l}\lambda_l(\nabla^2 f(\x^* + p(\x-\x^*))) dp.
\end{align}
Therefore, the bounds on $\nu^s_i$ and $\nu^{us}_j$ for $\alpha = \frac{1}{L}$ can be given by:
\begin{align}
1 - \alpha \int_{p=0}^{p=1}\sup_{\lambda_l>0}\lambda_l(\nabla^2 f(\x^* + p(\x-\x^*))) dp &\leq \nu^s_i \leq 1 - \alpha \int_{p=0}^{p=1}\inf_{\lambda_l>0}\lambda_l(\nabla^2 f(\x^* + p(\x-\x^*))) dp \\
1 - \alpha \int_{p=0}^{p=1}L dp &\leq \nu^s_i \leq 1 - \alpha \int_{p=0}^{p=1}\beta dp \\
0 &\leq \nu^s_i \leq 1 - \frac{\beta}{L}
\end{align}
\begin{align}
1 - \alpha \int_{p=0}^{p=1}\sup_{\lambda_l<0}\lambda_l(\nabla^2 f(\x^* + p(\x-\x^*))) dp &\leq \nu^{us}_j \leq 1 - \alpha \int_{p=0}^{p=1}\inf_{\lambda_l<0}\lambda_l(\nabla^2 f(\x^* + p(\x-\x^*))) dp \\
1 - \alpha \int_{p=0}^{p=1}-\beta dp &\leq \nu^{us}_j \leq 1 - \alpha \int_{p=0}^{p=1}-L dp \\
1 + \frac{\beta}{L} &\leq \nu^{us}_j \leq 2
\end{align}
where we used the fact that $ \inf_{l}\abs{\lambda_l(\nabla^2 f(\x^* + p(\x-\x^*)))} > \beta$, i.e., the minimum absolute eigenvalue of the function $f(\cdot)$ in a neighborhood of $\x^*$ is greater than $\beta$ from Assumption \textbf{A4}. Also, we used $ \sup_{l}\abs{\lambda_l(\nabla^2 f(\x^* + p(\x-\x^*)))} \leq L$, from Assumption \textbf{A2}.

Hence, the R.H.S. in \eqref{rho_bound0} can be lower bounded as:
\begin{align}
\frac{{\langle  \hat{\u}, (\mathbf{D}(\x))^4\hat{\u}\rangle} - \langle  \hat{\u}, (\mathbf{D}(\x))^2\hat{\u}\rangle}{\sqrt{\langle  \hat{\u}, (\mathbf{D}(\x))^4\hat{\u}\rangle} + \sqrt{\langle  \hat{\u}, (\mathbf{D}(\x))^2\hat{\u}\rangle}} &= \frac{\bigg(\sum_{j \in \mathcal{I}_{US}}\bigg({(\nu^{us}_j)^4}-(\nu^{us}_j)^2\bigg)(\langle \hat{\u}, \e_j^{us}\rangle)^2  + \sum_{i \in \mathcal{I}_{S}} \bigg({(\nu^{s}_i)^4}-(\nu^{s}_i)^2\bigg)( \langle \hat{\u}, \e_i^{s}\rangle)^2\bigg)}{\sqrt{\langle  \hat{\u}, (\mathbf{D}(\x))^4\hat{\u}\rangle} + \sqrt{\langle  \hat{\u}, (\mathbf{D}(\x))^2\hat{\u}\rangle}}\\
 &\geq \frac{\bigg(\sum_{j \in \mathcal{I}_{US}}\bigg({(1+\frac{\beta}{L})^4}-(1+\frac{\beta}{L})^2\bigg)(\langle \hat{\u}, \e_j^{us}\rangle)^2  -\frac{1}{4} \sum_{i \in \mathcal{I}_{S}}(  \langle \hat{\u}, \e_i^{s}\rangle)^2\bigg)}{\sqrt{\langle  \hat{\u}, (\mathbf{D}(\x))^4\hat{\u}\rangle} + \sqrt{\langle  \hat{\u}, (\mathbf{D}(\x))^2\hat{\u}\rangle}} \label{rho_bound1}
\end{align}
where we used the fact that ${\nu^{us}_j}\geq {(1+\frac{\beta}{L})}$ and $\bigg({(\nu^{s}_i)^4}-(\nu^{s}_i)^2\bigg)\geq- \frac{1}{4}$ for $\nu_i^s <1$ (minimum of $h(y)= y^4 -y^2$ for $0\leq y<1$ is $-\frac{1}{4}$).

Next we minimize the numerator of the R.H.S. in \eqref{rho_bound1} in a way so as to get rid of the dependency on $\hat{\u}$. Recall that the minimization of $ \bigg(\sum_{j \in \mathcal{I}_{US}}\bigg({(1+\frac{\beta}{L})^4}-(1+\frac{\beta}{L})^2\bigg)(\langle \hat{\u}, \e_j^{us}\rangle)^2  -\frac{1}{4} \sum_{i \in \mathcal{I}_{S}}(  \langle \hat{\u}, \e_i^{s}\rangle)^2\bigg)$ is constrained by $$\sum_{j \in \mathcal{I}_{US}}(\langle \hat{\u}, \e_j^{us}\rangle)^2  +\sum_{i \in \mathcal{I}_{S}}( \langle \hat{\u}, \e_i^{s}\rangle)^2 = 1$$ and
$$ \sum_{j \in \mathcal{I}_{US}}(\nu^{us}_j\langle \hat{\u}, \e_j^{us}\rangle)^2  +\sum_{i \in \mathcal{I}_{S}}(\nu^{s}_i \langle \hat{\u}, \e_i^{s}\rangle)^2 > 1 \iff \sum_{j \in \mathcal{I}_{US}}((\nu^{us}_j)^2-1)(\langle \hat{\u}, \e_j^{us}\rangle)^2  +\sum_{i \in \mathcal{I}_{S}}((\nu^{s}_i)^2-1)( \langle \hat{\u}, \e_i^{s}\rangle)^2 > 0$$
where the second constraint comes from \eqref{checkresult2}. Relaxing the second constraint by using the bounds $ \nu^{s}_i \geq 0$, $\nu^{us}_j \geq (1+\frac{\beta}{L}) $ we get:
$$((1+{\beta}/{L})^2-1)\sum_{j \in \mathcal{I}_{US}}(\langle \hat{\u}, \e_j^{us}\rangle)^2  -\sum_{i \in \mathcal{I}_{S}}( \langle \hat{\u}, \e_i^{s}\rangle)^2 > 0.$$
Let $a = \sum_{j \in \mathcal{I}_{US}}(\langle \hat{\u}, \e_j^{us}\rangle)^2, b = \sum_{i \in \mathcal{I}_{S}}( \langle \hat{\u}, \e_i^{s}\rangle)^2$ then from the two constraints we have the following minimization problem for the numerator term in \eqref{rho_bound1}:
\begin{align*}
    & \min_{a,b\geq 0} \bigg(\bigg((1+\frac{\beta}{L})^4-(1+\frac{\beta}{L})^2\bigg)a -\frac{1}{4} b\bigg) \\
    \text{s.t. } & a+b =1 \\
    & ((1+{\beta}/{L})^2-1)a - b > 0.
\end{align*}
Solving this geometrically we obtain that the minimum is attained at the intersection of lines $a+b =1 $ and $ ((1+{\beta}/{L})^2-1)a - b = 0$ which gives $a = \frac{1}{(1+{\beta}/{L})^2}$ and $ b = 1- \frac{1}{(1+{\beta}/{L})^2}$. Substituting $a, b$ in our function $\bigg(\bigg({(1+\frac{\beta}{L})^4}-(1+\frac{\beta}{L})^2\bigg)a -\frac{1}{4} b\bigg) $ yields the following lower bound in \eqref{rho_bound1}:
\begin{align}
   \frac{\bigg(\sum_{j \in \mathcal{I}_{US}}\bigg({(1+\frac{\beta}{L})^4}-(1+\frac{\beta}{L})^2\bigg)(\langle \hat{\u}, \e_j^{us}\rangle)^2  -\frac{1}{4} \sum_{i \in \mathcal{I}_{S}}(  \langle \hat{\u}, \e_i^{s}\rangle)^2\bigg)}{\sqrt{\langle  \hat{\u}, (\mathbf{D}(\x))^4\hat{\u}\rangle} + \sqrt{\langle  \hat{\u}, (\mathbf{D}(\x))^2\hat{\u}\rangle}} & >  \frac{\bigg((1+\frac{\beta}{L})^2+ \frac{1}{4 (1+\frac{\beta}{L})^2}-\frac{5}{4}\bigg)}{\sqrt{\langle  \hat{\u}, (\mathbf{D}(\x))^4\hat{\u}\rangle} + \sqrt{\langle  \hat{\u}, (\mathbf{D}(\x))^2\hat{\u}\rangle}} \\
   & >  \frac{\bigg((1+\frac{\beta}{L})^2+ \frac{1}{4 (1+\frac{\beta}{L})^2}-\frac{5}{4}\bigg)}{6} \label{rho_bound4}
\end{align}
where in the last step we used the fact that the maximum eigenvalue of $ (\mathbf{D}(\x))^2$ is $4$ which implies $\sqrt{\langle  \hat{\u}, (\mathbf{D}(\x))^2\hat{\u}\rangle}<2 $ and $\sqrt{\langle  \hat{\u}, (\mathbf{D}(\x))^4\hat{\u}\rangle}<4 $.

Now, it can be verified that for values of $\frac{\beta}{L} > 0$, the right-hand side of \eqref{rho_bound4} is bounded away from $0$. Since $ \bar{\rho}(\x) = \frac{\sqrt{\langle  \hat{\u}, (\mathbf{D}(\x))^4\hat{\u}\rangle} }{\sqrt{\langle  \hat{\u}, (\mathbf{D}(\x))^2\hat{\u}\rangle} }$, then using \eqref{rho_bound4} and $\sqrt{\langle  \hat{\u}, (\mathbf{D}(\x))^2\hat{\u}\rangle}<2 $ we can write $$ \bar{\rho}(\x) = 1 + \frac{\sqrt{\langle  \hat{\u}, (\mathbf{D}(\x))^4\hat{\u}\rangle} - \sqrt{\langle  \hat{\u}, (\mathbf{D}(\x))^2\hat{\u}\rangle}}{ \sqrt{\langle  \hat{\u}, (\mathbf{D}(\x))^2\hat{\u}\rangle}} >1 + \frac{\bigg((1+\frac{\beta}{L})^2+ \frac{1}{4 (1+\frac{\beta}{L})^2}-\frac{5}{4}\bigg)}{12}$$ which is an expansion factor for any $\frac{\beta}{L} > 0$.
\end{proof}

We now extend the claim of Lemma \ref{suplem1} to the case of non-contraction, i.e., $\norm{\x^+-\x^*} = \norm{\x-\x^*} $. In words, we show that sequential monotonicity property from \eqref{parta.} holds even if the gradient trajectory has non-contraction dynamics w.r.t. stationary point $\x^*$ at some $k=K$.

\begin{lem}\label{suplem3}
For an iterative gradient mapping given by $\x^+ = \x - \alpha \nabla f(\x)$ in some neighborhood of $\x^*$, if $\norm{\x^+-\x^*} = \norm{\x-\x^*} $ then the following holds:
\begin{align}
\textbf{a.} \hspace{0.5cm} \norm{\x^{++}-\x^*} &\geq \bar{\rho}(\x)\norm{\x^+-\x^*} - \sigma(\x) \\
\textbf{b.} \hspace{0.5cm}  \norm{\x^{++}-\x^*} &> \norm{\x^+-\x^*}
\end{align}
where $\sigma(\x) = \mathcal{O}(\norm{\x-\x^*}^2)$ and $\bar{\rho}(\x)>1 $.
\end{lem}

\begin{proof}
Notice that while obtaining \eqref{rho_bound4} from \eqref{rho_bound1}, we utilized the given condition of \eqref{checkresult2} according to which we have: $$ \sum_{j \in \mathcal{I}_{US}}(\nu^{us}_j\langle \hat{\u}, \e_j^{us}\rangle)^2  + \sum_{i \in \mathcal{I}_{S}} (\nu^{s}_i  \langle \hat{\u}, \e_i^{s}\rangle)^2 > 1. $$ This condition implies that we have $\norm{\x^+-\x^*} > \norm{\x-\x^*}$. However, it could be the case that we have $\norm{\x^+-\x^*} = \norm{\x-\x^*}$ which would imply $$\langle  \hat{\u}, (\mathbf{D}(\x))^2\hat{\u}\rangle = \sum_{j \in \mathcal{I}_{US}}(\nu^{us}_j\langle \hat{\u}, \e_j^{us}\rangle)^2  + \sum_{i \in \mathcal{I}_{S}} (\nu^{s}_i  \langle \hat{\u}, \e_i^{s}\rangle)^2 = 1. $$ Using this condition, it can be readily checked that \eqref{rho_bound4} will still hold but only with a non-strict inequality, i.e., we will have:
 $$ \bigg(\sqrt{\langle  \hat{\u}, (\mathbf{D}(\x))^4\hat{\u}\rangle} - \sqrt{\langle  \hat{\u}, (\mathbf{D}(\x))^2\hat{\u}\rangle}\bigg) \geq  \frac{\bigg((1+\frac{\beta}{L})^2+ \frac{1}{4 (1+\frac{\beta}{L})^2}-\frac{5}{4}\bigg)}{6}.$$ Now since $\bar{\rho}(\x)=\frac{\sqrt{\langle  \hat{\u}, (\mathbf{D}(\x))^4\hat{\u}\rangle} }{\sqrt{\langle  \hat{\u}, (\mathbf{D}(\x))^2\hat{\u}\rangle} } =  \sqrt{\langle  \hat{\u}, (\mathbf{D}(\x))^4\hat{\u}\rangle}$, we will have that:
 \begin{align}
 \bar{\rho}(\x)& \geq \sqrt{\langle  \hat{\u}, (\mathbf{D}(\x))^2\hat{\u}\rangle} + \frac{\bigg((1+\frac{\beta}{L})^2+ \frac{1}{4 (1+\frac{\beta}{L})^2}-\frac{5}{4}\bigg)}{6} \\
 &= 1 + \frac{\bigg((1+\frac{\beta}{L})^2+ \frac{1}{4 (1+\frac{\beta}{L})^2}-\frac{5}{4}\bigg)}{6} > 1 + \frac{\bigg((1+\frac{\beta}{L})^2+ \frac{1}{4 (1+\frac{\beta}{L})^2}-\frac{5}{4}\bigg)}{12}. \label{rhomonotonic1}
 \end{align}
Now if $\sigma(\x)$ satisfies the condition \eqref{partb_variation} for this $ \bar{\rho}(\x)$ then we are guaranteed to have $ \norm{\x^{++}-\x^*}> \norm{\x^+ - \x^*}$ even when $ \norm{\x^{+}-\x^*}= \norm{\x - \x^*}$. This completes the proof of the claim.
\end{proof}

Now that we have established the result that if $ \norm{\x^{+}-\x^*}\geq \norm{\x - \x^*}$, then we are guaranteed to have $ \norm{\x^{++}-\x^*}> \norm{\x^+ - \x^*}$ provided $\sigma(\x)$ satisfies the condition \eqref{partb_variation}, we can apply this result recursively for any gradient trajectory generated by the sequence $\{\x_k\}$ in some neighborhood of $\x^*$. The next lemma provides a handle on the radius of this neighborhood inside which the sequential monotonicity property holds.

\begin{lem}\label{suplem4}
The sequential monotonicity property from Lemma \ref{suplem1} and \ref{suplem3} holds for the tuple $\{\x,\x^+,\x^{++}\}$ whenever $   \norm{\x-\x^*} \leq \frac{1}{ \varsigma M } \frac{\bigg((1+\frac{\beta}{L})^2+ \frac{1}{4 (1+\frac{\beta}{L})^2}-\frac{5}{4}\bigg)}{6}$ for some $\varsigma> 2$.
\end{lem}

\begin{proof}
To identify the radius of this neighborhood, we use \eqref{partb_variation} where we substitute $\sigma(\x)$ from \eqref{parta_partb} and $\bar{\rho}(\x)=\frac{\sqrt{\langle  \hat{\u}, (\mathbf{D}(\x))^4\hat{\u}\rangle} }{\sqrt{\langle  \hat{\u}, (\mathbf{D}(\x))^2\hat{\u}\rangle} } $ to get the condition:
\begin{align}
  (\bar{\rho}(\x)-1)\norm{\x^+-\x^*}>&\sigma(\x) =\frac{\sup_{j}\{\nu^{us}_j\} M L\alpha \norm{\x-\x^*}^2}{2} \\
  \implies \bigg(\sqrt{\langle  \hat{\u}, (\mathbf{D}(\x))^4\hat{\u}\rangle} - \sqrt{\langle  \hat{\u}, (\mathbf{D}(\x))^2\hat{\u}\rangle}\bigg)\norm{\x-\x^*} >&\sigma(\x) =\frac{\sup_{j}\{\nu^{us}_j\} M L\alpha \norm{\x-\x^*}^2}{2} \label{radiusmonotonic}
\end{align}
Now, in order to guarantee the condition \eqref{radiusmonotonic}, for some $\varsigma>2$,  we set $ \bigg(\sqrt{\langle  \hat{\u}, (\mathbf{D}(\x))^4\hat{\u}\rangle} - \sqrt{\langle  \hat{\u}, (\mathbf{D}(\x))^2\hat{\u}\rangle}\bigg)$ equal to $ \frac{1}{\varsigma}$ times its lower bound from \eqref{rho_bound4} and set $\sigma(\x)$ to its upper bound in \eqref{radiusmonotonic} to get the condition:
\begin{align}
\frac{1}{\varsigma} \frac{\bigg((1+\frac{\beta}{L})^2+ \frac{1}{4 (1+\frac{\beta}{L})^2}-\frac{5}{4}\bigg)}{6}\norm{\x-\x^*} &\geq \frac{2 M L\alpha \norm{\x-\x^*}^2}{2} \geq \frac{\sup_{j}\{\nu^{us}_j\} M L\alpha \norm{\x-\x^*}^2}{2} = \sigma(\x) \\
\frac{1}{ \varsigma M } \frac{\bigg((1+\frac{\beta}{L})^2+ \frac{1}{4 (1+\frac{\beta}{L})^2}-\frac{5}{4}\bigg)}{6} &\geq  \norm{\x-\x^*} \label{radiusmonotonic1}
\end{align}
where we used $\alpha = \frac{1}{L}$ and the bound $ \sup_{j}\{\nu^{us}_j\}= 1+ \alpha L \leq 2$. Now for $\frac{\beta}{L}>0$, if \eqref{radiusmonotonic1} is satisfied then the condition \eqref{radiusmonotonic} will hold true. Hence any gradient descent trajectory with $\alpha = \frac{1}{L}$ inside the ball $\mathcal{B}_{\xi}(\x^*)$ will exhibit strictly monotonic expansive dynamics once it has a non-contractive dynamics at any instant.
\end{proof}

Finally combining Lemmas \ref{suplem1}, \ref{suplem2}, \ref{suplem3} and \ref{suplem4}, Theorem \ref{thm2} is established.

\section{Proof of Lemma \ref{polyaklem}} \label{Appendix C}
Before starting the proof of Lemma \ref{polyaklem} we first show that unlike the expansion phase of the trajectory where the iterates satisfy strong monotonicity property from \eqref{partb.}, the iterates belonging to the contraction phase of the trajectory may not necessarily satisfy such property.
 From theorem \ref{thm2} it was established that a gradient trajectory $\{\x_k\}$ with $\x_k \in {\mathcal{B}}_{\xi}(\x^*)$ has expansive dynamics for all $k>K$ if at $k=K$, the gradient trajectory has non-contraction dynamics\footnote{Note: here we assume that $ \x_0 \in \bar{\mathcal{B}}_{\xi}(\x^*) \backslash {\mathcal{B}}_{\xi}(\x^*)$.}. Let there be some $k=K_{\tau}$ such that the sequence $\{\norm{\x_k - \x^*}\}$ is strictly decreasing for all $k \leq K_{\tau}$ and is non-decreasing for $k=K_{\tau}$. Then from Theorem \ref{thm2} we have that $\{\norm{\x_k - \x^*}\}$ is strictly increasing for all $k > K_{\tau}$ provided $\x_k \in {\mathcal{B}}_{\xi}(\x^*)$. Since $\norm{\x_{K_{\tau}} - \x^*}$ is the minimum of the sequence $\{\norm{\x_{k} - \x^*}\}$ with $\x_k \in {\mathcal{B}}_{\xi}(\x^*)$, let $k=K_c$ and $k=K_e$ be the indices with $K_c \leq K_{\tau} \leq K_e$ defined as follows:
\begin{align}
K_c &= \sup \bigg\{k \leq K_{\tau} \bigg\vert  \x_k \in \bar{\mathcal{B}}_{\xi}(\x^*) \backslash \mathcal{B}_{\epsilon}(\x^*) \bigg\} \\
K_e &= \inf \bigg\{k \geq K_{\tau} \bigg\vert   \x_k \in \bar{\mathcal{B}}_{\xi}(\x^*) \backslash \mathcal{B}_{\epsilon}(\x^*) \bigg\}.
\end{align}
 Let the gradient trajectory exit the ball $\mathcal{B}_{\xi}(\x^*)$ at some iteration $\hat{K}_{exit}$. Then the total sojourn time for the gradient trajectory inside the compact shell $\bar{\mathcal{B}}_{\xi}(\x^*) \backslash \mathcal{B}_{\epsilon}(\x^*)$ is $K_c + (\hat{K}_{exit} - K_e) $.

 Since $K_c \leq K_{\tau}$, we have the condition that $\norm{\x_k -\x^*}$ is monotonically decreasing for all $0<k\leq K_c$. However, even with the monotonically decreasing sequence, it cannot be guaranteed that $\norm{\x_k -\x^*}$ will decrease with a geometric rate. This can checked very easily from \eqref{checkproof1} in the proof of theorem \ref{thm2}. From that condition, we are guaranteed geometric expansion since the factor $\bar{\rho}(\x) = \frac{\sqrt{\langle \hat{\u}, (\mathbf{D}(\x))^4\hat{\u} \rangle}}{\sqrt{\langle \hat{\u}, (\mathbf{D}(\x))^2\hat{\u} \rangle}} > 1$ from the inequality:
 \begin{align}
     \langle \hat{\u}, (\mathbf{D}(\x))^4\hat{\u} \rangle  - \langle  \hat{\u}, (\mathbf{D}(\x))^2\hat{\u}\rangle &= \underbrace{\langle \hat{\u}, ((\mathbf{D}(\x))^2-\mathbf{I})^2\hat{\u} \rangle}_{\geq 0} + \underbrace{\langle \hat{\u}, ((\mathbf{D}(\x))^2-\mathbf{I})\hat{\u} \rangle}_{>0} &> 0
 \end{align}
 provided $\norm{\x^+ - \x^*} > \norm{\x-\x^*}$ or equivalently $\langle \hat{\u}, ((\mathbf{D}(\x))^2-\mathbf{I})\hat{\u} \rangle > 0$. Recall that $\norm{\x^+ - \x^*}= \sqrt{\langle  \hat{\u}, (\mathbf{D}(\x))^2\hat{\u}\rangle}  \norm{\x-\x^*}$ from \eqref{subs_1}. However, when we have $\norm{\x^+ - \x^*} < \norm{\x-\x^*}$ or equivalently $\langle \hat{\u}, ((\mathbf{D}(\x))^2-\mathbf{I})\hat{\u} \rangle < 0$ then \eqref{checkproof1} becomes:
  \begin{align}
     \langle \hat{\u}, (\mathbf{D}(\x))^4\hat{\u} \rangle  - \langle  \hat{\u}, (\mathbf{D}(\x))^2\hat{\u}\rangle &= \underbrace{\langle \hat{\u}, ((\mathbf{D}(\x))^2-\mathbf{I})^2\hat{\u} \rangle}_{\geq 0} + \underbrace{\langle \hat{\u}, ((\mathbf{D}(\x))^2-\mathbf{I})\hat{\u} \rangle}_{<0} & \lessgtr 0
 \end{align}
 and therefore it cannot be stated with certainty that $ \bar{\rho}(\x) < 1$ when we have $\norm{\x^+ - \x^*} < \norm{\x-\x^*}$. Hence, we work with the function value sequence $\{f(\x_k)\}$ instead of the iterate sequence $\{\x_k\}$ in order to develop best possible rate of contraction.

 We now prove Lemma \ref{polyaklem}. Taking norm on \eqref{gradtaylor}, using the substitution $\mathbf{G} = \nabla^2 f(\x^* + p(\x-\x^*)) $ followed by taking the lower bound yields:
 \begin{align}
     \norm{\nabla f(\x)} &= \norm{\bigg( \int_{p=0}^{p=1}\nabla^2 f(\x^* + p(\x-\x^*))  dp\bigg)(\x-\x^*)} \\
     \implies \norm{\nabla f(\x)} &\geq  \norm{\bigg( \int_{p=0}^{p=1}\nabla^2 f(\x^* + p(\x-\x^*))  dp\bigg)^{-1}}_2^{-1}\norm{\x-\x^*} \\
          \implies \norm{\nabla f(\x)} &\geq  \bigg( \int_{p=0}^{p=1}\norm{\bigg(\nabla^2 f(\x^* + p(\x-\x^*))\bigg)^{-1}}_2^{-1}  dp \bigg)\norm{\x-\x^*} \\
     \implies \norm{\nabla f(\x)} & \geq  \bigg(\int_{p=0}^{p=1}\lambda_{min} \bigg ( \sqrt{\mathbf{G}\mathbf{G}^T}\bigg )dp \bigg)\norm{\x-\x^*} \\
     \implies \norm{\nabla f(\x)} &\geq  \beta \norm{\x-\x^*} \label{interimbound4}
 \end{align}
 where we used the fact that $ \lambda_{min} \bigg ( \sqrt{\mathbf{G}\mathbf{G}^T}\bigg ) = \beta $ since $ \lambda_{min} \bigg ( \nabla^2 f(\x^* + p(\x-\x^*)) \bigg ) = \beta$ for any $\x^* + p(\x-\x^*) \in \mathcal{W} $ from \textbf{Assumption A4}.

 Next, using gradient Lipschitz condition on $f(\cdot)$ for $\x_k$ and $\x^*$ along with \eqref{interimbound4} we get:
\begin{align}
    f(\x_k) - f(\x^*) \leq \frac{L}{2}\norm{\x_k - \x^*}^2  \leq \frac{L}{2 \beta^2} \norm{\nabla f(\x_k)}^2 \label{optimalsk1}
\end{align}
where \eqref{optimalsk1} holds for any $\x_k \in \mathcal{W}$.

It is important to note that though \eqref{optimalsk1} holds in general for any $\x_k \in \mathcal{W}$, yet it cannot be called the Polyak--{\L}ojasiewicz condition \cite{karimi2016linear} when $\{\x_k\}$ has expansive dynamics locally w.r.t. $\x^*$ because then $ f(\x_k) - f(\x^*)$ may not be positive. In particular Lemma \ref{lemma3} shows that $f(\x_{K_{exit}}) < f(\x^*)$ where $K_{exit}$ is the exit time from the ball $\mathcal{B}_{\epsilon}(\x^*)$ so $f(\x_k) < f(\x^*)$ for all $k> K_{exit}$ by monotonicity of $\{f(\x_k)\}$. Hence \eqref{optimalsk1} becomes trivial in the expansion phase of the trajectory inside the shell $\bar{\mathcal{B}}_{\xi}(\x^*) \backslash \mathcal{B}_{\epsilon}(\x^*)$ due to the fact that $ f(\x_k) - f(\x^*)<0$ for $k>K_{exit}$.

Finally it remains to show that $  f(\x_k) - f(\x^*) > 0$ for the contraction phase provided $ K_c < K_e$ so that \eqref{optimalsk1} is indeed the Polyak--{\L}ojasiewicz condition in this case. We accomplish this by lower bounding the term $ f(\x_{K_c})- f(\x^*)$. Then $  f(\x_k) - f(\x^*) > 0$ for $k < K_c$ will follow immediately from the monotonicity of the sequence $\{f(\x_k)\}$. Observe that the trajectory $\{\x_k\}$ will enter the ball $\mathcal{B}_{\epsilon}(\x^*)$ when $ K_c < K_e$ and to do so it has to contract at $k=K_c$ since $K_c$ is the last iteration for which the trajectory contracts inside the shell $\bar{\mathcal{B}}_{\xi}(\x^*) \backslash \mathcal{B}_{\epsilon}(\x^*)$. Therefore we have that $\norm{\x_{K_c} -\x^*} > \norm{\x_{K_c+1} -\x^*} $. Further simplifying this condition we get:
\begin{align}
    \norm{\x_{K_c} -\x^*}^2 &> \norm{\x_{K_c+1} -\x^*}^2 \label{switchpolyak}\\
    \implies \norm{\x_{K_c} -\x^*}^2 &> \norm{\x_{K_c} - \alpha \nabla f(\x_{K_c}) -\x^*}^2 \\
   \implies \norm{\x_{K_c} -\x^*}^2 &> \norm{\x_{K_c} -\x^*}^2 + \norm{\alpha\nabla f(\x_{K_c})}^2 - 2\langle \alpha \nabla f(\x_{K_c}), \x_{K_c} -\x^*\rangle \\
   \implies \langle \x_{K_c} -\x^* , \nabla f(\x_{K_c}) \rangle & > \frac{\alpha}{2}\norm{\nabla f(\x_{K_c})}^2 \\
   \implies \bigg\langle \x_{K_c} -\x^* , \bigg(\int_{p=0}^1\nabla^2 f(\x^* + p(\x_{K_c}-\x^*))dp \bigg) (\x_{K_c} -\x^*) \bigg \rangle & > \frac{\alpha}{2}\norm{\nabla f(\x_{K_c})}^2 \\
    \implies \bigg\langle \x_{K_c} -\x^* , \nabla^2 f(\x^* )  (\x_{K_c} -\x^*) \bigg \rangle + \frac{M}{2}\norm{\x_{K_c} -\x^*}^3 & > \frac{\alpha}{2}\norm{\nabla f(\x_{K_c})}^2 \label{newbound11}
\end{align}
where we used the substitution $ \nabla f(\x_{K_c}) = \bigg(\int_{p=0}^1\nabla^2 f(\x^* + p(\x_{K_c}-\x^*))dp \bigg) (\x_{K_c} -\x^*)$ and the following bound:
\begin{align}
  \norm{ \bigg(\int_{p=0}^1\nabla^2 f(\x^* + p(\x_{K_c}-\x^*))dp \bigg) -\nabla^2 f(\x^*)} &\leq \int_{p=0}^1 \norm{ \nabla^2 f(\x^* + p(\x_{K_c}-\x^*))dp  -\nabla^2 f(\x^*)}dp\\ &\leq \int_{p=0}^1 Mp\norm{ \x_{K_c}-\x^*}dp \\
  & = \frac{M}{2}\norm{\x_{K_c} -\x^*}
\end{align}
in the last step.
Using Hessian Lipschitz condition on $\x_{K_c}$ and $ \x^*$ followed by substituting the bound \eqref{newbound11} we have that:
\begin{align}
    f(\x_{K_c}) & \geq f(\x^*) + \bigg\langle \x_{K_c} -\x^* , \nabla^2 f(\x^* )  (\x_{K_c} -\x^*) \bigg \rangle - \frac{M}{6}\norm{\x_{K_c} -\x^*}^3 \\
    \implies  f(\x_{K_c}) - f(\x^*) & \geq  \frac{\alpha}{2}\norm{\nabla f(\x_{K_c})}^2 - \frac{M}{2}\norm{\x_{K_c} -\x^*}^3 -\frac{M}{6}\norm{\x_{K_c} -\x^*}^3 \\
   \implies  f(\x_{K_c}) - f(\x^*) & \geq  \frac{\beta^2}{2L}\norm{\x_{K_c}-\x^*}^2 - \frac{2M}{3}\norm{\x_{K_c} -\x^*}^3 \label{newbound13}
\end{align}
where in the last step we used \eqref{interimbound4} and the substitution $\alpha = \frac{1}{L}$. Hence for $\norm{\x_{K_c} - \x^*} < \frac{3\beta^2}{4ML}$ we will have $ f(\x_{K_c}) - f(\x^*) > 0$.

\hfill
 \qedhere
 \qedsymbol

\section{Proof of Theorem \ref{thm3}} \label{Appendix CC}
We prove Theorem \ref{thm3} by first upper bounding $K_c$ and $ \hat{K}_{exit} - K_e$.

 \subsection*{Bound on $K_c$}

Using gradient Lipschitz condition on $f(\cdot)$ for $\x_k$ and $\x_{k+1}$ where $ \x_{k+1} = \x_{k} - \frac{1}{L} \nabla f(\x_k)$ followed by Lemma \ref{polyaklem} and inducting from $k=0$ to $k=K_c$ gives:
\begin{align}
    f(\x_{k+1}) - f(\x_{k}) \leq -\frac{1}{2L}\norm{\nabla f(\x_k)}^2 &\leq - \frac{\beta^2}{L^2} (f(\x_{k}) - f(\x^*)) \\
   \implies f(\x_{k+1}) - f(\x^*) & \leq \bigg( 1 - \frac{\beta^2}{L^2}\bigg) \bigg( f(\x_{k}) - f(\x^*) \bigg) \\
   \implies f(\x_{K_c}) - f(\x^*) & \leq \bigg( 1 - \frac{\beta^2}{L^2}\bigg)^{K_c} \bigg( f(\x_{0}) - f(\x^*) \bigg) \\
   \implies K_c &\leq \frac{\log(f(\x_{K_c}) - f(\x^*)) - \log(f(\x_{0}) - f(\x^*)) }{\log\bigg( 1 - \frac{\beta^2}{L^2}\bigg)}. \label{optimalsk2}
\end{align}
By gradient Lipschitz condition for $\x_0$ and $\x^*$, we have the condition:
\begin{align}
    f(\x_{0}) - f(\x^*) \leq \frac{L}{2}\norm{\x_0 - \x^*} = \frac{L}{2}\xi^2 \label{newbound12}
\end{align}
where we used the fact that the iterate $\x_0$ sits on the boundary of the ball $\mathcal{B}_{\xi}(\x^*)$.
Finally substituting the bounds \eqref{newbound12}, \eqref{newbound13} into \eqref{optimalsk2} yields the following contraction rate:
\begin{align}
     K_c &\leq \frac{\log\bigg(\frac{L}{2}\xi^2\bigg) - \log\bigg(\frac{\beta^2}{2L}\norm{\x_{K_c}-\x^*}^2 - \frac{2M}{3}\norm{\x_{K_c} -\x^*}^3\bigg) }{\log\bigg( 1 - \frac{\beta^2}{L^2}\bigg)^{-1}}.
\end{align}
Since $\epsilon \leq \norm{\x_{K_c} -\x^* } < \xi$, we can further upper bound $K_c$ as:
\begin{align}
    K_c &\leq \frac{\log\bigg(\frac{L}{2}\xi^2\bigg) - \log\bigg(\frac{\beta^2}{2L}\epsilon^2 - \frac{2M}{3}\epsilon^3\bigg) }{\log\bigg( 1 - \frac{\beta^2}{L^2}\bigg)^{-1}}. \label{contractionratenew}
\end{align}
Notice that while developing \eqref{contractionratenew} we used Lemma \ref{polyaklem} which requires $K_c < K_e$. For the case when $K_c = K_e$ the trajectory never enters the ball $\mathcal{B}_{\epsilon}(\x^*)$ and Lemma \ref{polyaklem} no longer holds true. However in that case one can repeat the argument from \eqref{switchpolyak} onward in the proof of Lemma \ref{polyaklem} by considering $K_c-1$ instead of $K_c$ and get the same upper bound \eqref{contractionratenew} on $K_c -1$. Therefore combining the two cases we effectively get:
\begin{align}
    K_c &\leq \frac{\log\bigg(\frac{L}{2}\xi^2\bigg) - \log\bigg(\frac{\beta^2}{2L}\epsilon^2 - \frac{2M}{3}\epsilon^3\bigg) }{\log\bigg( 1 - \frac{\beta^2}{L^2}\bigg)^{-1}} + 1.
\end{align}

The bound on $\epsilon$ given by $\epsilon < \frac{3\beta^2}{4ML}$ follows from Lemma \ref{polyaklem} and the fact that $\epsilon \leq \norm{\x_{K_c} -\x^* } $.

  \subsection*{Bound on $\hat{K}_{exit} - K_e$}
  Recall that from \eqref{partb.} in theorem \ref{thm2} we have $\norm{\x^{++}-\x^*} > \bar{\rho}(\x)\norm{\x^+-\x^*} - \sigma(\x)$ whenever $\norm{\x^{+}-\x^*} \geq \norm{\x-\x^*}$. Now for $ K_e \leq k \leq \hat{K}_{exit}$, the sequence $\{\norm{\x_k-\x^*}\}$ is non-decreasing from the definition of $K_e$. Hence, \eqref{partb.} holds for all such $\x_k$ which have $ K_e \leq k \leq \hat{K}_{exit}$. Using \eqref{partb.} with $\x^+ = \x_{k-1}$ and $\x^{++}= \x_{k}$ for $ K_e +1 \leq k \leq \hat{K}_{exit}$ yields:
  \begin{align}
      \norm{\x_{k}-\x^*} &> \bar{\rho}(\x_{k-2})\norm{\x_{k-1}-\x^*} - \sigma(\x_{k-2})  \\
       \norm{\x_{k}-\x^*} &> \bar{\rho}(\x_{k-2})\norm{\x_{k-1}-\x^*} - M \norm{\x_{k-2}-\x^*}^2 \\ \norm{\x_{k}-\x^*} + M \norm{\x_{k-2}-\x^*}^2 &> \bar{\rho}(\x_{k-2})\norm{\x_{k-1}-\x^*}   \\
       \norm{\x_{k}-\x^*} + M \norm{\x_{k}-\x^*}^2 &> \bar{\rho}(\x_{k-2})\norm{\x_{k-1}-\x^*}   \\
   \norm{\x_{k}-\x^*} &> \frac{\bar{\rho}(\x_{k-2})}{1 + M \norm{\x_{k}-\x^*}}\norm{\x_{k-1}-\x^*}  > \frac{\bar{\rho}(\x_{k-2})}{1 + M \xi}\norm{\x_{k-1}-\x^*}  \label{expansionrate1}
  \end{align}
  where we used the bound on $\sigma(\x)$ from \eqref{parta_partb} given by $\sigma(\x_{k-2}) = M \norm{\x_{k-2}-\x^*}^2 \leq M (\xi)^2$ followed by the condition $ \norm{\x_{k}-\x^*}>\norm{\x_{k-2}-\x^*}$ arising from the fact that $\{\norm{\x_{k}-\x^*}\}$ is a monotonically increasing sequence for $  K_e +2 \leq k \leq \hat{K}_{exit}$ and finally the substitution $ \norm{\x_{\hat{K}_{exit}}-\x^*} = \xi$. Now applying the bound \eqref{expansionrate1} recursively for $  K_e +2 \leq k \leq \hat{K}_{exit}$ yields:
  \begin{align}
      \norm{\x_{\hat{K}_{exit}}-\x^*} &>  \prod_{k=K_e+2}^{\hat{K}_{exit}-1} \frac{\bar{\rho}(\x_{k-2})}{1 + M \xi}\norm{\x_{K_e+1}-\x^*} \\
      \norm{\x_{\hat{K}_{exit}}-\x^*} &>   \bigg(\frac{\inf_{K_e +2 \leq k \leq \hat{K}_{exit}}\{\bar{\rho}(\x_{k-2})\}}{1 + M \xi}\bigg)^{\hat{K}_{exit} - K_e-2}\norm{\x_{K_e+1}-\x^*} \\
     \hat{K}_{exit} - K_e -2&< \frac{\log \bigg(\norm{\x_{\hat{K}_{exit}}-\x^*}\bigg) - \log \bigg(\norm{\x_{K_e+1}-\x^*}\bigg)}{\log \bigg(\frac{\inf\{\bar{\rho}(\x_{k-2})\}}{1 + M \xi}\bigg)} <  \frac{\log (\xi) - \log (\epsilon)}{\log \bigg(\frac{\inf\{\bar{\rho}(\x_{k-2})\}}{1 + M \xi}\bigg)} \\
     \hat{K}_{exit} - K_e &<   \frac{\log (\xi) - \log (\epsilon)}{\log \bigg(\frac{\inf\{\bar{\rho}(\x_{k-2})\}}{1 + M \xi}\bigg)} + 2
     \label{expansionrate}
  \end{align}
  where in the last step we used $\norm{\x_{\hat{K}_{exit}}-\x^*} = \xi $, $\norm{\x_{K_e+1}-\x^*} \geq \epsilon $ and the range of infimum is omitted after second step. Note that we require the condition $\bigg(\frac{\inf\{\bar{\rho}(\x_{k-2})\}}{1 + M \xi}\bigg) > 1$, however this is trivially satisfied which can be easily checked from \eqref{rho_bound4} and \eqref{radiusmonotonic1}.

For $\xi  \leq  \frac{1}{ \varsigma  M } \frac{\bigg((1+\frac{\beta}{L})^2+ \frac{1}{4 (1+\frac{\beta}{L})^2}-\frac{5}{4}\bigg)}{6}$ where $\varsigma > 2$, we get the condition:
\begin{align}
   \frac{\bar{\rho}(\x)}{1+M\xi}   &> \frac{1 + \frac{\bigg((1+\frac{\beta}{L})^2+ \frac{1}{4 (1+\frac{\beta}{L})^2}-\frac{5}{4}\bigg)}{12}}{1 + \frac{\bigg((1+\frac{\beta}{L})^2+ \frac{1}{4 (1+\frac{\beta}{L})^2}-\frac{5}{4}\bigg)}{6\varsigma}}>1. \label{expansionfactorexactbound}
\end{align}
  Finally adding \eqref{contractionratenew} and \eqref{expansionrate}, we get the following bound:
  \begin{align}
    K_{shell}  & \leq \frac{\log\bigg(\frac{L}{2}\xi^2\bigg) - \log\bigg(\frac{\beta^2}{2L}\epsilon^2 - \frac{2M}{3}\epsilon^3\bigg) }{\log\bigg( 1 - \frac{\beta^2}{L^2}\bigg)^{-1}} + \frac{\log (\xi) - \log (\epsilon)}{\log \bigg(\frac{\inf\{\bar{\rho}(\x_{k-2})\}}{1 + M \xi}\bigg)} + 2 \label{exittime_hat}
  \end{align}
  where $ K_{shell} =  K_c + \hat{K}_{exit} - K_e$.
\hfill
 \qedhere
 \qedsymbol

\section{Proof of Lemmas \ref{lemma1}-\ref{lemma5}} \label{Appendix D}
Before proving Lemma \ref{lemma1} and \ref{lemma2} we need the relative error bound on zeroth order approximation of the gradient trajectory.
Expanding the expression $  \u_K = \prod_{k=0}^{K-1} \bigg[\A_k + \epsilon \P_k   \bigg] \u_0$ from section \ref{sectionrelativeerror} to zeroth order we get the following bound on tail error:
\begin{align}
    \u_K &= \prod_{k=0}^{K-1} \bigg[\A_k + \epsilon \P_k   \bigg] \u_0\\
	    &= \prod_{k=0}^{K-1} \A_k \u_{0} + \mathcal{O}\bigg(\norm{\A}_2^K(K\epsilon) \frac{\norm{\P}_2}{\norm{\A}_2} \norm{\u_0} \bigg)\\
	\implies \norm{ \u_K - \prod_{k=0}^{K-1} \A_k \u_{0}}   &=   \mathcal{O}\bigg(\norm{\A}_2^K(K\epsilon) \epsilon \bigg)
\end{align}
where the above bound is obtained by following steps similar to \eqref{errormargin1}. Then using this tail error bound along with \eqref{errormargin2} we get the following bound on relative error for zeroth order approximation:
\begin{align}
     \frac{\norm{ \u_K - \prod_{k=0}^{K-1} \A_k \u_{0}}}{\norm{\u_K}} &\leq \frac{1}{\epsilon\bigg( 1 + \frac{\beta}{L} - \frac{\epsilon M}{2L}\bigg)^K \sqrt{\sum_{j \in \mathcal{N}_{US} } ({\theta}^{us}_{j})^2}  - \mathcal{O}\bigg(\norm{\A}_2^{K}(K\epsilon) \epsilon \bigg)}\mathcal{O}\bigg(\bigg(2 + \frac{\epsilon M}{2 L}\bigg)^K(K\epsilon) \epsilon \bigg) \\
    &\leq \frac{1}{ \sqrt{\sum_{j \in \mathcal{N}_{US} } ({\theta}^{us}_{j})^2}  - \mathcal{O}\bigg(\frac{\bigg(2 + \frac{\epsilon M}{2 L}\bigg)^K}{\bigg( 1 + \frac{\beta}{L} - \frac{\epsilon M}{2L}\bigg)^K}(K\epsilon) \bigg)}\mathcal{O}\bigg(\frac{\bigg(2 + \frac{\epsilon M}{2 L}\bigg)^K}{\bigg( 1 + \frac{\beta}{L} - \frac{\epsilon M}{2L}\bigg)^K}(K\epsilon) \bigg) \\
    &\leq \frac{1}{ \sqrt{\sum_{j \in \mathcal{N}_{US} } ({\theta}^{us}_{j})^2}  - \mathcal{O}\bigg(\frac{1}{\sqrt{\epsilon}}\bigg(\log\bigg(\frac{1}{\epsilon} \bigg)\epsilon\bigg)  \bigg)}\mathcal{O}\bigg(\frac{1}{\sqrt{\epsilon}}\bigg(\log\bigg(\frac{1}{\epsilon} \bigg)\epsilon\bigg)  \bigg) \label{relerrornew}
\end{align}
where we have substituted the upper bound on $K_{exit}$ from \eqref{linearexittimebound} into $K$. Hence for $\sqrt{\sum_{j \in \mathcal{N}_{US} } ({\theta}^{us}_{j})^2}  > \mathcal{O}\bigg(\frac{1}{\sqrt{\epsilon}}\bigg(\log\bigg(\frac{1}{\epsilon} \bigg)\epsilon\bigg)  \bigg) $ we have that:
\begin{align}
 \norm{\u_K}\bigg( 1 - \mathcal{O}\bigg(\frac{1}{\sqrt{\epsilon}}\bigg(\log\bigg(\frac{1}{\epsilon} \bigg)\epsilon\bigg)  \bigg) \bigg)  \leq  \norm{ \prod_{k=0}^{K-1} \A_k \u_{0}} \leq \norm{\u_K}\bigg( 1 + \mathcal{O}\bigg(\frac{1}{\sqrt{\epsilon}}\bigg(\log\bigg(\frac{1}{\epsilon} \bigg)\epsilon\bigg)  \bigg) \bigg). \label{zerothorder}
\end{align}
Now $\A_k = \sum\limits_{i \in \mathcal{N}_S} c_i^s(k)\v_i \v_i^T + \sum\limits_{j \in \mathcal{N}_{US}} c_j^{us}(k)\v_j \v_j^T $ where $\v_i$ and $\v_j$ are the eigenvectors corresponding to the stable and unstable subspaces of $\nabla^2 f(\x^*)$ and for $\alpha = \frac{1}{L}$ we have the bounds $ 1 + \frac{\beta}{L} - \frac{\epsilon M}{2L}  \leq c_j^{us}(k) \leq 2 + \frac{\epsilon M}{2L}$ and $  -\frac{\epsilon M}{2L} \leq c_i^{s}(k) \leq 1 - \frac{\beta}{L} + \frac{\epsilon M}{2L}$. Therefore we also get the bound:
\begin{align}
  \inf \norm{ \prod_{k=0}^{K-1} \A_k \u_{0}}  \leq  \norm{ \prod_{k=0}^{K-1} (\mathbf{I}-\alpha \nabla^2 f(\x^*)) \u_{0}} \leq \sup \norm{ \prod_{k=0}^{K-1} \A_k \u_{0}}.
\end{align}
Combining this with \eqref{zerothorder} we get:
\begin{align}
    \norm{\u_K}\bigg( 1 - \mathcal{O}\bigg(\frac{1}{\sqrt{\epsilon}}\bigg(\log\bigg(\frac{1}{\epsilon} \bigg)\epsilon\bigg)  \bigg) \bigg)  \leq  \norm{ \prod_{k=0}^{K-1} (\mathbf{I}-\alpha \nabla^2 f(\x^*)) \u_{0}} \leq \norm{\u_K}\bigg( 1 + \mathcal{O}\bigg(\frac{1}{\sqrt{\epsilon}}\bigg(\log\bigg(\frac{1}{\epsilon} \bigg)\epsilon\bigg)  \bigg) \bigg). \label{zerothordernew}
\end{align}
\subsection*{Proof of Lemma \ref{lemma1}}
For values of $\epsilon$ sufficiently small and $\sqrt{\sum_{j \in \mathcal{N}_{US} } ({\theta}^{us}_{j})^2}  > \mathcal{O}\bigg(\frac{1}{\sqrt{\epsilon}}\bigg(\log\bigg(\frac{1}{\epsilon} \bigg)\epsilon\bigg)  \bigg) $, using \eqref{zerothordernew} we have the following approximation:
\begin{align}
 \frac{\norm{ \prod_{k=0}^{K-1} (\mathbf{I}-\alpha \nabla^2 f(\x^*)) \u_{0}}}{\bigg( 1 + \mathcal{O}\bigg(\frac{1}{\sqrt{\epsilon}}\bigg(\log\bigg(\frac{1}{\epsilon} \bigg)\epsilon\bigg)  \bigg) \bigg)}   \leq   \norm{\u_K} &\leq \frac{\norm{ \prod_{k=0}^{K-1} (\mathbf{I}-\alpha \nabla^2 f(\x^*)) \u_{0}}}{\bigg( 1 - \mathcal{O}\bigg(\frac{1}{\sqrt{\epsilon}}\bigg(\log\bigg(\frac{1}{\epsilon} \bigg)\epsilon\bigg)  \bigg) \bigg)}   \\
\implies  \norm{\u_K} & \approx \norm{ \prod_{k=0}^{K-1} (\mathbf{I}-\alpha \nabla^2 f(\x^*)) \u_{0}} = \norm{  (\mathbf{I}-\alpha \nabla^2 f(\x^*))^K \u_{0}}
\end{align}
where $ \mathcal{O}\bigg(\frac{1}{\sqrt{\epsilon}}\bigg(\log\bigg(\frac{1}{\epsilon} \bigg)\epsilon\bigg)  \bigg) \bigg)$ term is neglected w.r.t. $1$ for sufficiently small $\epsilon$ and $K < K_{exit} \lessapprox \mathcal{O}(\log(\epsilon^{-1}))$.
Now, if $\u_0$ has a  projection value close to $0$ on the unstable subspace of $\nabla^2 f(\x^*)$, then $\norm{\u_{K}}$ first  approximately decreases exponentially such that $\x_{K}$ reaches some  $\x_{critical}$ and from there onward it approximately increases exponentially until saddle region is escaped.
For the case when $\x_{critical} \to \x^*$, we will have $\norm{\x_{critical} - \x^*} \to 0$. The escape time for the $\epsilon$--precision trajectories from this region $\mathcal{B}_{\epsilon^{\prime}}(\x^*)$ where $\epsilon^{\prime} = \norm{\x_{critical} - \x^*}$ will be upper bounded by $K < \mathcal{O}(\log({\epsilon^{\prime}}^{-1})) $ from \eqref{linearexittimebound}. This upper bound goes to infinity when $\epsilon^{\prime} \to 0$ hence $\epsilon$--precision trajectories fail to escape the saddle neighborhood when $\x_{critical} = \x^*$.
It should also be noted that if for some $K$, $\u_K = \mathbf{0}$ or in other words $\x_{critical} = \x^*$, then for all $J > K$ we have $\u_{J} = \textbf{0}$ since $\nabla f(\x_J) = \mathbf{0}$ and the gradient trajectory can never escape the saddle region.

\hfill
 \qedhere
 \qedsymbol

\subsection*{Proof of Lemma \ref{lemma2}}
Let $\{\u_K\}$ be any gradient trajectory with linear exit time that satisfies the condition $\sqrt{\sum_{j \in \mathcal{N}_{US} } ({\theta}^{us}_{j})^2}  > \mathcal{O}\bigg(\frac{1}{\sqrt{\epsilon}}\bigg(\log\bigg(\frac{1}{\epsilon} \bigg)\epsilon\bigg)  \bigg) $. Now if this trajectory curves around $\x^*$ then the vectors $\u_{0}$ and $\u_{K}$ will form an obtuse angle for some finite values of $K$. Therefore in order to prove the first part, it is sufficient to show that:
$$\langle \u_{K}, \u_{0}\rangle \geq 0$$ for any value of $K$ such that $\norm{\u_{K}} < \epsilon$. Now, for sufficiently small $\epsilon$ where $\epsilon$ is upper bounded by Theorem \ref{thm1}, from \eqref{relerrornew} we have $\u_{K}= \prod_{k=0}^{K-1} \A_k \u_0 + \mathcal{O}\bigg(\frac{1}{\sqrt{\epsilon}}\bigg(\log\bigg(\frac{1}{\epsilon} \bigg)\epsilon^2\bigg)  \bigg) \approx  (\mathbf{I} - \alpha \nabla^2 f(\x^*))^K\u_0 $ where we used the fact that $ \prod_{k=0}^{K-1} \A_k = (\mathbf{I} - \alpha \nabla^2 f(\x^*))^K + \mathcal{O}( K\epsilon) $ and dropped the term $ \mathcal{O}\bigg(\frac{1}{\sqrt{\epsilon}}\bigg(\log\bigg(\frac{1}{\epsilon} \bigg)\epsilon^2\bigg)  \bigg)$ for sufficiently small $\epsilon$. Using this approximate $\u_{K}$ we get:
\begin{align}
\langle \u_{K}, \u_{0}\rangle & \approx \u_{0}^{T}(\mathbf{I} - \alpha \nabla^2 f(\x^*))^K\u_0 \geq 0
\end{align}
where the last inequality comes from the fact that $(\mathbf{I} - \alpha \nabla^2 f(\x^*))^K$ will be a positive semi-definite matrix for $\alpha \leq \frac{1}{L}$. Therefore, vectors $\u_{0}$ and $\u_{K}$ will form an acute angle between them for all values of $K$ such that $\norm{\u_{K}} < \epsilon$ and $K \leq K_{exit} = \mathcal{O}(\log(\epsilon^{-1}))$. Hence, the trajectory can never curve around $\x^*$.

The proof for second part follows the same method. Let us take any two points on the gradient trajectory denoted by vectors $\u_{K_1}$ and $\u_{K_2}$ w.r.t. stationary point $\x^*$. Then we have the following inner product:
\begin{align}
    \langle \u_{K_1}, \u_{K_2}\rangle & \approx \langle \u_{0},(\mathbf{I} - \alpha \nabla^2 f(\x^*))^{K_1+K_2}\u_0\rangle \geq 0
\end{align}
for $ K_1 + K_2 \leq \mathcal{O}(\log(\epsilon^{-1}))$.
Now with $ \langle \u_{K_1}, \u_{K_2}\rangle \gtrapprox 0$ for any $K_1,K_2$ where $ K_1 + K_2 \leq \mathcal{O}(\log(\epsilon^{-1}))$ such that $\norm{\u_{K_1}} < \epsilon$ and $\norm{\u_{K_2}} < \epsilon$, the angle between the vectors $\u_{K_1}$ and $\u_{K_2}$ can never approximately exceed $\frac{\pi}{2}$. Hence the entire gradient descent trajectory approximately lies inside some orthant of the ball $\mathcal{B}_{\epsilon}(\x^*)$.

\hfill
 \qedhere
 \qedsymbol

\subsection*{Proof of Lemma \ref{lemma3}}
Let us denote the exit point on the ball $\mathcal{B}_{\epsilon}(\x^*)$ by $\x_{K+1}$ where $\norm{\x_{K}-\x^*} \leq \epsilon$ and $\norm{\x_{K+1}-\x^*} > \epsilon$. Also, $ \norm{\x_{K+1}-\x^*} \leq \norm{\x_{K}-\x^*} + \frac{1}{L} \norm{\nabla f(\x_k)} \leq 2 \norm{\x_{K}-\x^*}$ which implies $ \norm{\x_{K}-\x^*} \geq \frac{\norm{\x_{K+1}-\x^*}}{2} \geq \frac{\epsilon}{2}$. Now applying the Hessian Lipschitz condition around $\x^*$ for $\x_K$, we get the following:
\begin{align}
f(\x_{K}) \leq & f(\x^*) + \langle \nabla f(\x^*), \x_{K}-\x^*\rangle + \frac{1}{2}\langle(\x_{K}-\x^*), \nabla^2 f(\x^*) (\x_{K}-\x^*)\rangle + \frac{M}{6}\norm{\x_{K}-\x^*}^3 \\
\leq & f(\x^*)  + \frac{\langle\x_{K}-\x^*, \nabla f(\x_K) \rangle }{2} +\frac{1}{2}\bigg\langle(\x_{K}-\x^*),\bigg(\nabla^2 f(\x^*)-\nabla^2 f(\x^*)-\mathcal{O}(\epsilon) \bigg) (\x_{K}-\x^*)\bigg\rangle + \frac{M}{6}\norm{\x_{K}-\x^*}^3 \\
  \leq & f(\x^*) + \frac{\langle\x_{K}-\x^*, \nabla f(\x_K) \rangle }{2} + \mathcal{O}(\epsilon^3)
\label{lemmatrivial1}
 \end{align}
where we have used $\nabla f(\x_K) =   \bigg(\nabla^2 f(\x^*)+\mathcal{O}(\epsilon) \bigg)(\x_{K}-\x^*)$ from Lemma \ref{lemma1} and substituted $\norm{\x_{K}-\x^*} \leq \epsilon$ in the last step.

Let us first analyze the term $  \frac{\langle\x_{K}-\x^*, \nabla f(\x_K) \rangle }{2}$. Now, $\norm{\x_K- \x^*} < \norm{\x_{K+1}- \x^*}$ since the gradient descent trajectory is exiting the ball $\mathcal{B}_{\epsilon}(\x^*)$ at iteration $K+1$ and therefore it has expansive dynamics at this iteration\footnote{Exit at iteration $K+1$ implies $\norm{\x_{K}- \x^*} < \norm{\x_{K+1}- \x^*}$.}. Squaring the condition $\norm{\x_K- \x^*} < \norm{\x_{K+1}- \x^*}$ yields:
\begin{align}
    \norm{\x_K- \x^*}^2 &< \norm{\x_{K+1}- \x^*}^2  \\
    \norm{\x_K- \x^*}^2 &< \norm{\x_{K}- \x^*}^2 + \norm{\alpha \nabla f(\x_K)}^2 - 2 \alpha\langle\x_{K}-\x^*, \nabla f(\x_K) \rangle \\
    \langle\x_{K}-\x^*, \nabla f(\x_K) \rangle &<  \frac{\alpha}{2}\norm{ \nabla f(\x_K)}^2. \label{lemmatrivial2}
\end{align}
Next, by the gradient Lipschitz continuity for $\x_{K}$ and $\x_{K+1}$, we have that:
\begin{align}
    f(\x_{K+1}) & \leq f(\x_{K}) + \langle \nabla f(\x_K), \x_{K+1} -\x_K\rangle + \frac{L}{2}\norm{\x_{K+1} -\x_K}^2  \\
        f(\x_{K+1}) & \leq f(\x_{K}) - \alpha\norm{\nabla f(\x_K)}^2  + \frac{L}{2}\norm{\alpha \nabla f(\x_K)}^2  \\
        f(\x_{K+1}) + \frac{1}{2L}\norm{ \nabla f(\x_K)}^2 & \leq f(\x_{K}) \label{lemmatrivial3}
\end{align}
where we substituted $\alpha = \frac{1}{L}$. Combining \eqref{lemmatrivial3} with \eqref{lemmatrivial1} followed by substitution of \eqref{lemmatrivial2} yields:
\begin{align}
    f(\x_{K+1}) + \frac{1}{2L}\norm{ \nabla f(\x_K)}^2 & \leq f(\x_{K})  \leq  f(\x^*) + \frac{\langle\x_{K}-\x^*, \nabla f(\x_K) \rangle }{2} + \mathcal{O}(\epsilon^3) \\
\implies f(\x_{K+1}) + \frac{1}{2L}\norm{ \nabla f(\x_K)}^2 & \leq  f(\x^*) + \frac{\alpha}{4}\norm{ \nabla f(\x_K)}^2 + \mathcal{O}(\epsilon^3) \\
\implies f(\x_{K+1})  & \leq   f(\x^*) - \frac{1}{4L}\norm{ \nabla f(\x_K)}^2 + \mathcal{O}(\epsilon^3).  \label{lemmatrivial4}
\end{align}
Next, using the bound $ \norm{ \nabla f(\x_K)} \geq \beta \norm{\x_K-\x^*}$ from \eqref{interimbound4} in \eqref{lemmatrivial4} and the fact that $ \norm{\x_K-\x^*} \geq \frac{\epsilon}{2}$ we obtain:
\begin{align}
    f(\x_{K+1})  & \leq   f(\x^*) - \frac{\beta^2}{4L}\norm{\x_K-\x^*}^2 + \mathcal{O}(\epsilon^3) \leq   f(\x^*) - \frac{\beta^2}{16L}\epsilon^2 + \mathcal{O}(\epsilon^3) \\
    \implies  f(\x_{K+1})  & <  f(\x^*)
\end{align}
for sufficiently small $\epsilon$.

\hfill
 \qedhere
 \qedsymbol

\subsection*{Proof of Lemma \ref{lemma4}}
Let us take any two points $\x_1, \x_2$ in the closed ball $\mathcal{\bar{B}}_{{\epsilon}}(\x^*)$. Using gradient Lipschitz condition, we get the following inequalities:
\begin{align}
f(\x_1) & \leq f(\x^*) + \langle \nabla f(\x^*),\x_1-\x^*\rangle + \frac{L}{2}\norm{\x_1 - \x^*}^2 \\
& \leq f(\x^*) + \frac{L}{2}\norm{\x_1 - \x^*}^2 \label{lemma5_ineq1}
\end{align}	
and
\begin{align}
    f(\x^*) & \leq f(\x_2) - \langle \nabla f(\x^*),\x_2-\x^*\rangle + \frac{L}{2}\norm{\x_2 - \x^*}^2 \\
    & \leq f(\x_2)  + \frac{L}{2}\norm{\x_2 - \x^*}^2 \label{lemma5_ineq2}
\end{align}
 Now adding \eqref{lemma5_ineq1} and \eqref{lemma5_ineq2} yields:
\begin{align}
   f(\x_1) -  f(\x_2) & \leq \frac{L}{2}\norm{\x_2 - \x^*}^2 + \frac{L}{2}\norm{\x_1 - \x^*}^2 . \label{functiongap}
\end{align}
Next, using the fact that $\norm{\x_2 - \x^*}\leq \epsilon$, $\norm{\x_1 - \x^*}\leq \epsilon$ in \eqref{functiongap}, we get the following upper bound:
\begin{align}
	f(\x_1) -  f(\x_2) & \leq  L\epsilon^2 .\label{improvedfunctiongap}
\end{align}Formally, this upper bound states that the function value gap between any two points in the closed ball $\mathcal{\bar{B}}_{{\epsilon}}(\x^*)$ surface cannot be more than $L \epsilon^2$. Also notice that the result in \eqref{improvedfunctiongap} only depends on the gradient Lipschitz condition and therefore will hold true for any $\epsilon$.
Next, we assume that our gradient trajectory is currently exiting the ball $\mathcal{B}_{{\epsilon}}(\x^*)$ at point $\x_{K}$ s.t. $ \norm{\x_{K-1} -\x^*} \leq \epsilon$ and $ \norm{\x_{K} -\x^*} > \epsilon$. Let us further assume that $\hat{K}$ iterations after the current iteration, the gradient trajectory re-enters the ball $\mathcal{B}_{{\epsilon}}(\x^*)$, i.e., $ \norm{\x_{K+\hat{K}} -\x^*} \leq \epsilon$ and $ \norm{\x_{K+\hat{K}-1} -\x^*} > \epsilon$. Using the update equation $\x_{k+1} = \x_{k} - \alpha \nabla f(\x_{k})$ for $0 \ll \alpha \leq \frac{1}{L}$ together with gradient Lipschitz condition, we get:
\begin{align}
	f(\x_{k+1}) & \leq f(\x_{k}) + \langle \nabla f(\x_{k}), \x_{k+1} - \x_{k}\rangle + \frac{L}{2}\norm{\x_{k+1} - \x_{k}}^2 \\
	\implies 	f(\x_{k+1}) & \leq f(\x_{k}) - \frac{\alpha L}{2}\bigg(\frac{2}{L} -\alpha\bigg)\norm{\nabla f(\x_{k})}^2
\end{align}
Taking the telescopic sum for these inequalities from $k=K$ to $k=K + \hat{K} -1$ gives the following lower bound on $f(\x_{K})-f(\x_{K+\hat{K}})$:
\begin{align}
f(\x_{K+\hat{K}}) & \leq f(\x_{K}) - \frac{\alpha L}{2}\bigg(\frac{2}{L} -\alpha\bigg)\sum_{k=K}^{K+\hat{K}-1}\norm{\nabla f(\x_{k})}^2 \label{telescope1}\\
\frac{\alpha L \beta^2}{2}\bigg(\frac{2}{L} -\alpha\bigg)\hat{K}\epsilon^2  &< \frac{\alpha L}{2}\bigg(\frac{2}{L} -\alpha\bigg)\sum_{k=K}^{K+\hat{K}-1}\norm{\nabla f(\x_{k})}^2  \leq f(\x_{K}) - f(\x_{K+\hat{K}}) \leq f(\x_{K-1}) - f(\x_{K+\hat{K}}) \label{tempcheck11}
\end{align}
where $f(\x_{K}) \leq f(\x_{K-1}) $ from monotonicity of $ \{f(\x_{K})\}$ and we have substituted the lower bound $$\norm{\nabla f(\x_{k})}\geq \beta \norm{\x_{k} - \x^*} \geq \beta\epsilon$$ from \eqref{interimbound4} since $\norm{\x_{k} - \x^*} > \epsilon$ for all $K \leq k \leq K+\hat{K}-1$.
Combining \eqref{tempcheck11} with \eqref{improvedfunctiongap} for $ \x_{K-1}, \x_{K+\hat{K}} \in \mathcal{B}_{\epsilon}(\x^*) $ yields the following condition on $\hat{K}$:
\begin{align}
	\frac{\alpha L \beta^2}{2}&\bigg(\frac{2}{L} -\alpha\bigg)\hat{K}\epsilon^2  <  L \epsilon^2 \\
	\hat{K} & < \frac{2}{\alpha \beta^2 \bigg(\frac{2}{L} -\alpha\bigg)}.
\end{align}
Now, for sake of simplicity we substitute $\alpha = \frac{1}{L}$\footnote{It is to be noted that we can carry out a similar analysis for any other $\alpha$ s.t. $0\ll\alpha \leq \frac{1}{L}$ and still obtain the same inference.}. This yields the following bound on $\hat{K}$:
\begin{align}
\hat{K} & < \frac{2}{\kappa^2}  \label{contradictbound}
\end{align}
where $\kappa = \frac{\beta}{L}$. This inequality claims that if the gradient trajectory re-enters the ball $\mathcal{B}_{{\epsilon}}(\x^*)$, it has to do so in fewer than $\frac{2}{\kappa^2}$ iterations. From here onward we will develop a proof which contradicts this claim.

Let us first define some $\xi > \epsilon$ such that $\xi = 2^{\frac{2}{\kappa^2}} \epsilon(1+b)$ where $\kappa = \frac{\beta}{L}$, $b=\frac{\norm{\x_K -\x^*}}{\epsilon}-1$ is a positive value and $\xi$ is upper bounded from theorem \ref{thm3}. Note that $\x_K$ as defined earlier is the exit point of the gradient trajectory, i.e., $ \norm{\x_{K-1}-\x^*} \leq \epsilon$ and $\norm{\x_K - \x^*}> \epsilon$. Now for any $ \epsilon \ll 2^{-\frac{2}{\kappa^2}}$ we will have $\xi = \mathcal{O}(\epsilon)$. Therefore a gradient trajectory moving outwards from the ball $\mathcal{B}_{{\epsilon}}(\x^*)$ is also bound to move out from the ball $\mathcal{B}_{\xi}(\x^*)$ since we have already proved this in Theorem \ref{thm2} for trajectories with expansive dynamics.

Under these conditions, let $J$ represent the minimum number of iterations required to exit the ball $\mathcal{B}_{\xi}(\x^*)$ for a trajectory which is just exiting $\mathcal{B}_{{\epsilon}}(\x^*)$ and is currently at the point $\x_{K}$ s.t. $ \norm{\x_K-\x^*}>\epsilon$. To this end, we rewrite the update equation of radial vector $\u_{k}$ for any $k \in \{K,K+1,...,K+J-1\}$:
\begin{align}
\u_{k+1} &= \u_{k} + (\x_{k+1} - \x_{k}) = \u_{k} - \alpha \nabla f(\x_{k}) \label{radialupdate}
\end{align}
where we have that $\u_{k} = \x_{k} - \x^*$. From the gradient Lipschitz condition we have the following bound for any $\u_{k}$:
\begin{align}
\norm{\nabla f(\x_{k})} \leq L \norm{\u_{k}}  \label{exactgradienttightbound}
\end{align}
where $\u_k = \x_k - \x^*$. Applying norm to \eqref{radialupdate} followed by triangle inequality and using the upper bound from \eqref{exactgradienttightbound} yields:
\begin{align}
\norm{\u_{k+1}} &= \norm{\u_{k} + (\x_{k+1} - \x_{k})}   \leq \norm{\u_{k}} + \alpha \norm{\nabla f(\x_{k})} \leq 2\norm{\u_{k}}
\end{align}
for $\alpha = \frac{1}{L}$. Applying this bound recursively from $k=K$ to $k=K+J-1$ and substituting $\norm{\u_{K}} = \epsilon(1+b)$, we have:
\begin{align}
\norm{\u_{K+J}} \leq 2^{J} \norm{\u_{K}} = 2^{J} \epsilon(1+b).
\end{align}
Since $J$ is the minimum number of iterations required to exit the $\xi$ radius ball for a trajectory which is just exiting the $\epsilon$ ball, we can set $2^{J} \epsilon(1+b) = \xi$. This yields:
\begin{align}
2^{J} \epsilon(1+b) = \xi &=  2^{\frac{2}{\kappa^2}} \epsilon(1+b) \\
J &= \frac{2}{\kappa^2}.
\end{align}
Now, the $\hat{K}$ we defined as the time to re-enter the ball $\mathcal{B}_{{\epsilon}}(\x^*)$ should be definitely greater than $J$ since any trajectory will certainly take more than $J$ iterations to traverse the shell present in between the concentric $\xi$ and $\epsilon$ radii balls.
\begin{align}
\hat{K} > J &= \frac{2}{\kappa^2}.
\end{align}
However, this inequality contradicts the claim that $\hat{K} < \frac{2}{\kappa^2}$ from \eqref{contradictbound} which completes our proof.

\hfill
 \qedhere
 \qedsymbol

\subsection*{Proof of Lemma \ref{lemma5}}
Recall that from \eqref{functiongap} and \eqref{improvedfunctiongap} in previous lemma, for any $\x_1, \x_2 \in \bar{\mathcal{B}}_{\xi}(\x^*)$ we have that:
\begin{align}
	f(\x_1) -  f(\x_2) & \leq L(\xi)^2. \label{improvedfunctiongaplarge}
\end{align}
Next, let $\hat{K}$ be the minimum number of iterations in which the gradient trajectory re-enters the ball $\mathcal{B}_{\xi}(\x^*)$. Then following the same set of steps as in the previous lemma for obtaining \eqref{telescope1}, we get:
\begin{align}
    f(\x_{K+\hat{K}}) & \leq f(\x_{K}) - \frac{\alpha L}{2}\bigg(\frac{2}{L} -\alpha\bigg)\sum_{k=K}^{K+\hat{K}-1}\norm{\nabla f(\x_{k})}^2 \label{telescope2}\\
\implies \frac{\alpha  L^3}{4}\bigg(\frac{2}{L} -\alpha\bigg)\hat{K}(\xi)^2  &< \frac{\alpha L}{2}\bigg(\frac{2}{L} -\alpha\bigg)\sum_{k=K}^{K+\hat{K}-1}\norm{\nabla f(\x_{k})}^2  \leq f(\x_{K}) - f(\x_{K+\hat{K}}) \leq f(\x_{K-1}) - f(\x_{K+\hat{K}}) \label{minimumreturn}
\end{align}
where we substituted $ \norm{\nabla f(\x_k)} \geq \gamma > \frac{1}{\sqrt{2}}L  \xi$ and $ f(\x_{K}) \leq f(\x_{K-1}) $ from monotonicity of $\{f(\x_k)\}$. Now if the trajectory re-enters the ball $\mathcal{B}_{\xi}(\x^*)$ in $ \hat{K}$ iterations, then $\x_{K-1}, \x_{K+\hat{K}} \in {\mathcal{B}}_{\xi}(\x^*)$ and hence $\x_{K-1}, \x_{K+\hat{K}}$ satisfy \eqref{improvedfunctiongaplarge}. Therefore combining \eqref{minimumreturn} with \eqref{improvedfunctiongaplarge} yields the bound:
\begin{align}
\frac{\alpha  L^3}{4}\bigg(\frac{2}{L} -\alpha\bigg)\hat{K}(\xi)^2  &<  L(\xi)^2 \\
\implies	\hat{K} & < \frac{4}{\alpha L^2 \bigg(\frac{2}{L} -\alpha\bigg)}. \end{align}
 Now for $\alpha = \frac{1}{L}$, we have that $ \hat{K}  <  4$. Therefore the gradient trajectory has to re-enter the ball $\mathcal{B}_{\xi}(\x^*)$ in three or less iterations. We now show that the gradient trajectory cannot return in three or less iterations.

Let the current iterate for the gradient trajectory be $\x^-$ such that $\norm{\x^--\x^*} < \xi$ and $\norm{\x -\x^*} \geq \xi$, i.e., the iterate $\x$ exits the ball $\mathcal{B}_{\xi}(\x^*)$ where $\xi$ is bounded from Theorem \ref{thm3}. Next, from Theorem \ref{thm2}, the iterate $\x^+$ in the sequence $\{\x^-,\x, \x^+\} $ will also have expansive dynamics, i.e., $\norm{\x^+-\x^*} > \norm{\x-\x^*}$. Let $\x^{++}$ denote the next iterate in the sequence $\{\x^-,\x,\x^+\}$. Now, if the following condition:
\begin{align}
    \langle \x^{++} - \x^{+}, \x^+ - \x^* \rangle \geq 0 \label{entryviolation1}
\end{align}
is satisfied, then $\x^{++} \not\in {\mathcal{{B}}}_{\xi}(\x^*)$. To check this, let the condition \eqref{entryviolation1} be given and we have the contradiction $\x^{++} \in {\mathcal{{B}}}_{\xi}(\x^*)$, i.e., $\norm{\x^{++} - \x^*} < \xi$. Then we can write the following inequality:
\begin{align}
  \norm{\x^{++} - \x^*}^2 &< (\xi)^2 \\
  \implies \norm{\x^{++} -\x^{+} + \x^{+}  - \x^*}^2 &< (\xi)^2  \\
  \implies \underbrace{\norm{\x^{++} -\x^{+}}^2}_{>0} + \underbrace{\norm{\x^{+}  - \x^*}^2}_{\geq (\xi)^2} + 2 \underbrace{\langle \x^{++} -\x^{+}, \x^{+}  - \x^* \rangle}_{\geq 0}  &< (\xi)^2
\end{align}
which is not possible (left hand side is greater than right hand side). Hence, $\x^{++} \not\in \bar{\mathcal{{B}}}_{\xi}(\x^*)$.

Now, we are left to prove \eqref{entryviolation1} condition, i.e., $\langle \x^{++} - \x^{+}, \x^+ - \x^* \rangle \geq 0$. Manipulating the left hand side of this condition and using the substitutions $\x^{++} -\x^{+} = - \alpha \nabla f(\x^+) $, $\x^{+} -\x = - \alpha \nabla f(\x) $ and $\nabla f(\x^+) = \nabla f(\x) + \bigg(\int_{p=0}^{1} \nabla^2 f(\x + p(\x^+-\x))dp\bigg) (\x^+ - \x)$, we obtain:
\begin{align}
    \langle \x^{++} - \x^{+}, \x^{+} - \x^* \rangle & = \langle -\alpha \nabla f(\x^+), \x^{+} - \x^* \rangle \\
    & = -\alpha \bigg \langle  \nabla f(\x) + \bigg(\int_{p=0}^{1} \nabla^2 f(\x + p(\x^+-\x))dp\bigg) (\x^+ - \x), \x^{+} - \x^* \bigg \rangle \\
    & = -\alpha \bigg \langle  \bigg(\mathbf{I} - \alpha \int_{p=0}^{1} \nabla^2 f(\x + p(\x^+-\x))dp\bigg) \nabla f(\x), \x^+ - \x^* \bigg \rangle \\
     & =  \bigg \langle  \x^+ - \x^*,\bigg(\mathbf{I} - \alpha \int_{p=0}^{1} \nabla^2 f(\x + p(\x^+-\x))dp\bigg) (-\alpha\nabla f(\x))  \bigg \rangle\\
     & =  \bigg \langle  \x^+ - \x^*,\bigg(\mathbf{I} - \alpha \int_{p=0}^{1} \nabla^2 f(\x + p(\x^+-\x))dp\bigg) (\x^{+}-\x)  \bigg \rangle \label{entryviolation2}
\end{align}
where $ \bigg(\mathbf{I} - \alpha \int_{p=0}^{1} \nabla^2 f(\x + p(\x^+-\x))dp\bigg) $ is a positive semi-definite matrix for $\alpha \leq \frac{1}{L}$.
Next, recall that from \eqref{expansionrate1} in the proof for theorem \ref{thm3} for any tuple $ \{\x^{-},\x,\x^{+}\}$ generated by the gradient descent method where $\x^- \in \mathcal{B}_{\xi}(\x^*)$, we have that:
\begin{align}
    \norm{\x^+ - \x^*} &\geq \bigg(\frac{\bar{\rho}(\x^-)}{1+M\xi}\bigg)\norm{\x - \x^*} >\norm{\x - \x^*}  \label{entryviolation3}
\end{align}
for $\frac{\bar{\rho}(\x^-)}{1+M\xi}>1 $.

Using this fact that $ \norm{\x^+ - \x^*}  >\norm{\x - \x^*}$ followed by the cosine identity of triangles we get:
\begin{align}
   \frac{\langle  \x^+ - \x^*, \x^{+}-\x   \rangle}{\norm{\x^+ - \x^*}\norm{\x^{+}-\x}} &=  \frac{\norm{\x^{+}-\x}^2 + \norm{\x^+ - \x^*}^2 - \norm{\x - \x^*}^2}{2\norm{\x^+ - \x^*}\norm{\x^{+}-\x}}>0 \\
   \implies \langle  \x^+ - \x^*, \x^{+}-\x   \rangle & > 0.
\end{align}

For any vectors $\a$ and $\b$ and any positive semi-definite matrix $\mathbf{A}$, if $\langle \a, \b \rangle \geq 0$ then $\langle \a, \mathbf{A}\b \rangle \geq 0$. Using this property for $\mathbf{A} =\bigg(\mathbf{I} - \alpha \int_{p=0}^{1} \nabla^2 f(\x + p(\x^+-\x))dp\bigg)$, $\b= \x^+ - \x$ and $\a= \x^+ - \x^*$, we get that $\bigg \langle  \x^+ - \x^*,\bigg(\mathbf{I} - \alpha \int_{p=0}^{1} \nabla^2 f(\x + p(\x^+-\x))dp\bigg) (\x^+ - \x)  \bigg \rangle \geq 0$ since $\langle \x^+ - \x, \x^+ - \x^* \rangle \geq 0$. Hence from \eqref{entryviolation2}, we have that:
\begin{align}
\langle \x^{++} - \x^{+}, \x^+ - \x^* \rangle
     & =  \bigg \langle  \x^+ - \x^*,\bigg(\mathbf{I} - \alpha \int_{p=0}^{1} \nabla^2 f(\x + p(\x^+-\x))dp\bigg) (\x^+ - \x)  \bigg \rangle \geq 0
\end{align}
which completes the proof.

\hfill
 \qedhere
 \qedsymbol

\section{}\label{Appendix E}
\subsection*{Proof of Lemma \ref{saddleescapelemma}}

To establish the linear exit time of the proposed algorithm from any strict saddle neighborhood it is sufficient to prove the curvature condition (refer Step \ref{algocurvaturecondition} from Algorithm \ref{algo_1}). Now, for $\norm{\nabla {f}(\x)} \leq \epsilon$ and $\Xi = 0$, we have that:
\begin{align}
\nabla f(\x) = \bigg(\nabla f(\x^*) + \int_{p=0}^{p=1}\nabla^2 f(\x^* + p(\x-\x^*))dp \bigg) (\x-\x^*) . \label{algohessian_1}
\end{align}
With $\epsilon$ very small and upper bounded by Theorem \ref{thm1}, using Lemma 3.3 from \cite{dixit2022exit} we can approximate the Hessian $\nabla^2 f(\x^* + p(\x-\x^*)) = \nabla^2 f(\x^*) + \mathcal{O}(\epsilon)  \approx \nabla^2 f(\x^*)$ for any $\x \in \mathcal{B}_{\epsilon}(\x^*)$. This is a valid approximation since we are no longer solving for rates of convergence and just need to approximately determine the unstable projection value. Therefore, the equation \eqref{algohessian_1} for $\x = \x_k$ is approximated as:
 \begin{align}
 \nabla f(\x_k) = \bigg(\nabla^2 f(\x^*) + \mathcal{O}(\epsilon)\bigg)(\x_k-\x^*)   \approx  \nabla^2 f(\x^*) (\x_k-\x^*)  \label{algohessian_2}
 \end{align}
where $\nabla f(\x^*)$ is zero vector.
With $\y_0 = \x_k$, $\y_1 = \x_{k+1}$ and the approximation \eqref{algohessian_2}, we have the following terms:
\begin{align}
\y_1 &= \x_{k+1} =\x_k - \alpha \nabla f(\x_k)\\
&= \x_k - \alpha \bigg(\nabla^2 f(\x^*) + \mathcal{O}(\epsilon)\bigg)(\x_k - \x^*)  \label{algohessian_3}\\
&\approx \x_k - \alpha \nabla^2 f(\x^*) (\x_k-\x^*) , \label{algohessian_4}
\end{align}
\begin{align}
\nabla f(\y_1) & = \nabla f(\x_{k+1}) = \bigg(\nabla^2 f(\x^*) + \mathcal{O}(\epsilon)\bigg)(\x_{k+1}-\x^*)\\
&= \bigg(\nabla^2 f(\x^*) + \mathcal{O}(\epsilon)\bigg)\bigg(\x_k - \alpha \bigg(\nabla^2 f(\x^*) + \mathcal{O}(\epsilon)\bigg)(\x_k - \x^*)-\x^*\bigg)\\
& \approx  \nabla^2 f(\x^*) \bigg(\x_k - \alpha \nabla^2 f(\x^*) (\x_k-\x^*)-\x^*\bigg). \label{algohessian_5}
\end{align}
Note that in the second last step we used the substitution from \eqref{algohessian_3}.
Now, we define the terms $V_1, V_2$ using $\y_0, \y_1$:
\begin{align}
V_1 = & \langle \y_1 - \y_0, \y_1 - \y_0\rangle \approx (\x_k-\x^*)^T  (\alpha\nabla^2 f(\x^*))^2 (\x_k-\x^*) \\
V_2 = & \alpha \langle \y_1 - \y_0, \nabla {f}(\y_1) - \nabla {f}(\y_0)\rangle \approx (\x_k-\x^*)^T  (\alpha\nabla^2 f(\x^*))^3 (\x_k-\x^*)
\end{align}
Next we use the following substitution:
\begin{align}
     \x_k - \x^* = \norm{\x_k-\x^* }\bigg(  \sum_{i \in \mathcal{N}_{S} }{\theta}^{s}_{i} \v_i(0) + \sum_{j \in \mathcal{N}_{US} }{\theta}^{us}_{j}\v_j(0) \bigg) \label{algoeigensub}
\end{align}
where $ \norm{\x_k-\x^* } \theta^{s}_{i}  = \langle (\x_k-\x^* ), \v_i(0)\rangle$, $ \norm{\x_k-\x^* } \theta^{us}_{j}  = \langle (\x_k-\x^* ), \v_j(0)\rangle$ and $\v_i(0), \v_j(0)$ are the eigenvectors of the scaled Hessian $\alpha \nabla^2 f(\x^*)$. On further simplifying $V_1, V_2$ using \eqref{algoeigensub} we get:
\begin{align}
V_1 \approx & \norm{\x_{k} - \x^*}^2 \bigg(\sum_{i \in \mathcal{N}_{S} }(\lambda_i^{s})^2({\theta}^{s}_{i})^2  + \sum_{j \in \mathcal{N}_{US} }(\lambda_j^{us})^2 ({\theta}^{us}_{j})^2 \bigg) \\
V_2 \approx & \norm{\x_{k} - \x^*}^2 \bigg( \sum_{i \in \mathcal{N}_{S} }(\lambda_i^{s})^3({\theta}^{s}_{i})^2  + \sum_{j \in \mathcal{N}_{US} }(\lambda_j^{us})^3 ({\theta}^{us}_{j})^2 \bigg)
\end{align}
where $\lambda_i^{s}$ and $\lambda_j^{us}$ are the eigenvalues of stable subspace $\mathcal{E}_{S}$ and unstable subspace $\mathcal{E}_{US}$ of the scaled Hessian $ \alpha\nabla^2 f(\x^*)$ respectively. These eigenvalues are bounded by :
\begin{align}
\frac{\beta}{L} & \leq \lambda_i^{s} \leq 1  \\
-1 & \leq \lambda_j^{us} \leq -\frac{\beta}{L}.
\end{align}
Evaluating $V_1 - V_2$ and using the fact that $\norm{\x_{k} - \x^*} \leq \frac{1}{\beta}\norm{\nabla f(\x_{k})} \leq \frac{L\epsilon}{\beta}$ from \eqref{interimbound4}, we get the following expression:
\begin{align}
V_1 - V_2 \lessapprox & \frac{\epsilon^2}{\kappa^2} \bigg(\sum_{i \in \mathcal{N}_{S} }((\lambda_i^{s})^2-(\lambda_i^{s})^3)({\theta}^{s}_{i})^2  + \sum_{j \in \mathcal{N}_{US} }((\lambda_j^{us})^2-(\lambda_j^{us})^3) ({\theta}^{us}_{j})^2 \bigg) \label{conditionbound}
\end{align}
where $\kappa = \frac{\beta}{L}$.
Now, the function $h(y)=y^2 - y^3$ attains a maximum value of $\frac{4}{27}$ in the interval $y \in (0,1]$ and a maximum value of $2$ in the interval $y \in [-1,0)$. Substituting $y = \lambda_i^{s}$ in the interval $y \in (0,1]$ and $y = \lambda_j^{us}$ in the interval  $y \in [-1,0)$, the upper bound for \eqref{conditionbound} becomes:
\begin{align}
V_1 - V_2 \lessapprox & \frac{\epsilon^2}{\kappa^2} \bigg(\sum_{i \in \mathcal{N}_{S} }\frac{4}{27}({\theta}^{s}_{i})^2  + \sum_{j \in \mathcal{N}_{US} }2 ({\theta}^{us}_{j})^2 \bigg)  \\
V_1 - V_2 \lessapprox & \frac{\epsilon^2}{\kappa^2} \bigg(\frac{4}{27}-\frac{4}{27}\bigg(\sum_{j \in \mathcal{N}_{US} }({\theta}^{us}_{j})^2\bigg)  + 2\sum_{j \in \mathcal{N}_{US} } ({\theta}^{us}_{j})^2 \bigg) \\
V_1 - V_2 \lessapprox & \frac{\epsilon^2}{\kappa^2} \bigg(\frac{4}{27}+\frac{50}{27}\bigg(\sum_{j \in \mathcal{N}_{US} }({\theta}^{us}_{j})^2\bigg) \bigg) \\
\sum_{j \in \mathcal{N}_{US} }({\theta}^{us}_{j})^2 \gtrapprox & \frac{\frac{27(V_1-V_2)\kappa^2}{\epsilon^2}  - 4}{50} . \label{conditionboundtight}
\end{align}

The right-hand side in \eqref{conditionboundtight} can be considered as as the lower bound estimate for $ \sum_{j \in \mathcal{N}_{US} }({\theta}^{us}_{j})^2$. Now, the sufficient condition for escaping the saddle neighborhood comes from the minimum unstable subspace  projection value in \eqref{projectionwellconditioned}. Let $ P_{min}(\epsilon)$ be a function of $\epsilon$ equal to the lower bound from \eqref{projectionwellconditioned}, then with the condition $  \frac{\frac{27(V_1-V_2)\kappa^2}{\epsilon^2}  - 4}{50} > P_{min}(\epsilon) $ and \eqref{conditionboundtight}, we can guarantee $ \sum_{j \in \mathcal{N}_{US} }({\theta}^{us}_{j})^2\gtrapprox P_{min}(\epsilon) $  which implies that we have a sufficient unstable  projection value to escape saddle region in almost linear time.

Notice that the curvature condition from the step  \ref{algocurvaturecondition} in Algorithm \ref{algo_1} checks the inequality $  \frac{\frac{27(V_1-V_2)\kappa^2}{\epsilon^2}  - 4}{50} < P_{min}(\epsilon) $ which if true could imply $ \sum_{j \in \mathcal{N}_{US} }({\theta}^{us}_{j})^2< P_{min}(\epsilon) $. Then the gradient trajectory may not necessarily have linear exit time from saddle neighborhood. Hence, we solve the eigenvector problem given by:
\begin{align}
    \x_{k+1} \in \argmin_{\norm{\x -\x_k}= \frac{\norm{\nabla f(\x_k)}}{\beta}  } \bigg( \frac{1}{2}(\x - \x_k)^{T}\textbf{H}(\x - \x_k)  \bigg) \label{constrainedeigen}
\end{align}
 which gives a solution with sufficient unstable projection. Notice that a possible solution to the unconstrained problem:
\begin{align}
    \x_{k+1} \in \argmin_{\x} \bigg( \frac{1}{2}(\x - \x_k)^{T}\textbf{H}(\x - \x_k)  \bigg) \label{unconstrainedeigen}
\end{align}
can be given by $\x_{k+1} -\x_k= b\norm{\x_k-\x^*}\e_j^{us}$ where $\e_j^{us}$ is any eigenvector of the scaled Hessian $\textbf{H} = \alpha\nabla^2 f(\x_k) \approx \alpha\nabla^2 f(\x^*) $ corresponding to its least eigenvalue and $b$ is any scalar. Although any vector in the subspace formed by the eigenvectors corresponding to the minimum eigenvalue can be used instead of $ \e_j^{us}$, for sake of  simplicity of the proof, we use the direction $ \e_j^{us}$. Hence from the unconstrained eigenvector problem \eqref{unconstrainedeigen}, we can write $\x_{k+1} - \x^* = \x_k - \x^* + b\norm{\x_k-\x^*}\e_j^{us}$. Using the substitution $  \x_k - \x^* = \norm{\x_k-\x^* }\bigg(  \sum_{i \in \mathcal{N}_{S} }{\theta}^{s}_{i} \v_i(0) + \sum_{j \in \mathcal{N}_{US} }{\theta}^{us}_{j}\v_j(0) \bigg)$ as before from \eqref{algoeigensub} we get:
\begin{align}
    \x_{k+1} - \x^*  &= \norm{\x_k-\x^* }\bigg(  \sum_{i \in \mathcal{N}_{S} }{\theta}^{s}_{i} \v_i(0) + \sum_{j \in \mathcal{N}_{US} }{\theta}^{us}_{j}\v_j(0) \bigg) + b\norm{\x_k-\x^*}\e_j^{us} \\
    &= \norm{\x_k-\x^* }\bigg(  \sum_{i \in \mathcal{N}_{S} }{\theta}^{s}_{i} \v_i(0) + \sum_{j \in \mathcal{N}_{US} }{\theta}^{us}_{j}\v_j(0) \bigg) + b\norm{\x_k-\x^*}\bigg(\v_l(0) + \mathcal{O}(\epsilon)\bigg) \\
     &= {\norm{\x_k-\x^* }}{\sqrt{1+ b^2}}\bigg(  \sum_{i \in \mathcal{N}_{S} }\frac{{\theta}^{s}_{i}}{\sqrt{1+ b^2}} \v_i(0) + \sum_{j \in \mathcal{N}_{US} }\frac{{\theta}^{us}_{j}}{\sqrt{1+ b^2}}\v_j(0) + \frac{b}{\sqrt{1+ b^2}} \v_l(0) \bigg) +  \mathcal{O}(\epsilon^2)  \label{algohessian6}\\
     &= \norm{\x_k-\x^* }\sqrt{1+ b^2}\bigg(  \sum_{i \in \mathcal{N}_{S} }\tilde{\theta}^{s}_{i} \v_i(0) + \sum_{j \in \mathcal{N}_{US} }\tilde{\theta}^{us}_{j}\v_j(0) \bigg) + \mathcal{O}(\epsilon^2) . \label{algohessian7}
\end{align}
where we have  $\sum_{i \in \mathcal{N}_{S} }(\tilde{\theta}^{s}_{i})^2 + \sum_{j \in \mathcal{N}_{US} }(\tilde{\theta}^{us}_{j})^2 = 1$ for some positive $\tilde{\theta}^{s}_{i}, \tilde{\theta}^{us}_{j} $. Notice that we used the eigenvector perturbation bound $\e_j^{us} = \v_l(0) + \mathcal{O}(\epsilon) $ in the second step and $\v_l(0)$ corresponds to the eigenvector for the smallest eigenvalue of $ \alpha \nabla^2 f(\x^*) $. Notice that $l \in \mathcal{N}_{US}$ where $l$ is the index of $\v_l(0)$ provided $\x_k$ lies within some saddle neighborhood and not in a local minimum neighborhood. If $\x_k$ were in a local minimum neighborhood, then the unstable subspace would have been the null space. Finally, in the second last step we normalized by dividing with $\sqrt{1+b^2}$ because we require the condition:
\begin{align}
    \sum_{i \in \mathcal{N}_{S} }\bigg(\frac{{\theta}^{s}_{i}}{\sqrt{1+ b^2}}\bigg)^2 + \underbrace{\sum_{j \in \mathcal{N}_{US} }\bigg(\frac{{\theta}^{us}_{j}}{\sqrt{1+ b^2}}\bigg)^2+ \bigg(\frac{b}{\sqrt{1+ b^2}}\bigg)^2}_{U_1} &= 1  \label{algoeigenparam}
\end{align}
where we have that $ \sum_{i \in \mathcal{N}_{S} }({\theta}^{s}_{i})^2 + \sum_{j \in \mathcal{N}_{US} }({\theta}^{us}_{j})^2 = 1$.
From \eqref{algohessian6} and \eqref{algohessian7} using coefficient comparison, it can be checked that $ \frac{{\theta}^{s}_{i}}{\sqrt{1+ b^2}} = \tilde{\theta}^{s}_{i}+ \mathcal{O}(\epsilon^2)$ for all $ i \in \mathcal{N}_{S} $. Using this relation in \eqref{algoeigenparam} we get that $U_1 = \sum_{j \in \mathcal{N}_{US} }(\tilde{\theta}^{us}_{j})^2 + \mathcal{O}(\epsilon^2)$. Next, dropping $\mathcal{O}(\epsilon^2)$ term from the right-hand side of  \eqref{algohessian7}, we have:
\begin{align}
  \x_{k+1}-\x^*
   & \approx \norm{\x_k-\x^* }\sqrt{1+ b^2}\bigg(  \sum_{i \in \mathcal{N}_{S} }\tilde{\theta}^{s}_{i} \v_i(0) + \sum_{j \in \mathcal{N}_{US} }\tilde{\theta}^{us}_{j}\v_j(0) \bigg)
\end{align}
where $\sum_{j \in \mathcal{N}_{US} }(\tilde{\theta}^{us}_{j})^2$ can be considered as the new unstable  projection value of $(\x_{k+1}-\x^*)$ and $\norm{ \x_{k+1}-\x^*} \approx  \norm{\x_k-\x^* }\sqrt{1+ b^2}$.
Now, we require that the future gradient trajectory that starts from the point $\x_{k+1}$ escapes the ball $\mathcal{B}_{\tilde{\epsilon}}(\x^*)$ in linear time where $\tilde{\epsilon} = \norm{\x_k-\x^* }\sqrt{1+ b^2}$. Therefore we get that:
\begin{align}
    U_1 \approx \sum_{j \in \mathcal{N}_{US} }(\tilde{\theta}^{us}_{j})^2 &\geq {P}_{min}(\tilde{\epsilon}) \\
  \implies  \sum_{j \in \mathcal{N}_{US} }\bigg(\frac{{\theta}^{us}_{j}}{\sqrt{1+ b^2}}\bigg)^2+ \bigg(\frac{b}{\sqrt{1+ b^2}}\bigg)^2 &\gtrapprox {P}_{min}(\tilde{\epsilon}) \\
    &= P_{min}(\norm{\x_k-\x^* }\sqrt{1+ b^2}) \\
    &> P_{min}\bigg(\norm{\nabla f(\x_k) }\frac{\sqrt{1+ b^2}}{L}\bigg) \label{algohessian8}
\end{align}
where in the last step we used $P_{min}(\norm{\x_k-\x^* }\sqrt{1+ b^2}) > P_{min}\bigg(\norm{\nabla f(\x_k) }\frac{\sqrt{1+ b^2}}{L}\bigg) $ due to the fact that the function $ P_{min}(\epsilon)$ monotonically increases with $\epsilon$ from \eqref{projectionwellconditioned} along with the property that $ \norm{\nabla f(\x_k) } \leq L \norm{\x_k - \x^*}$. Now \eqref{algohessian8} will hold true whenever:
\begin{align}
      \bigg(\frac{b}{\sqrt{1+ b^2}}\bigg)^2  &> P_{min}\bigg(\norm{\nabla f(\x_k) }\frac{\sqrt{1+ b^2}}{L}\bigg) \\
       b &> \frac{\sqrt{P_{min}\bigg(\norm{\nabla f(\x_k) }\frac{\sqrt{1+ b^2}}{L}\bigg)}}{\sqrt{1-P_{min}\bigg(\norm{\nabla f(\x_k) }\frac{\sqrt{1+ b^2}}{L}\bigg)}}.
      \label{algohessian9}
\end{align}
It can be checked that \eqref{algohessian9} will hold true for any positive $b$ as long as it is bounded away from $\epsilon$. Finally in the substitution $\x_{k+1} -\x_k= b\norm{\x_k-\x^*}\e_j^{us}$, we can use the lower bound $ \norm{\nabla f(\x_k) } \geq \beta \norm{\x_k - \x^*}$ from \eqref{interimbound4} and the gradient Lipschitz bound $ \norm{\nabla f(\x_k) } \leq L \norm{\x_k - \x^*}$ to get the range $\frac{\norm{\nabla f(\x_k) }}{L \norm{\x_k - \x^*}} \leq  b \leq \frac{\norm{\nabla f(\x_k) }}{\beta \norm{\x_k - \x^*}} $. Selecting the upper bound of $b$ gives $\x_{k+1} -\x_k= \frac{\norm{\nabla f(\x_k)}}{\beta}\e_j^{us}$ provided $\frac{\beta}{L} \gg 0$. This particular choice of $b$ is less conservative though it should be selected carefully and the selection criterion may vary from one problem to another. For the particular case of well-conditioned saddle neighborhood, a large $b$ and hence a large step-size can be afforded. Notice that  $\frac{\beta}{L}\leq b \leq \frac{L}{\beta}$ and any $b$ in this range will satisfy \eqref{algohessian9} provided $ \frac{\beta}{L} \gg 0$. Since $\x_{k+1}$ is the desired solution, taking norm on both sides of $\x_{k+1} -\x_k= \frac{\norm{\nabla f(\x_k)}}{\beta}\e_j^{us}$ gives the constraint $ \norm{ \x_{k+1} -\x_k} = \frac{\norm{\nabla f(\x_k)}}{\beta}$ in the Step \ref{ballconstraint} of Algorithm \ref{algo_1}.

Since evaluating the eigenvector $\e_j^{us}$ will involve Hessian inversion operations, it will be solved in polynomial time though this step is invoked only once in the saddle neighborhood if required and hence does not add much computational complexity per iteration (only $\mathcal{O}(n^2 \log n)$ complexity per saddle point).

Recall that the entire algorithmic analysis was carried out assuming there is just one eigenvector $\e_j^{us}$ corresponding to the smallest eigenvalue of the Hessian $\nabla^2 f(\x^*)$. However, the same analysis can be done for the case of a subspace corresponding to the smallest eigenvalue. The bounds on $b$ will still be the same however the steps involved are somewhat tedious and lengthy hence purposefully left out from the proof.

 For the case of a local minimum we will have $\sum_{j \in \mathcal{N}_{US} }({\theta}^{us}_{j})^2 = 0$ since there is no unstable subspace. Substituting it in \eqref{conditionboundtight} yields:
\begin{align}
    \frac{4\epsilon^2}{27\kappa^2}\gtrapprox V_1 - V_2. \label{localmin1}
\end{align}
Hence for $ \frac{4\epsilon^2}{27\kappa^2} \lessapprox V_1 - V_2$ we cannot have a local minimum neighborhood. Hence if \eqref{localmin1} holds, then the region can be both a saddle neighborhood or a local minimum region. Therefore, the Step \ref{algocurvaturecondition} in Algorithm \ref{algo_1} also checks if $ \frac{4\epsilon^2}{27\kappa^2} < V_1 - V_2$ so as to rule out the possibility of local minimum. If however we have the inequality $ \frac{4\epsilon^2}{27\kappa^2} > V_1 - V_2$ then a secondary condition $\lambda_{min} ( \textbf{H} ) < 0 $ ascertains it as a saddle neighborhood.
This completes the proof.

\hfill
 \qedhere
 \qedsymbol

\subsection*{Proof of Lemma \ref{functionseqmonotonic} }\label{monotonicdecreasef}
It can be very easily established that $ f(\x_{K+1}) \leq f(\x_K)$ where $\x_{K+1}$ comes from the Step \ref{ballconstraint} in Algorithm \ref{algo_1}.

Since $\x_{K+1}$ is generated from Step \ref{ballconstraint} of Algorithm \ref{algo1} we can use the particular update $ \x_{K+1} - \x_{K} = \frac{\norm{\nabla f(\x_K)}}{\beta}\e_j^{us}  $ (the more general update \ref{ballconstraint} is avoided for sake of simplicity) where $\e_j^{us}$ is an eigenvector of $\nabla^2 f(\x_K)$ belonging to its unstable subspace and $\langle \e_j^{us}, \x_K- \x^* \rangle \lessapprox \mathcal{O}(\frac{\epsilon} {\sqrt{\log(\epsilon^{-1})}}) $ (this approximate bound implies $\x_K-\x^*$ does not have the required unstable subspace  projection value from Theorem \ref{thm1}). As a consequence we will have $\langle \nabla f(\x_K), \x_{K+1}-\x_K \rangle \lessapprox \mathcal{O}(\frac{\epsilon^2} {\sqrt{\log(\epsilon^{-1})}}) $ from the following steps where we use the substitutions $ \nabla f(\x_K) = (\nabla^2 f(\x^*) + \mathcal{O}(\epsilon) )(\x_{K}-\x^*)$ and $ \nabla^2 f(\x_K) = (\nabla^2 f(\x^*) + \mathcal{O}(\epsilon) )$ from matrix perturbation theory.
\begin{align}
   \langle \nabla f(\x_K), \x_{K+1}-\x_K \rangle = & \langle \nabla f(\x_K),  \frac{\norm{\nabla f(\x_K)}}{\beta}\e_j^{us} \rangle \\
   = & \frac{\norm{\nabla f(\x_K)}}{\beta}\langle  \e_j^{us} , (\nabla^2 f(\x^*) + \mathcal{O}(\epsilon) )(\x_{K}-\x^*)\rangle \\
   = & \frac{\norm{\nabla f(\x_K)}}{\beta}\langle  \e_j^{us} , (\nabla^2 f(\x_K) + \mathcal{O}(\epsilon) )(\x_{K}-\x^*)\rangle \\
    = & \frac{\norm{\nabla f(\x_K)}}{\beta}\langle  \lambda_j^{us} \e_j^{us} , (\x_{K}-\x^*)\rangle + \mathcal{O}(\epsilon^3)
    \lessapprox \mathcal{O}(\frac{\epsilon^2} {\sqrt{\log(\epsilon^{-1})}}) \label{seqmontone3}
\end{align}
where $ \nabla^2 f(\x_K)\e_j^{us} = \lambda_j^{us} \e_j^{us} $ and $ \mathcal{O}(\frac{\epsilon^2} {\sqrt{\log(\epsilon^{-1})}}) > \mathcal{O}(\epsilon^3) $.

Finally using Hessian Lipschitz condition for $\x_{K+1}$ about $\x_K$ along with \eqref{seqmontone3} we get:
\begin{align}
    f(\x_{K+1}) \leq & f(\x_K) + \langle \nabla f(\x_K), \x_{K+1}-\x_K\rangle + \underbrace{\frac{1}{2}\langle(\x_{K+1}-\x_K), \nabla^2 f(\x_K)  (\x_{K+1}-\x_K)\rangle}_{<0} + \underbrace{\frac{M}{6}\norm{\x_{K+1}-\x_K}^3}_{\mathcal{O}(\epsilon^3)} \\
  \leq  & f(\x_K)  + \mathcal{O}(\frac{\epsilon^2} {\sqrt{\log(\epsilon^{-1})}})  + \frac{\norm{\nabla f(\x_K)}^2}{2 \beta^2} \underbrace{\langle \e_j^{us},\nabla^2 f(\x_K)\e_j^{us}\rangle }_{<- \beta} + \mathcal{O}(\epsilon^3)\\
  \leq & f(\x_K) + \mathcal{O}(\frac{\epsilon^2} {\sqrt{\log(\epsilon^{-1})}}) - \mathcal{O}(\norm{\nabla f(\x)}^2) + \mathcal{O}(\epsilon^3) \\
  \leq & f(\x_K) + \mathcal{O}(\frac{\epsilon^2} {\sqrt{\log(\epsilon^{-1})}}) - \mathcal{O}(\epsilon^2) + \mathcal{O}(\epsilon^3) = f(\x_K) + \mathcal{O}(\frac{\epsilon^2} {\sqrt{\log(\epsilon^{-1})}}) - \mathcal{O}(\epsilon^2) \label{seqmontone2}
\end{align}
where we used the facts that $\norm{\x_{K+1} - \x_K} =  \mathcal{O}(\epsilon)$, $ \norm{\nabla f(\x_K)} =  \mathcal{O}(\epsilon)$, $ \langle \e_j^{us},\nabla^2 f(\x_K)\e_j^{us}\rangle = \lambda_j^{us} < - \beta $ and $\frac{1}{2}\langle(\x_{K+1}-\x_K), \nabla^2 f(\x_K)  (\x_{K+1}-\x_K)\rangle < 0$ from the Step \ref{ballconstraint} of Algorithm \ref{algo1}. Now for sufficiently small $\epsilon$, the term $ \frac{\epsilon^2} {\sqrt{\log(\epsilon^{-1})}} \to 0$ much faster than $\epsilon^2$ goes to $0$. Hence for sufficiently small $\epsilon$ we will have $ f(\x_{K+1}) \leq f(\x_K)$.
For all other iterations when gradient descent update is used, the sequence $\{f(\x_k)\}$ decreases monotonically.

\hfill
 \qedhere
 \qedsymbol

 \section{Asymptotic convergence}\label{Appendix J}

 \subsection*{Proof of Lemma \ref{lem6} }

 Let $\{\x_k\}$ be the sequence generated by Algorithm \ref{algo_1}. Then by Lemma \ref{saddleescapelemma} this sequence exits the $\epsilon$ neighborhood of any strict saddle point $\x^*$ of a locally analytic Morse function in approximately linear time where $\epsilon$ is bounded from Theorem \ref{thm1}. Further, $\epsilon$ can be chosen in a way such that if the iterate $\x_k$ exits the ball $  \mathcal{B}_{\epsilon}(\x^*)$ at some $k=K$ then the trajectory of $\{\x_k\}$ cannot return to this neighborhood $  \mathcal{B}_{\epsilon}(\x^*)$ for any $k>K$. Such a choice of $\epsilon$ is guaranteed from Lemma \ref{lemma4}. Hence the sequence $\{\x_k\}$ cannot converge to the strict saddle point $\x^*$ which completes the proof of the first part of the lemma.

For the second part notice that if any subsequence $ \{\x_{m_k}\}$ of the sequence $ \{\x_{k}\}$ converges to $\x^*$ then $\x^* \in \{\x_{m_k}\} \hspace{0.1cm} \text{i.o.} $ or equivalently $\x^* \in \{\x_{k}\} \hspace{0.1cm} \text{i.o.} $. Since $\x^*$ is a fixed point of the iteration $ \x_{k+1} = \x_k - \alpha \nabla f(\x_k)$, this would imply that if $ \x_k = \x^* $ for some $k=K$ then $ \x_k = \x^* $ for all $k>K$ or equivalently $\x_k \to \x^*$, a contradiction. Therefore no subsequence $ \{\x_{m_k}\}$ of the sequence $ \{\x_{k}\}$ can converge to the strict saddle point $\x^*$ which completes the proof.

 \hfill
 \qedhere
 \qedsymbol

 \subsection*{Proof of Lemma \ref{lem7} }

 The sequence $\{f(\x_k)\}$ decreases monotonically from Lemma \ref{functionseqmonotonic}. 
Since $f$ is coercive i.e. $ \lim_{\norm{\x}\to \infty} f(\x) = \infty$ and $f$ is continuous (and hence lower semi-continuous), we will have $f(\x) \geq \inf_{\x} f(\x)> - \infty$ i.e. the infimum of the function {values} exists \cite{kinderlehrer2000introduction}. Then by the monotone convergence theorem, $\lim_{k \to \infty} f(\x_k)$ exists and is finite. Since $f$ is coercive and continuous, its sublevel sets given by $\{\x \hspace{0.1cm} \vert \hspace{0.1cm} f(\x)\leq b\}$ for any $b < \infty$ are compact. Since $\lim_{k \to \infty} f(\x_k)$ exists and is finite, by the monotonicity of $\{f(\x_k)\}$ it will belong to the  compact sublevel set $ \{\x \hspace{0.1cm} \vert \hspace{0.1cm} f(\x)\leq f(\x_0)\}$, which completes the proof.

 \hfill
 \qedhere
 \qedsymbol

\subsection*{Proof of Lemma \ref{lem8} }

Let $\x_0$ be the initialization of Algorithm \ref{algo_1}, then by the previous lemma the sequence $\{f(\x_k)\}$ converges over the compact sublevel set $ \{\x \hspace{0.1cm} \vert \hspace{0.1cm} f(\x)\leq f(\x_0)\}$. Combining this fact and the monotonicity of the sequence $\{f(\x_k)\}$ we have that $\x_k \in \{\x \hspace{0.1cm} \vert \hspace{0.1cm} f(\x)\leq f(\x_0)\}$ for all $k$.
Since a Morse function on a compact manifold has finitely many critical points \cite{matsumoto2002introduction}, the compact sublevel set $ \{\x \hspace{0.1cm} \vert \hspace{0.1cm} f(\x)\leq f(\x_0)\}$ can have at most finitely many saddle points.

\hfill
 \qedhere
 \qedsymbol

\subsection*{Proof of Theorem \ref{thmglobmain} }
 In order to prove asymptotic convergence of the sequence $\{\x_k\}$ generated by Algorithm \ref{algo_1} to a critical point we only need to show that the sequence $\{\x_k\}$ satisfies all the conditions from Theorem \ref{thmglob}. First, from Lemma \ref{lem8} all points of the sequence $\{\x_k\}$ are contained in a compact set $ {D} \subset {X}$ where $ D = \{\x \hspace{0.1cm} \vert \hspace{0.1cm} f(\x)\leq f(\x_0)\}$ and $X = \mathbb{R}^n$. Next, the continuous function $Z = f$ satisfies the strict monotonicity property where $\{f(\x_k)\}$ is a strictly decreasing sequence provided $\x_k \not\in S $ and the solution set $S \subset D$ is the set of critical points of $f$ with $f(\x_k) = f(\x_{k+1})$ for $\x_k \in S$.

Finally we are left to show that the mapping $\A$ where $\x_{k+1} = \A \x_k$ is closed outside $S$. It is easy to check that the mapping $\A$ from Algorithm \ref{algo_1} is compact {when $\A := \mathrm{id} - \alpha \nabla f $}. Notice that for the gradient descent update, the map $\A := \mathrm{id} - \alpha \nabla f$ is continuous due to $f \in \mathcal{C}^2$. Since $\x_k \in D = \{\x \hspace{0.1cm} \vert \hspace{0.1cm} f(\x)\leq f(\x_0)\}$ for all $k$, the map $\A := \mathrm{id} - \alpha \nabla f $ takes $D$ to itself, i.e. $ \A : D \mapsto D $ where $D$ is compact and Hausdorff \footnote{A Hausdorff space is a topological space with a separation property: any two distinct points can be separated by disjoint open sets.}. Then by the closed map lemma (Lemma A.52 in \cite{lee2013smooth}), $\A := \mathrm{id} - \alpha \nabla f$ is a closed map in $D$ and hence closed in $D \backslash S$.

From the second-order step in Algorithm \ref{algo_1}, $\x_{k+1} \in \argmin_{\norm{\x -\x_k}= \frac{\norm{\nabla f(\x_k)}}{\beta}  } \bigg( \frac{1}{2}(\x - \x_k)^{T}\nabla^2 f(\x_k)(\x - \x_k)  \bigg) = \A(\x_k)$ and it remains to show that this mapping is continuous. The second-order step can be simplified as $ \x_{k+1} \in \x_k - \frac{\norm{\nabla f(\x_k)}}{\beta} \argmin_{\norm{\x}>0 } \frac{\x^T\nabla^2 f(\x_k)\x}{\norm{\x}^2}$. Since $f$ is Hessian Lipschitz, the eigenvectors of $\nabla^2 f(\x)$ will vary continuously with $\x$; hence $ \argmin_{\norm{\x}>0 } \frac{\x^T\nabla^2 f(\cdot)\x}{\norm{\x}^2}$ is a continuous function and $\norm{\nabla f(\cdot)}$ is a continuous function by continuity of $\nabla f(\cdot)$ and norm. Product of continuous functions is continuous therefore the map $\A$ associated with the second order step is continuous. As before the map $\A$ takes $D$ to itself where $D$ is compact and Hausdorff. Then by the closed map lemma, for the second order step, $\A$ is closed in $D \backslash S$. Since $\{\x_k\} \subset D$, which is compact, there exists a convergent subsequence $ \{\x_{m_k}\}$ of $\{\x_{k}\}$ and from Theorem \ref{thmglob} we have $\lim_{k \to \infty}\x_{m_k} \in S \subset D$ where $S$ is the set of critical points of $f$.

Finally from Lemma \ref{lem6}, since $\{\x_{m_k}\}$ does not converge to any strict saddle point, we have $ \x_{m_k} \to \x^*$, where $\x^*$ is a local minimum. Since $\x^* \in \{\x_{m_k}\}$ i.o. hence $\x^* \in \{\x_{k}\}$ i.o., but $\x^*$ is a fixed point of $ \A := \mathrm{id} - \alpha \nabla f$ (at the fixed point of Algorithm \ref{algo_1} the mapping $\A$ is identically $\mathrm{id} - \alpha \nabla f $). Hence $\x_k = \x^*$ for all $k \geq K$ for some large $K$, implying $\x_{k} \to \x^* $ and this completes the proof.

\hfill
 \qedhere
 \qedsymbol

\section{Convergence Rate to a Local Minimum (Theorem \ref{thm4} and \ref{thm5})} \label{Appendix F}

\subsection*{Proof of Theorem \ref{thm4} }
For any $\x, \y$ in $ \bar{\mathcal{B}}_{R_0}(\x^*_0) $ using \eqref{functionlipschitz} we have the following condition:
\begin{align}
   f(\x) - f(\y) \leq L\textbf{diam}(\mathcal{U}) \norm{\x - \y} \leq 2L\textbf{diam}(\mathcal{U}) R_0. \label{thm4bound1}
\end{align}
Next, let the trajectory re-enter the ball $\mathcal{B}_{R_0}(\x^*_0)$ after $J$ iterations and the current iteration index be $K$ where we have that $\x_K, \x_{K+J}$ belong to $ \bar{\mathcal{B}}_{R_0}(\x^*_0) $ whereas $ \x_{K+J-1} \not\in \bar{\mathcal{B}}_{R_0}(\x^*_0)$. Using gradient Lipschitz continuity on $\x_k$ and $\x_{k+1}$ we get:
\begin{align}
   & f(\x_{k+1}) -f(\x_k) \leq   \langle \nabla f(\x_k) , \x_{k+1}- \x_k\rangle + \frac{L}{2}\norm{\x_{k+1} - \x_k}^2 \\
  & \sum\limits_{k=K}^{K+J-1} \bigg(\langle \nabla f(\x_k) ,  \x_k- \x_{k+1}\rangle - \frac{L}{2}\norm{\x_{k+1} - \x_k}^2\bigg) \leq \sum\limits_{k=K}^{K+J-1} \bigg( f(\x_k)  - f(\x_{k+1}) \bigg) \\
 & \sum\limits_{k=K}^{K+J-1} \bigg(\langle \nabla f(\x_k) ,  \x_k- \x_{k+1}\rangle - \frac{L}{2}\norm{\x_{k+1} - \x_k}^2\bigg) \leq  f(\x_K)  - f(\x_{K+J}) \leq 2L\textbf{diam}(\mathcal{U}) R_0 \label{thm4bound2}
\end{align}
where in the last step we used \eqref{thm4bound1}. Now from Algorithm \ref{algo_1} let $\{k_l\}$ be the subsequence of $\mathcal{I}$ where $ \mathcal{I} = \{K,\dots,K+J-1\} $ for which we have the update $\x_{k+1} = \x_k - \alpha \nabla f(\x_k)$ and $\mathcal{I} \backslash \{k_l\}$ be the subsequence for which we have $\x_{k+1} -\x_k= \frac{\norm{\nabla f(\x_k)}}{\beta}\e_j^{us}$ (this update is a particular case of the Step \ref{ballconstraint} from Algorithm \ref{algo_1})\footnote{The more general update Step \ref{ballconstraint} from Algorithm \ref{algo_1} will also yield the same bound after taking norm but is not used here in the interest of simplifying analysis}. Further let $\{k_{l_j}\}$ be the subsequence of $\{k_l\}$ where $\norm{\nabla f(\x_k)} > \gamma$ and let $r_k = \langle \nabla f(\x_k) ,  \x_k- \x_{k+1}\rangle - \frac{L}{2}\norm{\x_{k+1} - \x_k}^2$. Now the left-hand side of \eqref{thm4bound2} can be written as:
\begin{align}
\sum_{k \in \mathcal{I}} r_k   &= \sum_{k \in \{k_{l_j}\}} r_k +\sum_{k \in \{k_l\} \backslash \{k_{l_j}\}} r_k + \sum_{k \in \mathcal{I} \backslash \{k_l\}} r_k \\
\sum_{k \in \mathcal{I}} r_k & = \sum_{k \in \{k_{l_j}\}} \bigg(\frac{1}{\alpha}\langle \x_k- \x_{k+1}  ,  \x_k- \x_{k+1}\rangle - \frac{L}{2}\norm{\x_{k+1} - \x_k}^2\bigg) +\sum_{k \in \{k_l\} \backslash \{k_{l_j}\}} \frac{1}{2L}\norm{\nabla f(\x_k)}^2  \nonumber \\& +\sum_{k \in \mathcal{I} \backslash \{k_l\}} \bigg( \langle \nabla f(\x_k) ,  \frac{\norm{\nabla f(\x_k)}}{\beta}\e_j^{us}\rangle - \frac{L}{2}\norm{\frac{\norm{\nabla f(\x_k)}}{\beta}\e_j^{us}}^2\bigg) \\
\sum_{k \in \mathcal{I}} r_k & = \sum_{k \in \{k_{l_j}\}}  \frac{1}{2}\norm{\nabla f (\x_k)}\norm{\x_{k+1} -\x_k} +\sum_{k \in \{k_l\} \backslash \{k_{l_j}\}} \frac{1}{2L}\norm{\nabla f(\x_k)}^2  \nonumber \\& + \sum_{k \in \mathcal{I} \backslash \{k_l\}}\bigg( \langle \nabla f(\x_k) ,  \frac{\norm{\nabla f(\x_k)}}{\beta}\e_j^{us}\rangle - \frac{L}{2\beta^2}\norm{\nabla f(\x_k)}^2\bigg) \\
\sum_{k \in \mathcal{I}} r_k & > \frac{\gamma}{2}  \sum_{k \in \{k_{l_j}\}} \norm{\x_{k+1} -\x_k} +\sum_{k \in \{k_l\} \backslash \{k_{l_j}\}} \frac{1}{2L}\norm{\nabla f(\x_k)}^2  - \sum_{k \in \mathcal{I} \backslash \{k_l\}} \bigg( \frac{1}{\beta} + \frac{L}{2 \beta^2}\bigg) \norm{\nabla f(\x_k)}^2.  \label{thm4bound3}
\end{align}
Substituting \eqref{thm4bound3} into \eqref{thm4bound2} yields:
\begin{align}
   & \frac{\gamma}{2}  \sum_{k \in \{k_{l_j}\}} \norm{\x_{k+1} -\x_k} +\sum_{k \in \{k_l\} \backslash \{k_{l_j}\}} \frac{1}{2L}\norm{\nabla f(\x_k)}^2  - \sum_{k \in \mathcal{I} \backslash \{k_l\}} \bigg( \frac{1}{\beta} + \frac{L}{2 \beta^2}\bigg) \norm{\nabla f(\x_k)}^2  \leq 2L\textbf{diam}(\mathcal{U}) R_0 \\
        & \frac{\gamma}{2}  \sum_{k \in \{k_{l_j}\}} \norm{\x_{k+1} -\x_k}  - \sum_{k \in \mathcal{I} \backslash \{k_l\}} \bigg( \frac{1}{\beta} + \frac{L}{2 \beta^2}\bigg) \norm{\nabla f(\x_k)}^2    \leq 2L\textbf{diam}(\mathcal{U}) R_0 \\
        & \frac{\gamma}{2}  \sum_{k \in \{k_{l_j}\}} \norm{\x_{k+1} -\x_k}  - \sum_{k \in \mathcal{I} \backslash \{k_l\}} \bigg( \frac{1}{\beta} + \frac{L}{2 \beta^2}\bigg) L^2\epsilon^2   \leq 2L\textbf{diam}(\mathcal{U}) R_0   \label{thm4bound4}
\end{align}
where in the last step we used the fact that $ \norm{\nabla f(\x_k)} \leq L \epsilon$ for $k \in \mathcal{I} \backslash \{k_l\} $. Also note that for all $k \in \mathcal{I} \backslash \{k_{l}\} $ we will have $ \x_k \in \bigcup_{\substack{\x^*_i \in \mathcal{S}_* \\ \norm{\x^*_i - \x^*_0} > R_0}} \mathcal{B}_{{\epsilon}}(\x^*_i) $. Similarly for all $k \in \mathcal{I} \backslash \{k_{l_j}\} $ we will have $ \x_k, \x_{k+1}$ in the region $  \bigcup_{\substack{\x^*_i \in \mathcal{S}_* \\ \norm{\x^*_i - \x^*_0} > R_0}} \mathcal{B}_{\xi}(\x^*_i) $ along with $ \mathcal{B}_{\xi}(\x^*_r) \cap \mathcal{B}_{\xi}(\x^*_s) = \phi$ for any $ \x^*_r,\x^*_s$ in $\mathcal{S}_*$.

Now adding $ \frac{\gamma}{2}  \sum_{k\in \mathcal{I} \backslash \{k_{l_j}\}} \norm{\x_{k+1} -\x_k}$ to both sides of \eqref{thm4bound4} we get:
\begin{align}
     &\frac{\gamma}{2}  \sum_{k\in \mathcal{I} \backslash \{k_{l_j}\}} \norm{\x_{k+1} -\x_k}+\frac{\gamma}{2}  \sum_{k \in \{k_{l_j}\}} \norm{\x_{k+1} -\x_k}     \leq 2L\textbf{diam}(\mathcal{U}) R_0+ \sum_{k \in \mathcal{I} \backslash \{k_l\}}\bigg( \frac{1}{\beta} + \frac{L}{2 \beta^2}\bigg) L^2\epsilon^2 \nonumber \\ & \hspace{8cm}+ \frac{\gamma}{2}  \sum_{k\in \mathcal{I} \backslash \{k_{l_j}\}} \norm{\x_{k+1} -\x_k} \\
    &\frac{\gamma}{2}  \sum_{k\in \mathcal{I} } \norm{\x_{k+1} -\x_k}     \leq 2L\textbf{diam}(\mathcal{U}) R_0+ \sum_{k \in \mathcal{I} \backslash \{k_l\}}\bigg( \frac{1}{\beta} + \frac{L}{2 \beta^2}\bigg) L^2\epsilon^2 + \frac{\gamma}{2}  \sum_{k\in \mathcal{I} \backslash \{k_{l_j}\}} \norm{\x_{k+1} -\x_k}  \\
     & \frac{\gamma}{2}  \sum_{k\in \mathcal{I} } \norm{\x_{k+1} -\x_k}     \leq 2L\textbf{diam}(\mathcal{U}) R_0+ \sum_{k \in \mathcal{I} \backslash \{k_l\}}\bigg( \frac{1}{\beta} + \frac{L}{2 \beta^2}\bigg) L^2\epsilon^2 + \gamma  \sum_{k\in \mathcal{I} \backslash \{k_{l_j}\}} \xi \label{thm4bound5}
\end{align}
where in the last step we used the fact that $\norm{(\x_{k+1} -\x_k)} \leq 2\xi $ since $ \x_k, \x_{k+1}$ lie inside some ball $\mathcal{B}_{\xi}(\x^*_i) $ for $k \in \mathcal{I} \backslash \{k_{l_j}\} $. If the trajectory $\{\x_k\}$ encounters $N$ such $\mathcal{B}_{\xi}(\x^*_i) $ balls then \eqref{thm4bound5} can be further simplified as:
\begin{align}
  \frac{\gamma}{2}  \sum_{k\in \mathcal{I} } \norm{\x_{k+1} -\x_k}   &   \leq 2L\textbf{diam}(\mathcal{U}) R_0+ N \bigg( \frac{1}{\beta} + \frac{L}{2 \beta^2}\bigg) L^2\epsilon^2   + \gamma  N ({K}_{exit}+K_{shell}) \xi \label{thm4bound6}
\end{align}
where exit time from $\mathcal{B}_{{\epsilon}}(\x^*) $ ball is $ K_{exit}$ from Theorem 3.2 of \cite{dixit2022exit}, exit time from $\mathcal{B}_{\xi}(\x^*) $ ball is $ {K}_{exit}+K_{shell}$ after adding results from Theorem \ref{thm3} and Theorem 3.2 of \cite{dixit2022exit}, and we have that $ \sum_{k \in \mathcal{I} \backslash \{k_l\}} \leq N $, $ \sum_{k\in \mathcal{I} \backslash \{k_{l_j}\}} \leq N ({K}_{exit}+K_{shell})$.

Note that $ \sum_{k\in \mathcal{I} } \norm{\x_{k+1} -\x_k}$ is the total path length of the trajectory inside the shell $ \mathcal{B}_{R_{\omega}}(\x^*_0) \backslash \mathcal{B}_{R_{0}}(\x^*_0)$ where we have that $ R_{\omega} = \max_{k \in \mathcal{I}} \norm{\x_k - \x^*_0}$ and $ R_{0} = \norm{\x_K - \x^*_0} = \norm{\x_{K+J} - \x^*_0}$. Hence, for some $K_{\omega} = \arg\max_{k \in \mathcal{I}} \norm{\x_k - \x^*_0}$ we will have the condition:
\begin{align}
 \sum_{k\in \mathcal{I} } \norm{\x_{k+1} -\x_k} &= \sum_{k=K}^{ K_{\omega}-1} \norm{\x_{k+1} -\x_k} + \sum_{k= K_{\omega}}^{K+J} \norm{\x_{k+1} -\x_k}  \\
& \geq \norm{\sum_{k=K}^{ K_{\omega}-1} \x_{k+1} -\x_k} + \norm{\sum_{k= K_{\omega}}^{K+J} \x_{k+1} -\x_k}  \\
& \geq \norm{ \x_{K_{\omega}} -\x_K} + \norm{ \x_{K+J} - \x_{K_{\omega}}} \\
& \geq \norm{ \x_{K_{\omega}} -\x^*_0} - \norm{ \x_K - \x^*_0}  + \norm{ \x_{K_{\omega}} -\x^*_0} -\norm{ \x_{K+J} -\x^*_0} \\
& = 2( R_{\omega} - R_0). \label{thm4bound7}
\end{align}
Substituting \eqref{thm4bound7} into \eqref{thm4bound6} yields:
\begin{align}
     \gamma( R_{\omega} - R_0) \leq \frac{\gamma}{2}  \sum_{k\in \mathcal{I} } \norm{\x_{k+1} -\x_k}   &   \leq 2L\textbf{diam}(\mathcal{U}) R_0+ N \bigg( \frac{1}{\beta} + \frac{L}{2 \beta^2}\bigg) L^2\epsilon^2  + \gamma  N ({K}_{exit}+K_{shell}) \xi.  \label{thm4bound8}
\end{align}
Next, recall that the distance between any two stationary points is greater than $R$. Hence, between two points $\x, \y$ with $\norm{\x-\y} \leq D$, there can be at most $\frac{D}{R}$ stationary points along the straight line joining $\x, \y$. Now if the points $\x, \y$ are connected by a path formed from the sequence of points $\{\v_k\}_{k=1}^{P}$ then there can be at most $\frac{\sum\limits_{p=1}^{P-1} \norm{\v_{k+1} -\v_k} }{R}$ stationary points on the path connecting $\x, \y$. Using this result in \eqref{thm4bound8} yields the following bound on $N$:
\begin{align}
    & \frac{\gamma}{2}N  \leq \frac{\gamma \sum_{k\in \mathcal{I} } \norm{\x_{k+1} -\x_k}}{2R}   \leq  2L\textbf{diam}(\mathcal{U}) \frac{R_0}{R} + N \bigg( \frac{1}{\beta} + \frac{L}{2 \beta^2}\bigg) \frac{L^2\epsilon^2}{R}  +  \gamma  N ({K}_{exit}+K_{shell}) \frac{\xi}{R} \\
    & N \bigg( \frac{\gamma}{2}-\bigg( \frac{1}{\beta} + \frac{L}{2 \beta^2}\bigg) \frac{L^2\epsilon^2}{R} - \gamma ({K}_{exit}+K_{shell}) \frac{\xi}{R}\bigg) \leq  2L\textbf{diam}(\mathcal{U}) \frac{R_0}{R} \\
    & N \leq \frac{2L\textbf{diam}(\mathcal{U}) \frac{R_0}{R}}{\bigg( \frac{\gamma}{2}-\bigg( \frac{1}{\beta} + \frac{L}{2 \beta^2}\bigg) \frac{L^2\epsilon^2}{R} - \gamma ({K}_{exit}+K_{shell}) \frac{\xi}{R}\bigg)} \label{thm4bound9}
\end{align}
provided $\bigg( \frac{\gamma}{2}-\bigg( \frac{1}{\beta} + \frac{L}{2 \beta^2}\bigg) \frac{L^2\epsilon^2}{R} - \gamma ({K}_{exit}+K_{shell}) \frac{\xi}{R}\bigg) > 0 $ which will hold true for $\xi \ll R$.

Finally, combining \eqref{thm4bound8} and \eqref{thm4bound9} yields the result:
\begin{align}
    R_{\omega} & \leq R_0 +  2L\textbf{diam}(\mathcal{U}) \frac{R_0}{\gamma}+ N_0 K_{exit}\bigg( \frac{1}{\beta} + \frac{L}{2 \beta^2}\bigg) \frac{L^2\epsilon^2}{\gamma}  +  N_0 ({K}_{exit}+K_{shell}) \xi
\end{align}
where $N_0 = \frac{2L\textbf{diam}(\mathcal{U}) \frac{R_0}{R}}{\bigg( \frac{\gamma}{2}-\bigg( \frac{1}{\beta} + \frac{L}{2 \beta^2}\bigg) \frac{L^2\epsilon^2}{R} - \gamma ({K}_{exit}+K_{shell}) \frac{\xi}{R}\bigg)}$ is the upper bound on the number of stationary point neighborhoods encountered by the trajectory of $\{\x_k\}$.

\hfill
 \qedhere
 \qedsymbol

\subsection*{Proof of Theorem \ref{thm5}}
To obtain the total number of iterations in which the sequence $\{\x_k\}$ converges to some $\epsilon$ neighborhood of a local minimum which is within a $\zeta$ neighborhood of $\x_0$, we first obtain the number of iterations the sequence $\{\x_k\}$ spends in the region $\mathcal{U} \backslash \bigcup_{j=1}^{l} \bar{\mathcal{B}}_{\xi}(\x^*_j) $, i.e., the region with $\norm{\nabla f(\x)}> \gamma$. Let $K_1$ be the number of such iterations and $T$ be the number of saddle neighborhoods encountered by the trajectory of $\{\x_k\}$.

In order to obtain $K_1$ we make use of \eqref{thm4bound4} for $R_0 = \zeta$ to get:
\begin{align}
    &\frac{\gamma}{2}  \sum_{k \in \{k_{l_j}\}} \norm{\x_{k+1} -\x_k}  - \sum_{k \in \mathcal{I} \backslash \{k_l\}} \bigg( \frac{1}{\beta} + \frac{L}{2 \beta^2}\bigg) L^2\epsilon^2   \leq 2L\textbf{diam}(\mathcal{U}) \zeta \\
    \implies &\frac{\gamma}{2}  \sum_{k \in \{k_{l_j}\}} \norm{\alpha \nabla f(\x_k)}  -  T\bigg( \frac{1}{\beta} + \frac{L}{2 \beta^2}\bigg) L^2\epsilon^2   \leq 2L\textbf{diam}(\mathcal{U}) \zeta \\
    \implies &  K_1   \leq 4L\textbf{diam}(\mathcal{U}) \frac{\zeta L}{\gamma^2} +  2T\bigg( \frac{1}{\beta} + \frac{L}{2 \beta^2}\bigg) \frac{\epsilon^2  }{\gamma^2} \label{k1time}
\end{align}
where we used the fact that $  \sum_{k \in \{k_{l_j}\}} \norm{ \nabla f(\x_k)} > \gamma K_1$ by definition of the subsequence $\{k_{l_j}\} $ in \eqref{thm4bound4} and $ \sum_{k \in \mathcal{I} \backslash \{k_l\}} = T < N_0 =\frac{2 L\textbf{diam}(\mathcal{U}) \frac{\zeta}{R}}{\bigg( \frac{\gamma}{2}-\bigg( \frac{1}{\beta} + \frac{L}{2 \beta^2}\bigg) \frac{L^2\epsilon^2}{R} - \gamma ({K}_{exit}+K_{shell}) \frac{\xi}{R}\bigg)}$ by Theorem \ref{thm4} for $R_0=\zeta$ where $T$ is the number of saddle neighborhoods encountered by the trajectory of $\{\x_k\}$. Since we have a bound on the number of saddle neighborhoods $T$ and we also know the travel time within each saddle neighborhood we are only left to find the rate within the neighborhood of a local minimum.

\subsection*{Local minimum neighborhood}
When the trajectory $\{\x_k\}$ is within a $\xi$ neighborhood of local minimum $\x^*_{optimal}$ for some $k=K$, we have linear rate of convergence to the neighborhood $\mathcal{B}_{\epsilon}(\x^*_{optimal})$ from the following steps:
\begin{align}
    \x_{k+1} - \x^*_{optimal} & = \bigg(\mathbf{I}- \alpha \int_{p=0}^{p=1}\nabla^2 f(\x^*_{optimal} + p(\x_k-\x^*_{optimal}))dp \bigg)(\x_{k} - \x^*_{optimal}) \\
    \implies \norm{ \x_{k+1} - \x^*_{optimal}} & \leq \underbrace{\norm{\mathbf{I}- \alpha\bigg( \int_{p=0}^{p=1}\nabla^2 f(\x^*_{optimal} + p(\x_k-\x^*_{optimal}))dp \bigg)}_2}_{= 1- \frac{\beta}{L}} \norm{\x_{k} - \x^*_{optimal}} \\
    \implies \norm{ \x_{K+K_{convex}} - \x^*_{optimal}} & \leq \bigg(1- \frac{\beta}{L} \bigg)^{K_{convex}}\norm{\x_{K} - \x^*_{optimal}} \\
    \implies K_{convex} & \leq \frac{\log(\norm{\x_{K} - \x^*_{optimal}})-  \log(\norm{\x_{K+K_{convex}} - \x^*_{optimal}})}{\log\bigg(1- \frac{\beta}{L} \bigg)^{-1}} \leq \frac{\log\bigg(\frac{\xi}{\epsilon} \bigg)}{\log\bigg(1- \frac{\beta}{L} \bigg)^{-1}} \label{convexrate}
\end{align}
where $\x_{K} \in \mathcal{B}_{\xi}(\x^*_{optimal})$ and $ \norm{\x_{K+K_{convex}} - \x^*_{optimal}} = \epsilon$.
Note that in the second step we used the facts that $\alpha = \frac{1}{L}$, $\lambda_{min}(\int(.) ) \geq \int \lambda_{min}(.)$ and $\lambda_{min}\bigg( \nabla^2 f(\x^*_{optimal} + p(\x_k-\x^*_{optimal})) \bigg) = \beta$ for any $\x^*_{optimal} + p(\x_k-\x^*_{optimal})$ in the convex neighborhood $ \mathcal{B}_{\xi}(\x^*_{optimal})$ from \textbf{Assumption A4}.

Finally putting everything together and using Theorem 3.2 from \cite{dixit2022exit}, Theorem \ref{thm3}, travel time from \eqref{k1time} and the convergence rate within a convex neighborhood from \eqref{convexrate}, the total time for the trajectory of $\{\x_k\}$ to converge to an $\epsilon$ neighborhood of $\x^*_{optimal}$ is bounded by:
\begin{align}
    K_{max} & \leq  T\bigg( {K}_{exit}+K_{shell}  \bigg) + K_1 + K_{convex} \\
           & <  T\bigg( {K}_{exit}+K_{shell}  \bigg)+ 4L\textbf{diam}(\mathcal{U}) \frac{\zeta L}{\gamma^2} +  2T\bigg( \frac{1}{\beta} + \frac{L}{2 \beta^2}\bigg) \frac{\epsilon^2  }{\gamma^2} + \frac{\log\bigg(\frac{\xi}{\epsilon} \bigg)}{\log\bigg(1- \frac{\beta}{L} \bigg)^{-1}}
\end{align}
where $ T < \frac{2L\textbf{diam}(\mathcal{U}) \frac{\zeta}{R}}{\bigg( \frac{\gamma}{2}-\bigg( \frac{1}{\beta} + \frac{L}{2 \beta^2}\bigg) \frac{L^2\epsilon^2}{R} - \gamma ({K}_{exit}+K_{shell}) \frac{\xi}{R}\bigg)}$ is the total number of saddle neighborhoods encountered.

We complete the proof of Theorem \ref{thm5} by proving one last claim. Recall that $K_{exit}$ was the exit time of the $\epsilon$--precision trajectory from the ball $\mathcal{B}_{\epsilon}(\x^*)$ while we proved Theorem \ref{thm5} for the exact gradient trajectory. Hence, we need to justify the use of the upper bound on $K_{exit}$ from \eqref{linearexittimebound} in Theorem \ref{thm5}.

Let $K_{exit}^{\omicron}$ be the actual exit time of the gradient trajectory $\{\u_K\}$ from the ball $\mathcal{B}_{\epsilon}(\x^*)$, i.e., $K_{exit}^{\omicron} = \inf_{K>0}\bigg\{K \bigg\vert \norm{\u_K} \geq \epsilon\bigg\} $ where $\u_K = \x_K -\x^*$ is the radial vector and $\norm{\u_0} = \epsilon$. Since $K_{exit}$ is the exit time of the $\epsilon$--precision trajectory $\{\tilde{\u}_K\}$ from the ball $\mathcal{B}_{\epsilon}(\x^*)$, i.e., $K_{exit} = \inf_{K>0}\bigg\{K \bigg\vert \norm{\tilde{\u}_K} \geq \epsilon\bigg\} $, by the definition of exit time we have that $\norm{ \tilde{\u}_{K_{exit}}} \geq \epsilon$.

Now if the initial unstable subspace  projection value $ \sum_{j \in \mathcal{N}_{US} } ({\theta}^{us}_{j})^2$ satisfies the condition of Theorem \ref{thm1} then from the relative error bound \eqref{relativeerrorprojectionbound} we have that:
\begin{align}
     \frac{\norm{\u_K - \tilde{\u}_K}}{\norm{\u_K}} &\leq \mathcal{O}\bigg(\frac{1}{\sqrt{\epsilon}}\bigg(\log\bigg(\frac{1}{\epsilon} \bigg)\epsilon\bigg)^2  \bigg) \\
    \implies 1 -  \mathcal{O}\bigg(\frac{1}{\sqrt{\epsilon}}\bigg(\log\bigg(\frac{1}{\epsilon} \bigg)\epsilon\bigg)^2  \bigg) \leq \frac{ \norm{\tilde{\u}_K}}{\norm{\u_K}}   &\leq 1 +  \mathcal{O}\bigg(\frac{1}{\sqrt{\epsilon}}\bigg(\log\bigg(\frac{1}{\epsilon} \bigg)\epsilon\bigg)^2  \bigg) \\
    \implies \frac{ \norm{\tilde{\u}_K}}{1 +  \mathcal{O}\bigg(\frac{1}{\sqrt{\epsilon}}\bigg(\log\bigg(\frac{1}{\epsilon} \bigg)\epsilon\bigg)^2  \bigg)}     \leq  \norm{\u_{K}}  & \leq \frac{ \norm{\tilde{\u}_K}}{1 -  \mathcal{O}\bigg(\frac{1}{\sqrt{\epsilon}}\bigg(\log\bigg(\frac{1}{\epsilon} \bigg)\epsilon\bigg)^2  \bigg)} \label{boundedvariation1a}\\
     \implies \frac{ \epsilon}{1 +  \mathcal{O}\bigg(\frac{1}{\sqrt{\epsilon}}\bigg(\log\bigg(\frac{1}{\epsilon} \bigg)\epsilon\bigg)^2  \bigg)}     \leq  \norm{\u_{K_{exit}}}  & \leq \frac{ (1+d)\epsilon}{1 -  \mathcal{O}\bigg(\frac{1}{\sqrt{\epsilon}}\bigg(\log\bigg(\frac{1}{\epsilon} \bigg)\epsilon\bigg)^2  \bigg)} \label{boundedvariation1b}
\end{align}
where we substituted $K= K_{exit}$ and used the bound $(1+d)\epsilon \geq \norm{ \tilde{\u}_{K_{exit}}} \geq \epsilon$ for some $d>0$ in the last step.
Next, from the definition of $K_{exit}^{\omicron}$ we have that $\norm{ {\u}_{K_{exit}^{\omicron}}} \geq \epsilon$. Hence, unless we have $ \norm{ {\u}_{K_{exit}}} \geq \epsilon$ (which implies $ K_{exit}^{\omicron} \leq K_{exit}$), the gradient trajectory $\{\u_K\}$ will take not more than $ K_{exit}^{\omicron} - K_{exit}$ iterations to travel the shell $\mathcal{B}_{\epsilon}(\x^*) \backslash \mathcal{B}_{\norm{ {\u}_{K_{exit}}}}(\x^*) $. Next, $ K_{exit}^{\omicron} - K_{exit}$ can be upper bounded by Theorem \ref{thm3} provided the gradient trajectory has expansive dynamics at $K_{exit}$ (from Theorem \ref{thm2}).

Now for sufficiently small $\epsilon$ and $K_{exit} \geq 2$ (the minimal condition that ensures the gradient trajectory at-least enters the ball $ \mathcal{B}_{\epsilon}(\x^*)$), there exists some $K=K^{\upsilon}$ with $K^{\upsilon} < K_{exit} $ such that:
\begin{align}
    \frac{ \norm{\tilde{\u}_{K^{\upsilon}}}}{1 -  \mathcal{O}\bigg(\frac{1}{\sqrt{\epsilon}}\bigg(\log\bigg(\frac{1}{\epsilon} \bigg)\epsilon\bigg)^2  \bigg)} & \leq \frac{ \norm{\tilde{\u}_{K_{exit}}}}{1 +  \mathcal{O}\bigg(\frac{1}{\sqrt{\epsilon}}\bigg(\log\bigg(\frac{1}{\epsilon} \bigg)\epsilon\bigg)^2  \bigg)}. \label{boundedvariation1c}
\end{align}
Combining \eqref{boundedvariation1c} with \eqref{boundedvariation1a} for $ K = K_{exit}$ and $K=K^{\upsilon}$ we get:
\begin{align}
     \norm{{\u}_{K^{\upsilon}}} \leq \frac{ \norm{\tilde{\u}_{K^{\upsilon}}}}{1 -  \mathcal{O}\bigg(\frac{1}{\sqrt{\epsilon}}\bigg(\log\bigg(\frac{1}{\epsilon} \bigg)\epsilon\bigg)^2  \bigg)} & \leq  \frac{ \norm{\tilde{\u}_{K_{exit}}}}{1 +  \mathcal{O}\bigg(\frac{1}{\sqrt{\epsilon}}\bigg(\log\bigg(\frac{1}{\epsilon} \bigg)\epsilon\bigg)^2  \bigg)} \leq  \norm{{\u}_{K_{exit}}} \\
     \implies \norm{{\u}_{K^{\upsilon}}} &\leq \norm{{\u}_{K_{exit}}} .
\end{align}
which implies that the gradient trajectory has expansive dynamics at $K= K_{exit}$ from Theorem \ref{thm2}. Hence, the gradient trajectory will also have expansive dynamics from $K= K_{exit}$ to $K= K_{exit}^{\omicron}$. Using Theorem \ref{thm3} for $\xi = \norm{\u_{ K_{exit}^{\omicron}-1}}$, $\epsilon =\norm{\u_{ K_{exit}}}  $, $\hat{K}_{exit} = K_{exit}^{\omicron}-1  $ and $K_e=K_{exit}   $ we get:
\begin{align}
    K_{exit}^{\omicron}-1 - K_{exit} = \hat{K}_{exit}  - K_e &\leq \frac{\log (\norm{\u_{ K_{exit}^{\omicron}-1}}) - \log (\norm{\u_{ K_{exit}}})}{\log \bigg(\frac{\inf\{\bar{\rho}(\x_{k-2})\}}{1 + M \xi}\bigg)} + 2 \\
    & < \frac{\log (1 +  \mathcal{O}\bigg(\frac{1}{\sqrt{\epsilon}}\bigg(\log\bigg(\frac{1}{\epsilon} \bigg)\epsilon\bigg)^2  \bigg)) }{\log \bigg(\frac{\inf\{\bar{\rho}(\x_{k-2})\}}{1 + M \xi}\bigg)} + 2
    \lessapprox 2 \label{gfunction}
\end{align}
where we used the bound $ \norm{\u_{ K_{exit}^{\omicron}-1}} < \epsilon$ from the definition of $K_{exit}^{\omicron} $, the lower bound on $ \norm{\u_{ K_{exit}}}$ from \eqref{boundedvariation1b} in the second last step and dropped the term $ \log (1 +  \mathcal{O}\bigg(\frac{1}{\sqrt{\epsilon}}\bigg(\log\bigg(\frac{1}{\epsilon} \bigg)\epsilon\bigg)^2  \bigg))$ for sufficiently small $\epsilon$. Hence we have the condition $ K_{exit}^{\omicron}  \lessapprox  K_{exit}+ 3 $ where the constant $3$ can be dropped w.r.t. order $\mathcal{O}(\log(\epsilon^{-1}))$ term after substituting the upper bound on $K_{exit}$ from \eqref{linearexittimebound}. This completes the proof.

\hfill
 \qedhere
 \qedsymbol
\end{appendices}



\begin{thebibliography}{10}
\providecommand{\url}[1]{#1}
\csname url@samestyle\endcsname
\providecommand{\newblock}{\relax}
\providecommand{\bibinfo}[2]{#2}
\providecommand{\BIBentrySTDinterwordspacing}{\spaceskip=0pt\relax}
\providecommand{\BIBentryALTinterwordstretchfactor}{4}
\providecommand{\BIBentryALTinterwordspacing}{\spaceskip=\fontdimen2\font plus
\BIBentryALTinterwordstretchfactor\fontdimen3\font minus
  \fontdimen4\font\relax}
\providecommand{\BIBforeignlanguage}[2]{{%
\expandafter\ifx\csname l@#1\endcsname\relax
\typeout{** WARNING: IEEEtran.bst: No hyphenation pattern has been}%
\typeout{** loaded for the language `#1'. Using the pattern for}%
\typeout{** the default language instead.}%
\else
\language=\csname l@#1\endcsname
\fi
#2}}
\providecommand{\BIBdecl}{\relax}
\BIBdecl

\bibitem{curry1944method}
H.~B. Curry, ``The method of steepest descent for non-linear minimization
  problems,'' \emph{Quarterly of Applied Mathematics}, vol.~2, no.~3, pp.
  258--261, 1944.

\bibitem{hestenes1952methods}
M.~R. Hestenes \emph{et~al.}, ``Methods of conjugate gradients for solving
  linear systems,'' \emph{Journal of research of the National Bureau of
  Standards}, vol.~49, no.~6, pp. 409--436, 1952.

\bibitem{rosenbrock1960automatic}
H.~Rosenbrock, ``An automatic method for finding the greatest or least value of
  a function,'' \emph{The Computer Journal}, vol.~3, no.~3, pp. 175--184, 1960.

\bibitem{karmarkar1984new}
N.~Karmarkar, ``A new polynomial-time algorithm for linear programming,'' in
  \emph{Proceedings of the sixteenth annual ACM symposium on Theory of
  computing}.\hskip 1em plus 0.5em minus 0.4em\relax ACM, 1984, pp. 302--311.

\bibitem{mehrotra1992implementation}
S.~Mehrotra, ``On the implementation of a primal-dual interior point method,''
  \emph{SIAM Journal on optimization}, vol.~2, no.~4, pp. 575--601, 1992.

\bibitem{nesterov1994interior}
Y.~Nesterov and A.~Nemirovskii, \emph{Interior-point polynomial algorithms in
  convex programming}.\hskip 1em plus 0.5em minus 0.4em\relax Siam, 1994,
  vol.~13.

\bibitem{lee2001algorithms}
D.~D. Lee and H.~S. Seung, ``Algorithms for non-negative matrix
  factorization,'' in \emph{Advances in neural information processing systems},
  2001, pp. 556--562.

\bibitem{sanger1989optimal}
T.~D. Sanger, ``Optimal unsupervised learning in a single-layer linear
  feedforward neural network,'' \emph{Neural networks}, vol.~2, no.~6, pp.
  459--473, 1989.

\bibitem{broyden1970convergence}
C.~G. Broyden, ``The convergence of a class of double-rank minimization
  algorithms 1. general considerations,'' \emph{IMA Journal of Applied
  Mathematics}, vol.~6, no.~1, pp. 76--90, 1970.

\bibitem{dixit2022exit}
R.~Dixit, M.~Gurbuzbalaban, and W.~U. Bajwa, ``Exit time analysis for
  approximations of gradient descent trajectories around saddle points,''
  \emph{arXiv preprint arXiv:2006.01106}, 2022.

\bibitem{lee2017first}
J.~D. Lee, I.~Panageas, G.~Piliouras, M.~Simchowitz, M.~I. Jordan, and
  B.~Recht, ``First-order methods almost always avoid saddle points,''
  \emph{arXiv preprint arXiv:1710.07406}, 2017.

\bibitem{kelley1966stable}
A.~Kelley, ``The stable, center-stable, center, center-unstable, unstable
  manifolds,'' \emph{Journal of Differential Equations}, 1966.

\bibitem{polyak1964some}
B.~T. Polyak, ``Some methods of speeding up the convergence of iteration
  methods,'' \emph{USSR Computational Mathematics and Mathematical Physics},
  vol.~4, no.~5, pp. 1--17, 1964.

\bibitem{nesterov2006cubic}
Y.~Nesterov and B.~T. Polyak, ``Cubic regularization of newton method and its
  global performance,'' \emph{Mathematical Programming}, vol. 108, no.~1, pp.
  177--205, 2006.

\bibitem{karimi2016linear}
H.~Karimi, J.~Nutini, and M.~Schmidt, ``Linear convergence of gradient and
  proximal-gradient methods under the polyak-{\l}ojasiewicz condition,'' in
  \emph{Joint European Conference on Machine Learning and Knowledge Discovery
  in Databases}.\hskip 1em plus 0.5em minus 0.4em\relax Springer, 2016, pp.
  795--811.

\bibitem{Lojasiewicz1959probleme}
S.~{\L}ojasiewicz, ``Sur le probl{\`e}me de la division,'' \emph{Studia
  Mathematica}, vol.~18, pp. 87--136, 1959.

\bibitem{attouch2013convergence}
H.~Attouch, J.~Bolte, and B.~F. Svaiter, ``Convergence of descent methods for
  semi-algebraic and tame problems: proximal algorithms, forward--backward
  splitting, and regularized gauss--seidel methods,'' \emph{Mathematical
  Programming}, vol. 137, no. 1-2, pp. 91--129, 2013.

\bibitem{bolte2007Lojasiewicz}
J.~Bolte, A.~Daniilidis, and A.~Lewis, ``The {\l}ojasiewicz inequality for
  nonsmooth subanalytic functions with applications to subgradient dynamical
  systems,'' \emph{SIAM Journal on Optimization}, vol.~17, no.~4, pp.
  1205--1223, 2007.

\bibitem{kifer1981exit}
Y.~Kifer, ``The exit problem for small random perturbations of dynamical
  systems with a hyperbolic fixed point,'' \emph{Israel Journal of
  Mathematics}, vol.~40, no.~1, pp. 74--96, 1981.

\bibitem{hu2017fast}
W.~Hu and C.~J. Li, ``On the fast convergence of random perturbations of the
  gradient flow,'' \emph{arXiv preprint arXiv:1706.00837}, 2017.

\bibitem{ShiSuEtAl.arxiv20}
\BIBentryALTinterwordspacing
B.~Shi, W.~J. Su, and M.~I. Jordan, ``On learning rates and {S}chr\"{o}dinger
  operators,'' \emph{arXiv preprint}, 2020. [Online]. Available:
  \url{https://arxiv.org/abs/2004.06977}
\BIBentrySTDinterwordspacing

\bibitem{du2017gradient}
S.~S. Du, C.~Jin, J.~D. Lee, M.~I. Jordan, A.~Singh, and B.~Poczos, ``Gradient
  descent can take exponential time to escape saddle points,'' in
  \emph{Advances in neural information processing systems}, 2017, pp.
  1067--1077.

\bibitem{jin2017escape}
C.~Jin, R.~Ge, P.~Netrapalli, S.~M. Kakade, and M.~I. Jordan, ``How to escape
  saddle points efficiently,'' in \emph{Proceedings of the 34th International
  Conference on Machine Learning-Volume 70}.\hskip 1em plus 0.5em minus
  0.4em\relax JMLR. org, 2017, pp. 1724--1732.

\bibitem{zhou2017stochastic}
Z.~Zhou, P.~Mertikopoulos, N.~Bambos, S.~Boyd, and P.~W. Glynn, ``Stochastic
  mirror descent in variationally coherent optimization problems,'' in
  \emph{Advances in Neural Information Processing Systems}, 2017, pp.
  7040--7049.

\bibitem{daneshmand2018escaping}
H.~Daneshmand, J.~Kohler, A.~Lucchi, and T.~Hofmann, ``Escaping saddles with
  stochastic gradients,'' \emph{arXiv preprint arXiv:1803.05999}, 2018.

\bibitem{reddi2017generic}
S.~J. Reddi, M.~Zaheer, S.~Sra, B.~Poczos, F.~Bach, R.~Salakhutdinov, and A.~J.
  Smola, ``A generic approach for escaping saddle points,'' \emph{arXiv
  preprint arXiv:1709.01434}, 2017.

\bibitem{jin2017accelerated}
C.~Jin, P.~Netrapalli, and M.~I. Jordan, ``Accelerated gradient descent escapes
  saddle points faster than gradient descent,'' \emph{arXiv preprint
  arXiv:1711.10456}, 2017.

\bibitem{xu2018first}
Y.~Xu, J.~Rong, and T.~Yang, ``First-order stochastic algorithms for escaping
  from saddle points in almost linear time,'' in \emph{Advances in Neural
  Information Processing Systems}, 2018, pp. 5530--5540.

\bibitem{allen2018natasha}
Z.~Allen-Zhu, ``Natasha 2: Faster non-convex optimization than sgd,'' in
  \emph{Advances in Neural Information Processing Systems}, 2018, pp.
  2675--2686.

\bibitem{allen2018neon2}
Z.~Allen-Zhu and Y.~Li, ``Neon2: Finding local minima via first-order
  oracles,'' in \emph{Advances in Neural Information Processing Systems}, 2018,
  pp. 3716--3726.

\bibitem{fang2019sharp}
C.~Fang, Z.~Lin, and T.~Zhang, ``Sharp analysis for nonconvex sgd escaping from
  saddle points,'' \emph{arXiv preprint arXiv:1902.00247}, 2019.

\bibitem{paternain2019newton}
S.~Paternain, A.~Mokhtari, and A.~Ribeiro, ``A newton-based method for
  nonconvex optimization with fast evasion of saddle points,'' \emph{SIAM
  Journal on Optimization}, vol.~29, no.~1, pp. 343--368, 2019.

\bibitem{mokhtari2018escaping}
A.~Mokhtari, A.~Ozdaglar, and A.~Jadbabaie, ``Escaping saddle points in
  constrained optimization,'' in \emph{Advances in Neural Information
  Processing Systems}, 2018, pp. 3629--3639.

\bibitem{anandkumar2016efficient}
A.~Anandkumar and R.~Ge, ``Efficient approaches for escaping higher order
  saddle points in non-convex optimization,'' in \emph{Conference on learning
  theory}, 2016, pp. 81--102.

\bibitem{carmon2018accelerated}
Y.~Carmon, J.~C. Duchi, O.~Hinder, and A.~Sidford, ``Accelerated methods for
  nonconvex optimization,'' \emph{SIAM Journal on Optimization}, vol.~28,
  no.~2, pp. 1751--1772, 2018.

\bibitem{liu2018adaptive}
M.~Liu, Z.~Li, X.~Wang, J.~Yi, and T.~Yang, ``Adaptive negative curvature
  descent with applications in non-convex optimization,'' \emph{Advances in
  Neural Information Processing Systems}, vol.~31, pp. 4853--4862, 2018.

\bibitem{zhang2021escape}
C.~Zhang and T.~Li, ``Escape saddle points by a simple gradient-descent based
  algorithm,'' \emph{Advances in Neural Information Processing Systems},
  vol.~34, 2021.

\bibitem{luenberger1984linear}
D.~G. Luenberger, Y.~Ye \emph{et~al.}, \emph{Linear and nonlinear
  programming}.\hskip 1em plus 0.5em minus 0.4em\relax Springer, 1984, vol.~2.

\bibitem{corless1996lambertw}
R.~M. Corless, G.~H. Gonnet, D.~E. Hare, D.~J. Jeffrey, and D.~E. Knuth, ``On
  the lambertw function,'' \emph{Advances in Computational mathematics},
  vol.~5, no.~1, pp. 329--359, 1996.

\bibitem{ma2020implicit}
C.~Ma, K.~Wang, Y.~Chi, and Y.~Chen, ``Implicit regularization in nonconvex
  statistical estimation: Gradient descent converges linearly for phase
  retrieval, matrix completion, and blind deconvolution.'' \emph{Foundations of
  Computational Mathematics}, vol.~20, no.~3, 2020.

\bibitem{matsumoto2002introduction}
Y.~Matsumoto, \emph{An introduction to Morse theory}.\hskip 1em plus 0.5em
  minus 0.4em\relax American Mathematical Soc., 2002, vol. 208.

\bibitem{MBT}
``Degenerate perturbation theory,''
  \url{http://farside.ph.utexas.edu/teaching/qmech/Quantum/node105.html#e12.89},
  accessed: 2019-08-19.

\bibitem{MBT1}
``Matrix perturbation theory,''
  \url{https://ocw.mit.edu/courses/nuclear-engineering/22-51-quantum-theory-of-radiation-interactions-fall-2012/lecture-notes/MIT22_51F12_Ch11.pdf},
  accessed: 2019-08-19.

\bibitem{lee2016gradient}
J.~D. Lee, M.~Simchowitz, M.~I. Jordan, and B.~Recht, ``Gradient descent
  converges to minimizers,'' \emph{arXiv preprint arXiv:1602.04915}, 2016.

\bibitem{rastrigin1974systems}
L.~A. Rastrigin, ``Systems of extremal control,'' \emph{Nauka}, 1974.

\bibitem{muhlenbein1991parallel}
H.~M{\"u}hlenbein, M.~Schomisch, and J.~Born, ``The parallel genetic algorithm
  as function optimizer,'' \emph{Parallel computing}, vol.~17, no. 6-7, pp.
  619--632, 1991.

\bibitem{hoffmeister1990genetic}
F.~Hoffmeister and T.~B{\"a}ck, ``Genetic algorithms and evolution strategies:
  Similarities and differences,'' in \emph{International Conference on Parallel
  Problem Solving from Nature}.\hskip 1em plus 0.5em minus 0.4em\relax
  Springer, 1990, pp. 455--469.

\bibitem{magnus1985matrix}
J.~R. Magnus and H.~Neudecker, ``Matrix differential calculus with applications
  to simple, hadamard, and kronecker products,'' \emph{Journal of Mathematical
  Psychology}, vol.~29, no.~4, pp. 474--492, 1985.

\bibitem{kinderlehrer2000introduction}
D.~Kinderlehrer and G.~Stampacchia, \emph{An introduction to variational
  inequalities and their applications}.\hskip 1em plus 0.5em minus 0.4em\relax
  SIAM, 2000.

\bibitem{lee2013smooth}
J.~M. Lee, ``Smooth manifolds,'' in \emph{Introduction to Smooth
  Manifolds}.\hskip 1em plus 0.5em minus 0.4em\relax Springer, 2013, pp. 1--31.

\end{thebibliography}
\end{document}